\documentclass[10pt]{amsart}
\usepackage[english]{babel}
\usepackage[a4paper,twoside,marginparwidth=50pt]{geometry}
\usepackage{microtype}
\usepackage{enumerate}
\usepackage{bbm}
\usepackage[utf8]{inputenc}
\usepackage[T1]{fontenc}

\usepackage{xspace}
\patchcmd{\subsection}{-.5em}{.5em}{}{}
\usepackage[hyperfootnotes=false]{hyperref} 
\usepackage{xcolor}
\hypersetup{
	colorlinks = true,
	linkcolor = {blue},
	urlcolor = {red},
	citecolor = {blue}
}

\makeatletter
\newcommand{\mysubsubsection}[1]{\subsubsection*{\bfseries #1}}
\renewcommand{\tocsection}[3]{
  \indentlabel{\@ifnotempty{#2}{\ignorespaces#1 #2\quad}}\bfseries#3}
\renewcommand{\tocsubsection}[3]{
  \indentlabel{\@ifnotempty{#2}{\ignorespaces#1 #2\quad}}#3}

\newcommand\@dotsep{4.5}
\def\@tocline#1#2#3#4#5#6#7{\relax
  \ifnum #1>\c@tocdepth
  \else
    \par \addpenalty\@secpenalty\addvspace{#2}
    \begingroup \hyphenpenalty\@M
    \@ifempty{#4}{
      \@tempdima\csname r@tocindent\number#1\endcsname\relax
    }{
      \@tempdima#4\relax
    }
    \parindent\z@ \leftskip#3\relax \advance\leftskip\@tempdima\relax
    \rightskip\@pnumwidth plus1em \parfillskip-\@pnumwidth
    #5\leavevmode\hskip-\@tempdima{#6}\nobreak
    \leaders\hbox{$\m@th\mkern \@dotsep mu\hbox{.}\mkern \@dotsep mu$}\hfill
    \nobreak
    \hbox to\@pnumwidth{\@tocpagenum{\ifnum#1=1\bfseries\fi#7}}\par
    \nobreak
    \endgroup
  \fi}
\AtBeginDocument{
\expandafter\renewcommand\csname r@tocindent0\endcsname{0pt}
}
\def\l@subsection{\@tocline{2}{0pt}{2.5pc}{5pc}{}}
\makeatother

\usepackage{amsmath}
\usepackage{amsthm}
\usepackage{amssymb}
\usepackage{mathrsfs} 
\usepackage{eucal}
\usepackage{mathtools}

\usepackage{graphicx}
\usepackage{tikz}

\newcounter{results}[section] 

\theoremstyle{plain}
\newtheorem{theorem}[results]{Theorem}
\newtheorem{lemma}[results]{Lemma}
\newtheorem{proposition}[results]{Proposition}
\newtheorem{corollary}[results]{Corollary}

\theoremstyle{remark}
\newtheorem{remark}[results]{Remark}

\theoremstyle{definition}
\newtheorem{definition}[results]{Definition}
\newtheorem{assumption}[results]{Assumption}

\numberwithin{results}{section}
\numberwithin{equation}{section}

\ExplSyntaxOn
\NewDocumentCommand{\makeabbrev}{mmm}
 {
  \yoruk_makeabbrev:nnn { #1 } { #2 } { #3 }
 }

\cs_new_protected:Npn \yoruk_makeabbrev:nnn #1 #2 #3
 {
  \clist_map_inline:nn { #3 }
   {
    \cs_new_protected:cpn { #2 } { #1 { ##1 } }
   }
 }
 \ExplSyntaxOff
 
\makeabbrev{\mathbf}{mbf#1}{a,b,c,d,e,f,g,h,i,j,k,l,m,n,o,p,q,r,s,t,u,v,w,x,y,z,A,B,C,D,E,F,G,H,I,J,K,L,M,N,O,P,Q,R,S,T,U,V,W,X,Y,Z}

\newcommand{\N}{\mathbb{N}}
\newcommand{\PP}{\mathbb{P}}
\newcommand{\Q}{\mathbb{Q}}
\newcommand{\R}{\mathbb{R}}
\newcommand{\T}{\mathbb{T}}
\newcommand{\Z}{\mathbb{Z}}

\renewcommand{\AA}{\mathscr{A}}

\newcommand{\HH}{\mathscr{H}}
\newcommand{\LL}{\mathscr{L}}

\newcommand{\cA}{{\ensuremath{\mathcal A}}}
\newcommand{\cD}{{\ensuremath{\mathcal D}}}
\newcommand{\cF}{{\ensuremath{\mathcal F}}}
\newcommand{\cG}{{\ensuremath{\mathcal G}}}
\newcommand{\cH}{{\ensuremath{\mathcal H}}}
\newcommand{\cP}{{\ensuremath{\mathcal P}}}
\newcommand{\cW}{{\ensuremath{\mathcal W}}}

\newcommand{\sfd}{{\sf d}}

\newcommand{\sfB}{{\sf B}}
\newcommand{\sfE}{{\sf E}}
\newcommand{\sfH}{{\sf H}}
\newcommand{\sfX}{{\sf X}}
\newcommand{\sfY}{{\sf Y}}

\newcommand{\mm}{\mathfrak m}
\newcommand{\fru}{{\mathfrak u}}

\newcommand{\rme}{{\mathrm e}}
\newcommand{\rmC}{{\mathrm C}}
\newcommand{\rmD}{{\mathrm D}}
\newcommand{\rmF}{{\mathrm F}}

\newcommand{\ddelta}{{\mbox{\boldmath$\delta$}}}
\newcommand{\ggamma}{{\mbox{\boldmath$\gamma$}}}

\newcommand{\eps}{\varepsilon}

\newcommand{\de}{\, \mathrm d}
\renewcommand{\d}{{\mathrm d}}

\newcommand{\prob}{\mathcal P}

\newcommand{\bmm}{\overline{\mm}}

\newcommand{\Lip}{\mathop{\rm Lip}\nolimits}
\newcommand{\lip}{\mathop{\rm lip}\nolimits}

\newcommand{\supp}{\mathop{\rm supp}\nolimits}

\newcommand{\boldnabla}{\boldsymbol\nabla}

\newcommand{\emp}{\mathrm{em}}

\newcommand{\pCE}{\mathsf{pC\kern-1pt E}}
\newcommand{\CE}{\mathsf{C\kern-1pt E}}

\newcommand{\trid}{\star}

\newcommand{\set}[1]{\left\{#1\right\}}

\newcommand{\ceil}[1]{\left\lceil #1 \right\rceil}
\newcommand{\norm}[1]{ \left \| #1 \right \|}

\newcommand{\sclprd}[2]{\ensuremath{\left [ #1, #2\right ]}} %

\newenvironment{frcseries}{\fontfamily{frc}\selectfont}{}
\newcommand{\textfrc}[1]{{\frcseries#1}}
\newcommand{\mathfrc}[1]{\text{\textfrc{#1}}}
\newcommand{\cursm}{\textrm{\rm \mathfrc{m}}}

\newcommand{\cercle}[1]
{\tikz[baseline=(X.base)] 
  \node[draw,circle,inner sep=0pt](X){#1};}
  
\title[HJB equations in Wasserstein--Sobolev spaces]{
HJB equations driven by the Dirichlet--Ferguson Laplacian in Wasserstein--Sobolev spaces}
\author{François Delarue}
\address{François Delarue: Université C\^ote d’Azur, CNRS, Laboratoire J. A. Dieudonné, 06108 Nice (France)}
\email{francois.delarue@univ-cotedazur.fr}

\author{Mattia Martini}
\address{Mattia Martini: École polytechnique, Centre de Mathématiques Appliquées, 91128 Palaiseau (France)}
\email{mattia.martini@polytechnique.edu}

\author{Giacomo Enrico Sodini}
\address{Giacomo Enrico Sodini: Institut für Mathematik - Fakultät für Mathematik - Universität Wien, Oskar-Morgenstern-Platz 1, 1090 Wien (Austria)}
\email{giacomo.sodini@univie.ac.at}

\subjclass{Primary: 35R15, 60H15, 49L25, 49N80; Secondary: 60J60, 35K90, 46E35.}
 \keywords{Dirichlet–Ferguson measure; Wasserstein diffusion; partial differential equations on spaces of probability measures; mean field control; Hamilton–Jacobi equation; Kolmogorov equation; Wasserstein--Sobolev space; interacting particle system.}

\begin{document} 

\begin{abstract}
We study linear and nonlinear PDEs defined on the space of probability measures $\prob(\T^d)$ over the flat torus $\T^d$, equipped with the Dirichlet--Ferguson measure $\cD$. We first develop an analytic framework based on the Wasserstein--Sobolev space $H^{1,2}(\prob(\T^d), W_2, \cD)$ associated with the Dirichlet form induced by the infinite-dimensional Laplacian acting on functions of measures. Within this setting, we establish existence and uniqueness results for transport–diffusion and Hamilton--Jacobi equations in the Wasserstein space. Our analysis connects the PDE approach with a corresponding massive interacting particle system, providing a probabilistic (Kolmogorov-type) representation of strong solutions. Finally, we extend the theory to semilinear equations and mean-field optimal control problems, together with consistent finite-dimensional approximations.
\end{abstract}

\maketitle
\tableofcontents
\thispagestyle{empty}

\section{Introduction}

This work lies at the intersection of two lines of research that have each witnessed significant progress in the analysis and understanding of mean-field interaction models—or, more generally, of dynamics describing the statistical evolution of an infinite population: on the one hand, the theory of mean-field control; on the other, the notion of Wasserstein diffusions. Both are briefly reviewed in the following lines.
\vskip 4pt

\paragraph{\it \bf Generator induced by mean-field dynamics} In its microscopic version, a mean-field interaction model is commonly understood as a particle system in which the particles interact with one another in a weak and symmetric manner. Mathematically, the global state is thus summarized by the empirical measure of the particles. Whenever the conditions imposed on the system—whether initial conditions or subsequent perturbations, all possibly random—exhibit some form of exchangeability, that is, are statistically invariant under permutation, the empirical measure is expected to converge as the number of particles tends to infinity. When the initial conditions and perturbations are not only exchangeable but also independent, the limiting statistical distribution is deterministic; this corresponds to a law of large numbers. Otherwise, the limit may remain random and retains, according to de Finetti’s theorem, the trace of the randomness common to all the conditions imposed on the system. In both cases, the limiting distribution is referred to as the mean-field limit. Mathematically, the identification and characterization of this mean-field limit, as well as the justification of the limiting procedure itself, have motivated numerous studies, dating back to the seminal works of \cite{Kac56,mckean1966class,McKean67}.
In the regime where the law of large numbers applies and in the case where the particles are dynamic—that is, they evolve over time, for instance under the influence of transport and diffusion—, the mean-field limit is typically described as the solution of a so-called McKean--Vlasov equation, that is, a stochastic differential equation whose coefficients depend on the law of its solution (see \cite{sznitman1991topics}); or, equivalently, as a nonlinear Fokker–Planck equation (see, among many references, the book \cite{BogachevKrylovRocknerShaposhnikov}). When the sources of randomness acting on the system are driven by a common noise, the McKean--Vlasov equation is called conditional, and the associated Fokker–Planck equation becomes stochastic (see \cite{carmonadelarue2}). In both cases, the dynamics of the mean-field limit evolve in the space ${\mathcal P}$ of probability measures over the underlying state space (for reasons that will be explained below, we will work in fact on the space ${\mathcal P}(\mathbb{T}^d)$ of probability measures on the $d$-dimensional flat torus $\mathbb{T}^d := \mathbb{R}^d / \mathbb{Z}^d$). Provided that those dynamics are unique, the mean-field limit induces a semigroup, whose generator is a differential operator acting on functions defined over ${\mathcal P}$. This operator has recently been the subject of intense research. In particular, the associated \emph{harmonic functions}, solving PDEs on ${\mathcal P}$ or time-space PDEs on 
${\mathbb R}_+ \times {\mathcal P}$, provide valuable insight into the convergence of the finite-dimensional system to the infinite-dimensional one (see \cite{DelarueTse,mischler2013kac,mischler2015new}).
When the mean-field limit satisfies a deterministic Fokker--Planck equation, these PDEs can be interpreted as  transport equations on ${\mathcal P}$. When the Fokker--Planck equation is stochastic, they also include a diffusion term, which is, in most models, highly degenerate (see, for instance, \cite{cardaliaguetdelaruelasrylions}). One of the aims of the present paper is precisely to study a version of these linear PDEs in presence of a ``non-degenerate’’ diffusion on ${\mathcal P}$, called a Wasserstein diffusion, which significantly enriches the analysis of harmonic functions.
\vskip 4pt

\paragraph{\it \bf Beyond the linear case: mean-field control} This is another goal of the present work: to go beyond the linear case. Indeed, the original (finite-particle) system may be subject to mechanisms other than transport or diffusion; see \cite{ChaintronDiez1,ChaintronDiez2} for an overview. In the situation of interest here, transport and diffusion are typically combined with an optimization step: when the system is finite, the particles are dynamically controlled, with the objective of minimizing a cost functional associated with the entire system. This amounts to solving a control problem for a large interacting particle system with mean-field interactions. The limiting problem (i.e., after passage to the mean-field limit) is referred to as the \emph{mean-field control problem}. 
Since the optimization step is intrinsically nonlinear, the passage from the finite-dimensional to the infinite-dimensional model relies on arguments that are considerably more subtle than the law of large numbers. This is why the theory of mean-field control has attracted significant attention in recent years, initiated in particular by Lions in his lectures on mean-field game theory---a closely related framework with many connections. 
In this context, the role played by harmonic functions in the analysis of the mean-field limit under pure transport-diffusion dynamics is taken over, in the controlled case, by the \emph{value function} of the control problem, which is expected to satisfy, at least in the viscosity sense, a Hamilton--Jacobi equation posed on ${\mathcal P}$. This Hamilton--Jacobi equation is of first order in the absence of common noise in the particle system, and of second order---typically highly degenerate in known examples---in the presence of common noise; see 
\cite{cardaliaguetdelaruelasrylions,ConfortiKraaijTonon,CossoGozziKharroubiPham, DaudinJacksonSeeger,GangboMayorgaSwiech}. More generally, it can be understood as a semilinear PDE (in the gradient of the solution) on ${\mathcal P}$. 
As shown in \cite{cardaliaguetdelaruelasrylions,CardaliaguetDaudinJacksonSouganidis,cecchin2025quantitativeconvergencemeanfield,DaudinDelarueJackson}, the regularity of the value function in the mean-field control problem (and thus of the solution to the Hamilton--Jacobi equation) determines the convergence rate from the finite-dimensional control problem to its infinite-dimensional limit. The regularity of solutions to semilinear PDEs on ${\mathcal P}$ is therefore a key ingredient in the analysis. 
Once again, one of the objectives of this paper is precisely to study the regularity of solutions to semilinear equations driven by a \emph{nondegenerate} diffusion operator---referred to below as a \emph{Wasserstein diffusion}.
\vskip 4pt

\paragraph{\it \bf Wasserstein diffusions.}
The introduction of a diffusion acting on the space of probability measures is a central ingredient of our work. 
While the construction of measure-valued processes is certainly not new in probability theory (see \cite{Dawson}), 
the key point here is to employ a diffusion with a clear impact on the resolution of the linear and semilinear PDEs discussed in the previous paragraphs. 
Ideally, the  diffusion should enjoy sufficient smoothing properties to overcome potential singularities in the coefficients of these PDEs. 
Intuitively, it should resemble the analogue of a Brownian motion on ${\mathcal P}$, but this picture raises multiple difficulties, due in particular to the absence of a canonical reference measure on ${\mathcal P}$, similar to the Lebesgue measure in Euclidean spaces. 
As a consequence, none of the existing constructions—an overview of which is given below—can be regarded as canonical. 
Nevertheless, the underlying idea is to define a diffusion process whose local variations are governed by the $2$-Wasserstein distance $W_2$ on ${\mathcal P}$ 
(at least when the underlying state space is ${\mathbb R}^d$ or ${\mathbb T}^d$). 
From the viewpoint of stochastic process theory, any associated harmonic function $\varphi$ should evolve, along the trajectories of the diffusion, 
as a  martingale; the local variance of the latter should be given by the square of the norm of the “intrinsic gradient” of $\varphi$.  
To avoid any ambiguity, we emphasize that the 
intrinsic derivative mentioned here is precisely the one that appears in the formulation of the aforementioned PDEs on ${\mathcal P}$. We revisit in detail in Subsection~\ref{sec:diffprob} the notion of differentiability for functions defined on ${\mathcal P}$, and
we clarify
in Appendix~\ref{app:lder}
the identification  with the Fréchet derivative of the lifting of $\varphi$ 
to the space of square-integrable random variables.
From the viewpoint of functional analysis, the square of the norm of the “intrinsic gradient” of $\varphi$ 
is typically studied in a global manner via the theory of Dirichlet forms. 
In that setting, the construction of the Dirichlet form may even precede that of the process itself—one of the main difficulties again lying 
in the choice of a reference measure on ${\mathcal P}$. 
To give a brief overview, the constructions of a Wasserstein diffusion achieved in \cite{DelloSchiavo22, DelloSchiavo24,vonRenesseSturm,Sturm,Ren}
are based on Dirichlet form techniques, whereas those in \cite{KonarovskyivonRenesse} follow a more pathwise approach. 
These works can further be categorized according to the dimension of the underlying state space:
since ${\mathcal P}({\mathbb R})$ endowed with $W_2$ is isometric to the subset of $L^2((0,1))$ consisting of nondecreasing functions,
the one-dimensional case is, in some respects, more tractable.
The papers \cite{KonarovskyivonRenesse,vonRenesseSturm} indeed focus on this setting.
In the present work, we build upon the diffusion 
on $\cP(\T^d)$ introduced for dimensions $d \geq 2$ in \cite{DelloSchiavo22} (the state space is resticted to the torus as some compactness is needed),
and we review in Section~\ref{se:2} the main principles underlying its definition as well as the key properties it satisfies.
\vskip 4pt

\paragraph{\bf Dirichlet-Ferguson diffusion.}
The approach developed in \cite{DelloSchiavo22} presents a major advantage: although the diffusion process is constructed using Dirichlet form theory, its pathwise description is explicit. 
In particular, it is rather straightforward to perturb the dynamic by adding a velocity field. 
In what follows, this property will allow us, among other things, to consider a stochastic control problem on ${\mathcal P}(\T^d)$, 
in which the drift  itself plays the role of the control variable. 
Naturally, the choice of one diffusion model over another is somewhat arbitrary, and the resulting drifted diffusion process 
retains the structural features of the chosen diffusion. 
In the present case, the diffusion introduced in \cite{DelloSchiavo22}—which we refer to as the \emph{Dirichlet--Ferguson diffusion}, in reference to the choice of the underlying measure—lives in fact only on the subspace  ${\mathcal P}^{\rm pa}(\T^d)$ 
of 
$\cP(\T^d)$
consisting of purely atomic probability measures (on the torus). 
Accordingly, the diffusion process can be represented as a countable “particle system.” 
Its structure is, as mentioned above, particularly simple: in the absence of drift, the particles evolve as independent Brownian motions with distinct intensities, 
while the corresponding Dirac masses, located at each particle’s position, are aggregated into a countable convex combination with non-uniform weights.  
The essence of the model lies precisely in the choice of these intensities and weights:  
\begin{enumerate}[i.]
\item the intensities are taken as the inverses of the square roots of the weights. 
This scaling relation explains why, in the various Itô expansions, the quadratic variation is governed by the square of the intrinsic gradient norm; 
it also corresponds to the same rescaling used in the Dean–Kawasaki model (see \cite{DelloSchiavo24,KonarovskyivonRenesseDK};
\item the weights (and hence the intensities) are randomized: specifically, they follow a Poisson-Dirichlet distribution,independently of the particle positions. 
This is the origin of the Dirichlet--Ferguson measure defined on ${\mathcal P}(\T^d)$; 
its realizations can be viewed as random purely atomic probability measures whose weights follow a Poisson-Dirichlet distribution and whose atom locations are independent, $\T^d$-uniformly distributed random variables, the weights and the atom locations being independent. We use the letter ${\mathcal D}$ 
to denote such a measure.
\end{enumerate}
Below, 
the purely diffusive particle system, that is, in the absence of drift, will be referred to as the \emph{free Dirichlet--Ferguson particle system} (or simply the \emph{free particle system}).
\vskip 4pt

\paragraph{\bf Regularizing effect and Wasserstein-Sobolev spaces.}
The regularizing properties of a diffusion process can be highlighted in several ways. 
A common approach consists in showing that the semigroup generated by the diffusion maps functions of a given regularity level 
to functions of higher regularity. 
A related and subtle issue is how the regularity of the input and output functions is quantified. 
For example, in the case of a finite-dimensional Brownian motion, bounded functions are mapped, for any strictly positive time~$t$, 
to $\rmC^\infty$ functions, and the Lipschitz seminorm of the output functions behaves at worst like $1/\sqrt{t}$ when $t$ tends to $0$. 
This result plays a central role in PDE theory, in particular in the study of semilinear equations, such as Hamilton–Jacobi–Bellman equations driven by a quadratic Hamiltonian. In \cite{DelarueHammersley}, a diffusion process taking values in ${\mathcal P}({\mathbb R})$ is constructed, exhibiting a similar smoothing phenomenon, but with the factor $1/\sqrt{t}$ replaced by $1/t^{3/4}$.
Comparable properties have also been established by \cite{Stannat} for a certain class of Fleming–Viot processes, 
although with an exponentially fast blow-up rate.
In the present setting, one cannot expect a smoothing result of the same strength (see \cite{DelloSchiavo24}). 
Nevertheless, the presence of the Laplacian suggests that the semigroup, when applied to an $L^2$ function on ${\mathcal P}$ endowed with  ${\mathcal D}$, 
should admit an intrinsic derivative which itself belongs to $L^2$ in both time and space. 
In other words, the idea is to use a Sobolev-type space $H^{1,2}$, 
consisting of functions whose intrinsic derivative has a square-integrable norm under ${\mathcal D}$, 
and that the semigroup maps $L^2$ functions into $L^2([0,T],H^{1,2})$ for any finite time horizon $[0,T]$. 
The next step is to show how linear transport–diffusion equations, and more generally semilinear equations with bounded coefficients, 
can be solved in the space $L^2([0,T],H^{1,2})$.
\vskip 4pt

\paragraph{\bf Contributions.} 
Our first contribution is to clarify the definition and the analytical structure of the space $H^{1,2}$. 
This is the object of Section \ref{se:2}, where we present the notions of metric Sobolev spaces and Cheeger energy used to define $H^{1,2}$. 
Theorem \ref{thm:mainold} then establishes the precise link between this Sobolev space and the Dirichlet form. 
Our second contribution is to solve, on the subspace ${\mathcal P}^{\rm pa}(\T^d) \subset {\mathcal P}(\T^d)$ of purely atomic probability measures (on $\T^d$), 
linear transport–diffusion type equations with bounded velocity fields, and to prove existence and uniqueness of solutions. 
As shown in Theorem \ref{thm:pde}, these solutions are continuous in time with values in $L^2({\mathcal P}^{\rm pa}(\T^d),{\mathcal D})$, 
and square-integrable in time with values in $H^{1,2}$.
Moreover, they can be represented via a Kolmogorov (or, more generally, Feynman–Kac) formula 
involving a drifted version of the 
free Dirichlet--Ferguson particle system  (see Theorem \ref{thm:repr}). 
The solutions of this drifted particle system are obtained through a Girsanov transformation and are connected to the PDE through an Itô-type formula; 
existence and uniqueness hold in the weak sense. 
All these results are developed in detail in Section
\ref{sec: massive}. 
Our third contribution is to address a semilinear version of the transport–diffusion equations on the same space ${\mathcal P}^{\rm pa}(\T^d)$, 
when the coefficients are merely square-integrable and the nonlinearity grows at most linearly with respect to the gradient. 
Existence and uniqueness are established in the same sense as in the linear case (see Proposition~\ref{prop:well_pos_HJ}). 
When the nonlinearity has an Hamiltonian structure compatible with mean-field control theory, 
the solution is interpreted as the value function of a stochastic control problem defined on the drifted Dirichlet--Ferguson particle system, 
with the drift identified as the control variable;  see Theorem \ref{thm: verif}, in which we also identify explicitly the optimal feedback function. 
Finally, our fourth contribution is to provide, both in the linear and semilinear settings, 
an interpretation of the PDEs under study as limits of finite-dimensional equations. 
This result is, at least conceptually, analogous to propagation of chaos results that connect finite systems with their infinite-population limits in classical mean-field models. 
However, in our case, there is no averaging effect stemming from a law of large numbers; 
the passage to the limit instead relies on the stability properties of the PDE set on ${\mathcal P}^{\rm pa}(\T^d)$. 
This return to the finite-dimensional framework offers an instructive perspective on the structure of the equations investigated here.
We stress that all the results established in the paper, including the approximation results, hold true  without any continuity or convexity condition on the coefficients (except the standard convexity assumption of the Lagrangian in the control variable). 
\vskip 4pt

\paragraph{\bf Comparison with the existing literature, and perspectives.} 
This paper is naturally very close to the works \cite{DelloSchiavo22, DelloSchiavo24}, and explicitly builds on several results established therein. 
In fact, the main directions on which our work relies were developed before \cite{DelloSchiavo24} was posted on arXiv; this partly explains the significant overlap between the two. 
To be precise, \cite{DelloSchiavo24} goes, in some respects, considerably further than we do: roughly speaking, \cite{DelloSchiavo24} investigates drifted versions of the free particle system in which the velocity field derives from a possibly singular potential. 
In comparison, the velocity fields considered here are not necessarily potential (which makes our setting more general), but when they are, they stem from less singular potentials than the repulsive Riesz-type potentials studied in \cite{DelloSchiavo24}. 
In fact, in \cite{DelloSchiavo24}, the potential structure enables a systematic use of Dirichlet form theory. 
By contrast---although this point is also discussed in \cite{DelloSchiavo24}---we deliberately work directly, and repeatedly, at the level of the trajectories of the associated process. 
This is a key feature that allows us, in Section \ref{sec: control}, to analyze controlled systems, which are clearly not addressed in \cite{DelloSchiavo24}. 
From a pedagogical standpoint, we also believe that the pathwise perspective emphasized here offers a complementary viewpoint to the works \cite{DelloSchiavo22,DelloSchiavo24}, from which the reader may benefit. 

Another significant difference between our paper and \cite{DelloSchiavo22,DelloSchiavo24} lies in the central role we give to the associated PDEs. 
While the PDE results obtained in Section
\ref{sec:bke} can, in their most basic form, be seen as consequences of a combination of the results in \cite{DelloSchiavo22,DelloSchiavo24} and earlier works on metric Sobolev spaces (notably \cite{FSS22}, from which we take the construction of the space $H^{1,2}$), 
our paper is, to the best of our knowledge, the first to present general well-posedness results for transport-diffusion type PDEs on the space of probability measures. 
This, in our view, constitutes our main conceptual contribution: we establish a bridge between two research communities that have developed substantially different tools for the analysis of mean-field models. 
As recalled at the beginning of the introduction, the study of PDEs on the space of probability measures has seen remarkable progress in recent years, and we hope that the present work will open new directions along this line. Indeed, the paper provides new solvability and approximation results, all of which remain valid, thanks to the presence of the diffusion, even in potentially irregular settings. This naturally raises several challenging questions. We mention two of them, which we leave for future work.
The first concerns the structure of the equation itself. Many of the equations studied on ${\mathcal P}(\T^d)$ or ${\mathcal P}(\R^d)$ correspond, from an Eulerian perspective, to diffusive (Fokker–Planck type) dynamics on these spaces, and therefore include an additional differential term arising from the derivative of the entropy. To be clear, our results do not cover this case.
A second question pertains to the notion of solution. In Section~\ref{sec: control}, we introduce a notion of strong solution for the Hamilton–Jacobi equations under consideration; it would certainly be interesting to compare this with the notion of viscosity solution used in the absence of diffusion, and to clarify, in this context, the role played by the measure ${\mathcal D}$.

\section{Functions and measures on probabilities}
\label{se:2}
We work on the $d$-dimensional flat torus $\T^d:= \R^d / \Z^d$, $d \in \N \setminus \{0,1\}$ whose elements are denoted by $x=(x^1, \dots, x^d)$. We write $x \cdot y$ for the usual scalar product between elements $x,y \in \R^d$, and $|\cdot|$ for the Euclidean norm. The standard (probability) volume measure on $\T^d$ is denoted by $\mathrm{vol}_d$. Functions on $\T^d$ will be usually denoted by letters such as $f,g,h, \dots$, while functions on the product space $[0,1] \times \T^d$ (and its subsets) will be denoted by the same letters with the addition of a hat: $\hat{f}, \hat{g}, \hat{h}, \dots$. Vectors of functions of both kinds are obtained by `bolding' the corresponding letter such as $\mbff=(f_1, \dots, f_k)$ and $\hat \mbff=(\hat{f}_1, \dots, \hat{f}_k)$, $k \in \N_+:=\N\setminus\{0\}$.
Sequences are denoted by $(a_i)_i$, and they are meant as mappings over $\N_+$.
Finally, $T>0$ denotes a fixed finite time horizon.

\subsubsection{Spaces of functions} If $\sfE$ is a  finite dimensional manifold or a Banach space, $\sfB$ is a Banach space, and $k$ is in ${\mathbb N} \cup \{\infty\}$, we denote by $\rmC^k(\sfE; \sfB)$ (resp.~$\rmC^k_b(\sfE; \sfB)$, resp.~$\rmC^k_c(\sfE; \sfB)$) the space of $k$-times continuously differentiable (resp.~$k$-times continuously differentiable and with bounded derivatives up to order $k$, resp.~$k$-times continuously differentiable and with compactly supported derivatives up to order $k$) functions from $\sfE$ to $\sfB$; if $k=0$ and/or $\sfB=\R$, we remove $k$ and/or $\sfB$ from the notation.\\
For a metric space $(\sfX, \sfd)$, $\Lip_b(\sfX, \sfd)$ denotes the space of real valued $\sfd$-Lipschitz and bounded functions. \\
If $(\sfY, \mathcal{Y}, \mathrm{n})$ is a measure space and $\sfB$ is a Banach space, we denote by $L^p(\sfY, \mathrm{n}; \sfB)$, $p \in [1, \infty]$, the Banach space of $p$-summable functions with values in $\sfB$ identified up to equality $\mathrm{n}$-a.e.; whenever $\sfB=\R$, we simply write $L^p(\sfY, \mathrm{n})$ in place of $L^p(\sfY, \mathrm{n}; \R)$. If $\sfY\subseteq[0,T]$, we denote by $\mathscr{L}^\sfY$ the Lebesgue measure on $\sfY$. Unless explicitly stated otherwise, all integrations over $\sfY$, are with respect to this measure, which we omit from the notation for $L^p$ spaces.

Given a set $\sfX$, a Banach space $(\sfB,\lvert\cdot\rvert_\sfB)$ and a function $f\colon \sfX\to \sfB$, we set $\norm{f}_{\infty} := \sup_{x\in X}\lvert f(x)\rvert_\sfB$. Given a Hilbert space $\sfH$ and a function $f\colon\sfH\to(-\infty,\infty]$, we denote by $D(f)$ the effective domain of $f$, that is the subset of $x \in \sfH$ such that $f(x)\neq\infty$. The Fréchet subdifferential of $f$ is denoted by $\partial f$ and $D(\partial f) \subset D(f)$ stands for the subset of points $x \in \sfH$ where $\partial f (x) \ne \emptyset.$
Finally, we denote by $\mathrm{AC}_{loc}((0,T); \sfH)$ the subset of $\rmC((0,T); \sfH)$ of functions whose restriction to $K$ is absolutely continuous for every compact subset $K \subset (0,T)$.

\subsubsection{Probabilities}\label{sec:wass}
For a Polish space ${\mathsf S}$, we denote by $\mathcal{B}(\mathsf S)$ the Borel $\sigma$-field of $\mathsf S$ and by $\prob({\mathsf S})$ the space of Borel probability measures on ${\mathsf S}$, endowed with the Borel $\sigma$-field induced by the topology of weak convergence. We recall that this Borel $\sigma$-field 
coincides 
with the $\sigma$-field generated by the mappings 
$\{m \in \prob({\mathsf S}) \mapsto m(E), \ E \in {\mathcal B}(\mathsf S)\}$, see for instance \cite[Proposition 7.25]{BertsekasShreve}. When ${\mathsf S}={\mathbb T}^d$, we write $\prob^{pa}(\T^d)$ for the subset of $\prob(\T^d)$ consisting of purely atomic measures (i.e., of at most countable convex combinations of Delta masses). We endow $\prob(\T^d)$ with the $L^2$-Wasserstein distance (which in this case metrizes the weak convergence of measures against functions in $\rmC(\T^d)$) defined as
\begin{equation}\label{eq:wd}
 W_2(\mu, \nu):= \left ( \inf \left \{ \int_{ \T^d \times \T^d }  |x-y|^2 \de \ggamma(x, y) : \ggamma \in \Gamma(\mu, \nu) \right \}\right )^{1/2}, \quad \mu, \nu \in \prob(\T^d), 
\end{equation}
where $\Gamma(\mu, \nu)$ denotes the subset of Borel probability measures on  $\T^d \times \T^d$ having  $\mu$ and $\nu$ as marginals. These  probability measures are called \emph{transport plans} or simply \emph{plans} between $\mu$ and $\nu$.  It turns out that $(\prob(\T^d), W_2)$ is a compact, complete and separable metric space, see \cite[Proposition 7.1.5]{AGS08}. We refer, e.g.,~to \cite{Villani09,AGS08,Villani03} for a comprehensive treatment of the theory of Optimal Transport and Wasserstein spaces. Whenever $\mu \in \prob(\T^d)$, $x \in \T^d$, and $r \in [0,1]$, we denote
\begin{align}\label{eq:Intro:CoCo}
\mu_x:= \mu(\set{x})\in [0,1] \qquad \text{and} \qquad \mu^x_r:= (1-r)\mu+r\delta_x\in\prob(\T^d),
\end{align}
where $\delta_x$ is the Dirac mass at $x$. Moreover, for a Borel set $A\subset\T^d$, we denote by $\mu\lvert_A$ the (non-normalized) restriction of $\mu$ to $A$. If $\sfX,\sfY$ are measurable spaces, $\mathrm{n}$ is a measure on $\sfX$, and $f: \sfX \to \sfY$ is a measurable function, we denote by $f_\sharp \mathrm{n}$ the measure on $\sfY$ defined as
\[ f_\sharp \mathrm{n} := \mathrm{n} \circ f^{-1}.\]
We notice that, if $\sfX$ and $\sfY$ are Polish spaces and $E$ is a Borel subset of $Y$, then the mapping 
$\mathrm{n} \in \prob(\mathsf X) \mapsto f_\sharp \mathrm n(E) = \mathrm{n}(f^{-1}(E)) \in {\mathbb R} $ is measurable. In particular, the mapping $\mathrm n \in \prob(\mathsf{X}) \mapsto f_\sharp \mathrm{n} \in \prob(\mathsf Y)$ is measurable.

\subsection{The Dirichlet--Ferguson measure}\label{sec:diri}
The aim of this subsection is to introduce the Dirichlet--Ferguson measure, with which we will endow $\prob(\T^d)$. To do so, we follow the same presentation as in \cite{DelloSchiavo22}.
In the following, we endow~$I^\infty:= [0,1]^{\infty}$ with the product topology and set
\begin{equation}\label{eq:DeltaTau}
\begin{aligned}
S^\infty:=\set{\mbfs:= (s_i)_i\in I^\infty : \sum_{i=1}^\infty s_i=1}, \qquad
T^\infty:=\set{\mbfs= (s_i)_i \in S^\infty : s_i \geq s_{i+1} \geq 0 \text{~~for all~}i} .
\end{aligned}
\end{equation}
The set~$T^\infty_o\subset T^\infty$ is defined similarly to~$ T^\infty$ with~$>$ in place of~$\geq$. 
We denote by~$\Upsilon\colon  I^\infty\to  T^\infty\subset  I^\infty$ the reordering map, 
which is known to be measurable thanks to~\cite[p.~91]{DonJoy89}. On the infinite product~$(\T^d)^\infty$, we consider the product probability measure $\mathrm{vol}_d^\infty:= \bigotimes^\infty_{i=1} \mathrm{vol}_d$.
We also denote
\[ (\T^d)^\infty_o := \{ \mbfx := (x_i)_i \in (\T^d)^\infty : x_i \ne x_j \text{ if } i \ne j \}.\]

We consider the map
\begin{align}\label{eq:Intro:Phi}
\emp \colon  S^\infty\times(\T^d)^\infty \to \prob(\T^d),\quad \emp(\mbfs,\mbfx):= \sum_{i=1}^\infty s_i\delta_{x_i}.
\end{align}
We set 
\[ \prob_o(\T^d) := \emp(T^\infty_o \times (\T^d)^\infty_o) \]
and we observe that it is a Borel subset of $\prob(\T^d)$: indeed, the so called atomic topology $\tau_a$ on $\prob(\T^d)$ (cf.~\cite[Section 3.1.1]
{DelloSchiavo24}) is such that $(\prob_o(\T^d), \tau_a)$ and $(\prob(\T^d), \tau_a)$ are Polish spaces (see \cite[Corollary 3.3]{DelloSchiavo24}); in particular, by e.g.~\cite[Theorem 4.3.24]{Engelking89}, $\prob_o(\T^d)$ is a $G_\delta$ subset of $(\prob(\T^d), \tau_a)$ and thus an element of the Borel $\sigma$-field induced by $\tau_a$ on $\prob(\T^d)$. Since the latter coincides with the Borel $\sigma$-field induced by the weak topology on $\prob(\T^d)$, see \cite[Proposition 3.1(iii)]{DelloSchiavo24}, the set $\prob_o(\T^d)$ is also a Borel subset of $\prob(\T^d)$ equipped with the weak topology (in fact any Borel subset of $(\prob_o(\T^d), \tau_a)$ is a Borel subset of $\prob(\T^d)$).

Since the map $\emp$ is injective on $ T^\infty_o \times (\T^d)^\infty_o$ (because the elements
of $T^\infty_o$, encoding the masses, are strictly decreasing, and the elements
of $(\T^d)^\infty_o$, encoding the positions, are pairwise distinct), we can define an inverse $\emp^{-1}: \prob_o(\T^d)\to  T^\infty_o \times (\T^d)^\infty_o$. For a measure $\mu \in \prob_o(\T^d)$, we denote 
\begin{equation}\label{eq:phiinv}
(\mbfs(\mu), \mbfx(\mu)):= \emp^{-1}(\mu).   
\end{equation}
Note that, since the injective map $\emp: T_o^\infty \times (\T^d)_o^\infty \to \prob_o(\T^d)$ is clearly Borel, then its inverse is also Borel, see e.g.~\cite[Theorem 6.8.6]{Bogachev07}.

\smallskip
Let $\mathrm{U}$ be the uniform distribution on $[0,1]$ and let us denote by $\mathrm{U}^\infty$ the corresponding product measure on~$ I^\infty$. We introduce the \emph{Poisson--Dirichlet measure}~$\Pi_1$ as done in~\cite{DonGri93}.

\begin{definition}[Poisson--Dirichlet measure]
Let~$\Lambda\colon  I^\infty\to  S^\infty$ be defined, for $\mbfr:= (r_i)_i\in  I^\infty$, by
\begin{equation}\label{eq:LambdaReArr}
\begin{aligned}
\Lambda_1(r_1):=& r_1,\qquad \Lambda_k(r_1,\dotsc, r_k):= r_k\prod_i^{k-1} (1-r_i), \qquad 
\Lambda(\mbfr):=& (\Lambda_1(r_1), \Lambda_2(r_1,r_2),\dotsc).
\end{aligned}
\end{equation}
The \emph{Poisson--Dirichlet measure~$\Pi_1$ with parameter~$1$} on~$ T^\infty$, concentrated on~$ T^\infty_o$, is
\begin{align*}
\Pi_1:= ({\Upsilon}\circ{\Lambda})_\sharp \mathrm{U}^\infty.
\end{align*}
\end{definition}
We are in position to define the Dirichlet--Ferguson measure.
\begin{definition}[Dirichlet--Ferguson measure]\label{def:dfm} The Dirichlet--Ferguson measure $\mathcal{D}$ on $\prob(\T^d)$ is defined as
\[ \mathcal{D}:= \emp_\sharp (\Pi_{1}\otimes \mathrm{vol}_d^\infty).\]
\end{definition}
Note that $\mathcal{D}$ depends on $d$, but we do not stress this dependence in the notation, thus always assuming that $d$ is a fixed parameter.
\begin{remark} One could define similarly a Dirichlet--Ferguson measure depending on two parameters $\theta$ and $\nu$, with $\theta>0$ and $\nu \in \prob(\T^d)$ being a diffuse probability measure, by replacing $1$ with $\theta$, $\mathrm{vol}_d$ with $\nu$, and $\mathrm{U}$ with the beta distribution $\mathrm{B}_\theta$ on $[0,1]$ with shape parameters $1$ and $\theta$. Since $\mathrm{vol}_d(\T^d)=1$ and to be consistent with the presentation in \cite{DelloSchiavo22}, we choose to work with this canonical choice of parameters.
\end{remark}

\begin{remark}
\label{rem:Poisson-Dirichlet:rep}
    Definition \ref{def:dfm} can be reinterpreted from a probabilistic perspective, characterizing the Dirichlet--Ferguson measure as the law of a suitable random variable taking values in $\prob(\T^d)$, obtained as a purely atomic measure with random weights and atom positions. Let $\mathbf{Y} := (Y_i)_i$ and $\mbfx := (x_i)_i$ be sequences of independent random variables with laws $\mathrm{U}$ and $\mathrm{vol}_d$ (i.e., the uniform distribution on $\T^d$), respectively. We define the $S^\infty$-valued random variable $\Lambda(\mathbf{Y}) = (\Lambda_1(Y_1), \Lambda_2(Y_1, Y_2), \Lambda_3(Y_1, Y_2, Y_3), \dots) = (Y_1, Y_2(1 - Y_1), Y_3(1 - Y_2)(1 - Y_1), \dots)$, and denote by $\mbfs$ its reordered version, that is, the $T^\infty$-valued random variable $\mbfs := \Upsilon(\Lambda(\mathbf{Y}))$. Notice that $\mbfs \sim \Pi_1$ whilst $\mbfx\sim\mathrm{vol}_d^\infty$. We can then construct a purely atomic measure by assigning the random weights $\mbfs$ to the atoms 
$\mbfx$, namely, $\sum_{i=1}^\infty s_i \delta_{x_i} = \emp(\mbfs, \mbfx)$. The law of this random variable is precisely the Dirichlet--Ferguson measure $\cD$.
\end{remark}
We present a characterization of the \emph{Dirichlet--Ferguson measure}~$\mathcal{D}$ introduced by T.~S.~Ferguson in his seminal work~\cite{Fer73}, see also \cite[Proposition 4.9]{DelloSchiavo22}.
\begin{theorem}[A characterization of~$\mathcal{D}$]\label{thm:dfm} Let~$\mm$ be a probability measure on~$\prob(\T^d)$. 
Then, the following are equivalent:
\begin{itemize}
\item~$\mm$ is the Dirichlet--Ferguson measure $\mathcal{D}$;
\item~$\mm$ satisfies the \emph{Mecke-type identity} or \emph{Georgii--Nguyen--Zessin formula},
\begin{align}\label{eq:Mecke}
\int_{\prob(\T^d)} \int_{\T^d} u(\mu, x, \mu_x) \de\mu(x)\de\mm(\mu) =\int_{\prob(\T^d)} \int_{\T^d} \int_0^1 u(\mu^x_r, x, r) \de r\de\mathrm{vol}_d(x)\de\mm (\mu)
\end{align}
for any semi-bounded measurable~$u\colon\prob(\T^d)\times \T^d\times [0,1]\to \R$.
\end{itemize}
\end{theorem}
Note that, clearly, $\mathcal{D}$ is concentrated on discrete measures and, in fact, it can be shown that it is concentrated on $\prob_o(\T^d) = \emp( T^\infty_o \times (\T^d)^\infty_o)$, see \cite[Proposition 4.9]{DelloSchiavo22}. 

\subsection{Differentiability of functions on probabilities}\label{sec:diffprob}
In this subsection, we discuss two approaches to the definition of differentiability for functions $u: \prob(\T^d) \to \R$ which will be used in the sequel. We compare the notion of derivative introduced here with the one of \emph{L-derivative} in Appendix \ref{app:lder}.
Let us first introduce some relevant classes of functions on probabilities.
\subsubsection{Cylinder functions}\label{sec:fonp}
As prototypical examples of cylinder functions, we may consider functions induced by elements $f \in \rmC^\infty(\T^d)$ and $\hat f \in \rmC^\infty([0,1] \times \T^d)$:
\begin{equation}\label{eq:thefirsttrid}
f^\trid :  \prob(\T^d)\ni \mu \mapsto \int_{\T^d} f(x) \de \mu(x), \quad \hat f^\trid :  
 \prob(\T^d) \ni \mu \mapsto \int_{\T^d}\hat{f}(\mu_x, x) \de \mu(x).
\end{equation}
If $\mbff=(f_1, \dots, f_k) \in \rmC^\infty(\T^d; \R^k)$ and $\hat \mbff = (\hat f_1, \dots, \hat f_k) \in \rmC^\infty([0,1] \times \T^d; \R^k)$, $k \in\N_+$, we set
\[ \mbff^\trid := ( f_1^\trid, \dots, f_k^\trid),\quad \quad \hat \mbff^\trid := ( \hat f_1^\trid, \dots, \hat f_k^\trid).\]
The extended space $[0,1] \times \mathbb{T}^d$ makes it possible to define cylinder functions that depend explicitly on the weights of the atoms of a (possibly atomic) probability measure. A natural way to restrict attention to finitely many atoms is to consider only those depending on weights exceeding a given threshold $\varepsilon$.

We are in position to define several algebras of smooth cylinder functions.

\begin{definition}[Smooth cylinder functions] \label{def:cylfunc} We define, for every $\eps \in [0,1]$, the following algebras of smooth cylinder functions on $\prob(\T^d)$:
\begin{align*}
\mathfrak{Z}^\infty&:= \set{ u\colon \prob(\T^d)\to\R : \begin{gathered}  u=F\circ \mbff^\trid,\, F \in\rmC^{\infty}_{c}(\R^{k};\R),
\\
k\in \N_+,\, \mbff \in \rmC^{\infty}(\T^d; \R^k)\end{gathered}},\\
\hat{\mathfrak{Z}}_\eps^\infty&:= \set{ \hat u\colon \prob(\T^d)\to\R : \begin{gathered}  \hat u=F\circ \hat\mbff^\trid,\, F \in\rmC^{\infty}_{c}(\R^{k};\R), \, k \in \N_+,
\\
\hat\mbff\in \rmC^{\infty}([0, 1] \times \T^d; \R^k), \, \supp(\hat\mbff) \subset [\eps, 1] \times \T^d\end{gathered}},
\intertext{and}
\hat{\mathfrak{Z}}_c^\infty&:= \set{ \hat u\colon \prob(\T^d)\to\R : \begin{gathered} \hat  u=F\circ \hat \mbff^\trid,\, F \in\rmC^{\infty}_{c}(\R^{k};\R),\,
k\in \N_+,\\ \hat \mbff\in \rmC^{\infty}([0,1] \times \T^d; \R^k), \, \supp(\hat\mbff) \subset (0, 1] \times \T^d \end{gathered}}.
\end{align*}
\end{definition}
Since every function in $\rmC^\infty(\T^d)$ can be trivially identified with a function in $\rmC^\infty([0,1] \times \T^d)$, we note that $\hat{\mathfrak{Z}}_\gamma^\infty \subset \hat{\mathfrak{Z}}_\eps^\infty \subset \bigcup_{\delta \in (0,1)}\hat{\mathfrak{Z}}_\delta^\infty =  \hat{\mathfrak{Z}}_c^\infty \subset \hat{\mathfrak{Z}}_0^\infty $ for every $0<\eps<\gamma \le 1$ and $\mathfrak{Z}^\infty \subset \hat{\mathfrak{Z}}_0^\infty$. Note also that $\mathfrak{Z}^\infty \not\subset \hat{\mathfrak{Z}}_c^\infty$ and $\mathfrak{Z}^\infty \not\subset \hat{\mathfrak{Z}}_\eps^\infty$ for every $\eps \in (0,1]$. Functions in $\hat{\mathfrak{Z}}_c^\infty$ are usually denoted by $\hat{u}$ to emphasize that they are obtained from functions $\hat{f}$ defined on the product space $[0,1] \times \mathbb{T}^d$, in contrast to functions in $\mathfrak{Z}^\infty$, which are typically denoted by $u$. However, when considering a general element of $\hat{\mathfrak{Z}}_0^\infty$,   we use the notation $\hat{u}$ even if it could refer to an element of $\mathfrak{Z}^\infty$.
Functions in $\hat{\mathfrak{Z}}_0^\infty$ are Borel-measurable, but in general not continuous w.r.t.~any reasonable topology (e.g.~the weak topology on $\prob(\T^d)$), see \cite[Remark 5.3]{DelloSchiavo22}. Functions in $\mathfrak{Z}^\infty$ are continuous w.r.t.~the weak topology and even Lipschitz continuous w.r.t.~the Wasserstein distance, see Subsection \ref{sec:wass} and \cite[Section 4.1]{FSS22, S22}. This last property is crucial when considering the metric approach to the definition of differentiability of functions on $\prob(\T^d)$, see Section \ref{sec:metr}.

Let us also mention that many other sets of cylinder functions could be constructed by assuming a different regularity for the external functions $F$ and/or the internal ones $f$ and $\hat{f}$, but the classes of cylinder functions defined above are sufficient for our purposes.
Let us recall that if $\mm$ is any Borel probability measure on $\prob(\T^d)$, then
\begin{equation}\label{eq:density}
\hat{\mathfrak{Z}}^\infty_c \text{ is dense in } L^p(\prob(\T^d), \mm) \text{ for every } p \in [1,\infty),
\end{equation}
see e.g.~\cite[Lemma 5.4]{DelloSchiavo22}.

Finally, we introduce classes of cylinder functions depending on time: let $\mathfrak{Z}$ be any of the classes $\mathfrak{Z}^\infty$, $\hat{\mathfrak{Z}}_c^\infty$, $\hat{\mathfrak{Z}}_\eps^\infty$, $\eps \in [0,1)$, as in Definition \ref{def:cylfunc}; then we say that a function $u: I \times \prob(\T^d) \to \R$, where $I \subset \R$ is an interval, belongs to the class $\mathfrak{T Z}(I)$ if $u_t:=u(t, \cdot) \in \mathfrak{Z}$ for every $t \in I$ and $u(\cdot, \mu)\in \Lip_b(I, \sfd_{eucl})$ for every $\mu \in \prob(\T^d)$. 

We also notice that the density property \eqref{eq:density}
can be extended to functions depending not only on the measure argument but also on time and space.
Indeed, leveraging \eqref{eq:density}, we can approximate real-valued functions depending on time and space by functions belonging to
\begin{equation}
\hat{\mathcal{A}} := \rmC^\infty([0,T]) \otimes \hat{\mathfrak{Z}}^\infty_c \otimes \rmC^\infty(\T^d).
\end{equation}
In order to deal with vector-valued functions,
we also denote, for every $m \in \N_+$, by $\hat{\mathcal{A}}_m$ the set of functions $b=(b_i)_{i=1}^m$ from $[0,T] \times \prob(\T^d) \times \T^d$ taking values in $\R^m$, such that $b_i \in \hat{\mathcal{A}}$ for every $i=1, \dots, m$. We then have the following density result:

\begin{proposition}\label{prop:approx_trunc}
Let $m \in \N_+$, $b:[0,T] \times \prob(\T^d) \times \T^d \to \R^m$ be a measurable function, 
and
$\mathcal{Q}$ be a Borel probability measure on $[0,T] \times \prob(\T^d) \times \T^d$, 
such
that $b \in L^p([0,T] \times \prob(\T^d) \times \T^d, \mathcal{Q}; \R^m)$ for some $p \in [1,\infty)$. Then there exist sequences $(\eps_n)_n \subset (0,1)$ and $(b^n)_n \subset \hat{\mathcal{A}}_m$ such that $b^n \to b $ in $L^p([0,T] \times \prob(\T^d) \times \T^d, \mathcal{Q}; \R^m)$ as $n \to \infty$ with $b^n_t(\cdot, x) \in (\hat{\mathfrak{Z}}_{\eps_n}^\infty)^m$ for every $(t,x) \in [0,T] \times \T^d$ and every $n \in \N_+$. Moreover, if $\|b\|_\infty < \infty$, then such a sequence $(b^n)_n$ can be chosen such that $\|b^n\|_\infty \le \|b\|_\infty$ for every $n \in \N_+$.        
\end{proposition} 
\begin{proof}
Without loss of generality we can assume $m=1$; call $X:=[0,T] \times \prob(\T^d) \times \T^d$. Consider the product algebra of continuous functions on $X$ given by
\[ \mathcal{A} := \rmC^\infty([0,T]) \otimes \rmC(\prob(\T^d)) \otimes \rmC^\infty(\T^d).\]
Since each factor is a unital algebra of continuous functions separating points in the corresponding space, and $X$ is compact, $\mathcal{A}$ is uniformly dense in $\rmC(X)$, and thus (see e.g.~\cite[Lemma 2.27]{Savare22}) dense in $L^p(X, \mathcal{Q})$.
Let $\pi^2: X \to \prob(\T^d)$ be the map sending $(t,\mu,x)$ to $\mu$. By
\eqref{eq:density}, observe that $\hat{\mathfrak{Z}}^\infty_c$ is dense in $L^p(\prob(\T^d), \pi^2_\sharp \mathcal{Q})$, and then it is also $L^p(\prob(\T^d), \pi^2_\sharp \mathcal{Q})$-dense in $\rmC(\prob(\T^d))$. Thus we deduce that also $\hat{\mathcal{A}}$ is dense in $L^p(X, \mathcal{Q})$. This gives the existence of a sequence $(b_n)_n \subset \hat{\mathcal{A}}$ approximating any $b \in L^p(X,\mathcal{Q})$. Note that every $b^n$ is a sum of $K_n \in \N_+$ functions of the form
\[ \fru_t^{i,n}(\mu, x) = g^{i,n}(t) \hat{u}^{i,n}(\mu) f^{i,n}(x), \quad g^{i,n} \in \rmC^\infty([0,T]), \hat{u}^{i,n} \in \hat{\mathfrak{Z}}^\infty_c, f^{i,n} \in \rmC^\infty(\T^d), \quad i=1, \dots, K_n.\]
Since $\hat{u}^{i,n} \in \hat{\mathfrak{Z}}^\infty_{\eps^{i,n}}$ for some $\eps^{i,n} \in (0,1)$, it is clear that, by setting $\eps_n:= \min_{i=1, \dots, K_n} \eps^{i,n} \in (0,1)$, we 
obtain
\[ b^n_t(\cdot, x) \in \hat{\mathfrak{Z}}^\infty_{\eps_n} \quad  \text{ for every } (t,x) \in [0,T] \times \T^d.\]
This proves the first part of the statement. 
The second part also follows by the same argument as in the proof of \cite[Lemma 2.27]{Savare22}, by simply composing an approximating sequence $(b^n)_n \subset \hat{\mathcal{A}}$ of $b$ with a sequence of polynomials approximating the truncation function
$g(r):=-\|b\|_\infty \vee r \wedge \|b\|_\infty$.
\end{proof}

\subsubsection{The geometric approach}\label{sec:geo}
We follow the presentation in \cite{DelloSchiavo22}, see also the references therein. The main idea is to start from \emph{directional} differentiability w.r.t.~a class of natural variations on $\prob(\T^d)$: these are given by (the flows of) smooth, orientation preserving vector fields $\mathfrak{X}^\infty(\T^d)$ on $\T^d$. For an element $w \in \mathfrak{X}^\infty(\T^d)$, we denote by $X^w_t$ the flow of $w$ at time $t\ge 0$, that is $(X^w_t)_{t \geq 0}$ solves the Ordinary Differential Equation (ODE):
\[ X^w_0(x)=x, \quad \dot{X}^w_t =w (X^w_t), \quad t>0,\, x \in \T^d.\]
Let $u: \prob(\T^d) \to \R$ and $\mu \in \prob(\T^d)$; we say that $u$ is differentiable in the direction $w \in \mathfrak{X}^\infty(\T^d)$ at the point $\mu$ if the derivative
\[ \partial_w u(\mu) := \frac{\de}{\de t} \biggr |_{t=0} u((X^w_t)_\sharp \mu)\]
exists and is finite. Using the fact that $\mathfrak{X}^\infty(\T^d)$ is dense in the Hilbert space $T_\mu \prob(\T^d):=L^2(\T^d, \mu; \R^d)$, 
we deduce that, if $u$ is differentiable in the direction $w$ at the point $\mu$ for every $w \in \mathfrak{X}^\infty(\T^d)$ 
and the map
\[ w \mapsto \partial_w u(\mu)\]
is bounded on $\mathfrak{X}^\infty(\T^d)$, then the Riesz's representation theorem gives the existence of an element $(\boldnabla u)_\mu \in T_\mu \prob(\T^d)$ such that 
\[ \langle (\boldnabla u)_\mu, w \rangle_{T_\mu \prob(\T^d)} = \partial_w u (\mu) \quad \text{for every } w \in \mathfrak{X}^\infty(\T^d).\]
In this case, we say that $u$ is differentiable at $\mu$ and that $(\boldnabla u)_\mu$ is the gradient of $u$ at $\mu$. The definition of $T_\mu \prob(\T^d)$ used here is an extension, to vector fields of non-gradient type, of the classical definition of a
tangent space 
given
(among others) in \cite{Otto01, AGS08}. For further comments on the terminology as well as for other notions of tangent spaces, see the appendix of~\cite{DelloSchiavo20} and the references therein.

It is not difficult to show that functions in $\hat{\mathfrak{Z}}_0^\infty$, as introduced in Definition \ref{def:cylfunc}, are differentiable in $\prob(\T^d)$.  Also, we can compute explicitly their gradients, see \cite[Section 5.2, Lemma 5.8]{DelloSchiavo22}, \cite[Definition 4.4]{FSS22} and \cite[Lemma 5.17]{DS25} for a proof.

\begin{lemma}\label{le:hasgrad} Let $\hat{u} = F \circ \hat{\mbff}^\trid \in \hat{\mathfrak{Z}}_0^\infty$ with $F \in \rmC_c^\infty(\R^k; \R)$, $k \in \N_+$. Then $\hat{u}$ is differentiable at any $\mu \in \prob(\T^d)$ and it holds
\begin{equation}\label{eq:l:TridHorGrad:1}
\begin{aligned}
(\boldnabla \hat u)_\mu(x) =\ \sum_{i=1}^k (\partial_i F)(\hat\mbff^\trid(\mu)) \nabla \hat f_i(\mu_x,x), \quad (\mu,x) \in \prob(\T^d) \times \T^d,
\end{aligned}
\end{equation}
where $\nabla \hat{f}_i$ denotes the gradient of $\hat{f}_i$ w.r.t.~to the $x$-variable. Moreover, for $\cD \otimes \mathrm{vol}_d \otimes\mathrm{U}$-a.e.~$(\mu, x,r)$, the map $z \mapsto \hat u(\mu_r^z)$ is differentiable in a neighbourhood of $x$ and it holds
\begin{equation}\label{eq:relgrad}
    \nabla_z \rvert_{z=x} (z \mapsto \hat u(\mu_r^z)) = r (\boldnabla \hat u)_{\mu_r^x} (x).
\end{equation}
\end{lemma}

Once we have a notion of gradient for a suitable class of functions, we can introduce the quadratic form $(\mathcal{E},\hat{\mathfrak{Z}}_0^\infty)$ as
\begin{equation}\label{eq:em}
\mathcal{E}(\hat u,\hat v) := \int_{\prob(\T^d)} \langle (\boldnabla \hat u)_{\mu},(\boldnabla \hat v)_{\mu} \rangle_{T_\mu \prob(\T^d)} \de \mathcal{D}(\mu), \quad \hat u,\hat v \in \hat{\mathfrak{Z}}_0^\infty,    
\end{equation}
where $\mathcal{D}$ is the Dirichlet--Ferguson measure introduced in Definition \ref{def:dfm}.
As is common, we will denote $ \mathcal{E}(\hat u):=  \mathcal{E}(\hat u,\hat u)$.

\smallskip
One key question concerns the closability of $(\mathcal{E},\mathfrak{Z}^\infty)$, noticing that the form is here restricted to the smaller set of cylinder functions $\mathfrak{Z}^\infty \subset \hat{\mathfrak{Z}}_0^\infty$. In other words, the point is to see whether there exists a closed form $(\hat{\mathcal{E}},D(\hat{\mathcal{E}}))$ with $\mathfrak{Z}^\infty \subset D(\hat{\mathcal{E}})$ such that $\hat{\mathcal{E}}$ restricted to $\mathfrak{Z}^\infty$ coincides with $\mathcal{E}$. Here, `closed' means the following:
\begin{equation}\label{eq:closab}
 \left [ (u_n)_n \subset D(\hat{\mathcal{E}}), \, u \in D(\hat{\mathcal{E}}), \, u_n \to u \text{ in } L^2(\prob(\T^d), \mathcal{D}),\,\hat{\mathcal{E}}(u_n) \to 0 \right ] \Rightarrow u=0.     
\end{equation}

It is shown in \cite[Theorem 5.11]{DelloSchiavo22} that the form $\mathcal{E}$ admits a so called \emph{generator} $L_c$ on the set $\hat{\mathfrak{Z}}_c^\infty$: there exists a linear operator $ L_c: \hat{\mathfrak{Z}}_c^\infty \to L^2(\prob(\T^d), \mathcal{D})$ such that
\begin{equation}\label{eq:lc}
 \mathcal{E}(\hat u,\hat v) = - \langle \hat u, L_c \hat v \rangle_{L^2(\prob(\T^d), \mathcal{D})} \quad \text{ for every } \hat u,\hat v \in \hat{\mathfrak{Z}}_c^\infty.    
\end{equation}
Thanks to the Mecke's identity in Theorem \ref{thm:dfm}, the operator $L_c$ can be explicitly computed, see \cite[Formula (5.15)]{DelloSchiavo22}. It is given, for every $\hat{u} \in \hat{\mathfrak{Z}}_c^\infty$ and $\mathcal{D}\text{-a.e. } \mu \in \prob(\T^d)$, by
\begin{align}\label{eqn:gen_on_cil}
L_c (\hat u)_\mu:&= \int_{\T^d} \frac{\Delta_z\rvert_{z=x} \hat u(\mu+\mu_x\delta_z-\mu_x\delta_x)}{{\mu_x}^2} \de\mu(x)  = \sum_{x \in \supp(\mu)}\frac{\Delta_z\rvert_{z=x} \hat u(\mu+\mu_x\delta_z-\mu_x\delta_x)}{{\mu_x}}.
\end{align} 
This allows to show 
(see \cite[Theorem 5.11]{DelloSchiavo22}) that the form $(\mathcal{E},\hat{\mathfrak{Z}}_c^\infty) $ admits a closed extension $(\hat {\mathcal{E}},D(\hat {\mathcal{E}}))$, which is also a Dirichlet form (in the sense of \cite[Definition 4.5]{Ma-Rockner92}) with associated generator $\hat L$ and domain $D(\hat {\mathcal{E}})
\supset D(\hat{L}) \supset
\hat{\mathfrak{Z}}_c^\infty$, $\hat{L}$ being an extension of $L_c$. 
We also know from \cite[Theorem 3.15(i)]{DelloSchiavo24} that $L_c$ is \emph{essentially self adjoint} in $L^2(\prob(\T^d), \cD)$, that is, it admits a unique self-adjoint extension, which is $\hat{L}$, which therefore necessarily coincides with its adjoint $L_c^*$, see e.g.~\cite[Section VIII.2]{Reed-Simon}.

Finally, it is possible to show (see \cite[Lemma 5.16]{DelloSchiavo22}) that $\mathfrak{Z}^\infty \subset D(\hat{\mathcal{E}})$ (notice the absence of `hat' in the former) and
that $\hat{\mathcal{E}}$ coincides with $\mathcal{E}$ when restricted to $\mathfrak{Z}^\infty$. Since closability (of a form) is transmitted to restrictions (of the domain), we get that
\[ \text{ the quadratic form } (\mathcal{E}, \mathfrak{Z}^\infty) \text{ admits a closed extension } (\mathcal{E}_0, D(\mathcal{E}_0)) \text{ with generator } (L_0, D(L_0)).\]
Clearly $(\mathcal{E}_0, D(\mathcal{E}_0))$ is also a Dirichlet form. 
As it is obtained by restriction of the domain, 
it holds $\mathfrak{Z}^\infty \subset D(\mathcal{E}_0) \subset D(\hat{\mathcal{E}})$ and $D(L_0) \subset D(\mathcal{E}_0)$. Notice that, in principle,
the construction of the Dirichlet form by restriction of the domain does not allow to identity 
$(\mathcal{E}_0, D(\mathcal{E}_0))$ and $(\hat {\mathcal{E}},D(\hat {\mathcal{E}}))$. However, in this specific context, since we are also assuming that $d \ge 2$, the two forms coincide and so do their generators. 
In particular,
$\hat{\mathfrak{Z}}_c^\infty\subset D(\mathcal{E}_0)$. We refer to \cite[Corollary 5.19]{DelloSchiavo22} and \cite[Remark 5.21]{DS25}
for the main results on these facts.

\subsubsection{The metric approach}\label{sec:metr} There is a completely different way to treat the differentiability of functions defined on $\prob(\T^d)$, usually referred to as the metric approach to Sobolev spaces \cite{AGS14I,Cheeger99, Savare22}. We show in this article that this approach is particularly interesting since it is well suited to the treatment of PDEs on the space of probability measures, yielding the existence of a Hilbert space of (weakly) differentiable functions. We introduce this approach here in the particular case of the metric-measure space $(\prob(\T^d), W_2, \mathcal{D})$, where $W_2$ is the Wasserstein distance defined in \eqref{eq:wd} and $\mathcal{D}$ is the Dirichlet--Ferguson measure in Definition \ref{def:dfm}.

If $u: \prob(\T^d) \to \R$ belongs to $\Lip_b(\prob(\T^d), W_2)$, its asymptotic Lipschitz constant is defined as
\begin{equation}\label{eq:aslipintro}
  \lip u(\mu):=
  \limsup_{\nu,\sigma \to \mu,\ \nu\neq \sigma}\frac{|u(\nu)-u(\sigma)|}{W_2(\nu,\sigma)}, \quad \mu \in \prob(\T^d).
\end{equation}
The so called pre-Cheeger energy of $u \in \Lip_b(\prob(\T^d), W_2)$ is defined as
\begin{equation}\label{eq:prech}
\pCE(u):= \int_{\prob(\T^d)} (\lip u)^2 \de\mathcal{D}, \quad u \in \Lip_b(\prob(\T^d), W_2),  
\end{equation}
and it is a surrogate of the integral norm of the gradient. The \emph{Cheeger energy} of a function $u \in L^2(\prob(\T^d), \mathcal{D})$ is the $L^2$-relaxation of $\pCE$, i.e.
\begin{equation}\label{eq:relaxintroche}
  \CE(u) := \inf \left \{ \liminf_{n \to \infty}
    \pCE(u_n) : (u_n)_n \subset \Lip_b(\prob(\T^d), W_2), \, u_n \to u \text{ in } L^2(\prob(\T^d), \mathcal{D})
  \right \}.
\end{equation}
We can now provide the definition of a metric Sobolev space.
\begin{definition} The metric Sobolev space $H^{1,2}(\prob(\T^d), W_2, \mathcal{D})$ is the vector space of functions $u \in L^2(\prob(\T^d), \mathcal{D})$ with finite Cheeger energy,  endowed with the norm 
\begin{align} \label{eq:CheegerEnergy} |u|^2_{H^{1,2}}:= \int_{\prob(\T^d)} |u|^2 \de\mathcal{D}+ \CE(u),
\end{align}
which makes it a Banach space (cf.~\cite[Theorem 2.1.17]{pasquabook}).
\end{definition}
 Despite the implicit nature of $H^{1,2}(\prob(\T^d), W_2, \mathcal{D})$, in the next subsection we will present more explicit characterizations of the space and its functions, obtained by comparison with the geometric approach discussed in Subsection \ref{sec:geo}.

\smallskip
For ease of notation, in the rest of the article we use
\begin{equation}\label{eq:notation}
    H:= L^2(\prob(\T^d), \mathcal{D}), \quad H^{1,2}:=H^{1,2}(\prob(\T^d), W_2, \mathcal{D}),
\end{equation}
and we denote by $(\cdot, \cdot)_H$ and $|\cdot|_H$ the scalar product and the norm in $H$, respectively.

\subsubsection{Comparison of geometric and metric approaches}\label{sec:geovsmetr} We aim at comparing the objects
\[ (\mathcal{E}_0, D(\mathcal{E}_0)) \quad \text{ and } \quad (\CE, H^{1,2}),\]
whose definitions can be found at the very end of Subsection \ref{sec:geo} and in Subsection \ref{sec:metr}, respectively.
In this regard, the following two results are taken from the literature. 
The first one establishes a connection between the restriction of the pre-Cheeger energy \eqref{eq:prech} to cylinder functions
and the form $(\mathcal{E}, \mathfrak{Z}^\infty)$ defined in \eqref{eq:em}, see \cite[Proposition 4.9]{FSS22}, \cite[Proposition 4.8]{DS25}.

\begin{proposition}\label{prop:introwss} For every $u \in \mathfrak{Z}^\infty$, it holds
\begin{equation}\label{eq:pcerep}
    (\lip u (\mu))^{2} =\int_{{\mathbb T}^d} |(\boldnabla u)_\mu|^2 \de \mu, \quad \mu \in \prob(\T^d),
\end{equation} 
where $\boldnabla u$ is as in \eqref{eq:l:TridHorGrad:1}. 
In particular, $\pCE$ and $\mathcal{E}$ coincide on $\mathfrak{Z}^\infty$, which implies that
$(\pCE, \mathfrak{Z}^\infty)$ is a quadratic form.
\end{proposition}

The following statement gives the density property of cylinder functions, see \cite[Theorem 4.10]{FSS22}. This
allows us to identify the domain of the form
$\mathcal{E}_0$ with the metric Sobolev space $H^{1,2}$.

\begin{theorem}\label{thm:mainold}  The algebra of smooth cylinder functions $\mathfrak{Z}^\infty$ is dense in $H^{1,2}$. In particular, $H^{1,2}$ is a Hilbert space (i.e.~$\CE$ is a quadratic form in $H$) and, for every $u \in H^{1,2}$, there exists a sequence $(u_n)_n \subset  \mathfrak{Z}^\infty$ such that
\begin{equation}\label{eq:apprpr}
    u_n \to u \text{ in $H$} \quad \text{ and } \quad  \pCE(u_n) \to \CE(u) \quad \text{ as } n \to \infty.
\end{equation}
Moreover
\[ (\mathcal{E}_0, D(\mathcal{E}_0)) = (\CE, H^{1,2}).\]
\end{theorem}
Because of the above identification between the Dirichlet form and the Cheeger energy, we will refer to the generator $L_0$ of $\mathcal{E}_0$ with the symbol $\boldsymbol{\Delta}$ which in part of the literature denotes the so called metric-measure Laplacian of a metric-measure space (in this case $(\prob(\T^d), W_2, \mathcal{D}))$, see \cite[Definition 5.2.1]{pasquabook}. In the sequel we will also make direct use of some properties of $\boldsymbol{\Delta}$ coming from the metric theory, such as its identification with the unique element in the subdifferential of the Cheeger energy. We also have the following integration by parts formula
\begin{equation}\label{eq:ibp0}
\CE(u,v)=\mathcal{E}_0(u,v) = - \int_{\prob(\T^d)} u \boldsymbol{\Delta} v \de \mathcal{D} =- (u, \boldsymbol{\Delta}v)_H, \quad \forall \, u,v \in D(\boldsymbol{\Delta}) \supset \hat{\mathfrak{Z}}_c^\infty.    
\end{equation}

\medskip
We conclude this section by introducing the notion of gradient for a function in $H^{1,2}$, following \cite[Theorem 5.1 and Section 5.1]{FSS22} (although the proof therein is carried out for $\R^d$, it does not rely on the linear structure of $\R^d$ and can be extended to $\T^d$). In order to do so, it is useful to `extend' the measure $\mathcal{D}$ on $\prob(\T^d)$ to a measure $\overline{\mathcal{D}}$ on $\prob(\T^d) \times \T^d$ in the following way:
\begin{equation}\label{eq:bmm}
\overline{\mathcal{D}} := \int_{\prob(\T^d)} (\delta_\mu \otimes \mu) \de \mathcal{D}(\mu).
\end{equation}
Therefore, the integral of a Borel semi-bounded function $u:\prob(\T^d) \times \T^d \to \R$ w.r.t.~$\overline{\mathcal{D}}$ is given by
\[ \int_{\prob(\T^d) \times \T^d} u \de \overline{\mathcal{D}} = \int_{\prob(\T^d)}\int_{\T^d} u(\mu, x) \de \mu(x) \de \mathcal{D}(\mu).\]

\begin{remark}\label{rem:notation1} From the notational point of view we use $\langle \cdot, \cdot \rangle$ and $\| \cdot\|$ for the scalar product and the norm in $L^2(\prob(\T^d) \times \T^d, \overline{\mathcal{D}}; \R^d)$, respectively.
\end{remark}

\begin{proposition}\label{prop:vectorgrad} For every $u \in H^{1,2}$, there exists  a unique vector field $\rmD u \in L^2(\prob(\T^d) \times \T^d, \overline{\mathcal{D}}; \R^d)$ such that, for any sequence $(u_n)_n \subset \mathfrak{Z}^\infty$ converging to $u$ in the same sense as \eqref{eq:apprpr}, it holds
\begin{equation}\label{eq:gradappr}
\boldnabla u_n \to \rmD u \quad \text{ in } L^2(\prob(\T^d) \times \T^d, \overline{\mathcal{D}}; \R^d).  
\end{equation} 
In particular, for any $u, v \in H^{1,2}$, we have
\[\CE(u,v)= \langle \rmD u,\rmD v \rangle.\]
Moreover, for every $u \in \mathfrak{Z}^\infty$ there exists a sequence $(\hat u_n)_n \subset \hat{\mathfrak{Z}}_c^\infty$
such that $\hat u_n \to u$ in $H$ and $\boldnabla \hat u_n \to \boldnabla u$ in $L^2(\prob(\T^d) \times \T^d, \overline{\cD}; \R^d)$ as $n \to \infty$. Finally, $\rmD \hat u = \boldnabla \hat u$, for every $\hat u \in \mathfrak{Z}^\infty$ and every $\hat u \in \hat{\mathfrak{Z}}_c^\infty$. 

\end{proposition}
\begin{proof} 
The only results not proved in the aforementioned references are: (i) the approximation of $u \in \mathfrak{Z}^\infty$ by elements of $\hat{\mathfrak{Z}}_c^\infty$ in the sense described in the last part of the statement, and (ii) the identity  $\rmD \hat u = \boldnabla \hat u$ for every $\hat u \in \hat{\mathfrak{Z}}_c^\infty$. The first follows exactly as in the proof of \cite[Lemma 5.19]{DS25}. For the second, we argue as follows. By the closability of $(\mathcal{E}, \hat{\mathfrak{Z}}_c^\infty)$ and the equality $\mathcal{E}_0= \hat{\mathcal{E}}$ (see the notation and discussion in Subsection \ref{sec:geo}), we have
\[ 
\|\rmD \hat u\|^2 = \mathcal{E}_0(\hat u,\hat u) = \mathcal{E}(\hat u,\hat u) = \|\boldnabla \hat u\|^2.
\]
Therefore, in order to prove that 
$\rmD \hat u = \boldnabla \hat u$, it is enough to show that
\[  \langle \rmD \hat u,\boldnabla \hat u \rangle \ge \|\rmD \hat u\| \|\boldnabla \hat u\|.  \]
By definition of $\rmD \hat u$, there exists a sequence $(u_n)_n \subset \mathfrak{Z}^\infty$ such that $u_n \to \hat u$ in $H$ and
$\boldnabla u_n \to \rmD \hat u$ in $L^2(\prob(\T^d) \times \T^d, \overline{\mathcal{D}}; \R^d)$. By the approximation result in item (i) above, we can replace the sequence $(u_n)_n$ by a sequence $(\hat{u}_n)_n \subset \hat{\mathfrak{Z}}_c^\infty$, and preserve the convergence of the functions and their gradients. We get
\begin{align*}
    \langle \rmD \hat u,\boldnabla \hat u \rangle &= \lim_{n \to \infty}  \langle \boldnabla \hat{u}_n,\boldnabla \hat u \rangle =\lim_{n \to \infty} \CE(\hat u, \hat{u}_n) \\
    & = \lim_{n \to \infty} \left ( \frac{1}{2} \CE(\hat u+\hat{u}_n) - \frac{1}{2} \CE(\hat u) - \frac{1}{2}\CE(\hat{u}_n) \right ) \\
    & \ge \frac{1}{2} \CE(2\hat u) - \frac{1}{2} \CE(\hat u) - \frac{1}{2} \CE(\hat u) \\
    & = \CE(\hat u) =\|\rmD \hat u\| \|\boldnabla \hat u\|, 
\end{align*}
where we have used the fact that $\CE$ is a quadratic form to get the second line, together with the fact that it is lower semicontinuous w.r.t.~the convergence in $H$ to get the third line.
\end{proof}

\subsection{Free particle system}\label{ssec:freeparty}
The aim of this section is to present the relation between the objects introduced in the previous subsections and a system of free (i.e., undrifted and thus uncontrolled) particles. Although most of the material presented in this section relies on the results from \cite{DelloSchiavo22, DelloSchiavo24}, we reformulate parts of it to align with the objectives of the upcoming sections. As a result, the exposition is largely original. In addition, we introduce several technical tools---such as an Itô formula---that will be used in the following sections.

We recall that $\mathcal{D}$ is the Dirichlet--Ferguson measure as in Definition \ref{def:dfm} and $\bar{\mathcal{D}}$ is the measure as in \eqref{eq:bmm}. 
Given the same time horizon $T \in (0,\infty)$ as in the first paragraph of Section \ref{se:2}, an initial time
$t_0\in[0,T)$, a filtered probability space $(\Omega, \cF, \PP, \{\cF_t\}_{t\in[0,T]})$, and 
 a sequence $(B^i)_i$ of independent $d$-dimensional  $\{\cF_t\}_{t\in[0,T]}$-Brownian motions, let $(X^i)_i$ be the family of processes defined by
\begin{equation}\label{eqn: particles_free}
    \de X^i_t = \sqrt{\frac{2}{s_i}}\de B^i_t,\quad X^i_{t_0} = \imath(x_i), \quad t \in [t_0, T],\quad i\in\N_+,
\end{equation}
where $(\mbfs=(s_i)_i, \mbfx=(x_i)_i) \in  T^\infty_o \times (\T^d)^\infty_o$ is regarded as an initial datum (thus, the initial condition involves not only the starting points of the processes \((X^i)_i\), but also the associated weights, or inverse intensities, \((s_i)_i\)), 
and $\imath$ denotes the canonical embedding 
from ${\mathbb T}^d$ into 
${\mathbb R}^d$, obtained
by associating to an element of ${\mathbb T}^d$ its unique representative in 
$[0,1)^d$. Note that, using this definition, we will consider below stochastic initial conditions. Namely, 
in 
\eqref{eqn: particles_free}, 
$(\mbfs, \mbfx)$
will be 
sometimes regarded 
as an $\cF_{t_0}$-measurable $ T^\infty_o \times (\T^d)^\infty_o$-valued random variable distributed according to 
some
$\cursm\in\prob(T^\infty_o \times (\T^d)^\infty_o)$.
In this case,
we write $(\mbfs, \mbfx)\sim_\PP\cursm$ (i.e.~$(\mbfs, \mbfx)_\sharp \PP = \cursm$). We also notice that 
$(\mbfs,\mbfx)$
is then independent of $\sigma( 
\{
(B_t^i - B_{t_0}^i), 
i \in \N_+, 
t \in [t_0,T]
\})$.

With the particle system \eqref{eqn: particles_free}, we can associate the ${\mathcal P}({\mathbb T}^d)$-valued stochastic process $ (\mu^\infty_t:=\sum_{i=1}^{\infty}s_i\delta_{X^i_t})_{ t \in [t_0,T]} $, where $\delta_{x}$, for $x \in {\mathbb R}^d$ is regarded as a probability measure on ${\mathbb T}^d$, defined by 
\begin{equation*}
\delta_x (A) = \sum_{k \in {\mathbb Z}^d} {\mathbf 1}_A(x+k) \quad  \text{ for every Borel subset } A \subset \T^d.
\end{equation*}
We can also write
\begin{equation}\label{eq:project} \mu_t^\infty = \sum_{i=1}^\infty s_i \delta_{\varpi(X_t^i)},
\end{equation}
where $\varpi: \R^d \to \T^d$ is the canonical projection.
It is worth noting that writing $(\mbfs, \mbfx)\sim_\PP\cursm$ is equivalent to stating that $\mu^\infty_{t_0}$ is distributed according to $\mm := \emp_\sharp(\cursm)$, which we write $\mu_{t_0}\sim_\PP\mm$.

An important example is
$\cursm = \Pi_{1}\otimes \mathrm{vol}_d^\infty$ (see Definition \ref{def:dfm}), i.e.,
 $(\mbfs, \mbfx)\sim_\PP\Pi_{1}\otimes \mathrm{vol}_d^\infty$, in which case $\mu^\infty_{t_0}$ is distributed according to $\cD$. The following result, stated in the paragraph above \cite[Theorem 1.5]{DelloSchiavo24}, contains the invariance property of $\mathcal{D}$ for the process $\mu^\infty$. 

\begin{proposition}\label{prop:mut} The Dirichlet-Ferguson measure $\mathcal{D}$ is invariant for the $\prob(\T^d)$-valued process $\mu^\infty$. More precisely, if $\mu^\infty_{t_0}\sim_\PP\mathcal{D}$, then $\mu^\infty_t \sim_\PP\mathcal{D}$ for any $t\in[t_0,T]$.
\end{proposition}

\subsubsection{A canonical space for the particle system}\label{ssec:canonical}
The construction of the free particle system \eqref{eqn: particles_free} is carried out in a `strong' sense, meaning that it
can be achieved on any arbitrarily given probabilistic setting. 
{However, for the purposes of Section \ref{sec: massive}, in which we study a particle system that can only be solved in a weak sense, it is necessary to specify the canonical set-up associated with 
\eqref{eqn: particles_free}.
Therefore, we set
\begin{equation*}
\Xi^{0} := T^\infty_o \times 
({\mathbb T}^d)^{\infty}_o \times \rmC([0,T];{\mathbb R}^d)^{\infty},
\end{equation*}
which is the canonical space carrying the initial condition and the family of Brownian motions driving the free particle system, and then
\begin{equation*}
\Xi := \Xi^{0} \times \rmC([0,T];{\mathbb R}^d)^{\infty},
\end{equation*}
which is the canonical space carrying, in addition to the
initial condition and the Brownian motions, the trajectories of the 
free particle system.
We equip the two spaces with their Borel $\sigma$-fields, 
denoted by
$\cG^0$ and $\mathcal{G}$ respectively, and with their canonical filtrations, denoted by $\{\cG^0_t\}_{t\in[0,T]}$ and$\{\cG_t\}_{t\in[0,T]}$ respectively. The canonical process on $\Xi^{0}$ is denoted by 
\begin{equation*}
\left((\varsigma^j)_j,(\xi^{0,j})_j,((\beta^j_t)_{t \in [0,T]})_j
\right).
\end{equation*}
More precisely, for 
$\chi^0 \in \Xi^0$ and  
$j \in \N_+$, 
$\varsigma^j(\chi^0)$ (resp. $\xi^{0,j}(\chi^0)$, resp. $\beta^j(\chi^0)$)
is the $j$-th coordinate of the component of 
$\chi^0$ lying in $T^\infty_o$ (resp. in $(\T^d)^\infty_o$, resp. in $\rmC([0,T];{\mathbb R}^d)^{\infty}$). 
The canonical process on $\Xi$ is defined analogously and is denoted by 
\begin{equation*}
\left((\varsigma^j)_j,(\xi^{0,j})_j,((\beta^j_t)_{t \in [0,T]})_j,((\xi_t^j)_{t \in [0,T]})_{j}
\right).
\end{equation*}

\begin{remark}\label{rem:toruspb} Note that $\xi^j_t$ takes values in $\R^d$ for every $t \in [0,T]$ and every $j \in \N_+$. However, we will often consider it as implicitly composed with the projection map $\varpi$ as in \eqref{eq:project} as, for example, in \eqref{eq:muinfmap} (when defining a probability measure on $\prob(\T^d)$) or in \eqref{eq:firstito} (when evaluating $\rmD u_r$).
\end{remark}

At this stage, we notice that, on the smaller space $\Xi^0$, 
the free particle system \eqref{eqn: particles_free} can be easily constructed by means of the map
 $\varphi : (0,1) \times {\mathbb T}^d \times \rmC([0,T];{\mathbb R}^d) \to \rmC([0,T];{\mathbb R}^d)$ that sends
an element $(s,x,(w_t)_{t \in [0,T]})$
onto 
\begin{equation*}
\left( \imath(x) + \sqrt{\frac{2}{s}} w_t
\right)_{t \in [0,T]}.
\end{equation*}
Then, the free particle system can be regarded, on $\Xi^0$, as the mapping
 $\Phi : \Xi^0 \rightarrow \rmC([0,T];{\mathbb R}^d)^{\infty}$ defined by
\begin{equation*}
\Phi 
\left((s_j)_j,(x_{j}^0)_j,\left((w^j_t)_{t \in [0,T]}\right)_j
\right)
=
\left( \varphi\left(s_j,x_{j}^0,\left(w^j_t\right)_{t \in [0,T]}
\right)
\right)_j.
\end{equation*}
Note that both $\varphi$ and $\Phi$ are Borel measurable.
The next step in equipping $\Xi$ with a convenient probabilistic structure is to endow $\Xi^0$ with a suitable probability measure and then transfer it to $\Xi$ via $\Phi$.
The definition of a probability measure on $\Xi^0$ proceeds via the following standard `concatenation' mapping: for a given initial time $t_0\in[0,T)$, we let
\begin{align*}
    \mathrm{conc}^{t_0} :\rmC([0,T]; \R^d) \times \rmC([0,T]; \R^d) &\to \mathrm{C}([0,T]; \R^d) 
    \\
    ((w^1_t)_{t \in [0,T]}, (w^2_t)_{t \in [0,T]}) &\mapsto \left ( t \mapsto \begin{cases} w^1_t \quad & \text{ if } t \in [0,t_0] 
    \\ w^2_t-w^2_{t_0}+w^1_{t_0} \quad & \text{ if } t \in [t_0, T] \end{cases} \right ).
\end{align*}
On $\rmC([0,T]; \R^d)$ we consider the measure
\[{\mathcal W}_{t_0} := \mathrm{conc}^{t_0}_\sharp (\delta_0 \otimes \mathcal{W}), \]
being $\mathcal{W}$ the Wiener measure on $[0,T]$ and $\delta_0$ is regarded as the Dirac mass at the constant curve equal to $0$ on $[0,T]$. In other words, ${\mathcal W}_{t_0}$ is the probability measure on $\rmC([0,T]; \R^d)$ obtained  by concatenating the law of the path identical to zero on $[0,t_0]$ with the (shifted) Wiener law $\cW$ on $[t_0,T]$.
Finally for an initial measure 
$\cursm$ on $T^\infty_o \times ({\mathbb T}^d)^{\infty}_o$, we endow $\Xi^0$ 
with the measure $\cursm \otimes {\mathcal W}_{t_0}^{\infty}$, where ${\mathcal W}_{t_0}^{\infty}:= \bigotimes_{i=1}^\infty \mathcal{W}_{t_0}$, and then $\Xi$ with the measure
\begin{equation}\label{def:mathbfP:cursm}
 {\mathbf P}^{t_0,\cursm} :=
(\textrm{\rm id},\Phi)_{\sharp}  \left( \cursm \otimes \cW^{\infty}_{t_0} \right),   
\end{equation}
where $\textrm{\rm id}$ is here regarded as the identity mapping on $\Xi^0$. When dealing with a deterministic initial condition $(\mbfs, \mbfx) \in T_o^\infty \times (\T^d)^\infty_o$, we use the lighter notation
\[ {\mathbf P}^{t_0, (\mbfs,\mbfx)}:= {\mathbf P}^{t_0, \delta_{(\mbfs,\mbfx)}}. \]
With the notation introduced here, we can regard $\mu^\infty$ as a Borel map from $\Xi$ to $\rmC([0,T]; \prob(\T^d))$, defined by
\begin{equation}\label{eq:muinfmap}
  \mu^\infty (\chi) 
= \left( 
\mu^\infty_t(\chi) 
\right)_{t \in [0,T]} 
  := \left( \sum_{j=1}^\infty \varsigma^j(\chi) \delta_{\xi^j_t(\chi)} \right )_{t \in [0,T]}, \quad \chi \in \Xi.  
\end{equation}

\subsubsection{A chain rule}
We now prove an It\^o type formula for the process $\mu^\infty$ (defined in 
\eqref{eq:muinfmap}
right above). This result will play a key role in the next sections. We start with a technical result stating that any finite sequence $(s_1, \dots, s_N)$ of points in $[0,1]^N$ can be completed to a sequence in $[0,1]^\infty$ (resp.~that a finite sequence $(x_1, \dots, x_N)$ of points in $(\T^d)^N$ can be completed to a sequence in $(\T^d)^\infty$) in a suitable way. We stress that we are going to use only the result about points in $(\T^d)^N$, but we added the one about points in $[0,1]^N$ for completeness.
For every $N \in \N_+$, we denote with the same notation $\pi_N$ the projection from $[0,1]^\infty$ and $(\T^d)^\infty$ on their first $N$ components i.e. the map that sends a sequence $\mbfs = (s_i)_i$ (resp. $\mbfx = (x_i)_i$) onto the vector $(s_i)_{i=1}^N$ (resp $(x_i)_{i=1}^N$).
\begin{lemma}\label{le:measel} For every $N \in \N_+$ there exist measurable maps $\mathsf{M}_N :[0,1]^N \to [0,1]^\infty$ and $\mathrm{P}_N: (\T^d)^N \to (\T^d)^\infty$ such that:
\begin{enumerate}[1)]
    \item $(\mathsf{M}_N \circ \pi_N)(T_o^\infty) \subset T_o^\infty$ and $\pi_N \circ \mathsf{M}_N$ is the identity in $[0,1]^N$;
    \smallskip
    \item $(\mathsf{P}_N \circ \pi_N)((\T^d)_o^\infty) \subset (\T^d)_o^\infty$ and $\pi_N \circ \mathsf{P}_N$ is the identity in $(\T^d)^N$.
\end{enumerate} 
\end{lemma}
\begin{proof} We start with the construction of $\mathsf{M}_N$; if $(s_1, \dots, s_N) \in (0,1)^N$ and $R:= 1-\sum_{j=1}^N s_j >0$, then we set 
\[ r:= \begin{cases} \frac{1}{2} \quad &\text{ if } s_N \ge R, \\
1- \frac{s_N}{2R} \quad &\text{ if } s_N <R. \end{cases}\]
In this case, we define 
\[\mathsf{M}_N(s_1, \dots, s_N)_k := \begin{cases} s_k \quad & \text{ if } 1 \le k \le N, \\
R(1-r)r^{k-N-1} \quad &\text{ if } k > N. \end{cases}\]
In the other case, we set
\[\mathsf{M}_N(s_1, \dots, s_N)_k := \begin{cases} s_k \quad & \text{ if } 1 \le k \le N, \\
0 \quad &\text{ if } k > N. \end{cases}\]
To construct $\mathsf{P}_N$, we proceed as follows: let us denote
\[ D^N := \left \{ (x_1, \dots, x_N) \in (\T^d)^N : x_1=x_2 = \dots = x_N \right \},\]
and note in particular that $D^1 = \T^d$. Define also the function $\eps^N: (\T^d)^N \to (0,\infty)$ by
    \[ \eps^N(x_1, \dots, x_N):= \begin{cases} 1 \quad &\text{ if } (x_1, \dots, x_N) \in D^N, \\
     \frac{1}{3}\min \{ |x_1-x_i| : 2 \le i \le N, \, x_i \ne x_1 \} &\text{ if } (x_1, \dots, x_N) \notin D^N.\end{cases}\]
Given \((x_1, \dots, x_N)\), we define \(\mathrm{P}_N(x_1, \dots, x_N)\) componentwise by specifying each component \(z_j\) for \(j \in \N_+\): we set $z_i=x_i$ if $i \in \{1,\cdots,N\}$ and 
\[ z_j = x_1 + 2^{-j} e_1 \eps^N(x_1, \dots, x_N), \quad j > N, \]
where $e_1$ is the first vector of the standard Euclidean basis in $\R^N$. It is easy to check that $\mathrm{M}_N$ and $\mathrm{P}_N$ satisfy the requirements.
\end{proof}
Using the map $\mathsf P_N$ defined in the previous Lemma \ref{le:measel}, we next define a $\mbfs$-parametrized version of 
the function $\emp$ (see \eqref{eq:Intro:Phi}) that only depends on the $N$ first atoms and weights, letting 
\begin{equation}
    \label{eq:empN}
\begin{split}
\emp^N(\mbfs, (x_1, \dots, x_N)) &:=  \emp(\mbfs, \mathrm{P}_N(x_1, \dots, x_N))  
\\
&=  \sum_{j=1}^N s_j \delta_{x_j} + \sum_{j=N+1}^{\infty} s_j \delta_{z_j}, \quad (\mbfs, (x_1, \dots, x_N)) \in  T^\infty_o \times (\T^d)^N,
\end{split}
\end{equation}
 where $(z_j)_j$ 
 in the right-hand side are implicitly understood as 
 $(z_j)_j
 =\mathrm{P}_N(x_1, \dots, x_N)$. We point out that for any $(\mbfs, (x_i)_{i=1}^N)\in T_o^\infty \times\pi_N((\T^d)^\infty_o)$, the measure $\emp^N(\mbfs, (x_1, \dots, x_N))$ lies in $\prob_o(\T^d)$. In parallel, we also introduce a truncation level that only retains, in a sequence $\mbfs \in T^\infty_o$, the weights higher than a given threshold $\varepsilon \in (0,1]$, letting
\begin{equation}   \label{eq:Neps:mbfs}
    N(\eps, \mbfs):= \min \{j \in \N_+ : s_j < \eps \}-1 \in \{0, \dots, \ceil{1/\eps}\} \end{equation}
    with the inclusion  on the right-hand side following from the fact that $1= 
\sum_{j=1}^{\infty} s_j 
\geq \sum_{j=1}^{N(\eps, \mbfs)} s_j \geq N(\eps, \mbfs) \varepsilon$.
The following lemma clarifies the connection between the differential operators $\rmD$ and 
$\boldsymbol{\Delta}$, acting on functions defined on 
$\cP(\T^d)$, and the standard differential operators defined in the Euclidean setting.

\begin{lemma} \label{lemma: rel_der_cil} Let $\varepsilon \in (0,1]$, $(\mbfs, \mbfx) \in T^\infty_o \times (\T^d)_o^\infty$, 
$\hat u \in \hat{\mathfrak{Z}}_\eps^\infty$, and set $\mu := \emp(\mbfs, \mbfx)$. Then for every $N \ge N(\eps, \mbfs)$ we have
\begin{align} \label{eqn: grad_cil0}
       &\rmD \hat u(\mu,x_i) = \frac{1}{s_i}\nabla_i (\hat u \circ \emp^{N}) (\mbfs,(x_1, \dots, x_N)),  && \text{ for every } i=1, \dots, N, 
       \\
       &\rmD \hat u (\mu,x_i) = 0, && \text{ for every } i>N(\eps, \mbfs),\label{eqn: grad_cil}
       \\
       & \boldsymbol{\Delta} \hat u(\mu) = \sum_{i=1}^{N(\eps, \mbfs)} \frac{1}{s_i}\Delta_i (\hat u \circ \emp^N) (\mbfs,(x_1, \dots, x_N)),\label{eqn: lapl_cil} \\
       &\Delta_i (\hat u \circ \emp^N)(\mbfs, (x_1, \dots, x_N)) = 0, && \text{ for every } N(\eps, \mbfs) < i \le N,\label{eqn: lapl_cil0}
    \end{align}
where $\nabla_i$ and $\Delta_i$ denote the respective gradient and Laplace operator w.r.t.~the $i$-th
toroidal variable (it being understood that the derivatives that appear above exist in a classical sense).
\end{lemma}
To clarify further the scope of the result, it should be stressed that, when 
$N = N(\varepsilon,{\mbfs})$, then
$\emp^N(\mbfs,(x_1,\cdots,x_N))$ is a probability measure whose atoms of mass less than $\varepsilon$
depend on $(x_1,\cdots,x_N)$
(and may vary with it). However, these atoms count for nothing when $\emp^N$ is composed with $\hat{u}$, because the latter only depends on the atoms of 
mass greater than $\varepsilon$.

\begin{proof} 
Following the definition of 
$\hat{\mathfrak{Z}}_\eps^\infty$, we call $F$ and $\hat{\mbff}$ two functions such that 
$\hat{u} = F \circ \hat{\mathbf f}^\star$, 
where 
$F \in \rmC^\infty_c({\mathbb R}^k,{\mathbb R})$
and
$\hat{\mbff} = (\hat{f}_1, \dots, \hat{f}_k)$, for some $k \in \N_+$ and some $\hat{f}_1,\dots,\hat{f}_k \in \rmC^\infty([0,1] \times \T^d)$ with $\supp(\hat\mbff) \subset [\eps, 1] \times \T^d$. By construction, $\hat u \circ \emp^N(\mbfs, \cdot)$ is well defined on $(\T^d)^N$ and for every $(y_1, \dots, y_N) \in  (\T^d)^N $ we have
\[ (\hat u \circ \emp^N)(\mbfs, (y_1, \dots, y_N)) = F \left ( \sum_{i=1}^N s_i \hat f_1 (s_i, y_i), \dots, \sum_{i=1}^N s_i \hat f_k (s_i, y_i)  \right ) \]
so that we see that it is also smooth there.
Then, we have, for any $i \in \{1,\cdots,N\}$,
\begin{align*}
    \nabla_i (\hat u \circ \emp^{N}) (\mbfs,(x_1, \dots, x_N)) &= \nabla_z |_{z=x_i} \bigl[ F ( \hat{\mbff}^\star (\mu-s_i \delta_{x_i} + s_i \delta_z)
    ) \bigr]
    \\
    &= \sum_{\ell=1}^k \partial_\ell F( \hat{\mbff}^\star(\mu)) \nabla_z |
    _{z=x_i} 
    \bigl[
    \hat{f}_\ell^\star (\mu-s_i \delta_{x_i} + s_i \delta_z) 
    \bigr]
    \\
    &= \sum_{\ell=1}^k \partial_\ell F( \hat{\mbff}^\star(\mu)) \nabla_z |_{z=x_i}
    \bigl[
    \hat{f}_\ell(s_i, z) 
    \bigr]
    s_i \\
    &= s_i \sum_{\ell=1}^k \partial_\ell F( \hat{\mbff}^\star(\mu)) \nabla f_\ell(s_i, x_i) \\
    &= s_i \rmD \hat u (\mu, x_i),
\end{align*}
where the last equality comes from Lemma \ref{le:hasgrad}. This proves \eqref{eqn: grad_cil0}.
Using again Lemma \ref{le:hasgrad}, and observing that $\nabla f_\ell(s_j,x_j)=0$ for $j >N(\eps, \mbfs)$ (since $s_j < \varepsilon$), 
we 
establish 
\eqref{eqn: grad_cil} by repeating the computation in the last two lines. To prove \eqref{eqn: lapl_cil0} we observe that, since $\hat{u}(\mu-s_j \delta {x_j} + s_j \delta_z)=
(\hat{u} \circ \emp^N)
(\mbfs,(x_1,\dots,x_{j-1},z,x_{j+1},\dots,x_N))$ for every $j=1, \dots, N$ and $z \in \T^d$, and for $j > N(\eps, \mbfs)$ the first function is constant in $z$, we have
\[\Delta_i (\hat u \circ \emp^N) (\mbfs,(x_1, \dots, x_N)) = 0 \quad \text{ for every } N(\eps, \mbfs) < i \le N. \]
To prove \eqref{eqn: lapl_cil} we use the same equality and \eqref{eqn:gen_on_cil} to get
\begin{align*}
\boldsymbol{\Delta} \hat u(\mu) &= \int_{\T^d}\frac{\Delta_z\rvert_{z=x} \hat u(\mu + \mu_x\delta_z - \mu_x\delta_x)}{(\mu_x)^2}\de\mu(x) 
\\
& = \sum_{j=1}^{\infty} \frac{1}{s_j} \Delta_z |_{z=x_j} 
\bigl[
\hat u (\mu -s_j \delta_{x_j} + s_j \delta_z)
\bigr]
\\
& = \sum_{j=1}^{N(\eps, \mbfs)} \frac{1}{s_j} \Delta_z |_{z=x_j} \bigl[ (\hat u \circ \emp^N)((s_1, \dots, s_N), (x_1, \dots, x_{j-1}, z, x_{j+1}, \dots, x_N)) 
\bigr]
\\
& = \sum_{j=1}^{N(\eps, \mbfs)} \frac{1}{s_j}\Delta_i (\hat u \circ \emp^N) ((s_1, \dots, s_N),(x_1, \dots, x_N)).
\end{align*}\end{proof}
The next result relies on the definition of $\mathfrak{T Z}(I)$ given at the very end of Subsection \ref{sec:fonp}, where $I\subset\R$ is an interval, $\eps\in[0,1)$ and $\mathfrak{Z}$ is any of the classes $\mathfrak{Z}^\infty$, $\hat{\mathfrak{Z}}_c^\infty$, $\hat{\mathfrak{Z}}_\eps^\infty$. 
We also recall that  the tuple $((\varsigma^j)_j,(\xi^{0,j})_j,((\beta^j_t)_{t \in [0,T]})_j,((\xi_t^j)_{t \in [0,T]})_{j}
)$ denotes the canonical process on $\Xi$.
\begin{proposition}\label{prop: ito_cyl} 
 Let $t_0 \in [0,T)$ be an initial time and $(\mbfs,\mbfx) \in T_o^\infty \times (\T^d)_o^\infty$ be an initial datum.
 Then, for $\eps \in (0,1]$ and $u \in \mathfrak{T}\hat{\mathfrak{Z}}_\eps^\infty([t_0, T])$, the process $ (u_t(\mu_t^\infty))_{t \in [0,T]}$ can be expanded as
    \begin{equation}\label{eq:firstito}
    \begin{split}
        u_t(\mu^\infty_t)  &= u_{t_0}(\mu^\infty_{t_0}) + \int_{t_0}^t \left (\partial_r u_r(\mu^\infty_r) 
         +  \boldsymbol{\Delta}u_r(\mu^\infty_r)\right )\de r
         \\
         &\hspace{15pt} + \sum_{i = 1}^{\infty}\int_{t_0}^t \sqrt{2 \varsigma_i} \rmD u_r(\mu_r^\infty, \xi^i_r) \cdot \de \beta^i_r,\quad t\in[t_0,T],\quad \mathbf{P}^{t_0, (\mbfs,\mbfx)}\text{-a.s.}
   \end{split}
    \end{equation}
where, for any 
$\mu \in \prob(\T^d)$, the function $r  \in [t_0,T] \mapsto \partial_r u_r(\mu)$ is regarded as (a representative of) the Sobolev derivative of the function
$r \in [t_0,T] \mapsto u_r(\mu)$. (We explain in the proof how the representative, which is uniquely defined up to a null Borel subset, can be canonically chosen.) 
\end{proposition}
 Note that in the above display we are using the convention explained in Remark \ref{rem:toruspb} to evaluate $\rmD u_r(\mu_r^\infty, \cdot)$ at $\xi_r^i$.
\begin{proof} 
For simplicity
we assume that $t_0$ is equal to $0$.
Moreover, for the first part of the proof, we also suppose that $u$ does not depend on time (i.e.~$u\in\hat{\mathfrak{Z}}_\eps^\infty$).
We set $N:=\lceil 1/\varepsilon \rceil$, $(\boldsymbol{\xi}_t^{N} := (\xi^1_t,\dots,\xi^{N}_t))_{t \in [0,T]}$. Recall also the notation $\varpi\colon\R^d\to\T^d$ as in \eqref{eq:project} for the (smooth) canonical projection from $\R^d$ to $\R^d / \Z^d = \T^d$ and let $\varpi_N\colon (\R^d)^N\to(\T^d)^N$ be the (smooth) canonical projection from $(\R^d)^N$ to $(\R^d / \Z^d)^N = (\T^d)^N$ (i.e. $\varpi$ applied applied componentwise). Note that, for $v\in\rmC^\infty((\T^d)^N)$, one has $\nabla_i (v\circ\varpi) = (\nabla_iv)\circ \varpi$, where on the left-hand side the gradient is understood as an operator on $\R^d$, while on the right-hand side it is understood as an operator on $\T^d$. An analogous relation holds for $\Delta_i$.

By applying classical It\^o's formula to $u \circ \emp^N(\mbfs,\cdot)\circ\varpi_N$, which belongs to 
$\rmC^\infty((\R^d)^N)$, and by using the same notation $\nabla_i$ and $\Delta_i$ as in Lemma \ref{lemma: rel_der_cil}, we get
\begin{align*}
    u(\mu^\infty_t) &= (u \circ \emp^N(\mbfs,\cdot)\circ\varpi_N)(\boldsymbol{\xi}_t^{N}) \\
    &=  (u \circ \emp^N(\mbfs,\cdot)\circ\varpi_N)(\boldsymbol{\xi}_0^{N}) +\sum_{i=1}^{N}\int_0^t \sqrt{\frac{2}{s_i}}\nabla_i (u \circ \emp^N(\mbfs,\cdot)\circ\varpi_N)(\boldsymbol{\xi}_r^{N})\cdot \de \beta^i_r \\
    & \quad + \sum_{i=1}^N \int_0^t \frac{1}{s_i}\Delta_i (u \circ \emp^N(\mbfs,\cdot)\circ\varpi_N)(\boldsymbol{\xi}_r^{N})\de r \\
        &=  (u \circ \emp^N)(\varpi_N(\boldsymbol{\xi}_0^{N}))+\sum_{i=1}^{N}\int_0^t \sqrt{\frac{2}{s_i}}\nabla_i (u \circ \emp^N)(\varpi_N(\boldsymbol{\xi}_r^{N}))\cdot \de \beta^i_r \\
    & \quad + \sum_{i=1}^N \int_0^t \frac{1}{s_i}\Delta_i (u \circ \emp^N)(\varpi_N(\boldsymbol{\xi}_r^{N}))\de r \\
    & = u(\mu_{0}^\infty) + \sum_{i=1}^{\infty}\int_0^t\sqrt{2s_i} \rmD u_r(\mu_r^\infty, \xi^i_r)\cdot \de  \beta^i_r + \int_0^t \boldsymbol{\Delta}u(\mu^\infty_r)\de r, \quad t\in[0,T],\quad \mathbf{P}^{0,(\mbfs,\mbfx)}\text{-a.s.}
\end{align*}
where to obtain the last line we used Lemma \ref{lemma: rel_der_cil}.

Extending the proof to the case when $u$ 
also depends on time requires some care, because elements of 
$\mathfrak{T}\hat{\mathfrak{Z}}_\eps^\infty)[t_0, T])$ are just assumed to be Lipschitz in 
time, uniformly in 
the measure argument. 
Generally speaking, the point is to replace the standard version of Itô's formula
by the It\^o-Krylov one (see, e.g., \cite[Theorem 1, p.~122]{KrylovBookControl}).
Here, this is possible because the process $(\boldsymbol{\xi}_t^{N})_{t \in [0,T]}$
is non-degenerate. In brief, 
Sobolev differentiability
in time is then enough to get the above expansion. 
We thus regard $r \mapsto \partial_r [u_r \circ \emp^N(\mbfs,\cdot)\circ\varpi_N ] (x_1,\cdots,x_N)$
as a version of the Sobolev derivative (in $r$) of the function $r \mapsto [u_r \circ \emp^N(\mbfs,\cdot)\circ\varpi_N ](x_1,\cdots,x_N).$
This derivative is uniquely defined up to a null Borel subset of $[0,T] \times ({\mathbb R}^d)^N$, which 
is enough for our purpose
because
$(\boldsymbol{\xi}_t^{N})_{t \in [0,T]}$ does not see null Borel subsets of 
$({\mathbb R}^d)^N$. Thanks to the identity $u_r(\mu)=
(u_r \circ\emp^N(\mbfs,\cdot)\circ\varpi_N)
(x_1,\dots,x_N)$, we can use this version when writing 
$\partial_r {u}_r(\mu)$.
\end{proof}

\section{The backward Kolmogorov equation} \label{sec:bke}
In this section we discuss existence, uniqueness and stability for a family of PDEs set on the space of probability measures. In Section \ref{sec: massive} we will relate the solution to such PDEs to an interacting particle system.
In the whole section we let $T$ be the same finite time horizon as at the beginning of Section \ref{se:2}.
Recall that $\mathcal{D}$ denotes the Dirichlet--Ferguson measure (see Definition~\ref{def:dfm}), 
$\overline{\mathcal{D}}$ the measure in~\eqref{eq:bmm}, 
and $H$ and $H^{1,2}$ the spaces introduced in~\eqref{eq:notation}. Moreover, $\rmD$ is defined in Proposition~\ref{prop:vectorgrad}, while $\boldsymbol{\Delta}$ is introduced below Theorem~\ref{thm:mainold}.
Given two functions $b_1, b_2 \in L^2(\prob(\T^d) \times \T^d, \overline{\mathcal{D}}; \R^d)$, we denote by $[b_1,b_2] \in L^2({\mathcal P}(\T^d), \cD)$ the function
\begin{equation}
\label{eq:[b1,b2]}
\sclprd{b_1}{b_2}_\mu := \int_{\prob(\T^d)} b_1(\mu, x) \cdot b_2(\mu, x) \de \mu(x), \quad \mathcal{D}\text{-a.e. }\mu \in \prob(\T^d). 
\end{equation}

We report here a useful result coming from \cite[Theorems 3.4, 3.6]{Brezis73}.

\begin{theorem}\label{thm:brezispde}
    Let $\mathcal{\sfH}$ be a separable Hilbert space, $\varphi: \mathcal{\sfH} \to [0,\infty]$ be a lower semicontinuous and convex functional, 
     $x_0 \in \overline{D(\varphi)}$, and 
     $f \in L^2([0,T]; \sfH)$. Then,  there exists a unique $x \in \rmC([0,T];\sfH) \cap \mathrm{AC}_{loc}((0,T); \sfH)$ such that $x_t \in D(\partial \varphi)$, for a.e.~$t \in (0,T)$, and
    \[ \partial_t x_t + \partial \varphi(x_t) \ni f_t \quad \text{ for a.e.~} t \in (0,T), \quad x_t  \bigr |_{t=0} = x_0. \]
   For such an $x$, it also holds that $t \mapsto \varphi(x_t)$ belongs to $L^1([0,T])$.
\end{theorem}

Note that the inclusion in the above statement is necessary since, in general, the subdifferential of $\varphi$ at $x_t$, denoted by $\partial \varphi (x_t) \subset \sfH$, is a set rather than a point. Therefore, we require that the element $f_t-\partial_t x_t \in \sfH$ belongs to this set.

We use the result above (which statement is very general) to establish the first contributions of the paper. The next one guarantees the solvability of the heat equation associated with \(\boldsymbol{\Delta}\), even when the source term and boundary condition are only in \(H\), see notation~\eqref{eq:notation}.

\begin{proposition}\label{prop:pde} Let $f \in L^2([0,T]; H)$ and $g \in H$. 
Then, there exists a unique $ u \in \rmC([0,T];H) \cap \mathrm{AC}_{loc}((0,T); H) \cap L^2([0,T]; H^{1,2})$ such that $u_t \in D(\boldsymbol{\Delta})$ for a.e.~$t \in (0,T)$ and 
\begin{equation}\label{eqn:bkw_heta_forcing} 
\partial_t u_t + \boldsymbol{\Delta} u_t= f_t \quad \text{for a.e.~} t \in (0,T), \quad u_t  \bigr |_{t=T} = g.
\end{equation}
\end{proposition}
\begin{proof}
First, let us observe that solving \eqref{eqn:bkw_heta_forcing} is equivalent to solving the forward equation
    \[
        \partial_t \tilde{u}_t - \boldsymbol{\Delta} \tilde{u}_t= - f_{T-t}\quad \text{for a.e.~} t \in (0,T), \quad \tilde{u}_t  \bigr |_{t=0} = g.
    \]
    Indeed, to pass from the forward formulation to the backward one, it suffices to set $u_t:= \tilde{u}_{T-t}$.
    Then, it is enough to apply Theorem \ref{thm:brezispde} with $\mathsf{H}:=H$, $\varphi:= \CE$ (so that $\partial \varphi = \{-\boldsymbol{\Delta}\}$ as stated in \cite[Proposition 5.2.4]{pasquabook}), $x_0 := g \in \overline{D(\varphi)}=H$ (see \cite[Proposition 5.2.6]{pasquabook}). This proves the existence of a solution $\tilde u \in \rmC([0,T];H) \cap \mathrm{AC}_{loc}((0,T);H)$ to the forward equation. To conclude, it remains to show that $\tilde{u} \in L^2([0,T]; H^{1,2})$. Since $t \mapsto \varphi(\tilde{u}_t) = \CE(\tilde{u}_t)$ belongs to $L^1([0,T])$, we already know that 
\[ \int_0^T \CE(\tilde{u}_t) \de t = \int_0^T \|\rmD \tilde{u}_t\| \de t < \infty.
\]
Therefore, we
are just left to show the strong measurability of $t \mapsto \rmD \tilde{u}_t$, cf.~\cite[Appendix E.5]{Evans02}. Since $H^{1,2}$ is separable, by Pettis' theorem (see e.g.~\cite[Theorem 7, p. 733]{Evans02}, this is equivalent to proving weak measurability.
 Using the strong measurability of $t \mapsto u_t$ and the integration by parts formula in \eqref{eq:ibp0}, we deduce that
 \[ t \mapsto \langle \rmD \tilde{u}_t, \boldnabla \hat v \rangle  = - (\tilde u_t, {\boldsymbol \Delta} \hat{v}) \text{ is measurable, for every } \hat v \in \hat{\mathfrak{Z}}_c^\infty. \]
Let us denote by $\mathrm{T}$ the subspace
\[ \mathrm{T}:= L^2(\prob(\T^d) \times \T^d, \overline{\cD}; \R^d) \text{ - closure of } \left \{ \boldnabla \hat v : \hat v \in \hat{\mathfrak{Z}}^\infty_c \right \}.\]
It follows that
\[ t \mapsto \langle \rmD \tilde{u}_t, \xi \rangle \quad \text{ is measurable, for every } \xi \in \mathrm{T}. \]
By Proposition \ref{prop:vectorgrad}, we have that $\rmD \tilde{u}_t \in \mathrm{T}$ for every $t \in [0,T]$, so that, denoting by $P_\mathrm{T}: L^2(\prob(\T^d) \times \T^d, \overline{\cD}; \R^d) \to \mathrm{T}$ the orthogonal projection on $\mathrm{T}$, we get
\[ t \mapsto \langle \rmD \tilde{u}_t, \xi \rangle = \langle \rmD \tilde{u}_t, P_{\mathrm{T}}\xi \rangle \quad \text{ is measurable, for every } \xi \in  L^2(\prob(\T^d) \times \T^d, \overline{\cD}; \R^d).\]
This shows that $t \mapsto \rmD \tilde{u}_t$ is weakly measurable and concludes the proof.
\end{proof}

\begin{remark}[On the choice of a representative of $\rmD u$]  
In the setting of Proposition \ref{prop:pde}, the map $t \mapsto \rmD u_t$ is  strongly measurable from $[0,T]$ to $L^2(\prob(\T^d) \times \T^d, \overline{\cD}; \R^d)$. 
It is natural to ask whether there exists a Borel function $(t,\mu,x) \mapsto \xi (t, \mu, x)$ such that (the $L^2(\prob(\T^d) \times \T^d, \overline{\cD}; \R^d)$-equivalence class represented by) $\xi(t, \cdot)$ coincides with $\rmD u_t$ for a.e.~$t \in [0,T]$. By \cite[Propositions 1.3.19, 1.3.20, 1.3.21]{pasquabook}, the answer is positive. Therefore, in the sequel, for a function $u \in L^2([0,T]; H^{1,2})$, the notation $\rmD u$ will refer to one such Borel representative.
\end{remark}

Our first main result establishes a form of regularization under the action of the operator \({\boldsymbol \Delta}\). The transport-diffusion equation on \(\mathcal{P}(\mathbb{T}^d)\), driven by \({\boldsymbol \Delta}\) and perturbed by a bounded, merely measurable velocity field, admits a unique solution in the sense defined below.

\begin{theorem}\label{thm:pde}Let $b: [0,T] \times  \prob(\T^d) \times \T^d \to \R^d$ be a bounded and measurable function, $f \in L^2([0,T]; H)$, and $g \in H$. Then, there exists a unique $u \in \rmC([0,T];H) \cap \mathrm{AC}_{loc}((0,T); H) \cap L^2([0,T]; H^{1,2})$
 such that $u_t \in D(\boldsymbol{\Delta})$ for a.e.~$t \in (0,T)$ and
\begin{equation}\label{eq:pdewithf}
\partial_t u_t + \sclprd{b_{t}}{\rmD u_t} + \boldsymbol{\Delta} u_t = f_t\quad \text{for a.e.~} t \in (0,T), \quad u_t  \bigr |_{t=T} = g.     
\end{equation}

\end{theorem}
\begin{proof}
For simplicity of exposition, we focus on the forward equation  
\begin{equation}
\partial_t u_t - \sclprd{b_{t}}{\rmD u_t} - \boldsymbol{\Delta} u_t = f_t \quad \text{for a.e.~} t \in (0,T), \quad  u_t \big|_{t=0} = g,     \label{eq:pde:forward:F}
\end{equation}  
which can be shown to be equivalent to \eqref{eq:pdewithf} by 
the same
simple time reversal as in the proof of Proposition \ref{prop:pde}. 
The analysis of 
\eqref{eq:pde:forward:F} relies on a classical fixed point argument, applied iteratively on successive short time intervals. To formulate this, 
we set, for any $\bar{T}>0$, $X_{\bar{T}} := L^2([0,\bar{T}];H^{1,2})$, with the corresponding Hilbertian norm being denoted by $\|\cdot\|_{X_{\bar{T}}}$.
Let us also recall that, throughout, $( \cdot, \cdot )_H$
and $|\cdot|_H$ stand respectively
for the scalar product and the norm in $H$, and that $\|\cdot\|$  denotes the norm on the space $L^2(\prob(\T^d)\times \T^d, \overline{\mathcal{D}}; \R^d)$, as defined
in Remark \ref{rem:notation1}. We divide the proof into several steps.\\

\textit{Claim 1}: For any $\bar{T} \in (0,T)$, $w_0 \in H$ and $v \in X_{\bar{T}}$, there exists a unique $u^v \in X_{\bar{T}}$ solving
\begin{equation}\label{eqn: diff_v}
    \partial_t u^v_t - \boldsymbol{\Delta} u^v_t =f^v_t \quad  \text{ for a.e.~} t \in (0,\bar{T}), \quad u^v_t \bigr |_{t=0} = w_0, 
\end{equation}
where 
$f^v_t :=  \sclprd{b_{t}}{\rmD v_t} + f_t$, for $t \in [0,T]$. This defines a map $\Phi_{\bar{T},w_0}: X_{\bar{T}} \to X_{\bar{T}}$ associating to every $v \in X_{\bar{T}}$ the function $u^v$ as above.\\
\smallskip

\textit{Proof of claim 1:} This follows directly from Proposition \ref{prop:pde}, with $f:=f^v$ therein,
noting that
$f^v$ belongs to 
$L^2([0,\bar T];H)$ because
$b$ is bounded and 
$u \in L^2([0,T];H^{1,2})$.

\quad \\
\textit{Claim 2}: Let $\bar{T}:= (4\sqrt{2}\|b\|_\infty)^{-2} \wedge T/2$. Then, for every $w_0 \in H$, the map $\Phi_{\bar{T}, w_0}$ as in claim 1 is a contraction on $X_{\bar{T}}$.\\
\textit{Proof of claim 2:} 
For $w_0 \in H$, and $v,\tilde{v}\in X_{\bar{T}}$, set $u := \Phi_{\bar{T}, w_0}(v)$, $\tilde{u} := \Phi_{\bar{T}, w_0}(\tilde{v})$, $h :=  \sclprd{b}{\rmD v}+f$, and $\tilde{h} :=  \sclprd{b}{\rmD \tilde{v}}+f$. 
Considering \eqref{eqn: diff_v} for both $u$ and $\tilde{u}$ and 
subtracting the resulting equations, we get
\begin{equation*}
        \partial_t (u_t - \tilde{u}_t) - \boldsymbol{\Delta} (u_t - \tilde{u}_t) = h_t - \tilde{h}_t \quad \text{ for a.e.~} t \in (0,\bar{T}), \quad (u_t - \tilde{u}_t) \bigr |_{t=0} = 0.
 \end{equation*}
    We integrate the equation against $u - \tilde{u}$ on $[0,t] \times \prob(\T^d)$, for $t \in [0,\bar{T}]$. 
    By \eqref{eq:ibp0} we get
    \begin{equation}
        |u_t -\tilde{u}_t|^2_H + 2\int_0^t \norm{\rmD u_r - \rmD \tilde{u}_r}^2\de r = 2 \int_0^t (h_t - \tilde{h}_t,u_r - \tilde{u}_r)_H\de r.
    \end{equation}
    Then, by Young's inequality, it holds
    \begin{equation}\label{eqn: young}
    \begin{aligned}
        |u_t -\tilde{u}_t|^2_H &+ \int_0^t \norm{\rmD u_r - \rmD \tilde{u}_r}^2\de r \\
        &\leq \frac{1}{2\bar{T}}\int_0^t |u_r - \tilde{u}_r|^2_H\de r + 2\bar{T}\int_0^t |h_t - \tilde{h}_t|^2_H\de r\\
        &\leq \frac{1}{2\bar{T}}\int_0^t |u_r - \tilde{u}_r|^2_H\de r + 2\bar{T}\norm{ b }_\infty^2\int_0^t \norm{\rmD v_r - \rmD\tilde{v}_r}^2\de r \\
        & \leq \frac{1}{2\bar{T}}\int_0^{\bar{T}} |u_r - \tilde{u}_r|^2_H\de r + 2\bar{T}\norm{ b }_\infty^2 \norm{v - \tilde{v}}^2_{X_{\bar{T}}},
    \end{aligned}
    \end{equation}
    where we have used the definitions of $h$ and $\tilde{h}$ together with the 
    fact that $b$ is bounded. In particular, neglecting the second term on the left-hand side and integrating over the interval $[0, \bar{T}]$, we obtain
\begin{equation}\label{eqn: l2_est_diff}
        \int_0^{\bar{T}} |u_t - \tilde{u}_t|_{H}^2 \de t\leq 4 \bar{T}^2 \norm{ b }^2_\infty \norm{v - \tilde{v}}^2_{X_{\bar{T}}}.
    \end{equation}
Returning to \eqref{eqn: young}, inserting \eqref{eqn: l2_est_diff}, neglecting now the first term on the left-hand side, and evaluating the second term on the left-hand side at time $t = \bar{T}$, we obtain

    \begin{equation}\label{eqn: l2_est_diff_2}
        \int_0^{\bar{T}} \norm{\rmD u_r - \rmD\tilde{u}_r}^2\de r \leq \frac{1}{2\bar{T}}\int_0^{\bar{T}} |u_r - \tilde{u}_r|^2_H \de r + 2\bar{T}\norm{ b }_\infty^2\norm{v - \tilde{v}}^2_{X_{\bar{T}}} \le 4\bar{T} \norm{ b }_\infty^2\norm{v - \tilde{v}}^2_{X_{\bar{T}}}.
    \end{equation}
    Finally, putting together \eqref{eqn: l2_est_diff} and \eqref{eqn: l2_est_diff_2}, we get
    \begin{align*}
        \norm{\Phi_{\bar{T}, w_0}(v) - \Phi_{\bar{T}, w_0}(\tilde{v})}_{X_{\bar{T}}}^2 & = \int_0^{\bar{T}}\left ( |u_t - \tilde{u}_t|^2_H + \norm{\rmD u_t - \rmD\tilde{u}_t}^2\right ) \de t\\
         &\leq 4\norm{ b }^2_\infty(\bar{T} + \bar{T}^2)\norm{v - \tilde{v}}^2_{X_{\bar{T}}}\\
         & \leq \frac{1}{4}\norm{v - \tilde{v}}^2_{X_{\bar{T}}}.
    \end{align*}
\textit{Conclusion}: Claim 2 and Banach fixed point theorem yield the existence and uniqueness (as stated in the present theorem) of a solution on every interval of the form $[t_0, t_0 + \bar{T}] \subset [0,T]$. Therefore the 
desired unique solution $u$ is obtained by successively patching together the solutions
on these intervals, noting also that $\bar{T}$ only depends on $\|b\|_\infty$ and $T$.
\end{proof}  

As a particular case, we can solve \eqref{eq:pdewithf} when the source term is equal to $0$. Since we will often refer to this in what follows, we state it as a corollary.
\begin{corollary}\label{cor:pde}Let $b: [0,T] \times \prob(\T^d) \times \T^d \to \R^d$ be a bounded and measurable function and $g \in H$. Then, there exists a unique $u \in \rmC([0,T];H) \cap \mathrm{AC}_{loc}((0,T); H) \cap L^2([0,T]; H^{1,2})$
 such that $u_t \in D(\boldsymbol{\Delta})$ for a.e.~$t \in (0,T)$ and
\begin{equation}\label{eq:pdeaaa}
\partial_t u_t + \sclprd{b}{\rmD u_t} + \boldsymbol{\Delta} u_t =0\quad \text{for a.e.~} t \in (0,T), \quad u_t  \bigr |_{t=T} = g.     
\end{equation}
\end{corollary}
We conclude this section with a general stability result for the transport-diffusion equation \eqref{eq:pdewithf} under perturbations of the drift $b$, the source term $f$, and the initial datum $g$.

\begin{proposition}\label{prop_stab_new}
    Let $b,b^n: [0,T] \times \prob(\T^d) \times \T^d \to \R^d$, $n \in \N_+$, be bounded and measurable functions such that $b^n \to b$ in $L^2([0,T] \times \prob(\T^d)\times\T^d, \mathscr{L}^{[0,T]} \otimes  \overline{\mathcal{D}}; \R^d)$ and $M:=\sup_{n \in \N_+} \|b^n\|_\infty < \infty$. Let $g, g^n \in H$, $f,f^n \in L^2([0,T];H)$, $n \in \N_+$,  be such that $g^n \to g$ in $H$ and $f^n \to f$ in $L^2([0,T]; H)$ as $n \to \infty$. If  $u$ and $u^n$, $n \in \N_+$, are the unique solutions to \eqref{eq:pdewithf} with drifts $b$ and $b^n$, source term $f$ and $f^n$, and final datum $g$ and $g^n$, respectively, then 
    \begin{equation}
\label{eq:prop:3.6:main:bound}
\begin{split}
    &\sup_{t \in [0,T]} |u^n_t-u_t|_H^2 + \int_0^T\|\rmD u_t^n - \rmD u_t \|^2 \de t  \\
    & \qquad \qquad \le C(M, T) \left (  |g-g^n|_H^2 + \int_0^T |f_t-f_t^n|_H^2 \de t + \mathcal{R}(b, b^n, \rmD u, T) \right ) \to 0 \text{ as } n \to \infty,
    \end{split}
    \end{equation}
    where 
    \[
   \mathcal{R(}b,b^n,\rmD u,T):= \int_0^T \int_{\prob(\T^d)} \int_{\T^d} \lvert b_r(\mu, x) - b_r^n(\mu, x) \rvert^2 \lvert \rmD u_r(\mu,x) \rvert^2 \de \mu(x) \de \cD(\mu) \de r,
    \]
    and 
    $C(M,T):= 1+ \rme^{T(2+M^2)} + (2+M^2)T\rme^{T(2+M^2)}$.
In particular, the left-hand side of \eqref{eq:prop:3.6:main:bound} tends to $0$ as $n$ tends to $\infty$.
\end{proposition}
\begin{proof} By subtracting the equations solved on $\prob(\T^d) \times (0,T)$ by $u$ and $u^n$ respectively, we obtain that, for a fixed $n \in \N_+$, $u-u^n$ satisfies 
    \begin{equation*}
        \partial_t(u_t - u^n_t)+\boldsymbol{\Delta}(u_t - u^n_t) = - (  \sclprd{b_{t}} {\rmD u_t} - \sclprd{b_{t}^n}{ \rmD u^n_t}) +( f_t-f_t^n)  \quad  \text{ for a.e.~} t \in (0,T),\quad (u_t-u^n_t)\bigr |_{t=T} = g-g^n.
    \end{equation*}
    Then, by integrating the equation against $u-u^n$ on $[t,T]\times\prob(\T^d)$, for a given $t \in [0,T]$, and by integrating by parts the term driven by ${\boldsymbol \Delta}$ using \eqref{eq:ibp0}, we get
     \begin{align}
        \frac{1}{2}|u_t - u^n_t|^2_H &+ \int_t^T \norm{\rmD u_r - \rmD u^n_r}^2 \de r\nonumber \\
        & =  \frac{1}{2}|g-g^n|^2_H + \int_t^T (f_r^n-f_r, u_r-u_r^n)_H \de r + \int_t^T (\sclprd{b_r}{\rmD u_r} - \sclprd{b_r^n}{\rmD u^n_r},u_r - u^n_r)_H\de r\nonumber\\
        & \le \frac{1}{2}|g-g^n|^2_H + \frac{1}{2}\int_0^T |f_r-f_r^n|^2_H \de r + \frac{1}{2} \int_t^T |u_r-u_r^n|^2_H \de r\label{eqn: stab_rhs0}\\
        &\quad + \int_t^T \big( \sclprd{b_{r} - b_{r}^n}{\rmD u_r} + \sclprd{b_{r}^n}{\rmD u_r - \rmD u^n_r}, u_r - u^n_r\big )_H\de r\label{eqn: stab_rhs},
    \end{align}
    where we also used Young's inequality.
    Let us focus on \eqref{eqn: stab_rhs}. Regarding the second summand, 
    we have, by Young's inequality,
    \begin{align}\label{eq:3132}
        \int_t^T \big(\sclprd{b_{r}^n }{\rmD u_r - \rmD u^n_r},u_r - u^n_r\big)_H\de r \leq \frac{1}{2}\int_t^T \norm{\rmD u_r - \rmD u^n_r}^2\de r + \frac{M^2}{2}\int_t^T |u_r-u^n_r|^2_H\de r.
    \end{align}
To estimate the first summand in \eqref{eqn: stab_rhs}, we proceed as follows: by Young's and Jensen inequalities, 
    \begin{equation}
    \label{eq:3131}
    \begin{split}
        \int_t^T  &\big( \sclprd{b_{r} - b_{r}^n}{ \rmD u_r}, u_r - u^n_r\big)_H\de r \\
         &\qquad\leq \frac{1}{2}\int_t^T \int_{\prob(\T^d)} \int_{\T^d} \lvert b_r(\mu, x) - b_r^n(\mu, x) \rvert^2 \lvert \rmD u_r(\mu,x) \rvert^2 \de \mu(x) \de \cD(\mu) \de r 
         \\
         &\qquad \hspace{15pt} + \frac{1}{2}\int_t^T |u_r - u^n_r|^2_H\de r\\
         & \qquad\leq \frac{1}{2} \mathcal{R}(b, b^n, \rmD u, T) + \frac{1}{2}\int_t^T |u_r - u^n_r|^2_H\de r.
         \end{split}
    \end{equation}
    Combining \eqref{eqn: stab_rhs0}, \eqref{eqn: stab_rhs}, \eqref{eq:3132}, and \eqref{eq:3131}, we get
    \begin{align*}\label{eq:readytogron}
        |u_t-u_t^n|_H^2 + \int_t^T \|\rmD u_r- \rmD u_r^n\|^2 \de r &\le  (2+ M^2) \int_t^T |u_r-u_r^n|_H^2 \de r +|g-g^n|_H^2 + \int_0^T |f_r-f_r^n|_H^2 \de r \\
        & \quad + \mathcal{R}(b, b^n, \rmD u, T).
    \end{align*}
    The inequality in the statement follows then by the above inequality and a standard application of Gronwall's lemma. To conclude the proof, we are left to show that
    \begin{equation}\label{eq:subsub}
        \mathcal{R}(b, b^n, \rmD u, T) \to 0 \quad \text{ as } n \to \infty.
    \end{equation}
    Since $b^n \to b$ in $L^2([0,T] \times \prob(\T^d)\times\T^d, \mathscr{L}^{[0,T]} \otimes  \overline{\mathcal{D}}; \R^d)$, we have
    \[ \int_0^T \int_{\prob(\T^d)} \int_{\T^d} \left | b_r(\mu,x)-b_r^n(\mu, x) \right |^2 \de \mu(x) \de \cD(\mu) \de r \to 0 \quad \text{ as } n \to \infty,\]
which implies that, as 
$n \to \infty$,
\begin{equation*}
\left( (r,\mu,x) \mapsto   \left | b_r(\mu,x)-b_r^n(\mu, x) \right |^2  \right) \rightarrow 0,
\end{equation*}
in
${\mathscr L}^{[0,T]} \otimes 
\overline{\mathcal D}$-measure over 
$[0,T] \times \cP(\T^d) \times \T^d$. Since $b$ and $(b^n)_n$ are uniformly bounded by the constant $M$,
and the function 
\begin{equation*}
(r,\mu,x) \mapsto 
 |\rmD u_r(\mu, x) |^2 
\end{equation*}
is integrable on $[0,T] \times \cP(\T^d) \times \T^d$ under ${\mathscr L}^{[0,T]} \otimes \overline{\mathcal D}$, we deduce that 
the collection 
\begin{equation*}
\left( (r,\mu,x) \mapsto \left | b_r(\mu,x)-b_r^n(\mu, x) \right |^2 
|\rmD u_r(\mu, x) |^2\right)_n
\end{equation*}
is uniformly integrable under 
${\mathscr L}^{[0,T]} \otimes \overline{\mathcal D}$. As it converges to $0$
in ${\mathscr L}^{[0,T]} \otimes \overline{\mathcal D}$-measure, this proves that it also converges in $L^1$-norm under ${\mathscr L}^{[0,T]} \otimes \overline{\mathcal D}$-measure, which completes the proof.
\end{proof}

\section{The (massive) interacting particle system}\label{sec: massive}

The purpose of this section is to provide a probabilistic interpretation of the solution to the transport-diffusion equation \eqref{eq:pdeaaa} using a particle system built from the free particle system presented in Subsection \ref{ssec:freeparty}. More generally, the entire section is based on the material introduced earlier, and in particular on 
the notation defined in Section \ref{se:2}. In the whole section we let $T$ be the same finite time horizon as at the beginning of Section \ref{se:2}.

Given a collection of uniformly bounded and measurable drifts $(b^i\colon[0,T]\times\prob(\T^d)\times \T^d\to\R^d)_i$ (corresponding, at least when all the $b^i$'s are the same, to the velocity field in the PDE \eqref{eq:pdeaaa}), an initial time $t_0\in[0,T)$, and an initial distribution $\cursm\in\prob(T_o^\infty\times (\T^d)^\infty_o)$, we aim at solving (in a suitable weak sense to be precised later) the following system
\begin{equation}\label{eqn: particles_drift}
    \begin{cases}
        \de X^i_t & = b_t^i(\mu^\infty_t, X^i_t)\de t +  \sqrt{\frac{2}{s_i}}\de W^i_t,\quad t\in[t_0,T],\, i \in \N_+,\\
        X^i_{t_0} & = \imath(x_i), \quad i \in \N_+,
    \end{cases}
\end{equation}
with $(W_i)_i$ being a family of independent $d$-dimensional $\{\cF_t\}_{t \in [t_0,T]}$-Brownian motions, the initial datum $(\mbfs=(s_i)_i, \mbfx=(x_i)_i)$ has distribution $\cursm$, and where $\mu^\infty_t := \sum_{j=1}^{\infty} s_j\delta_{X^j_t}$, $t \in [t_0,T]$ is the $\prob(\T^d)$-valued process defined as in \eqref{eq:project}.
The necessity of allowing $b$ to depend on $i$ will become apparent in Section~\ref{sec: control}. We emphasize that, by uniform bounded collection, we mean that $\sup_{i \in {\mathbb N}_+} \| b^i \|_\infty < + \infty$. Note also that we are using the convention in Remark \ref{rem:toruspb} to evaluate $b_t^i(\mu_t^\infty, \cdot)$ at $X_t^i$.

\subsection{Weak existence and uniqueness in law}\label{ssec:well_posedness_partycles}

\subsubsection{Notion of weak solution and related statements}

Here we address the well-posedness of equation \eqref{eqn: particles_drift}. Given the low regularity of the drift, it is natural to look for weak solutions. In particular, for a collection of uniformly bounded and measurable drifts $(b^i\colon[0,T]\times\prob(\T^d)\times \T^d\to\R^d)_{i}$, an initial time $t_0\in[0,T)$, and an initial distribution $\cursm\in\prob(T^\infty_o\times(\T^d)_o^{\infty})$, we say that the tuple $({\Omega},{\cF},{\Q},\{{\cF}_t\}_{t\in[0,T]}, (({\mbfs},{\mbfx}), ({X}^i)_i, ({W}^i)_i))$ is a \emph{weak solution} of the equation \eqref{eqn: particles_drift} driven by the data $(b^i)_i$, $t_0$, $\cursm$ if:
\begin{enumerate}[i.]
    \item $({W}^i)_i$ is a family of independent $d$-dimensional $\{{\cF}_t\}_{t\in[t_0,T]}$-Brownian motions on the filtered probability space $({\Omega},{\cF},{\Q},\{{\cF}_t\}_{t\in[t_0,T]})$;
    \item The initial datum $({\mbfs},{\mbfx})$ is a $T^\infty_o\times(\T^d)^\infty_o$-valued ${\cF}_{t_0}$-measurable random variable distributed according to $\cursm$ under ${\Q}$;
    \item It holds:\begin{equation*}
        {X}^i_t = \imath({x}^i) + \int_{t_0}^t b_r^i({\mu}^\infty_r, {X}^i_r) \de r +  \sqrt{\frac{2}{{s_i}}} W_t^i,\quad t\in[t_0,T],\quad i\in\N_+,\quad {\Q}\text{-a.s.}
    \end{equation*}
    where ${\mu}^\infty_t := \sum_{j=1}^{\infty} {s}_j\delta_{{X}^j_t}$, $t \in [t_0,T]$ is the $\prob(\T^d)$-valued process defined as in \eqref{eq:project}.
\end{enumerate}
Note that the solution depends on $t_0$, $\cursm$, and $(b^i)_i$. For the moment, we suppress this dependence to keep the notation lighter, but we will make it explicit whenever necessary.

As in the finite-dimensional setting, weak existence and uniqueness in law follow from a standard application of Girsanov theorem. Weak existence is ensured by the following statement.
\begin{proposition}\label{prop: weak_existence} Let $(b^i:[0,T]\times\prob(\T^d) \times \T^d \to \R^d)_i$ be a collection of uniformly bounded and measurable functions, $t_0\in [0,T)$ be an initial time, and $\cursm\in\prob(T^\infty_o \times (\T^d)^\infty_o)$ be an initial distribution.
Then, there exists a weak solution to the infinite system \eqref{eqn: particles_drift} driven by the data $(b^i)_i$, $t_0$, $\cursm$.
\end{proposition}
\begin{proof} Without loss of generality, let us take $t_0=0$.
As in Subsection \ref{ssec:freeparty}, let us then consider a filtered probability space $(\Omega, \cF, \PP, \{\cF_t\}_{t\in[0,T]})$, an $\cF_0$-measurable initial condition $(\mbfs,\mbfx)\sim_\PP\cursm$ and a sequence of
independent $d$-dimensional  $\{\cF_t\}_{t\in[0,T]}$-Brownian motions $(B^i)_i$. 
We recall that the corresponding free particle system $(X^i)_i$ is given by
\begin{equation}\label{eqn: particles_free_new}
    \d X^i_t =  \sqrt{\frac{2}{s_i}}\de B^i_t,\quad t\in[0,T],\quad  X^i_0 =\imath(x_i), 
    \quad i \in \N_+.
\end{equation}
We introduce the local martingale
\begin{equation}\label{eqn: girsanov_mg}
    Z_t:=\exp\left\{\sum_{i = 1}^{\infty} \int_0^t\sqrt{\frac{s_i}{2}} b_r^i(\mu^\infty_r, X^i_r) \cdot\de B^i_r - \frac{1}{2}\sum_{i=1}^{\infty} \int_0^t \frac{s_i}{2}\lvert  b_r^i(\mu^\infty_r, X^i_r) \rvert^2\de r\right\}, 
    \quad t \in [0,T],
\end{equation}
where $\mu^\infty_t := \sum_{j=1}^{\infty} s_j\delta_{X^j_t}$, $t \in [0,T]$ is as in \eqref{eq:project}. Since $\sum_{i=1}^{\infty} s_i = 1$ and $\sup_{i \in {\mathbb N}_+} \norm{b^i}_\infty<\infty$, the sum of the stochastic integrals inside the bracket forms a martingale whose quadratic variation is bounded by a deterministic constant. Therefore, $Z$ is a true martingale with $Z_0=1$. Thus, by the infinite dimensional version of
Girsanov theorem (see e.g.~\cite[Proposition I.0.6]{LiuRock}), the measure ${\mathbb Q} := Z_T  {\mathbb P}$ is a probability measure, under which the processes
\begin{equation}\label{eqn: girsanov}
    W^i_t : = B^i_t - \int_0^t \sqrt{\frac{s_i}{2}}  b_r^i(\mu^\infty_r, X^i_r)\de r,\quad t \in [0,T],\quad i\in\N_+,
\end{equation}
are independent ${\{\cF_t\}_{t\in[0,T]}}$-Brownian motions.
Then, the result is a consequence of the following two observations.
First, notice that the law of $(\mbfs,\mbfx)$
remains unchanged under the new probability $\Q$.
Second, by plugging \eqref{eqn: girsanov} into \eqref{eqn: particles_free_new}, we obtain, for any $i \in \N_+$,
\begin{equation*}
    \de X^i_t = b_t^i(\mu^\infty_t, X^i_t)\de t +\sqrt{\frac{2}{s_i}}\de W^i_t, \quad t \in [0,T],
\end{equation*}
which means that $(\Omega,\cF,\Q, \{\cF_t\}_{t\in[0,T]}, ((\mbfs,\mbfx), (X^i)_i, (W^i)_i))$ is a weak solution to \eqref{eqn: particles_drift}.\end{proof}
We now address uniqueness (in law).
\begin{proposition}
\label{prop:weak:uniqueness}
Let ($b^i:[0,T]\times\prob(\T^d) \times \T^d \to \R^d)_{i}$ be a collection of uniformly bounded and measurable functions, $t_0\in [0,T)$ be an initial time, and $\cursm\in\prob(T^\infty_o \times (\T^d)^\infty_o)$ be an initial distribution. If \[(\Omega^{(k)},\cF^{(k)},\Q^{(k)}, \{\cF^{(k)}_t\}_{t\in[t_0,T]}, ((\mbfs^{(k)},\mbfx^{(k)}),(X^{i,(k)})_i, (W^{i,(k)})_i)),\quad k=1,2,
\]are
two weak solutions to \eqref{eqn: particles_drift}, both driven by the data $(b^i)_i, t_0,\cursm$, then $((\mbfs^{(1)},\mbfx^{(1)}), (X^{i,(1)})_i,(W^{i,(1)})_i)$ and $((\mbfs^{(2)},\mbfx^{(2)}),(X^{i,(2)})_i,(W^{i,(2)})_i)$ have the same law on $\Xi$ under the probability measures ${\mathbb Q}^{(1)}$ and 
${\mathbb Q}^{(2)}$ respectively.
\end{proposition}
\begin{proof} Following the proof of Proposition \ref{prop: weak_existence}, we assume $t_0=0$ for simplicity.
Moreover,
from the two properties
$\sup_{i \in {\mathbb N}_+}
\| b^i \|_\infty<\infty$ and $\sum_{i=1}^\infty s^{(k)}_i = 1$, 
we notice that,
for $k = 1,2$, the process defined by 
    \begin{equation*}
    \zeta^{(k)}_t:=\exp\left\{-\sum_{i = 1}^{\infty} \int_0^t\sqrt{\frac{s^{(k)}_i}{2}}b_r^i(\mu^{\infty,(k)}_r, X^{i,(k)}_r) \cdot \de W^{i,(k)}_r - \frac{1}{2}\sum_{i=1}^{\infty} \int_0^t \frac{s^{(k)}_i}{2}\lvert b_r^i(\mu^{\infty,(k)}_r, X^{i,(k)}_r) \rvert^2\de r\right\}, \quad 
    t \in [0,T],
    \end{equation*}
    with $\mu^{\infty,(k)}_t := \sum_{j=1}^{\infty} s^{(k)}_j\delta_{X^{j,(k)}_t}$, $t \in [0,T]$ as in \eqref{eq:project},
    is a $\Q^{(k)}$-martingale starting from $\zeta^{(k)}_0 = 1$. Then, by the same infinite dimensional version of
    Girsanov theorem as in the proof of Proposition \ref{prop: weak_existence}, the processes $(B^{i,(k)})_i$, defined 
    for $k=1,2$ by\begin{equation}\label{eqn: girsanov_inv}
        B^{i,(k)}_t : = W^{i,(k)}_t + \int_0^t\sqrt{\frac{s^{(k)}_i}{2}} b_r^i(\mu^{\infty,(k)}_r, X^{i,(k)}_r)\de r,\quad t \in [0,T],\quad i\in\N_+
    \end{equation}
    are, for each $k \in \{1,2\}$, independent $d$-dimensional $\{\cF_t\}_{t \in [0,T]}$-
    Brownian motions under the new probability measure $\PP^{(k)}:=\zeta_T^{(k)} {\mathbb Q}^{(k)}$ (i.e., 
    $\de \PP^{(k)} / \de\Q^{(k)} = \zeta^{(k)}_T$). By plugging \eqref{eqn: girsanov_inv} into \eqref{eqn: particles_drift}, we obtain
    \begin{equation*}
        \de X^{i,(k)}_t = \sqrt{\frac{2}{s^{(k)}_i}}\de B^{i,(k)}_t,\quad X^{i,(k)}_0 = \imath(x^{(k)}_i), \quad t \in [0,T], \quad i \in \N_+,
    \end{equation*}
which proves (by returning
back to 
\eqref{eqn: girsanov_inv}) 
that  the tuples 
\[((\mbfs^{(1)},\mbfx^{(1)}), (X^{i,(1)})_i,(B^{i,(1)})_i) \text{ and }((\mbfs^{(2)},\mbfx^{(2)}), (X^{i,(2)})_i,(B^{i,(2)})_i)\] have the same law (on 
the space $\Xi$) under the probability measures ${\mathbb P}^{(1)}$ and 
${\mathbb P}^{(2)}$ respectively. Moreover, we can rewrite the inverse of 
$\zeta_t^{(k)}$ in the form of an exponential martingale under ${\mathbb P}^{(k)}$, namely
\begin{equation*}
\frac1{\zeta_t^{(k)}}=\exp\left\{\sum_{i = 1}^{\infty} \int_0^t\sqrt{\frac{s^{(k)}_i}{2}}b_r^i(\mu^{\infty,(k)}_r, X^{i,(k)}_r) \cdot \de B^{i,(k)}_r - \frac{1}{2}\sum_{i=1}^{\infty} \int_0^t \frac{s^{(k)}_i}{2}\lvert b_r^i(\mu^{\infty,(k)}_r, X^{i,(k)}_r) \rvert^2\de r\right\}, \ \
t \in [0,T].
\end{equation*}
Thanks to 
\cite[Lemmata 4.3.2 \& 4.3.3]{StroockVaradhan}, this
shows that $1/\zeta_T^{(1)}$
and $1/\zeta_T^{(2)}$
can be expressed as a common measurable function of $((\mbfs^{(k)},\mbfx^{(k)}),(X^{i,(k)})_i,(B^{i,(k)})_i)$.
Now, 
for each $i \in \N_+$, we introduce the map (with the superscript ${\boldsymbol b}$ in 
$\psi_i$ below being a shorthand notation for ${\boldsymbol b}=(b^j)_j$)
\begin{equation}\label{eqn:psi_i_new}
\begin{split}
\psi_i^{\boldsymbol b} : 
T^\infty_o\times\rmC([0,T];{\R}^d)^{\infty} \times 
\rmC([0,T];{\R}^d)
&\rightarrow \rmC([0,T];{\mathbb R}^d)
\\
\left((s_j)_j, ((x_t^j)_{t \in [0,T]})_j,
(w_t)_{t \in [0,T]}\right) &\mapsto 
\left(
w_t - \sqrt{\frac{s_i}{2}} \int_0^t b_r^i\left( 
\sum_{j =1 }^\infty
s_j \delta_{x_r^j}
,x_r^i
\right) \de r
\right)_{t \in [0,T]},
\end{split}
\end{equation}
which is clearly progressively-measurable. Again, we have used the convention that $x_t^j$ is implicitly composed with the projection map as explained in Remark \ref{rem:toruspb}. Then, we can let 
\begin{equation}\label{eqn:big_psi_i_new}
\begin{split}
\Psi_i^{\boldsymbol b} : 
T^\infty_o\times(\T^d)^\infty_o\times\rmC([0,T];{\R}^d)^{\infty} \times 
\rmC([0,T];{\R}^d)
&\rightarrow T^\infty_o\times(\T^d)^\infty_o\times\rmC([0,T];{\mathbb R}^d)\times\rmC([0,T];{\mathbb R}^d)
\\
\left((s_j)_j,(x_j^0)_j, (x^j)_j,
w\right) &\mapsto 
\left((s_j)_j,(x_j^0)_j, x^i, \psi_i^{\boldsymbol b}\left((s_j)_j, (x^j)_j,
w\right)\right).
\end{split}
\end{equation}
Writing
\begin{equation*}
(\mbfs^{(k)},\mbfx^{(k)}, X^{i,(k)},W^{i,(k)})
 = \Psi_i^{\boldsymbol b}
 \left( \mbfs^{(k)},\mbfx^{(k)}, 
(X^{j,(k)})_j,B^{i,(k)}
 \right),
\end{equation*}
we deduce that for any integer $n \geq 1$, and any Borel subsets 
$\Gamma_1,\cdots,\Gamma_n$ of 
$T_o^\infty\times(\T^d)^\infty_o\times\rmC([0,T];{\mathbb R}^d) \times \rmC([0,T];{\mathbb R}^d)$,
\begin{align*}
        &\Q^{(1)}\left(
        \left\{
        (\mbfs^{(1)},\mbfx^{(1)},X^{1,(1)},W^{1,(1)}) \in 
        \Gamma_1,
        \dots, 
        (\mbfs^{(1)},\mbfx^{(1)},X^{n,(1)},W^{n,(1)}) \in \Gamma_n
        \right\}
        \right)
        \\
        &=
\Q^{(1)}\left(
\left\{ 
(\mbfs^{(1)},\mbfx^{(1)},X^{1,(1)},B^{1,(1)}) \in 
\Psi_1^{-1}(\Gamma_1),
        \dots, 
        (\mbfs^{(1)},\mbfx^{(1)},X^{n,(1)},B^{n,(1)}) \in \Psi_n^{-1}( \Gamma_n)
        \right\}
        \right)
        \\
        & = \int_{\Omega^{(1)}} \frac{1}{\zeta^{(1)}_T}\mathbf{1}_{\{(\mbfs^{(1)},\mbfx^{(1)},X^{1,(1)},B^{1,(1)}) \in 
\Psi_1^{-1}(\Gamma_1),
        \dots, 
        (\mbfs^{(1)},\mbfx^{(1)},X^{n,(1)},B^{n,(1)}) \in \Psi_n^{-1}( \Gamma_n)\}}\de \PP^{(1)}\\
        & = \int_{\Omega^{(2)}} \frac{1}{\zeta^{(2)}_T}\mathbf{1}_{\{(\mbfs^{(2)},\mbfx^{(2)},X^{1,(2)},B^{1,(2)}) \in 
\Psi_1^{-1}(\Gamma_2),
        \dots, 
        (\mbfs^{(2)},\mbfx^{(2)},X^{n,(2)},B^{n,(2)}) \in \Psi_n^{-1}( \Gamma_n)\}}\de \PP^{(2)}\\
        & = \Q^{(2)}\left(
        \left\{
        (\mbfs^{(2)},\mbfx^{(2)},X^{1,(2)},W^{1,(2)}) \in 
        \Gamma_1,
        \dots, 
        (\mbfs^{(2)},\mbfx^{(2)},X^{n,(2)},W^{n,(2)}) \in \Gamma_n
        \right\}
        \right),
    \end{align*}
where, to pass from the third to the fourth line, we used the fact that the integrands on both lines could be written as the images, by 
a common measurable mapping, of $(\mbfs^{(1)},\mbfx^{(1)},(X^{1,(1)},B^{1,(1)}),\dots,(X^{n,(1)},B^{n,(1)}))$
and $(\mbfs^{(2)},\mbfx^{(2)},(X^{1,(2)},B^{1,(2)}),\dots,(X^{n,(2)},B^{n,(2)}))$
respectively. This completes the proof.
\end{proof}

\subsubsection{Transfer onto the canonical space}\label{ssec:transfer}
For later use, it is important to transfer weak solutions onto the canonical space
$\Xi = T^\infty_o \times (\T^d)^\infty_o \times \rmC([0,T];\R^d)^{\infty} \times \rmC([0,T];\R^d)^{\infty}$ (see Subsection \ref{ssec:canonical}).
Given a (unique in law) weak solution $({\Omega},{\cF},{\Q},\{{\cF}_t\}_{t\in[t_0,T]}, (({\mbfs},{\mbfx}), ({X}^i)_i, ({W}^i)_i))$, we can indeed denote the law of the tuple $(({\mbfs},{\mbfx}), ({X}^i)_i, ({W}^i)_i)$ on $\Xi$ by $\mathbf{Q}^{t_0,\cursm}$, where we highlight the dependence on the initial time and distribution. Then, it is easy to check that $(\Xi,\cG,\{\cG_t\}_{t\in[0,T]}, \mathbf{Q}^{t_0,\cursm}, (((\varsigma^j)_j,(\xi^{0,j})_j),(\beta^j)_j,(\xi^j)_{j}
))$ is another weak solution to \eqref{eqn: particles_drift}, where we used the notation introduced in Subsection \ref{ssec:canonical} for the canonical process. Since the probability measure $\mathbf{Q}^{t_0,\cursm}$ fully characterizes this solution, in what follows we simply refer to the weak solution on the canonical space through its law $\mathbf{Q}^{t_0,\cursm}$. As made clear in the following statement, we can express $\mathbf{Q}^{t_0,\cursm}$ in terms of the law on the canonical space $\Xi$ of the free particle system \eqref{eqn: particles_free}, denoted by $\mathbf{P}^{t_0,\cursm}$.
\begin{proposition}\label{prop:pandq}
   Let $(b^i:[0,T]\times\prob(\T^d)\times \T^d \to \R^d)_i$ be a collection of uniformly bounded and measurable functions, $t_0\in [0,T)$ an initial time, and $\cursm\in\prob(\T^\infty_o \times (\T^d)^\infty_o)$ an initial distribution. Denote by $\mathbf{Q}^{t_0,\cursm}$ the law on $\Xi$ of a weak solution to \eqref{eqn: particles_drift}, and by $\mathbf{P}^{t_0,\cursm}$ the law on $\Xi$ of the free particle system \eqref{eqn: particles_free}, respectively. Then
   \begin{equation}\label{eq:new:expression:Qcursm0}
        {\mathbf Q}^{t_0,\cursm}= {\cercle{$\Psi$}}^{\boldsymbol b}_\sharp \left( \frac1{z_T^{t_0,\cursm}}  {\mathbf P}^{t_0,\cursm}\right),
    \end{equation}
    where the mapping ${\cercle{$\Psi$}}^{\boldsymbol b}\colon\Xi\to\Xi$ is defined by
    \begin{equation}\label{eq:notation:cercle}
        {\cercle{$\Psi$}}^{\boldsymbol b} : \left((s_j)_j,(x_j^0)_j,(w^j)_j,(x^j)_j\right) \mapsto \left( (s_j)_j,(x_j^0)_j,\left(\psi_j^{\boldsymbol b}((s_i)_i,(x^i)_i,w^j)\right)_j,(x^j)_j \right),
    \end{equation}
    $\psi_i^{\boldsymbol b}$ is given by \eqref{eqn:psi_i_new} for any $i\in\N_+$, and
    \begin{equation}\label{eqn:z_T}
         \frac1{z^{t_0,\cursm}_T} = \exp\left( \sum_{i=1}^\infty \int_{t_0}^T \sqrt{\frac{\varsigma^i}{2}}\, b_r^i(\mu^\infty_r,\xi_r^i)\cdot \de \beta_r^i - \frac12 \sum_{i=1}^\infty \int_{t_0}^T \frac{\varsigma^i}{2}\, \vert b_r^i(\mu^\infty_r,\xi_r^i)\vert^2 \de r \right),
    \end{equation}
    with $\mu^\infty := \sum_{j=1}^\infty\varsigma^i\delta_{\xi^i}$ (see \eqref{eq:muinfmap}).  
\end{proposition}
Note that in \eqref{eqn:z_T} (and in the rest of this section) we are using the same convention explained in Remark \ref{rem:toruspb} to evaluate $b^i_r(\mu_r^\infty, \cdot)$ at $\xi_r^i$.
\begin{proof}
For $n\in\N_+$ and two families $\Gamma_1^x,\cdots,\Gamma_n^x$ and 
$\Gamma_1^w,\cdots,\Gamma_n^w$ of Borel subsets of $\rmC([0,T];{\mathbb R}^d)$, let us consider the two cylinders $\Gamma_1^x \times \cdots \times \Gamma_n^x \times \rmC([0,T];{\mathbb R}^d)^{\{n+1,\cdots\}}$
and 
$\Gamma_1^w \times \cdots \times \Gamma_n^w \times \rmC([0,T];{\mathbb R}^d)^{\{n+1,\cdots\}}$. Following the proof of Proposition
\ref{prop:weak:uniqueness}, we know that for an initial distribution 
$\cursm$ on $T^\infty_o \times ({\mathbb T}^d)^{\infty}_o$ and for 
a Borel subset $\Gamma_0$ of 
$T^\infty_o \times ({\mathbb T}^d)^{\infty}_o$,
the law of the weak solution to \eqref{eqn: particles_drift}, with 
$\cursm$ as initial condition,
is given by 
\begin{equation}
\begin{split}
&{\mathbf Q}^{t_0,\cursm}\left( 
\Gamma_0 \times 
\Gamma_1^x \times \cdots \times \Gamma_n^x \times \rmC([0,T];{\mathbb R}^d)^{\{n+1,\cdots\}}\times
\Gamma_1^w \times \cdots \times \Gamma_n^w \times \rmC([0,T];{\mathbb R}^d)^{\{n+1,\cdots\}}
\right) 
\\
&=  \int_{\Xi}
\frac1{z^{t_0,\cursm}_T}
\prod_{i=1}^n 
{\mathbf 1}_{(\Psi^{\boldsymbol b}_i)^{-1}(\Gamma_0\times\Gamma^x_i \times \Gamma^w_i)}
((\varsigma^j)_j,
(\xi^{0,j})_j,\xi^i,\beta^i) 
\de {\mathbf P}^{t_0,\cursm},
\end{split}
\end{equation}
where $\frac{1}{z_T^{t_0,\cursm}}$ is the martingale given by \eqref{eqn:z_T} and $\Psi_i^{\boldsymbol b}$ is defined by \eqref{eqn:big_psi_i_new}. Then \eqref{eq:new:expression:Qcursm0} follows.
\end{proof}

Of course, it must be stressed that the random variable $z_T^{t_0,\cursm}$ is here regarded as a measurable function on $\Xi$. At this stage, the measurable function representing $z_T^{t_0,\cursm}$ depends on $\cursm$, and one of the purposes of the forthcoming statements is to provide a version that is independent of $\cursm$.

\begin{remark}[Dependence of $\mathbf{Q}^{t_0, \cursm}$ and $z_T^{t_0, \cursm}$ on $(b^i)_i$] The map $z_T^{t_0, \cursm}: \Xi \to \R$ and the probability $\mathbf{Q}^{t_0, \cursm}$ depend also on the collection of drifts $(b^i)_i$, but we do not stress this dependence in the notation, unless it is needed.
\end{remark}
Similarly to what we did for $\mathbf P^{t_0, \cursm}$, we adopt a lighter notation when dealing with a deterministic initial condition $(\mbfs, \mbfx) \in T_o^\infty \times (\T^d)^\infty_o$, setting
\begin{equation}
\label{eq:Q:t0:s,x}
 {\mathbf Q}^{t_0, (\mbfs,\mbfx)}:= {\mathbf Q}^{t_0, \delta_{(\mbfs,\mbfx)}}. 
 \end{equation}

\subsubsection{A disintegration formula}

This paragraph is devoted to establishing a disintegration formula for ${\mathbf Q}^{t_0, \cursm}$, closely resembling a weak Markov property. Our proof relies on the representation formula \eqref{eq:new:expression:Qcursm0}. While the result could alternatively be derived by adapting the notion of martingale problem to the present framework, it is more convenient here to exploit the explicit expression of the solutions.
Let us start with the analogous result for ${\mathbf P}^{t_0, \cursm}$.
\begin{lemma}
\label{lem:disintegration:Pcursm0}
Let $t_0 \in [0,T)$; then the map $\cursm \in {\mathcal P}(T^\infty_o \times (\T^d)^\infty_o)
\mapsto {\mathbf P}^{t_0, \cursm}$ is measurable, i.e., for any Borel subset $E \subset \Xi$, the map 
$\cursm \in {\mathcal P}(T^\infty_o \times (\T^d)^\infty_o)
\mapsto {\mathbf P}^{t_0, \cursm}(E)$ is Borel measurable.
Moreover, for any $\cursm \in {\mathcal P}(T^\infty_o \times (\T^d)^\infty_o)$ and any Borel subset $E \subset \Xi$
\begin{equation}
\label{eq:disintegration:Pcursm0}
{\mathbf P}^{t_0, \cursm}(E)
= \int_{T^\infty_o \times (\T^d)^\infty_o}
{\mathbf P}^{t_0, (\mbfs,\mbfx)}(E)
\de \cursm(\mbfs,\mbfx).
\end{equation}
\end{lemma}
Observe that, $T^\infty_o \times (\T^d)^\infty_o$ being a Polish space, the notion of measurability that is used in the statement is consistent with the one described in Subsection \ref{sec:wass}.

\begin{proof}
We first notice that the mapping $\cursm \mapsto \cursm \otimes \mathcal{W}_{t_0}^\infty$ is continuous, where $\mathcal{W}_{t_0}^\infty$ has been defined in Subsection \ref{ssec:canonical}. We deduce that, for any Borel subset $E \subset \Xi$, the mapping 
\begin{equation*}
\cursm
\mapsto {\mathbf P}^{t_0, \cursm}(E)
= \left( \cursm \otimes \mathcal{W}_{t_0}^\infty
\right)\left( (\textrm{\rm id},\Phi)^{-1}(E) \right)
\end{equation*}
is measurable, where we recall \eqref{def:mathbfP:cursm} for the definition 
of 
$(\textrm{\rm id},\Phi)$. This proves the first claim in the statement. 

As for \eqref{eq:disintegration:Pcursm0}, we notice that, for any Borel subset $E \subset \Xi$,
\begin{equation*}
\begin{split}
{\mathbf P}^{t_0, \cursm}(E) 
&= \int_{\Xi^0} {\mathbf 1}_E \circ (\textrm{\rm id},\Phi)
 \left( \mbfs,\mbfx,(w^i)_i \right) \de \left( \cursm \otimes \mathcal{W}_{t_0}^\infty \right)
\left( \mbfs,\mbfx,(w^i)_i \right)
\\
&= \int_{T^\infty_o \times (\T^d)^\infty_o}
\left[\int_{\rmC([0,T];{\mathbb R}^d)^\infty}
{\mathbf 1}_E \circ (\textrm{\rm id},\Phi)
 \left( \mbfs,\mbfx,(w^i)_i \right)  \de \mathcal{W}_{t_0}^\infty
\left((w^i)_i \right)
\right] \de \cursm(\mbfs,\mbfx)
\\
&=  \int_{T^\infty_o \times (\T^d)^\infty_o}
{\mathbf P}^{t_0, (\mbfs,\mbfx)}(E)
\de \cursm(\mbfs,\mbfx),
\end{split}
\end{equation*}
which completes the proof. 
\end{proof}

Before establishing an analogous disintegration formula for $\mathbf Q^{t_0, \cursm}$, it is necessary to analyze the measurability of the map $\cursm \mapsto \mathbf Q^{t_0, \cursm}$, which in turn requires several technical preliminaries. The first step is to consider the measurability of $\cursm \mapsto z_T^{t_0, \cursm}$, where $z_T^{t_0, \cursm}$ is defined in Proposition \ref{prop:pandq}. Observe that $z_T^{t_0, \cursm}$ depends on the 
collection of drifts $(b^i)_i$, although we suppress this dependence in the notation for the time being.

\begin{lemma}
\label{lem:universal:zeta0}
Let $(b^i:[0,T]\times\prob(\T^d)\times \T^d \to \R^d)_i$ be a collection of uniformly bounded and measurable functions and $t_0\in [0,T)$ be an initial time. There exists a measurable mapping $\zeta_T^{t_0} : \Xi \rightarrow {\mathbb R}$ such that, for
${\mathcal D}$-a.e.~$m \in \cP(\T^d)$, ${\mathbf P}^{t_0, \emp^{-1}(m)}(\{\zeta_T^{t_0}=z_T^{t_0, \emp^{-1}(m)}\})=1$.
\end{lemma}

The main difficulty in the above statement lies in constructing a mapping $\zeta_T^{t_0}$ that is independent of $m$, at least when $m$ ranges over a ${\mathcal D}$-full subset of ${\mathcal P}({\mathbb T}^d)$. The existence of such a $\zeta_T^{t_0}$, however, follows directly from the next lemma. For clarity, recall that $\Xi$ is equipped with its Borel $\sigma$-field $\mathcal{G}$ and the canonical filtration $\{{\mathcal{G}_t}\}_{t \in [0,T]}$.
\begin{lemma}
\label{lem:universal:mathcalI0} Let $(b^i:[0,T]\times\prob(\T^d)\times \T^d \to \R^d)_i$ be a collection of measurable functions
that are all dominated in norm by some $b \in L^2([0,T],L^2(\prob(\T^d)\times \T^d, \overline{\mathcal{D}}; \R^d))$,
and  $t_0\in[0,T]$ be an initial time. There exists a progressively-measurable mapping ${\mathcal I}^{t_0} : \Xi \rightarrow \rmC([0,T];{\mathbb R}^d)$ such that for ${\mathcal D}$-a.e. $m \in \cP(\T^d)$, 
\begin{equation*}
{\mathbf P}^{t_0, \emp^{-1}(m)}\left(\left\{\forall t \in [0,T], 
\quad {\mathcal I}_t = \sum_{i=1}^\infty 
\int_{0}^t 
\sqrt{\frac{\varsigma^i}2}
b_r^i(\mu^\infty_r,\xi^i_r) \cdot \de \beta_r^i
\right\}\right)=1. 
\end{equation*}
\end{lemma}

The fact that Lemma 
\ref{lem:universal:mathcalI0} implies 
Lemma \ref{lem:universal:zeta0} is quite obvious. It suffices to let 
\begin{equation*}
\zeta_T^{t_0} := \exp \left( - {\mathcal I}^{t_0}_T + \frac12 \sum_{i=1}^{\infty}
\int_{t_0}^T \frac{\varsigma^i}2 \vert b_r^i(\mu^\infty_r,\xi^i_r) \vert^2 
\de r \right). 
\end{equation*}

We thus concentrate on the proof of the first lemma.
\begin{proof}[Proof of Lemma \ref{lem:universal:mathcalI0}]
By the domination condition,
we first observe that
\begin{equation*}
\begin{split}
{\mathbb{E}}^{t_0, \emp^{-1}_\sharp (\cD)}
\left[\sum_{i=1}^\infty \int_0^T \varsigma^i 
\vert b_r^i(\mu^\infty_r,\xi_r^i) 
\vert^2 \de r
\right] 
&\leq {\mathbb{E}}
^{t_0, \emp^{-1}_\sharp (\cD)}
\left[ \int_0^T 
\left(
\int_{\T^d} 
\vert b_r (\mu^\infty_r,x) 
\vert^2 
\de \mu_r^\infty(x) 
\right)
\de r
\right]
\\
&=
\int_0^T 
\left[
\int_{\cP(\T^d)}
\left( \int_{\T^d}
\vert b_r(m,x) 
\vert^2
\de m(x)
\right)
\de \cD(m) 
\right] 
\de r 
< \infty,
\end{split}
\end{equation*}
where we used  
Proposition \ref{prop:mut} to get the second line. This implies that the stochastic integrals appearing in the statement are well-defined and form a continuous martingale.
Following \cite[Lemmata 4.3.2 \& 4.3.3]{StroockVaradhan}, we can construct a progressively-measurable mapping 
${\mathcal I}^{t_0}: \Xi \rightarrow \rmC([0,T];{\mathbb R}^d)$ such that 
\begin{equation*}
{\mathbf P}^{t_0, \emp^{-1}_\sharp (\cD)} \left( 
\left\{
\forall t \in [0,T], 
\quad 
{\mathcal I}^{t_0}_t = \sum_{i=1}^{\infty} \int_0^t 
\sqrt{\frac{\varsigma^i}{2}} b_r^i(\mu_r^\infty,\xi_r^i) \cdot \de \beta_r^i 
\right\}
\right) = 1.
\end{equation*}
Note that $\mathcal{I}_t^{t_0}=0$ for every $t\in[0,t_0]$, since, by the definition of $\mathbf P^{t_0, \emp^{-1}_\sharp (\cD)}$ (see \eqref{def:mathbfP:cursm}), the processes $(\beta^i)_i$ coincide with the path identical to zero on $[0,t_0]$. In fact, the same results say that there exists a sequence 
$({\mathcal I}^{t_0,(n)})_{n}$
of stochastic Riemann sums
such that 
\begin{equation*}
\forall \varepsilon >0, 
\quad 
\lim_{n \rightarrow \infty}
{\mathbf P}^{t_0, \emp^{-1}_\sharp (\cD)}
\left(
\left\{\sup_{t \in [0,T]}
\vert 
{\mathcal I}^{t_0,(n)}_t 
- 
{\mathcal I}^{t_0}_t
\vert \geq \varepsilon 
\right\}
\right)
=0.
\end{equation*}
By stochastic Riemann sums, we mean that each 
${\mathcal I}^{t_0,(n)}: \Xi \to \rmC([0,T]; \R^d)$ can be written in the form 
\begin{equation}
\label{eq:mathcalIn0}
{\mathcal I}^{t_0,(n)}_t = \sum_{i=1}^{\infty} \sum_{k=1}^{N(n)} H_{k}^{(n),i} \cdot \left( \beta_{t_k^{(n)}\wedge t}^i - \beta_{t_{k-1}^{(n)}\wedge t}^i 
\right), \quad t \in [0,T],
\end{equation} 
for 
$N(n) \in {\mathbb N}_+$  (typically, $N(n)$ tends to $\infty$ as $n$ tends to $\infty$), for
$0=t_0^{(n)}<t_1^{(n)}<\cdots<t_{N(n)}^{(n)}=T$ a partition of $[0,T]$ 
(typically, its step-size tends $0$ as $n$ tends to $\infty$),
and for 
$((H_k^{(n),i})_{i})_{k =1}^{N(n)}$
a collection of
$N(n)$ ${\mathbb R}^d$-valued 
bounded random variables such that 
$H_k^{(n),i}$ is ${\mathcal G}_{t_{k-1}^{(n)}}$-measurable
for each $k \in \{1,\cdots,N(n)\}$ and each $i \in {\mathbb N}_+$.
Using the disintegration formula \eqref{eq:disintegration:Pcursm0} for ${\mathbf P}^{t_0, \emp^{-1}_\sharp (\cD)}$, 
we can rewrite this as 
\begin{equation*}
\forall \varepsilon >0, 
\quad 
\lim_{n \rightarrow \infty}
\int_{\cP(\T^d)}
{\mathbf P}^{t_0, \emp^{-1}(m)}
\left(
\left\{\sup_{t \in [0,T]}
\vert 
{\mathcal I}^{t_0,(n)}_t 
- 
{\mathcal I}^{t_0}_t
\vert \geq \varepsilon 
\right\}
\right)
\de {\mathcal D}(m)
=0.
\end{equation*}
Up to a subsequence $k \mapsto n_k$, we easily deduce that, for 
${\mathcal D}$-a.e.~$m$,
\begin{equation*}
\forall \varepsilon \in (0,\infty) \cap \Q,
\quad 
\lim_{k \rightarrow \infty}
{\mathbf P}^{t_0, \emp^{-1}(m)}
\left(
\left\{\sup_{t \in [0,T]}
\vert 
{\mathcal I}^{t_0,(n_k)}_t 
- 
{\mathcal I}^{t_0}_t
\vert \geq \varepsilon 
\right\}
\right)
=0.
\end{equation*}
(Obviously, the same result holds true when $\varepsilon$ is taken in the entire $(0,\infty)$)
Since ${\mathcal I}^{t_0,(n_k)}$ is a stochastic Riemann sum, we deduce that, for $(\emp^{-1})_{\sharp} {\mathcal D}$-a.e.~$(\mbfs,\mbfx)$, 
${\mathcal I}^{t_0}$ provides a version of the stochastic integral $(\sum_{i=1}^\infty \int_0^t 
{\sqrt{\varsigma^i/2}}
b_r^i(\mu_r^\infty,\xi_r^i) \cdot \de \beta_r^i)_{t \in [0,T]}$. 
\end{proof}

As a result of the disintegration Lemma \ref{lem:disintegration:Pcursm0}, we notice that, for any $t_0 \in [0,T)$ and for any bounded $\cD$-density $p$ (that is, $p : {\mathcal P}(\T^d) \rightarrow [0,\infty)$ is a measurable function, with the 
property that $\int_{\cP(\T^d)} p(m) \de {\mathcal D}(m) = 1$), it holds
\begin{equation}\label{eq:diswithp}
\begin{split}
\text{ for every Borel subset } E \subset \Xi, \quad 
{\mathbf P}^{t_0, \emp^{-1}_\sharp (p {\mathcal D})}
(E) 
&= \int_{\cP(\T^d)} p(m) {\mathbf P}^{t_0, \emp^{-1}( m)}(E)
\de {\mathcal D}(m)
\\
&\leq \|p\|_\infty \int_{\cP(\T^d)}  {\mathbf P}^{t_0, \emp^{-1}(m)}(E)
\de {\mathcal D}(m) \\
&= \|p\|_\infty {\mathbf P}^{t_0, \emp^{-1}_\sharp  {\mathcal D}}(E).
\end{split}
\end{equation}
From \eqref{eq:diswithp} and Proposition \ref{prop:mut} we also immediately infer that
\begin{equation}\label{eq:acbound}
(\mu_t^{\infty})_{\sharp} {\mathbf P}^{t_0,{\emp_\sharp^{-1}} (p \cD)} \le \|p\|_\infty \cD \quad \text{ for every } t \in [t_0, T].
\end{equation}
 And, then repeating the proof of Lemmata \ref{lem:universal:zeta0}
and 
\ref{lem:universal:mathcalI0}, we obtain the following result. 
\begin{corollary}
\label{corol:universal:zeta0} In the same framework as
Lemma \ref{lem:universal:zeta0}, if $p$ is a bounded $\cD$-density, then
\begin{equation*}
{\mathbf P}^{t_0, \emp^{-1}_\sharp(p {\mathcal D})}\left(\left\{\zeta_T^{t_0}=z_T^{t_0, \emp^{-1}_\sharp(p {\mathcal D})}\right\}\right)=1.
\end{equation*}
\end{corollary}

We now state the main disintegration formula. 
\begin{proposition}
\label{prop:disintegration:Q0} Let 
$(b^i:[0,T]\times\prob(\T^d)\times \T^d \to \R^d)_i$ be a
collection of uniformly bounded and measurable functions, and $t_0\in [0,T)$ be an initial time. Then, for every Borel subset $E \subset \Xi$, the mapping
\begin{equation}
\label{eq:mapping:m:mapsto:Qt0m(E)}
m \in \cP(\T^d) \mapsto {\mathbf Q}^{t_0, \emp^{-1}(m)}(E)
\end{equation}
is measurable for the completion of the Borel $\sigma$-field on 
$\cP(\T^d)$ under ${\mathcal D}$. Moreover, if $p$ is a bounded $\cD$-density, then for every Borel subset $E \subset \Xi$
\begin{equation}
\label{eq:disintegration:Qcursm0}
{\mathbf Q}^{t_0, \emp^{-1}_\sharp (p {\mathcal D})}(E)
= \int_{\cP(\T^d)}
p(m) 
{\mathbf Q}^{t_0, \emp^{-1} (m)}(E)
\de {\mathcal D}(m).
\end{equation}
\end{proposition}
\begin{proof}
By combining \eqref{eq:new:expression:Qcursm0} 
 (but omitting, for simplicity, the superscript ${\boldsymbol b}$ in the notation $\cercle{$\Psi$}^{\boldsymbol b}$)
with Lemma \ref{lem:universal:zeta0}, we deduce that there exists a Borel 
subset $O$ of $\cP(\T^d)$ such that 
${\mathcal D}(O)=1$ and, for any $m \in O$
and any Borel subset $E \subset \Xi$,
\begin{equation*}
{\mathbf Q}^{t_0, \emp^{-1}(m)}(E)
= 
{\cercle{$\Psi$}}_\sharp \left(
\frac1{z_T^{t_0, \emp^{-1}(m)}}  {\mathbf P}^{t_0, \emp^{-1}(m)}\right)(E)
= 
\int_{\Xi}
\frac1{\zeta_T^{t_0}} {\mathbf 1}_{{\cercle{$\Psi$}}^{-1}(E)} 
\de {\mathbf P}^{t_0, \emp^{-1}(m)}.
\end{equation*}
We observe from Lemma \ref{lem:disintegration:Pcursm0} that the right-hand side is measurable w.r.t. 
$m$ (using in addition the fact that any positive measurable function can be approximated, pointwise, 
by non-decreasing limits of simple functions). 
This proves the measurability of the mapping 
\eqref{eq:mapping:m:mapsto:Qt0m(E)}.

As in the proof of the disintegration formula, we can
restart from the last display, but under the measure 
$\emp^{-1}_\sharp (p {\mathcal D})$
(which is possible thanks to Corollary \ref{corol:universal:zeta0}). And then, using also \eqref{eq:disintegration:Pcursm0}, we get
\begin{equation*}
\begin{split}
{\mathbf Q}^{t_0, \emp^{-1}_\sharp (p {\mathcal D})}(E)
&= 
\int_{\Xi}
\frac1{\zeta_T^{t_0}} {\mathbf 1}_{{\cercle{$\Psi$}}^{-1}(E)} 
\de {\mathbf P}^{t_0, \emp^{-1}_\sharp (p {\mathcal D})}
\\
&= \int_{\cP(\T^d)} 
\left( \int_{\Xi}
\frac1{\zeta_T^{t_0}} {\mathbf 1}_{{\cercle{$\Psi$}}^{-1}(E)}
\de {\mathbf P}^{t_0, \emp^{-1}(m)}
\right) p(m) \de {\mathcal D}(m)
\\
&= \int_{\cP(\T^d)} p(m) {\mathbf Q}^{t_0, \emp^{-1}(m)}(E) 
\de {\mathcal D}(m),
\end{split}
\end{equation*}
which
proves 
\eqref{eq:disintegration:Qcursm0}, and then
completes the proof.
\end{proof}

\subsubsection{On the choice of a version for $b$}
In this paragraph, we restrict ourselves to the case when all the entries of the collection $(b^i)_i$ are the same, i.e., there exists $b$ such that $b^i=b$ for each $i \in {\mathbb N}_+$. In this framework,
it is important (especially for the applications in Section \ref{sec: control}) {to see what our construction becomes when only a version of $b$ is given.
Assume indeed that $\tilde b$ is another measurable version of $b$, that is, for $\mathscr{L}^{[0,T]} \otimes {\mathcal D}$ almost every 
 $(t,\mu) \in [0,T] \times {\mathcal P}(\T^d)$ (so that $\mu$ is purely atomic),
 and for any $x \in {\mathbb T}^d$ such that $\mu_x >0$
 \begin{equation*}
b_t(\mu,x) = \tilde{b}_t(\mu,x),
 \end{equation*}
or, equivalently,
for $\mathscr{L}^{[0,T]} \otimes {\mathcal D}$-almost every
$(t,\mu)$, $b(t,\mu,\cdot)$ and $\tilde{b}(t,\mu,\cdot)$ coincide under $\mu$. Given these two measurable versions, we can consider the two associated systems of weak solutions (on the canonical space) to \eqref{eqn: particles_drift}
(with velocity fields that are independent of $i$), namely
\begin{equation*}
\left( {\mathbf Q}^{t_0,(\mbfs,\mbfx)} \right)_{\mbfs,\mbfx}
\quad \text{and} \quad
\left( \tilde{\mathbf Q}^{t_0,(\mbfs,\mbfx)} \right)_{\mbfs,\mbfx},
\end{equation*}
where $t_0\in[0,T)$ and $(\mbfs,\mbfx)$ varies in $T^\infty_o\times(\T^d)^\infty_o$.
The goal of this subsection is to prove the following consistency result.
\begin{proposition}
\label{prop:consistency:solutions}
Let $t_0\in [0,T)$, $b\colon[0,T]\times\prob(\T^d)\times\T^d\to\R^d$ be bounded and measurable, and $\tilde{b}\colon[0,T]\times\prob(\T^d)\times\T^d\to\R^d$ be a measurable version of $b$ in the sense described above. Then, for ${\mathcal D}$-a.e.~$m$, \begin{equation*}
{\mathbf Q}^{t_0,\emp^{-1}(m)} = \tilde {\mathbf Q}^{t_0,\emp^{-1}(m)},
\end{equation*}
where ${\mathbf Q}^{t_0,\emp^{-1}(m)}$ and $\tilde{\mathbf Q}^{t_0,\emp^{-1}(m)}$ denote the laws of the weak solutions to \eqref{eqn: particles_drift} with 
$i$-independent
drifts $b$ and $\tilde{b}$, respectively.
\end{proposition}

The proof of Proposition \ref{prop:consistency:solutions} relies on the following lemma, proved later in this subsection.
\begin{lemma}
\label{lem:consistency:absolutely:continuous:measures}
Let $t_0 \in [0,T)$, $p$ be a bounded $\cD$-density,
and denote by 
\[{\mathbf Q}^{t_0,{\emp_\sharp^{-1}} (p {\mathcal D})}
\quad 
\textrm{and}
\quad \tilde{\mathbf Q}^{t_0,{\emp_\sharp^{-1}}  (p {\mathcal D})}\]
the laws of the two weak solutions of \eqref{eqn: particles_drift}, both  initialized from the measure 
${\emp_\sharp^{-1}}  (p {\mathcal D})$ (equivalently, $\mu^\infty_{t_0}$ is sampled from $p {\mathcal D}$), and driven by the velocity fields $b$ and 
$\tilde b$ respectively. 
Then, 
\[{\mathbf Q}^{t_0,{\emp_\sharp^{-1}}  (p {\mathcal D})}=\tilde{\mathbf Q}^{t_0,{\emp_\sharp^{-1}}  (p {\mathcal D})}.\]
\end{lemma}

Assuming the above lemma, we now turn to the proof of Proposition \ref{prop:consistency:solutions}.

\begin{proof}[Proof of Proposition 
\ref{prop:consistency:solutions}.]
Let $p$ be a fixed bounded $\cD$-density. By the disintegration formula 
\eqref{eq:disintegration:Qcursm0}, it holds 
\begin{equation*}
 {\mathbf Q}^{t_0,{\emp_\sharp^{-1}}  (p {\mathcal D})} =
\int_{\cP(\T^d)}
p(m) 
{\mathbf Q}^{t_0,\emp^{-1}(m)}
\de {\mathcal D}(m).
\end{equation*}
Analogously, we can write the same formula for the weak solution associated to the version $\tilde b$ of the drift. Then, by Lemma \ref{lem:consistency:absolutely:continuous:measures}, we can equate the two formulas and obtain
\begin{equation*}
\begin{split}
\int_{\cP(\T^d)}
p(m) 
{\mathbf Q}^{t_0,\emp^{-1}(m)}
\de {\mathcal D}(m)=
\int_{\cP(\T^d)}
p(m) 
\tilde {\mathbf Q}^{t_0,\emp^{-1}(m)}
\de {\mathcal D}(m),
\end{split}
\end{equation*}
which completes the proof of Proposition \ref{prop:consistency:solutions}.
\end{proof}

The next step is to prove Lemma \ref{lem:consistency:absolutely:continuous:measures}, and to do so, we need to introduce another lemma.

\begin{lemma}
\label{lem:absolute:continuity}
With the same notation as in the statement of Lemma 
\ref{lem:consistency:absolutely:continuous:measures}, one has, for any 
$t \in [t_0,T)$, 
\begin{equation*}
(\mu_t^{\infty})_{\sharp} {\mathbf Q}^{t_0,{\emp_\sharp^{-1}} (p \cD)} \ll {\mathcal D},
\end{equation*}
i.e., the (time) marginal laws of 
$\mu^\infty$ (see \eqref{eq:muinfmap}) under 
${\mathbf Q}^{t_0,{\emp_\sharp^{-1}} (p \cD)}$
are absolutely 
continuous with respect to 
${\mathcal D}.$
\end{lemma}
\begin{proof}
For any Borel subset $E$ of $\prob(\T^d)$ satisfying $\cD(E)=0$, we aim to prove that, for a fixed $t\in[t_0,T]$,
\begin{equation}
\label{eq:absolu:continuity:proof:1}
{\mathbf Q}^{t_0,{\emp_\sharp^{-1}} (p \cD)}\left(\left\{ \mu_t^\infty\in E\right\}\right) = 0. 
\end{equation}
To simplify the exposition, let us denote 
${\emp_\sharp^{-1}} (p \cD)$
by $\cursm$. We recall from 
\eqref{eq:new:expression:Qcursm0} that 
${\mathbf Q}^{t_0,\cursm}$ is given by 
\begin{equation*}
{\mathbf Q}^{t_0,\cursm} =\cercle{$\Psi$}^b_\sharp \left( \frac1{z^{t_0,\cursm}_T}  {\mathbf P}^{t_0,\cursm} \right),
\end{equation*}
where ${\cercle{$\Psi$}^b}$ is as in Proposition \ref{prop:pandq} with $b$ $i$-independent.
Then, we observe that $\mu^\infty$ is invariant by ${\cercle{$\Psi$}^b}$, that is 
$\mu^\infty \circ {\cercle{$\Psi$}^b} = \mu^\infty$. As a result, we have
\begin{equation*}
\begin{split}
{\mathbf Q}^{t_0,\cursm}
\left( 
\left\{ 
\mu_t^\infty\in E
\right\}
\right)
=
\left(\frac1{z^{t_0,\cursm}_T}  {\mathbf P}^{t_0,\cursm}
\right) 
\left(
\left\{
\mu_t^\infty\in E
\right\}
\right).
\end{split}
\end{equation*}
In particular, to prove 
\eqref{eq:absolu:continuity:proof:1}, it suffices to prove that 
the right-hand side on the above display is equal to $0$. But, then, it is enough to show that 
\begin{equation*}
 {\mathbf P}^{t_0,\cursm}
\left(
\left\{
\mu_t^\infty\in E
\right\}
\right)=0
\end{equation*}
which follows from \eqref{eq:acbound}.
\end{proof}
Now, we can finally prove Lemma \ref{lem:consistency:absolutely:continuous:measures}.
\begin{proof}[Proof of Lemma \ref{lem:consistency:absolutely:continuous:measures}.]
By uniqueness in law, it is sufficient to prove  that, with probability 1 under 
${\mathbf Q}^{t_0,\cursm}$, with 
$\cursm := 
{\emp_\sharp^{-1}} (p {\mathcal D})$, for any 
$i \in \N_+$ and 
$t \in [t_0,T]$,
\begin{equation*}
\int_{t_0}^t b_r(\mu^\infty_r,\xi^i_r) \de r
= \int_{t_0}^t \tilde b_r(\mu^\infty_r,\xi^i_r) \de r. 
\end{equation*}
Obviously, it suffices to check that 
\begin{equation}\label{eq:bandtildeb}
    \int_{t_0}^T \int_{\Xi} \left | b_t (\mu_t^\infty, \xi^i_t) - \tilde{b}_t(\mu_t^\infty, \xi^i_t) \right | \de \mathbf Q^{t_0, \cursm} \de t =0 \quad \text{ for all } i \in \N_+.
\end{equation}
Since by Lemma \ref{lem:absolute:continuity} we have that $(\mu_t^\infty)_\sharp \mathbf Q^{t_0, \cursm} \ll \cD$, the  $\mathscr{L}^{[0,T]} \otimes \overline{\cD}$-a.e.~equality between $b$ and $\tilde{b}$ gives that
\begin{align*}
0&= \int_{t_0}^T \int_{\prob(\T^d)} \int_{\T^d} \left | b_t (\mu, x) - \tilde{b}_t(\mu, x) \right | \de \mu(x) \de \left [ (\mu_t^\infty)_\sharp \mathbf Q^{t_0, \cursm} \right ] (\mu) \de t\\
&= \int_{t_0}^T \int_{\Xi} \int_{\T^d} \left | b_t (\mu_t^\infty, x) - \tilde{b}_t(\mu_t^\infty, x) \right | \de \mu_t^\infty(x) \de \mathbf Q^{t_0, \cursm} \de t\\
&=\int_{t_0}^T \int_{\Xi} \sum_{i=1}^\infty \varsigma_i \left | b_t (\mu_t^\infty, \xi_t^i) - \tilde{b}_t(\mu_t^\infty, \xi_t^i) \right | \de \mathbf Q^{t_0, \cursm} \de t.
\end{align*}
In particular,
\[ \left | b_t (\mu_t^\infty, \xi_t^i) - \tilde{b}_t(\mu_t^\infty, \xi_t^i) \right | =0 \quad  \mathscr{L}^{[0,T]} \otimes {\mathbf Q}^{t_0, \cursm}\text{-a.e.} \]
This gives \eqref{eq:bandtildeb} and concludes the proof.
\end{proof}

\subsection{Chain rule and relation with the backward Kolmogorov equation}
The main purpose of this subsection is to identify the generator of the solution to
the particle system \eqref{eqn: particles_drift}
(when driven by a drift $b$ that is independent of $i$)
with the operator governing the PDE \eqref{eq:pdeaaa}. Throughout, we make a repeated use of the spaces $H$ and $H^{1,2}$ introduced in 
\eqref{eq:notation}. For convenience, we recall their definitions below:
\begin{equation*}
    H := L^2(\prob(\T^d), \mathcal{D}), \quad H^{1,2} := H^{1,2}(\prob(\T^d), W_2, \mathcal{D}).
\end{equation*}

The connection between the particle system \eqref{eqn: particles_drift} and the PDE \eqref{eq:pdeaaa} is established through an appropriate Kolmogorov formula, stated in Theorem \ref{thm:repr} below. This representation is obtained by expanding the solution to \eqref{eq:pdeaaa} along the particle system \eqref{eqn: particles_drift} via a suitable chain rule. The proof of this expansion relies on Proposition \ref{prop: ito_cyl} and first requires approximating the solution to the PDE \eqref{eqn:bkw_heta_forcing}, when driven by arbitrary boundary and source terms, by solutions to the same PDE but with cylindrical coefficients. As clarified in Proposition \ref{prop:cyl_solutions}, a key step is to show that, for cylindrical coefficients, the solution to \eqref{eqn:bkw_heta_forcing} itself inherits a form of cylindrical structure which we make precise below.

\begin{definition}\label{def:fwcyl} We say that a function $u: \prob_o(\T^d) \to \R$ is \emph{fiber-wise cylinder} if there exists $(u^\mbfs)_{\mbfs \in T_o^\infty} \subset \hat{\mathfrak{Z}}^\infty_c$ such that:
\begin{enumerate}[1)]
    \item $u(\mu) = u^{\mbfs(\mu)}(\mu)$ for every $\mu \in \prob_o(\T^d)$, where $\mbfs(\mu)$ is as in \eqref{eq:phiinv};
    \item setting
    \[ (\boldnabla u)_\mu := (\boldnabla u^{\mbfs(\mu)})_\mu, \quad (L_c u)_\mu := (L_c u^{\mbfs(\mu)})_\mu, \quad \mu \in \prob_o(\T^d), \]
    we have that $\boldnabla u \in L^2(\prob(\T^d) \times \T^d, \overline{\cD}; \R^d)$ and $u, L_c u \in H$.
\end{enumerate}
We denote the space of fiber-wise cylinder function with $\hat{\mathfrak{Z}}_{fw}^\infty$.
\end{definition}

Note that the definition of $\boldnabla u$ and $L_c u$ are well posed since, for a every $\mbfs \in T_o^\infty$, $u^\mbfs$ belongs to $\hat{\mathfrak{Z}}_{c}^\infty$ which is included in the domain of both the operators $\boldnabla$ and $L_c$, see \eqref{eqn:gen_on_cil} and Lemma \ref{le:hasgrad}.
Also, clearly, $\hat{\mathfrak{Z}}_{c}^\infty \subset \hat{\mathfrak{Z}}_{fw}^\infty$ and the definitions of $\boldnabla$ and $L_c$ are consistent.

\begin{proposition}\label{prop:ident} We have that $\hat{\mathfrak{Z}}_{fw}^\infty \subset  D(\boldsymbol{\Delta})$ and
\begin{equation}\label{eq:thekey}
\rmD u = \boldnabla u, \quad \boldsymbol{\Delta} u = L_c u \quad \text{ for every } u \in \hat{\mathfrak{Z}}_{fw}^\infty.
\end{equation}
\end{proposition}
We stress that the statement above means that, for every $u \in \hat{\mathfrak{Z}}_{fw}^\infty$, the $L^2(\prob(\T^d), \cD)$-equivalence class represented by the $\cD$-a.e.~defined function
$u$ belongs to the domain of the operator $\boldsymbol{\Delta}$ (and thus, in particular, also to the domain of $\rmD$) and the $L^2(\prob(\T^d) \times \T^d, \overline{\cD}; \R^d)$ function $\rmD u$ can be represented by the $\overline{\cD}$-a.e.~defined function $\boldnabla u$ (respectively, the $L^2(\prob(\T^d),\cD)$ function $\boldsymbol{\Delta} u$ can be represented by the $\cD$-a.e.~defined function $L_c u$) as in Definition \ref{def:fwcyl}.

\begin{proof}
 We claim that the statement follows if we show that
\begin{equation}\label{eq:thelast}
    ( L_c u , w)_H  = -\langle \boldnabla u, \boldnabla w \rangle \quad \text{ for every } u, w \in \hat{\mathfrak{Z}}_{fw}^\infty.
\end{equation}
Assume  indeed that \eqref{eq:thelast} holds true. Then, if $u \in \hat{\mathfrak{Z}}_{fw}^\infty$ is fixed, \eqref{eq:thelast} gives, by symmetry
\[ (L_c u, w)_H = (u, L_c w)_H \quad \text{ for every } w \in \hat{\mathfrak{Z}}_c^\infty,\]
which shows that $ u \in D(L_c^*)$ and $L_c^* u = L_c u$; since, as discussed at the end of Subsection \ref{sec:geo}, $\boldsymbol{\Delta}= L_c^*$, this gives that $u \in D(\boldsymbol{\Delta})$ and $\boldsymbol{\Delta} u = L_c u$. By the integration by parts formula in \eqref{eq:ibp0}, we 
deduce from \eqref{eq:thelast} again (still assuming that it holds true) that
\begin{equation}\label{eq:tangent}
\langle \rmD u, \rmD w \rangle = \langle \boldnabla u, \boldnabla w \rangle \quad \text{ for every }  w \in \hat{\mathfrak{Z}}_{fw}^\infty.
\end{equation}
Taking $w=u$, we see that $\|\rmD u\|= \|\boldnabla u\|$; taking $w=w_n$ for a sequence $(w_n)_n \subset \hat{\mathfrak{Z}}_c^\infty$ such that $\boldnabla w_n \to \rmD u$ in $L^2(\prob(\T^d) \times \T^d, \overline{\cD}; \R^d)$ as $n \to \infty$ (see Proposition \ref{prop:vectorgrad} and its proof) and passing to the limit as $n \to \infty$, we obtain
\[ \langle \boldnabla u, \rmD u \rangle = \langle \rmD u, \rmD u \rangle = \|\rmD u\|^2 = \| \boldnabla u \| \|\rmD u\|,\]
which gives $\boldnabla u = \rmD u$, 
and completes the proof (provided \eqref{eq:thelast} holds true). 
\smallskip

Therefore, we are only left to show \eqref{eq:thelast}; let $u, w \in \hat{\mathfrak{Z}}_{fw}^\infty$ and let, for every $\mbfs \in T_o^\infty$, $u^\mbfs, w^\mbfs \in \hat{\mathfrak{Z}}_c^\infty$ be the functions as in Definition \ref{def:fwcyl}, for $u$ and $w$ respectively.
We show that \eqref{eq:thelast} holds true by following \cite[Theorem 5.11]{DelloSchiavo22}.  We have
\begin{align*}
    (L_c u, w)_H &= \int_{\prob(\T^d)} (L_c u)_\mu w(\mu) \de \mathcal{D}(\mu) \\
    &= \int_{\prob(\T^d)} (L_c u^{\mbfs(\mu)})_\mu w^{\mbfs(\mu)}(\mu) \de \mathcal{D}(\mu) \\
    & = \int_{\prob(\T^d)} \int_{\T^d}\frac{\Delta_z\rvert_{z=x} u^{\mbfs(\mu)}(\mu+\mu_x\delta_z-\mu_x\delta_x)}{(\mu_x)^2} \de\mu(x) w^{\mbfs(\mu)}(\mu) \de \mathcal{D}(\mu) \\
    &=\int_0^1\int_{\prob(\T^d)} \int_{\T^d}\frac{\Delta_z\rvert_{z=x} u^{\mbfs(\mu_r^x)}(\mu_r^x+r\delta_z-r\delta_x)}{r^2}  w^{\mbfs(\mu_r^x)}(\mu_r^x) \de\mathrm{vol}_d(x) \de \mathcal{D}(\mu) \de r \\
    &=\int_0^1\int_{\prob(\T^d)} \int_{\T^d}\frac{\Delta_z\rvert_{z=x} u^{\mbfs(\mu,r)}(\mu_r^x+r\delta_z-r\delta_x)}{r^2}  w^{\mbfs(\mu,r)}(\mu_r^x) \de\mathrm{vol}_d(x) \de \mathcal{D}(\mu) \de r \\
    &=-\int_0^1\int_{\prob(\T^d)} \int_{\T^d} \frac{\nabla_z\rvert_{z=x} u^{\mbfs(\mu,r)}(\mu_r^x+r\delta_z-r\delta_x)}{r} \cdot   \frac{\nabla_z\rvert_{z=x} w^{\mbfs(\mu,r)}(\mu_r^x+r\delta_z-r\delta_x)}{r} \de\mathrm{vol}_d(x) \de \mathcal{D}(\mu) \de r\\
    &= -\int_0^1\int_{\prob(\T^d)} \int_{\T^d} (\boldnabla u^{\mbfs(\mu, r)})_{\mu_r^x}( x) \cdot (\boldnabla w^{\mbfs(\mu, r)})_{\mu_r^x}(x) \de\mathrm{vol}_d(x) \de \mathcal{D}(\mu) \de r\\
    &= -\int_0^1\int_{\prob(\T^d)} \int_{\T^d} (\boldnabla u^{\mbfs(\mu_r^x)})_{\mu_r^x} (x) \cdot (\boldnabla w^{\mbfs(\mu_r^x)})_{\mu_r^x} (x) \de\mathrm{vol}_d(x) \de \mathcal{D}(\mu) \de r\\
    &=- \int_{\prob(\T^d} \int_{\T^d} (\boldnabla u^{\mbfs(\mu)})_{\mu} (x) \cdot (\boldnabla u^{\mbfs(\mu)})_{\mu} (x) \de \mu(x) \de \cD(\mu) \\
    &=- \int_{\prob(\T^d} \int_{\T^d} (\boldnabla u)_{\mu} (x) \cdot (\boldnabla w)_{\mu} (x) \de \mu(x) \de \cD(\mu) \\
    &=-\langle \boldnabla u, \boldnabla w \rangle,
    \end{align*}
where we have used, in order, the definition of $L_c u$ (on the second line), the definition of $L_c u^\mbfs$ (on the third line), the Mecke identity \eqref{eq:Mecke} (to pass from the third to the fourth line), the fact that the set $\{(x, \mu) : \mu_x >0 \}$ is $\mathrm{vol}_d \otimes \mathcal{D}$-negligible (see~\cite[Proposition 4.9(iii)]{DelloSchiavo22} so that $\mbfs(\mu_x^r)$ does not depend on $x \in \T^d$ and therefore we can denote it by $\mbfs(\mu,r)$ (to get the fifth line), integration by parts on $\T^d$ (on the sixth line), the relation \eqref{eq:relgrad} (on the seventh line), again that $\mbfs(\mu_r^x)= \mbfs(\mu, r)$ (on the eight line), one more time the Mecke identity \eqref{eq:Mecke} (to obtain the ninth line), and, finally, the definition of $\boldnabla u$ and $\boldnabla w$ (to get the tenth line). This shows \eqref{eq:thelast} and concludes the proof.
\end{proof}

\begin{proposition}\label{prop:cyl_solutions} 
For a given $\eps \in (0,1)$, let $f \in \hat{\mathfrak{Z}}_\eps^\infty$ and $g \in \hat{\mathfrak{Z}}_\eps^\infty$. Then, if $u$ is the unique solution of
\[ \partial_t u_t + \boldsymbol{\Delta} u_t = f \quad \text{for a.e.~}t \in (0,T), \quad u_t \bigr |_{t=T} = g,\]
there exist Borel representatives of $u$, $\rmD u$, and $\boldsymbol{\Delta} u$
satisfying the following properties: for every $\mbfs \in  T^\infty_o$, there exists $v^{\mbfs} \in \mathfrak{T}\hat{\mathfrak{Z}}_{\eps/2}^\infty([0, T])$ such that
\begin{align}
 \label{eq:theprop}
u_t(\mu) &= v_t^{\mbfs}(\mu), \quad &&\text{ for all } \mu \in \emp(\mbfs, (\T^d)^\infty_o)\text{ and } t \in [0,T],\\ \label{eq:laprima}
\rmD u_t(\mu, \cdot) &= \rmD v_t^{\mbfs}(\mu, \cdot), \quad &&\text{ for all } \mu \in \emp(\mbfs, (\T^d)^\infty_o)\text{ and } t \in [0,T],\\ \label{eq:laseconda}
\boldsymbol{\Delta}u_t(\mu) &= \boldsymbol{\Delta} v_t^{\mbfs}(\mu), \quad &&\text{ for all } \mu \in \emp(\mbfs, (\T^d)^\infty_o)\text{ and } t \in [0,T].
\end{align}
In particular $u_t \in \hat{\mathfrak{Z}}_{fw}^\infty$ for every $t \in [0,T].$
\end{proposition}
\begin{proof}  
Recall that the map $\emp$ as in \eqref{eq:Intro:Phi} is injective on $ T^\infty_o \times (\T^d)^\infty_o$ and we write its Borel measurable inverse $\emp^{-1}: \prob_o(\T^d)\to  T^\infty_o \times (\T^d)^\infty_o$ in the form
\[ (\mbfs(\mu), \mbfx(\mu)):= \emp^{-1}(\mu), \quad \mu \in \prob_o(\T^d).\]
For every $\mbfs \in  T^\infty_o$, we set $n(\mbfs):= N(\eps, \mbfs)$, the latter being defined in \eqref{eq:Neps:mbfs}; intuitively, 
$n(\mbfs)$ is the largest index below which masses contained in the sequence $\mbfs$ are greater than or equal to $\eps$. Note that $\mbfs \mapsto n(\mbfs)$ is a Borel map.

\medskip

\textit{Construction of $v^\mbfs$}: let $(s_j)_j =\mbfs \in T_o^\infty$ be fixed. When $n(\mbfs)=0$, we set $v_t^\mbfs:=0$, for all
$t \in [0,T]$. 
When
$n(\mbfs)>0$, the construction of $v^\mbfs$ is more delicate and proceeds in several steps. We first set
\[
\mathrm{sep}(\mbfs) := \frac{1}{2} \left (\min \{ |s_i -s_j| : i,j \le n(\mbfs), \, i \ne j \} \wedge \min \{ s_j-\eps/2 : j \le n(\mbfs)\} \right ),
\]
which is related  both to the minimal distance between the masses greater than or equal to $\eps$, and  to the minimal distance from these masses to $\varepsilon/2$.
Given $i \in \{ 1, \dots, dn(\mbfs)\}$, we call $(k,j)$ the unique pair of integers $k \in \{0, \dots, n(\mbfs)-1\}$ and $j \in \{1, \dots, d\}$ such that $i= kd+j$, and then, we find a function $\hat{p}_i^{\mbfs} \in \rmC^\infty([0, 1]\times \T^d)$ with $\supp(\hat{p}_i^{\mbfs}) \subset [\eps/2, 1]\times \T^d$ satisfying, for any $(r,x) \in [\varepsilon/2,1] \times {\mathbb T}^d$,
\[
\hat{p}_i^{\mbfs}(r,x^1,\dots,  x^d) = \begin{cases}
     x^j/r\quad &\text{if } \ (r,x) \in [s_{k+1} -\mathrm{sep}(\mbfs)/2, (s_{k+1} + \mathrm{sep}(\mbfs)/2) \wedge 1] \times \T^d,\\
    0 \quad 
    &\text{if } \ (r,x) \in \left ( [0, s_{k+1} - \mathrm{sep}(\mbfs)] \cup [(s_{k+1} + \mathrm{sep}(\mbfs)) \wedge 1, 1]  \right ) \times \T^d, 
\end{cases}
\]
where $x^j$ denotes the $j$-th coordinate of $x \in \T^d$. We set $\hat{\mbfp}^{\mbfs}:=(\hat{p}_1^{\mbfs}, \dots, \hat{p}_{dn(\mbfs)}^{\mbfs})$.
It is clear from this definition, and with the notation 
\eqref{eq:thefirsttrid}, 
that,
\[ (\hat{\mbfp}^{\mbfs})^{\trid}(\mu) = (x_1, \dots, x_{n(\mbfs)}) \in (\T^d)^{n(\mbfs)} \quad \text{ for every } \mu =\emp(\mbfs, \mbfx), \quad (\mbfs, \mbfx) \in  T^\infty_o \times(\T^d)^\infty_o.\]

We define two smooth functions $f^{\mbfs}, g^\mbfs: (\T^d)^{n(\mbfs)} \to \R$ by letting
\begin{equation}
\label{eq:fmbfs:gmbfs}
f^\mbfs := (f \circ \emp^{n(\mbfs)})(\mbfs, \cdot), \quad g^\mbfs := (g \circ \emp^{n(\mbfs)})(\mbfs, \cdot), 
\end{equation}
where $\emp^n$ is as in \eqref{eq:empN}. We then consider, on $[0,T] \times (\T^d)^{n(\mbfs)}$, the PDE
\begin{equation}\label{eq:pden}
\partial_t h_t + \Delta^{\mbfs} h_t = f^{\mbfs} \quad \text{ for a.e.~} t \in (0, T), \quad h_t \bigr |_{t=T} = g^{\mbfs},
\end{equation}
where $\Delta^{\mbfs}$ is the operator given by 
\[ \Delta^{\mbfs} h := \sum_{j=1}^{n(\mbfs)} \frac{1}{s_i} (\Delta_{i} h), \quad h \in \rmC^2((\T^d)^{n(\mbfs)}),\]
with $\Delta_i$ being the Laplacian operator on $\T^d$, acting on the $i$-th ($d$-dimensional) component of   $(\T^d)^{n(\mbfs)}$. The existence of a smooth solution $h^{\mbfs}: [0,T] \times (\T^d)^{n(\mbfs)} \to \R$ to \eqref{eq:pden} follows from the fact that it can be expressed as
\[ h^{\mbfs}_t(x):= k^\mbfs_{T-t}(A_\mbfs^{-1/2}x), \quad \text{ for every } (t,x) \in [0,T] \times (\T^d)^{n(\mbfs)},\]
where $A_\mbfs$ is the diagonal matrix 
of size $n(\mbfs) \times n(\mbfs)$
with $(A_\mbfs)_{jj}=s_i^{-1}$ for $j \in \{kd+1, \dots, (k+1)d\}$, $k \in \{0, \dots, n(\mbfs)-1\}$, and $k^\mbfs:[0,T] \times (\R^d)^{n(\mbfs)} \to \R$ is the smooth solution of the heat equation
\begin{equation}\label{eq:pdenk}
\partial_t k_t - \Delta k_t = -\tilde{f}^{\mbfs}(A_\mbfs^{1/2}(\cdot)) \quad \text{ for a.e.~} t \in (0, T), \quad k_t \bigr |_{t=0} = \tilde{g}^{\mbfs}(A_\mbfs^{1/2}(\cdot)),
\end{equation}
with $\tilde{f}^\mbfs$ and $\tilde{g}^\mbfs$
being the periodic extensions of the smooth functions $f^\mbfs$ and $g^\mbfs$ to $(\R^d)^{n(\mbfs)}$; see e.g.~\cite[Exercise 9.1.3]{krylovholder} for an existence result to the PDE. Classical estimates on the solution to the heat equation \eqref{eq:pdenk} (see e.g.~\cite[Exercise 9.1.4]{krylovholder}) and the fact that $s_i \ge \eps$ for every $i \in \{1, \dots, n(\mbfs)\}$ with $n(\mbfs) \le \eps^{-1}$, give the existence of a constant $C$, depending on $\eps$, $f$, and $g$ but not on $\mbfs$, such that
\begin{equation}\label{eq:inftyest}
\sup_{t \in [0,T]} \left ( \sum_{i=1}^{n(\mbfs)} \frac{1}{s_i} \|\nabla_i h_t^\mbfs\|_\infty + \|\Delta^\mbfs h_t^\mbfs  \|_\infty +\|h_t^\mbfs\|_\infty \right ) \le C. 
\end{equation}
When $n(\mbfs)>0$, we set $v^\mbfs_t := h_t^\mbfs \circ (\hat{\mbfp}^{\mbfs})^{\trid}$, $t \in [0, T]$. It is clear from the construction that $v^\mbfs \in \mathfrak{T}\hat{\mathfrak{Z}}_{\eps/2}^\infty([0, T])$, in fact for every $\mu \in \prob_o(\T^d)$  the function $t \mapsto v^\mbfs_t(\mu)$ belongs to $\rmC^\infty([0,T])$, being $h^\mbfs$ smooth.\\
\quad \\
\textit{Construction of $u$, its derivatives, and their measurability}:
following item (1) in Definition \ref{def:fwcyl}, we define
the function $u: [0, T]\times \prob_o(\T^d) \to \R$ as
\begin{equation}\label{eq:udef}
u_t(\mu) := v^{\mbfs(\mu)}_t(\mu)  ,  \quad (t, \mu) \in [0, T] \times \prob_o(\T^d).    
\end{equation}
Let us also set
\begin{equation}\label{eq:uoper}
 (\boldnabla u_t)_\mu :=  (\boldnabla v^{\mbfs(\mu)}_t)_\mu, \quad (L_c u_t)_\mu:= (L_c v^{\mbfs(\mu)}_t)_\mu \quad \text{ for every } (t, \mu) \in [0,T] \times \prob_o(\T^d).
\end{equation}
We now show simpler representations of the above operators: for every $t \in [0,T]$ we have
\begin{align} \label{eq:gradu}
(\boldnabla u_t)_\mu(x_i) &= \frac{1}{s_i} \nabla_i h_t^\mbfs (x_1, \dots, x_{n(\mbfs)}) \quad &&\text{ for every } \mu = \emp(\mbfs, \mbfx), \, (\mbfs, \mbfx) \in T_o^\infty \times (\T^d)_o^\infty, \, i=1, \dots, n(\mbfs),\\ \label{eq:nablau}
(L_c u_t)_\mu &= \Delta^\mbfs h_t^\mbfs (x_1, \dots, x_{n(\mbfs)}) \quad &&\text{ for every } \mu = \emp(\mbfs, \mbfx), \, (\mbfs, \mbfx) \in T_o^\infty \times (\T^d)_o^\infty,
\end{align}
where, in both formulas, we mean that, if $n(\mbfs)=0$, then both right-hand sides are equal to the constant function $0$. The proof of the above formulas relies on the following two arguments: if $n(\mbfs)=0$, then $v^\mbfs=0$ and the formulas follow; if $n(\mbfs)>0$, \eqref{eq:gradu} is a simple consequence of the definition of $v^\mbfs$ and Lemma \ref{le:hasgrad}, while \eqref{eq:nablau} can be obtained by simply computing
\begin{align*}
(L_c v_t^\mbfs)_\mu&=  \int_{\T^d}\frac{\Delta_z\rvert_{z=x} v^\mbfs_t(\mu+\mu_x\delta_z-\mu_x\delta_x)}{(\mu_x)^2} \de\mu(x) = \sum_{j=1}^{\infty} \frac{1}{s_j} \Delta_z\rvert_{z=x_j} v^\mbfs_t \left ( \sum_{k=1}^{\infty} s_k \delta_{x_k} + s_j \delta_z - s_j \delta_{x_j}  \right ) \nonumber\\
& = \sum_{j=1}^{n(\mbfs)} \frac{1}{s_j} \Delta_j h_t^{\mbfs}(x_1, \dots, x_{n(\mbfs)}) = \Delta^{\mbfs} h_t^{\mbfs}(x_1, \dots, x_{n(\mbfs)}),
\end{align*}
where we have also used that $\mbfs = \mbfs(\mu) = \mbfs(\mu+ \mu_x\delta_z -\mu_x \delta_x)$. To show that $u$, $\boldnabla u$ and $L_c u$ are measurable functions, 
we consider, for every $k \in \N_+$, the 
set
\begin{equation*}
T^{k,\varepsilon}_o
:= \pi_k\left(\{ \mbfs \in T^\infty_o : n(\mbfs) = k\}\right),
\end{equation*}
where $\pi_k$ is the projection from $T_o^\infty$ to $[0,1]^k$ defined by $\pi_k(\mbfs)=(s_1,\dots,s_k)$
and, for every $(r_1, \dots, r_k) \in T_o^{k, \eps}$, $h^{r_1, \dots, r_k}$ is the smooth solution (see the discussion above for existence and uniqueness of it) of the PDE, on $[0,T] \times (\T^d)^k$,
\begin{equation}\label{eq:pdeagain}
\partial_t h_t + \sum_{j=1}^k \frac{1}{r_j} \Delta_j h_t = f^{r_1, \dots, r_k} \quad \text{ for a.e. } t \in (0,T), \quad h_t |_{t=T} =g^{r_1, \dots, r_k},    
\end{equation}
where $f^{r_1, \dots, r_k}, g^{r_1, \dots, r_k}: (\T^d)^k \to \R$ are defined as
\[ f^{r_1, \dots, r_k}:= f^{\mbfs}, \quad g^{r_1, \dots, r_k}:= g^{\mbfs}\]
for any $\mbfs \in \pi_k^{-1}(r_1, \dots, r_k)$, observing that $f^{\mbfs}$ and $f^{\mbfs'}$ are equal for any two $\mbfs,\mbfs' \in T^\infty_o$ such that $n(\mbfs)=n(\mbfs')=k$ and $\pi_k(\mbfs)=\pi_k(\mbfs')$. In particular, the PDE \eqref{eq:pdeagain} can be regarded as a parametrized parabolic PDE whose coefficients depend in a continuous way on the parameter $(r_1,\cdots,r_k) \in T^{k,\varepsilon}_o$, with the latter space  being (obviously) equipped with the induced topology. By standard Schauder theory, solutions to the PDE, regarded as  real-valued functions defined on $[0,T] \times ({\mathbb T}^d)^k$, are once differentiable in time and twice in space. Together with their derivatives (of order one in time and of order two in space), they  are jointly  H\"older continuous, uniformly in $(r_1,\cdots,r_k)$. As a result,  the mapping
that sends $(r_1,\dots,r_k)$ onto $(h^{r_1,\dots,r_k},(\partial_{i} h^{r_1,\dots,r_k})_{i=1,\dots,dk},(\partial^2_{ij} h^{r_1,\dots,r_k})_{i,j=1\dots,dk})$ is continuous, when all the functions on the right-hand side are regarded as elements of $\rmC([0,T],(\T^d)^k)$
equipped with the $L^\infty$-norm.
Therefore, the map $F_k : T_o^{k, \eps} \to \rmC([0,T] \times (\T^d)^k) \times \rmC([0,T] \times (\T^d)^k) \times \rmC([0,T] \times (\T^d)^k; \R^d)^k$, defined by
\[ F_k(r_1, \dots, r_k) := \left ( h^{r_1, \dots, r_k}, \sum_{j=1}^k \frac{1}{r_j} \Delta_j h^{r_1, \dots, r_k}, \left (\frac{1}{r_j} \nabla_j h^{r_1, \dots, r_k} \right )_{j=1, \dots, k} \right )\]
is continuous when 
the arrival space is  equipped with the uniform convergence topology.
Denoting the components of $F_k$ by $F_k^i$, $i=1,2,3$, and using \eqref{eq:gradu} and \eqref{eq:nablau} together with the identity $h^{\mbfs}=h^{\pi_k(\mbfs)}$ when $n(\mbfs)=k$, it is not difficult to check that $u$, $\boldnabla u$, and $L_c$ 
can be expressed in the following way:
\begin{align*}
u_t(\mu) &= \sum_{k=1}^{\ceil{1/\eps}} \mathbf{1}_{\{k\}}(n(\mbfs(\mu))) [F_k^1(s_1(\mu), \dots, s_k(\mu))](t, x_1(\mu), \dots, x_k(\mu)),\\
(L_c u_t)_\mu &= \sum_{k=1}^{\ceil{1/\eps}} \mathbf{1}_{\{k\}}(n(\mbfs(\mu))) [F_k^2(s_1(\mu), \dots, s_k(\mu))](t, x_1(\mu), \dots, x_k(\mu)),\\
(\boldnabla u_t)_\mu(x) &= \sum_{k=1}^{\ceil{1/\eps}} \mathbf{1}_{\{k\}}(n(\mbfs(\mu))) \sum_{j=1}^k \{[F_k^3(s_1(\mu), \dots, s_k(\mu))](t, x_1(\mu), \dots, x_k(\mu))\}_j \mathbf{1}_{\{x_j(\mu)\}}(x),
\end{align*}
for every $(t, \mu, x) \in [0,T] \times \prob_o(\T^d) \times \T^d$, where we have denoted with a subscript $j$ the components of $F^3_k$. The above expressions, the aforementioned continuity properties of $F_k$, and the measurability of $\emp^{-1}$, give the joint measurability of $u$, $\boldnabla u$ and $L_c u$
with respect to their respective entries, i.e., $t$, $\mu$ and possibly $x$.
Finally, the definition of $u$ and the measurability properties just established, together with the bounds in \eqref{eq:inftyest}, give that, for every $t \in [0,T]$, $u_t \in \hat{\mathfrak{Z}}_{fw}^\infty$. Hence, by Proposition \ref{prop:ident}, we get \eqref{eq:laprima} and \eqref{eq:laseconda}.\\
\quad \\
\textit{The equation solved by $u$}: we show that $u$ solves
\begin{equation}\label{eq:pdepointw}
\partial_t u_t(\mu) + (L_c u_t)_\mu = f(\mu), \quad u_T(\mu) = g(\mu) \quad \text{ for every } (t, \mu)\in (0, T) \times \prob_o(\T^d).
\end{equation}
Let $(t, \mu)\in (0, T) \times \prob_o(\T^d)$ with $\mu= \emp(\mbfs, \mbfx)$ for some $(\mbfs, \mbfx) \in T_o^\infty \times (\T^d)_o^\infty$. If $n(\mbfs)=0$, then
\[ \partial_t u_t (\mu) = \partial_t v^\mbfs (\mu) = (L_c u_t)_\mu = (L_c v_t^\mbfs)_\mu = 0,\]
where we used the definition of $u$ and \eqref{eq:nablau} together with the fact that, in this case, $v^\mbfs=0$; since, in this case, we also have  $f(\mu)=g(\mu)=0$ (as $f$ and $g$ belong to $\hat{\mathfrak{Z}}_\eps^\infty$), we deduce that \eqref{eq:pdepointw} is satisfied at $(t, \mu)$. If $n(\mbfs)>0$, we have
\[ \partial_t u_t(\mu) = \partial_t v^\mbfs_t (\mu) = \partial_t h^\mbfs (x_1, \dots, x_{n(\mbfs)}) = - \Delta^\mbfs h^\mbfs(x_1, \dots, x_{n(\mbfs)})+f^\mbfs(x_1, \dots, x_{n(\mbfs)}) = -(L_c u_t)_\mu + f(\mu),\]
where we have used the definition of $u$, \eqref{eq:nablau}, the PDE \eqref{eq:pden}, the definition
of $f^\mbfs$ in 
\eqref{eq:fmbfs:gmbfs}, and again the fact that $f \in \hat{\mathfrak{Z}}_{\eps}^\infty$. Using the same properties and the definition of $g^\mbfs$ (also in 
\eqref{eq:fmbfs:gmbfs}), it is easy to see that $u_T(\mu) = g(\mu)$. This shows that \eqref{eq:pdepointw} is satisfied at $(t, \mu)$ also in this case. Finally, \eqref{eq:pdepointw}, \eqref{eq:uoper}, and \eqref{eq:laseconda} show that $u$ solves the PDE in the statement thus concluding the proof. 
\end{proof}

\begin{remark} It follows immediately from the proof, that, for a given fixed $\mu \in \prob(\T^d)$,  the functions $t \mapsto v^{\mbfs}(\mu)$ belong to $\rmC^\infty([0,T])$ and not only to $\Lip_b([0,T], \sfd_{eucl})$ because of the smoothness of the mapping
$t \mapsto h^\mbfs(x_1, \dots, x_{n(\mbfs)})$ for a fixed value of $\mbfs \in T_o^\infty$.
\end{remark}

\begin{corollary}\label{prop:pde_approx_cyl}
Let $f \in L^2([0,T];H)$, $g\in H$, and $u$ be the unique solution of
    \[ \partial_t u_t + \boldsymbol{\Delta} u_t =  f_t \quad \text{for a.e.~}t \in (0,T), \quad u_t \bigr |_{t=T} = g.\]
Then, there exist sequences $(\eps_n)_n \subset (0,1)$, $(g^n)_n\subset\hat{\mathfrak{Z}}_c^\infty$ and $(f^n)_n\subset L^2([0,T];H)$, 
satisfying the following two properties:
\begin{enumerate}[i.]
    \item for any $n\in\N_+$,
given the
unique solution $u^n$ to 
     \[ \partial_t u^n_t + \boldsymbol{\Delta} u^n_t =  f^n_t \quad \text{for a.e.~}t \in (0,T), \quad u^n_t \bigr |_{t=T} = g^n,\]
there exist Borel representatives of $u^n$, $\rmD u^n$, and $\boldsymbol{\Delta} u^n$ and, for any $\mbfs \in T^\infty_o$,
there exists also a function $v^{\mbfs, n} \in \mathfrak{T}\hat{\mathfrak{Z}}_{\eps_n}^\infty([0, T])$ such that, at any 
pair $(t,\mu) \in 
[0,T] \times \emp(\mbfs, (\T^d)^\infty_o)$, it holds
\begin{equation}
u_t^n(\mu) = v_t^{\mbfs, n}(\mu), \quad \rmD u_t^n (\mu, \cdot) = \rmD v_t^{\mbfs, n} (\mu, \cdot), \quad \boldsymbol{\Delta} u_t^n (\mu) = \boldsymbol{\Delta} v_t^{\mbfs, n} (\mu);
\end{equation}
\item $g^n\to g$ in $H$, $f^n\to f$ in $L^2([0,T];H)$, and 
\[  \sup_{t \in [0,T]}  |u^n_t-u_t|^2_H + \int_{0}^T \|\rmD u_t^n - \rmD u_t \|^2 \de t \to 0 \]
as $n \to \infty$.
\end{enumerate}
\end{corollary}
\begin{proof} By density of $\hat{\mathfrak{Z}}_c^\infty$ in $L^2(\prob(\R^d), \mathcal{D})$ (see~\eqref{eq:density}) we can find sequences $(g^n)_n \subset \hat{\mathfrak{Z}}_c^\infty$ and $(f^n)_n \subset \mathcal{S}(0,T; \hat{\mathfrak{Z}}_c^\infty)$ such that $g^n \to g$ in $H$ and $f^n \to f$ in $L^2([0,T]; H)$ as $n \to \infty$, where we have denoted by $\mathcal{S}(0,T; \hat{\mathfrak{Z}}_c^\infty)$ the vector space of simple (in time) functions taking values in $\hat{\mathfrak{Z}}_c^\infty$. 
In words, the two inclusions 
$(g^n)_n \subset \hat{\mathfrak{Z}}_c^\infty$ and $(f^n)_n \subset \mathcal{S}(0,T; \hat{\mathfrak{Z}}_c^\infty)$ mean that for every $n \in \N_+$, there exist values $\eps'_n \in (0,1)$
and $K_n \in \N_+$, a partition $0=t_0^n < t_1^n < \dots < t_{K_n-1}^n < t_{K_n}^n=T$, and functions $(f_i^n)_{i=1}^{K_n} \subset \hat{\mathfrak{Z}}_{\eps'_n}^\infty$, such that $g^n \in \hat{\mathfrak{Z}}_{\eps'_n}^\infty$ and
\[ f^n(t) = \sum_{i=0}^{K_n-1} \chi_{[t_i^n, t_{i+1}^n)}(t) f_i^n \quad \text{ for every } t \in [0,T).\]
For a given $h \in \hat{\mathfrak{Z}}_\eps^\infty$, $\eps \in (0,1)$, let us denote by $S_t^{i,n} h$ the unique solution of 
\[ \partial_t v_t + \boldsymbol{\Delta} v_t =  f_i^n \quad \text{for a.e.~}t \in (0,T), \quad v_t \bigr |_{t=T} = h.\]
We deduce from Proposition \ref{prop:cyl_solutions} that there exist Borel representatives of $S_t^{i,n}h$, its gradient, and its Laplacian, and,
for every $\mbfs \in  T^\infty_o$,
there exists also
$v^{\mbfs,i,n,h} \in \mathfrak{T}\hat{\mathfrak{Z}}_{\delta_n}^\infty([0, T])$, 
with $\delta_n :=(\varepsilon_n\wedge \varepsilon)/2$,
such that 
\begin{align*}
S_t^{i,n} h(\mu) &= v_t^{\mbfs,i,n,h}(\mu), \quad &&\text{ for all } \mu \in \emp(\mbfs, (\T^d)^\infty_o)\text{ and }  t \in [0,T],\\ 
\rmD (S_t^{i,n} h) (\mu, \cdot) &= \rmD (v_t^{\mbfs,i,n,h})(\mu, \cdot), \quad &&\text{ for all } \mu \in \emp(\mbfs, (\T^d)^\infty_o)\text{ and }  t \in [0,T],\\
\boldsymbol{\Delta}(S_t^{i,n} h)(\mu) &= \boldsymbol{\Delta} (v_t^{\mbfs,i,n,h})(\mu), \quad &&\text{ for all } \mu \in \emp(\mbfs, (\T^d)^\infty_o)
\text{ and } t \in [0,T].
\end{align*}
Now define
\[ u_T^{n} := g^n, \quad u^{n}_t:=\sum_{i=0}^{K_n-1}S^{i,n}_{T-t_{i+1}^n+t} u^{n}_{t_{i+1}^n}\chi_{[t_i^n,t_{i+1}^n]}(t) \quad t \in [0,T]. \]
Combining 
the above construction with 
Proposition \ref{prop_stab_new}, we easily deduce that 
the sequence 
$(u^n)_n$ 
satisfies 
i and ii.
\end{proof}

Our goal is to relate \eqref{eq:pdeaaa} to the system of interacting particles \eqref{eqn: particles_drift} (when driven by a drift $b$ that is independent of $i$) with the language of the canonical process as in Subsection \ref{ssec:canonical}, see also Subsection \ref{ssec:transfer}.

\begin{proposition}\label{prop:ito_general}  Let $b:[0,T]\times\prob(\T^d) \times \T^d \to \R^d$ be a bounded and measurable function, $t_0\in [0,T)$ be an initial time, $p$ be a bounded $\cD$-density, $f \in L^2([0,T];H)$, and $g \in H$ be a final datum. If $u$ is the unique solution to \eqref{eq:pdewithf}, then the process $(u_t(\mu_t^\infty))_{t \in [t_0,T]}$ can be expanded 
in the following two equivalent ways:
\begin{align*}
    g(\mu^{\infty}_T) & =  u_{t}(\mu^{\infty}_{t}) + \int_t^T f_r(\mu_r^\infty) \de r -\int_{t}^T\int_{\T^d}\rmD u_r(\mu^{\infty}_r,x)\cdot b_r(\mu^{\infty}_r,x)\de\mu^{\infty}_r(x)\de r\\
    & \quad +\sum_{i=1}^{\infty}\int^T_{t} \sqrt{2\varsigma_i}\rmD u_r(\mu^{\infty}_r, \xi^i_r) \cdot \de \beta^i_r, \quad t \in [t_0, T],\, \mathbf{ P}^{t_0, \emp^{-1}_\sharp( p\cD)}\text{-a.e.},\\
    g(\mu^{\infty}_T) &  = u_{t}(\mu^{\infty}_{t})  +  \int_t^T f_r(\mu_r^\infty) \de r +\sum_{i=1}^{\infty}\int^T_{t} \sqrt{2\varsigma_i}\rmD u_r(\mu^{\infty}_r, \xi^i_r) \cdot\de \beta^i_r,\quad  t \in [t_0, T],\,\mathbf{ Q}^{t_0, \emp^{-1}_\sharp (p \cD)}\text{-a.e.} \, ,
\end{align*}
where 
$\mathbf{ Q}^{t_0, \emp^{-1}_\sharp (p \cD)}$
is defined as in Proposition 
\ref{prop:pandq}
(it being understood that 
all the 
$b^i$'s in 
\eqref{eqn: particles_drift}
are equal to $b$).

\end{proposition} Note that in the above display we are using the convention explained in Remark \ref{rem:toruspb} to evaluate $\rmD u_r(\mu_r^\infty, \cdot)$ at $\xi_r^i$.
\begin{proof}
Recalling the notation 
\eqref{eq:[b1,b2]},
let us set $t \in [0,T]\mapsto k_t := f_t -\sclprd{b}{\rmD u_t}\in H$.
Since $b$ is bounded and $u\in L^2([0,T];H^{1,2})$, $k$ belongs to $ L^2([0,T];H)$, which makes it possible to apply Corollary \ref{prop:pde_approx_cyl}. We then introduce $(\eps_n)_n \subset (0,1)$, $(g^n)_n\subset\hat{\mathfrak{Z}}^\infty_c$, $(k^n)_n \subset L^2([0,T];H)$,   $(u^n)_n \subset L^2([0,T];H^{1,2})$
\and
$(v^{\mbfs, n})_n \in \prod_{n=1}^\infty \mathfrak{T}\hat{\mathfrak{Z}}_{\eps_n}^\infty([0, T])$, the sequences provided by Corollary \ref{prop:pde_approx_cyl} (note $f_t$ there is $k_t$ here). Thanks to the latter, there exist Borel representatives of $u^n$, $\rmD u^n$, and $\boldsymbol{\Delta} u^n$ such that, for every $\mbfs \in  T^\infty_o$,
\begin{align}
 \label{eq:theprop2}
u_t^n(\mu) &= v_t^{\mbfs,n}(\mu), \quad &&\text{ for all } \mu \in \emp(\mbfs, (\T^d)^\infty_o) \
\text{and} \ t \in [0,T],\\ \label{eq:laprima2}
\rmD u_t^n(\mu, \cdot) &= \rmD v_t^{\mbfs,n}(\mu, \cdot), \quad &&\text{ for all } \mu \in \emp(\mbfs, (\T^d)^\infty_o) \
\text{and} \ t \in [0,T],\\ \label{eq:laseconda2}
\boldsymbol{\Delta}u_t^n(\mu) &= \boldsymbol{\Delta} v_t^{\mbfs,n}(\mu), \quad &&\text{ for all } \mu \in \emp(\mbfs, (\T^d)^\infty_o) \
\text{and} \ t \in [0,T].
\end{align}
Note that \eqref{eq:theprop2} also gives that
\begin{equation}
    \partial_t u_t^n(\mu) = \partial_t v_t^{\mbfs, n}(\mu) \quad \text{ for all } \mu \in \emp(\mbfs, (\T^d)^\infty_o) \
\text{and a.e.~} t \in [0,T], \label{eq:laterza}
\end{equation}
where both functions are regarded as representatives of Sobolev derivatives, see also the proof of Proposition \ref{prop: ito_cyl}.
For fixed $n \in \N_+$ and $(\mbfs, \mbfx) \in T_o^\infty \times (\T^d)_o^\infty$, since $v^{\mbfs, n} \in \mathfrak{T}\hat{\mathfrak{Z}}_{\eps_n}^\infty([0, T])$, we can apply Proposition \ref{prop: ito_cyl} and obtain that
\begin{equation}
    \begin{split}
        v_t^{\mbfs,n}(\mu^\infty_t)  &= v_{t_0}^{\mbfs,n}(\mu^\infty_{t_0}) + \int_{t_0}^t \left (\partial_r  v_t^{\mbfs,n}(\mu^\infty_r) 
         +  \boldsymbol{\Delta}v_r^{\mbfs,n}(\mu^\infty_r)\right )\de r
         \\
         &\hspace{15pt} + \sum_{i = 1}^{\infty}\int_{t_0}^t \sqrt{2 \varsigma_i} \rmD v_r^{\mbfs,n}(\mu_r^\infty, \xi^i_r) \cdot \de \beta^i_r,\quad t\in[t_0,T],\quad \mathbf{P}^{t_0,(\mbfs,\mbfx)}\text{-a.s.},
   \end{split}
    \end{equation}
where we recall Remark \ref{rem:toruspb} for the meaning of the evaluation of $\rmD v_r^{\mbfs, n}(\mu_r^\infty, \cdot)$ at $\xi_r^i$.  
Using the above relations \eqref{eq:theprop2}, \eqref{eq:laprima2}, \eqref{eq:laseconda2}, and \eqref{eq:laterza}, we can write
    \begin{equation}\label{eq:equalitywithn}
    \begin{split}
        u_t^{n}(\mu^\infty_t)  &= u_{t_0}^{n}(\mu^\infty_{t_0}) + \int_{t_0}^t \left (\partial_r u_r^n(\mu^\infty_r) 
         +  \boldsymbol{\Delta}u_r^{n}(\mu^\infty_r)\right )\de r
         \\
         &\hspace{15pt} + \sum_{i = 1}^{\infty}\int_{t_0}^t \sqrt{2 \varsigma_i} \rmD u_r^{n}(\mu_r^\infty, \xi^i_r) \cdot \de \beta^i_r,\quad t\in[t_0,T],\quad \mathbf{P}^{t_0,(\mbfs,\mbfx)}\text{-a.s.}
   \end{split}
\end{equation}
Observe that, in the above formula, the process defined by the stochastic integral depends in principle on $n$ and $(\mbfs, \mbfx)$. However, by Lemma \ref{lem:universal:mathcalI0} and Corollary \ref{corol:universal:zeta0} applied to $b= 2\rmD u^n$, we can assume that it only depends on $n$, provided that $(\mbfs, \mbfx) = \emp^{-1}(m)$ for $m$ varying in a full $\cD$-measure subset of $\prob(\T^d)$. In other words, for every $n \in \N_+$ there exists a progressively-measurable mapping $\mathcal{I}^{t_0, n} : \Xi \to \rmC([0,T]; \R^d)$ such that for $\cD$-a.e.~$m \in \prob(\T^d)$ and for every bounded $\cD$-density $q$ it holds
\begin{align*}
&\mathbf{P}^{t_0, \emp^{-1}(m)} \left ( \left \{ \forall t \in [t_0,T], \quad \mathcal{I}_t^{t_0,n} =  \sum_{i = 1}^{\infty}\int_{t_0}^t \sqrt{2 \varsigma_i} \rmD u_r^{n}(\mu_r^\infty, \xi^i_r) \right \} \right ) \\
&=\mathbf{P}^{t_0, \emp_\sharp^{-1}(q\cD)} \left ( \left \{ \forall t \in [t_0,T], \quad \mathcal{I}_t^{t_0,n} =  \sum_{i = 1}^{\infty}\int_{t_0}^t \sqrt{2 \varsigma_i} \rmD u_r^{n}(\mu_r^\infty, \xi^i_r) \right \} \right )=1.   
\end{align*}
Therefore, we can assume that for $\cD$-a.e.~$m \in \prob(\T^d)$, the equality in \eqref{eq:equalitywithn} is a $\mathbf{ P}^{t_0, \emp^{-1}(m)}$-a.e.~equality between Borel functions in $\Xi$ which only depends on $n$. By the disintegration formula  \eqref{eq:diswithp} for $\mathbf{P}^{t_0, \emp_\sharp^{-1}(p\cD)}$ and \eqref{eq:equalitywithn}, we get that for every $n \in \N_+$ it holds
\begin{equation}\label{eq:ito:under:em(pD)}
\begin{split}
        u_t^{n}(\mu^\infty_t)  - u_{t_0}^{n}(\mu^\infty_{t_0}) &= \int_{t_0}^t \left (\partial_r u_r^n(\mu^\infty_r) 
         +  \boldsymbol{\Delta}u_r^{n}(\mu^\infty_r)\right )\de r
         \\
         &\hspace{15pt} + \sum_{i = 1}^{\infty}\int_{t_0}^t \sqrt{2 \varsigma_i} \rmD u_r^{n}(\mu_r^\infty, \xi^i_r) \cdot \de \beta^i_r,\quad t\in[t_0,T],\quad \mathbf{P}^{t_0,\emp_\sharp^{-1}(p\cD)}\text{-a.s.}
   \end{split}
\end{equation}
Using the additive structure of the right-hand side, we can easily write a similar formula but for the increment 
$g^{n}(\mu_T^\infty) - u_t(\mu_t^\infty)$ (in place of 
$u_t(\mu_t^\infty) - u_{t_0}(\mu_{t_0}^\infty)$). And then, using the fact that $\partial_r u_r^n + \boldsymbol{\Delta} u_r^n = k_r^n$ (see Corollary \ref{prop:pde_approx_cyl}) together with the fact that $(\mu_t^\infty)_\sharp \mathbf{P}^{t_0, \emp^{-1}_\sharp (p \cD)} \ll \cD$ for every $t \in [t_0, T]$ by \eqref{eq:acbound}, we obtain for every $n \in \N_+$
\begin{equation}\label{eq:itowithn}
  g^n(\mu^\infty_T)  = u_{t}^{n}(\mu^\infty_{t}) + \int_{t}^T k_r^n(\mu_r^\infty) \de r + \sum_{i = 1}^{\infty}\int_{t}^T \sqrt{2 \varsigma_i} \rmD u_r^{n}(\mu_r^\infty, \xi^i_r) \cdot \de \beta^i_r,\quad t \in [t_0, T], \, \mathbf{P}^{t_0,\emp_\sharp^{-1}(p\cD)}\text{-a.s.}  
\end{equation}

Now we show that every term in the above equality converges as $n \to \infty$ to the corresponding term without the apex $n$. Using that $(\mu_t^\infty)_\sharp \mathbf{P}^{t_0, \emp_\sharp^{-1}(p\cD)} \le \|p\|_\infty \cD$ (again, coming from \eqref{eq:acbound}), we get
\begin{align*}
    \int_\Xi \lvert g^n(\mu^\infty_T) - g(\mu^\infty_T) \rvert^2 \de \mathbf{P}^{t_0, \emp_\sharp^{-1}(p\cD)}  &= \int_{\prob(\T^d)}\lvert  g^n(\mu) - g(\mu)\rvert^2\de[(\mu^{\infty}_T)_\sharp\mathbf{P}^{t_0, \emp_\sharp^{-1}(p\cD)} ](\mu) \\
    & \le \|p\|_\infty |g^n-g|_H^2 \to 0,\quad \text{as } n\to\infty,
\end{align*}
the convergence of the last term to $0$ following from the properties stated in Corollary \ref{prop:pde_approx_cyl}. Similarly, 
using again the conclusion of Corollary \ref{prop:pde_approx_cyl}, one can show that
\begin{align*}
&\int_\Xi \lvert u^n_{t}(\mu^\infty_{t}) - u_{t}(\mu^\infty_{t}) \rvert^2] \de \mathbf{P}^{t_0, \emp_\sharp^{-1}(p\cD)} \to 0,\quad \text{as } n\to\infty,
\\
&\int_\Xi \left\lvert\int_{t}^T (k^n_r(\mu^\infty_r)-k_r(\mu^\infty_r))\de r\right\rvert^2 \de \mathbf{P}^{t_0, \emp_\sharp^{-1}(p\cD)}
\\
&\hspace{30pt}\le T\int_0^T \int_{\Xi}\lvert k_r^n(\mu_r^\infty) - k_r(\mu_r^\infty) \rvert^2 \de \mathbf{P}^{t_0, \emp_\sharp^{-1}(p\cD)} \de r \to 0,\quad \text{as }n\to\infty.
\end{align*}
Finally, let us focus on the stochastic integral. By It\^o isometry and thanks, once again, to \eqref{eq:acbound}, it holds 
\begin{align*}
    &\int_\Xi \left\lvert \sum_{i = 1}^{\infty}\int_{t}^T \sqrt{2\varsigma_i} \left (\rmD u^n_r(\mu_r^\infty, \xi^i_r) - \rmD u_r(\mu_r^\infty, \xi^i_r)\right )\cdot \de \beta^i_r  \de r\right\rvert^2 \de \mathbf{P}^{t_0, \emp_\sharp^{-1}(p\cD)}  \\ &\qquad=2 \int_\Xi \int_{t}^T\sum_{i=1}^{\infty} \varsigma_i\lvert \rmD u^n_r(\mu_r^\infty, \xi^i_r) - \rmD u_r(\mu_r^\infty, \xi^i_r)\rvert^2\de r \de \mathbf{P}^{t_0, \emp_\sharp^{-1}(p\cD)}\\
    &\qquad =2 \int_{t}^T\int_\Xi \int_{\T^d}\lvert \rmD u^n_r(\mu_r^\infty, x) - \rmD u_r(\mu_r^\infty, x)\rvert^2\de\mu^\infty_r(x) \de \mathbf{P}^{t_0, \emp_\sharp^{-1}(p\cD)} \de r\\
    &\qquad \le 2\|p\|_\infty \int_{0}^T\int_{\prob(\T^d)}\int_{\T^d}\lvert \rmD u^n_r(\mu, x) - \rmD u_r(\mu, x)\rvert^2\de\mu(x)\de\mathcal{D(\mu)}\de r\\ &\qquad=2\|p\|_\infty \int_{0}^T\| \rmD u^n_r - \rmD u_r\|^2\de r\to 0,\quad\text{as } n\to\infty,
\end{align*}
where the last line follows from Corollary \ref{prop:pde_approx_cyl}. Thus, we can let $n$ tend to $ \infty$ in 
\eqref{eq:itowithn}. In the limit, we obtain the same formula, but without the apex $n$, as we announced earlier. 
Recalling that $k_t = f_t-\sclprd{b}{\rmD u_t}$, we obtain
\begin{align*}
    g(\mu^{\infty}_T) &=  u_{t}(\mu^{\infty}_{t}) +\int_t^T f_r(\mu_r^\infty) \de r  -\int_{t}^T\int_{\T^d} \rmD u_r(\mu^{\infty}_r,x)\cdot b(\mu^{\infty}_r,x)\de\mu^{\infty}_r(x)\de r \\
    & \quad +\sum_{i=1}^{\infty}\int^T_{t} \sqrt{2\varsigma_i}\rmD u_r(\mu^{\infty}_r, \xi^i_r)
    \cdot
    \de \beta^i_r, \quad t \in [t_0, T],\, \mathbf{P}^{t_0,\emp_\sharp^{-1}(p\cD)}\text{-a.s.}
\end{align*}
By Girsanov theorem, the sum of the last two terms on the right-hand side can be written as a stochastic integral, driven by a new Brownian motion under the probability $(z_T^{t_0,
\emp_\sharp^{-1}(p \cD)})^{-1} {\mathbf P}^{t_0,\emp_\sharp^{-1}(p \cD)}$, see 
\eqref{eqn:z_T}.
Using the transformation
$\cercle{$\Psi$}^{b}$,
see \eqref{eq:new:expression:Qcursm0} (with the bold superscript ${\boldsymbol b}$ being replaced by $b$ since the $b^i$'s are all equal to $b$), the resulting expansion can be transferred onto the canonical space.
We obtain
\begin{align*}
    g(\mu^{\infty}_T)= u_{t}(\mu^{\infty}_{t}) + \int_t^T f_r(\mu_r^\infty) \de r+ \sum_{i=1}^{\infty}\int^T_{t} \sqrt{2\varsigma_i}\rmD u_r(\mu^{\infty}_r, \xi^i_r)\cdot\de \beta^i_r,\quad t \in [t_0, T],\,\mathbf{Q}^{t_0, \emp_{\sharp}^{-1}(p\cD)}\text{-a.s.,}
\end{align*}
which completes the proof.
\end{proof}

The following result, which constitutes the main finding of this section, 
provides a Kolmogorov-type representation formula for the solution of the 
transport-diffusion equation~\eqref{eq:pdeaaa} in terms of the drifted particle 
system~\eqref{eqn: particles_drift}.

\begin{theorem}\label{thm:repr} Let $b:[0,T]\times\prob(\T^d) \times \T^d \to \R^d$ be a bounded and measurable function, and $g \in H$ be a final datum. If $u$ is the unique solution to \eqref{eq:pdeaaa}, then, for every $t_0 \in [0,T)$, there exists a Borel subset $O$ of $\cP(\T^d)$, with ${\mathcal D}(O)=1$, such that
\begin{equation}\label{eqn:repr}
    u_{t_0}(m) = \mathbb{E}^{{\mathbf Q}^{t_0,\emp^{-1}(m)}}
[g(\mu_T^\infty)],\quad \text{ for every } m \in O,
\end{equation}

where ${\mathbb{E}}^{{\mathbf Q}^{t_0,\emp^{-1}(m)}}$ is the expectation under ${\mathbf Q}^{t_0,\emp^{-1}(m)}$.
\end{theorem}
\begin{remark} The set $O$ appearing in the statement of Theorem~\ref{thm:repr} depends on the choice of $t_0 \in [0,T)$.
This dependence will also occur in several statements throughout the remainder of the paper.
Rather than making it explicit each time, as in Theorem~\ref{thm:repr}, we shall simply write that the equality holds “for every $t_0 \in [0,T)$ and $\cD$-a.e.~$m \in \prob(\T^d)$”.
\end{remark}

\begin{proof}
    The proof is divided into two steps: the first establishes a weaker version of a \eqref{eqn:repr}, which is then combined with \eqref{eq:disintegration:Qcursm0} in the second step to obtain the desired representation formula.\\
    \smallskip
    \textit{First step:}
    The objective of this step is to prove that
    $$u_{t_0}(\mu^\infty_{t_0}) = {\mathbb{E}}^{{\mathbf Q}^{t_0,\emp^{-1}_\sharp(\cD)}}[g(\mu^{\infty}_T)\rvert \mu^\infty_{t_0}], \quad \text{for every $t_0 \in [0,T)$ and ${\mathbf Q}^{t_0,\emp^{-1}_\sharp(\cD)}$~-a.e.}$$ 
    Let us denote $\cursm:= \emp_{\sharp}^{-1}(\cD)$; by Proposition \ref{prop:ito_general} with $f =0$, $p=1$, and $t=t_0$ 
\begin{align*}
   g(\mu^{\infty}_T) &  = u_{t_0}(\mu^{\infty}_{t_0})  + \sum_{i=1}^{\infty}\int^T_{t_0} \sqrt{2\varsigma_i}\rmD u_r(\mu^{\infty}_r, \xi^i_r) \cdot\de \beta^i_r,\quad  \mathbf{ Q}^{t_0, \cursm}\text{-a.e.}
\end{align*} 
If we assume that the stochastic integral is a $\mathbf{ Q}^{t_0, \cursm}$-martingale, then \eqref{eqn:repr} follows by taking the conditional expectation with respect to $\mu_{t_0}^\infty$. 

To prove the martingale property of the stochastic integral, we argue as follows.
Since the stochastic integral is already known to be a local martingale under $\mathbf{ Q}^{t_0, \cursm}$, it suffices to show that
\begin{equation}\label{eqn: unif_int}
    {\mathbb{E}}^{{\mathbf Q}^{t_0,\cursm}}\left[\sup_{ t\in [t_0,T]}\left\lvert\sum_{i=1}^{\infty}\int^t_{t_0} \sqrt{2 \varsigma_i}\rmD u_r(\mu^{\infty}_r, \xi^i_r)\cdot\de \beta^i_r\right\rvert\right]<\infty.
\end{equation}
In order to prove \eqref{eqn: unif_int}, let us recall the following two estimates. First, by Proposition \ref{prop:mut} and Theorem \ref{thm:pde}, we have 
\begin{equation}\label{eq:thefirstbound}
    {\mathbb{E}}^{{\mathbf P}^{t_0,\cursm}}\left[\sum_{i=1}^{\infty}\int_{t_0}^T\varsigma_i\lvert \rmD u_r(\mu^{\infty}_r, \xi^i_r)\rvert^2\de r\right]  \le \int_0^T \int_{\prob(\T^d)\times\T^d}\lvert \rmD u_r(\mu,x)\rvert^2 \de\overline{\mathcal{D}}(\mu,x)\de r <\infty.
\end{equation}
Second, recalling the definition of $z_{T}^{t_0, \cursm}$ in \eqref{eqn:z_T}, we have
\begin{equation}\label{eqn: bound_girsanov_exp}
  {\mathbb{E}}^{{\mathbf P}^{t_0,\cursm}} \left [ (z_{T}^{t_0, \cursm})^{-2} \right ] \le 2 \exp\{T\norm{b}_\infty^2\}.
\end{equation}
And then, by the definition of $\mathbf{Q}^{t_0, \cursm}$ in \eqref{eq:new:expression:Qcursm0}, we have
\begin{align}
     & {\mathbb{E}}^{{\mathbf Q}^{t_0,\cursm}}\left[\sup_{t \in [t_0,T]}\left\lvert\sum_{i=1}^{\infty}\int^t_{t_0} \sqrt{2\varsigma_i}\rmD u_r(\mu^{\infty}_r, \xi^i_r)\cdot\de \beta^i_r\right\rvert\right] \nonumber \\
    &\quad \leq  {\mathbb{E}}^{{\mathbf P}^{t_0,\cursm}}\left[\frac{1}{z_T^{t_0, \cursm}}\sup_{t \in [t_0,T]}\left\lvert\sum_{i=1}^{\infty}\int^t_{t_0} \sqrt{2\varsigma_i}\rmD u_r(\mu^{\infty}_r, \xi^i_r)\cdot\de \beta^i_r\right\rvert\right]\label{eqn: mg_est_1}\\
     &\qquad+ {\mathbb{E}}^{{\mathbf P}^{t_0,\cursm}}\left[\frac{1}{z_T^{t_0, \cursm}}\sup_{t \in [t_0,T]}\left\lvert\sum_{i=1}^{\infty}\int^t_{t_0} \varsigma_i\rmD u_r(\mu^{\infty}_r, \xi^i_r)\cdot b_r(\mu^{\infty}_r, \xi^i_r)\de r\right\rvert\right]\label{eqn: mg_est_2}.
\end{align}
By combining Cauchy-Schwarz inequality with Doob's inequality and the two bounds \eqref{eq:thefirstbound} and \eqref{eqn: bound_girsanov_exp}, we obtain 
\begin{align*}
    \eqref{eqn: mg_est_1}&\leq {\mathbb{E}}^{{\mathbf P}^{t_0,\cursm}}\left[(z_T^{t_0, \cursm})^{-2} \right ]^{1/2} {\mathbb{E}}^{{\mathbf P}^{t_0,\cursm}}\left[ \sup_{t \in [t_0,T]}\left\lvert\sum_{i=1}^{\infty}\int^t_{t_0} \sqrt{2 \varsigma_i}\rmD u_r(\mu^{\infty}_r, \xi^i_r)\cdot\de \beta^i_r\right\rvert^2\right]^{1/2}\\
    &\leq 4 e^{\frac{1}{2} T\norm{b}_\infty^2}{\mathbb{E}}^{{\mathbf P}^{t_0,\cursm}}\left[\sum_{i=1}^{\infty}\int_{t_0}^T \varsigma_i\lvert \rmD u_r(\mu^{\infty}_r, \xi^i_r)\rvert^2\de r\right]^{\frac{1}{2}}< \infty.
\end{align*}
Similarly, 
\begin{align*}
    \eqref{eqn: mg_est_2}&\leq {\mathbb{E}}^{{\mathbf P}^{t_0,\cursm}}\left[(z_T^{t_0, \cursm})^{-2} \right ]^{1/2} {\mathbb{E}}^{{\mathbf P}^{t_0,\cursm}}\left[\sup_{t \in [t_0,T]}\left\lvert\sum_{i=1}^{\infty}\int^t_{t_0} \varsigma_i\rmD u_r(\mu^{\infty}_r, \xi^i_r)\cdot b(\mu^{\infty}_r, \xi^i_r)\de r\right\rvert^2\right]^{\frac{1}{2}}\\
    &\leq \sqrt{2}e^{\frac{1}{2} T\norm{b}_\infty^2} \sqrt{T}\norm{b}_\infty{\mathbb{E}}^{{\mathbf P}^{t_0,\cursm}}\left[\sum_{i=1}^{\infty}\int_{t_0}^T \varsigma_i\lvert \rmD u_r(\mu^{\infty}_r, \xi^i_r)\rvert^2\de r\right]^{\frac{1}{2}}< \infty.
\end{align*}
Gathering the last three displays, we get \eqref{eqn: unif_int}, which 
concludes the proof of the first step.

    \textit{Second step:} From the first step, it follows that for any bounded and measurable $p\colon\prob(\T^d)\to\R$
    \begin{equation*}
        {\mathbb{E}}^{{\mathbf Q}^{t_0,\emp^{-1}_\sharp(\cD)}}\left[p(\mu^\infty_{t_0}) u_{t_0}(\mu^\infty_{t_0}) \right] = {\mathbb{E}}^{{\mathbf Q}^{t_0,\emp^{-1}_\sharp(\cD)}}\left[p(\mu^\infty_{t_0}) g(\mu_T^\infty)\right].
\end{equation*}
Then, by \eqref{eq:disintegration:Qcursm0}, we can rewrite the latter as
\begin{equation}
\label{eq:improved:Kolmogorov:1}
\int_{\prob(\T^d)} {\mathbb{E}}^{{\mathbf Q}^{t_0,\emp^{-1}(m)}}\left[p(\mu^\infty_{t_0})  u_{t_0}(\mu^\infty_{t_0}) \right] \de {\mathcal D}(m) = \int_{\prob(\T^d)}{\mathbb{E}}^{{\mathbf Q}^{t_0,\emp^{-1}(m)}}\left[p(\mu^\infty_{t_0}) g(\mu_T^\infty)\right]\de \cD(m).
\end{equation}
On the one hand, the left-hand side of \eqref{eq:improved:Kolmogorov:1} is equal to 
\begin{equation*}
\int_{\prob(\T^d)}
{\mathbb{E}}^{{\mathbf Q}^{t_0,\emp^{-1}(m)}}\left[p(\mu^\infty_{t_0}) u_{t_0}(\mu^\infty_{t_0}) \right] \de {\mathcal D}(m) = \int_{\prob(\T^d)}  p(m) u_{t_0}(m) \de {\mathcal D}(m),
\end{equation*}
since $\mu^\infty_{t_0} = m$ holds ${\mathbf Q}^{t_0,\emp^{-1}(m)}$-a.e..
On the other hand, the right-hand side of \eqref{eq:improved:Kolmogorov:1} is equal to
\begin{equation*}
\int_{\prob(\T^d)}
{\mathbb{E}}^{{\mathbf Q}^{t_0,\emp^{-1}(m)}}\left[ 
p(\mu^\infty_{t_0}) 
g(\mu^\infty_T)
\right]
\de {\mathcal D}(m)
=
\int_{\prob(\T^d)} 
p(m) 
{\mathbb{E}}^{{\mathbf Q}^{t_0,\emp^{-1}(m)}}\left[  
g(\mu^\infty_T)
\right]
\de {\mathcal D}(m).
\end{equation*}
Then, by combining the last two displays, we deduce that 
\begin{equation*}
\int_{\prob(\T^d)} 
p(m) 
{\mathbb{E}}^{{\mathbf Q}^{t_0,\emp^{-1}(m)}}\left[  
g(\mu^\infty_T)
\right]
\de {\mathcal D}(m)=
\int_{\prob(\T^d)} 
p(m) 
u_{t_0}(m)
\de {\mathcal D}(m),
\end{equation*}
and, finally, we deduce that, for ${\mathcal D}$-a.e.~$m$, 
\begin{equation*}
u_{t_0}(m)
= {\mathbb{E}}^{{\mathbf Q}^{t_0,\emp^{-1}(m)}}\left[  
g(\mu^\infty_T)
\right].
\end{equation*}
\end{proof}

We conclude this section with a different version of Proposition \ref{prop:ito_general} which will be needed when addressing the control problem in Section \ref{sec: control}. The main difference concerns the fact that the subset of $m \in \prob(\T^d)$ of full $\cD$-measure at which the It\^o expansion holds true, is independent of} the drift. In the statement below, the latter is denoted by ${\boldsymbol \alpha}=(\alpha^i)_i$ and is thus allowed to depend on $i$, as originally considered in the beginning of this section. Accordingly, we denote by 
${\mathbf Q}^{t_0,\emp^{-1}(m),{\boldsymbol \alpha}}$ the probability measure introduced in Proposition \ref{prop:pandq},
when $(b^i)_i$ is given by 
$(\alpha^i)_i$.
\begin{proposition}\label{prop:ito_general:new} Let $t_0 \in [0,T)$ be a initial time, $f \in L^2([0,T]; H)$ and $g \in H$. If $u$ is the unique solution of \eqref{eq:pdewithf} with $b=0$, then there exists a Borel subset $O$ of $\cP(\T^d)$, with ${\mathcal D}(O)=1$, such that, 
for any $m \in O$ and any collection of uniformly bounded and measurable functions 
$(\alpha^i : [0,T] \times \cP(\T^d) \times \T^d \rightarrow \R^d)_{i}$, it holds 
\begin{align*}
    g(\mu^\infty_T) & =  u_{t}(\mu^\infty_{t})  
    + \int_{t}^T f_r(\mu^\infty_r) \de r   
    + \sum_{i =1}^\infty
    \int_{t}^T 
    \varsigma _i
    \rmD u_r(\mu^\infty_r,\xi_r^i)\cdot \alpha_r^i(\mu^\infty_r,\xi_r^i) \ \de r 
    \\
    & \quad +\sum_{i=1}^{\infty}\int^T_{t} \sqrt{2 \varsigma_i}\rmD  u_r(\mu^\infty_r, \xi^i_r) \cdot \de \beta^i_r,\quad t\in[t_0,T],\quad{\mathbf Q}^{t_0,\emp^{-1}(m),\alpha}\text{-a.s.}
\end{align*}
\end{proposition}
Note that in the above display we are using the convention explained in Remark \ref{rem:toruspb} to evaluate $\rmD u_r(\mu_r^\infty, \cdot)$ and $\alpha_r^i(\mu_r^\infty, \cdot)$ at $\xi_r^i$.
\begin{proof}
The proof is divided into two steps. \\
\textit{First step:} Handling the case $\alpha^i= 0$, for all $i \in {\mathbb N}_+$, and defining the set $O$. \\
By Proposition \ref{prop:ito_general} with $b=0$ and $p= 1$, we know that, with probability 1 under ${\mathbf{P}}^{t_0, \emp^{-1}_\sharp(\cD)}$, for all $t \in [t_0,T]$, 
\begin{align*}
    g(\mu^{\infty}_T) & =  u_{t}(\mu^{\infty}_{t})  
    + \int_{t}^T f_r(\mu^{\infty}_r) \de r 
 +\sum_{i=1}^{\infty}\int^T_{t} \sqrt{2 \varsigma^i}{\rmD} u_r(\mu^{\infty}_r, \xi^i_r) \cdot \de \beta^i_r.
\end{align*}
Note that, the result above holds for one arbitrary version of ${\rmD} u$. 
Then, by the second display in the proof of Lemma \ref{lem:universal:mathcalI0}, with $b=2{\rmD}u$ (at least,
for one version of the latter), we deduce that there exists a progressively measurable mapping 
${\mathcal I}^{t_0} : \Xi \rightarrow \rmC([0,T];{\mathbb R}^d)$ such that ${\mathcal I}^{t_0}_T - {\mathcal I}^{t_0}_t$ coincides ${\mathbf P}^{t_0, \emp^{-1}_\sharp (\mathcal D)}$-a.s. with the stochastic integral appearing in the expansion right above. The statement of Lemma \ref{lem:universal:mathcalI0} says that 
${\mathcal I}_T^{t_0} - {\mathcal I}_t^{t_0}$ also coincides with the same stochastic integral, but 
${\mathbf P}^{t_0, \emp^{-1}(m)}$-a.s., for $m$ belonging to a Borel subset $O_1$ of $\cP(\T^d)$ such that 
${\mathcal D}(O_1)=1$.
We thus rewrite the above identity in the form 
\begin{align*}
{\mathbf P}^{t_0, \emp^{-1}_\sharp (\mathcal D)}
\left( \left\{\forall t \in [0,T], 
\quad     g(\mu^{\infty}_T)  =  u_t(\mu^{\infty}_t)  
    + \int_t^T f_r(\mu^{\infty}_r) \de r 
     + {\mathcal I}^{t_0}_T - {\mathcal I}^{t_0}_t
\right\} 
\right) = 1. 
\end{align*}
Using the disintegration formula \eqref{eq:disintegration:Pcursm0}, we can write this as
\begin{align*}
\int_{\cP(\T^d)}
{\mathbf P}^{t_0, \emp^{-1}(m)}
\left( \left\{\forall t \in [0,T], 
\quad     g(\mu^{\infty}_T)  =  u_t(\mu^{\infty}_t)  
    + \int_t^T f_r(\mu^{\infty}_r) \de r 
     + {\mathcal I}^{t_0}_T - {\mathcal I}^{t_0}_t
\right\} 
\right) \de {\mathcal D}(m) = 1,
\end{align*}
which implies that there exists a Borel subset $O_2$ of $\cP(\T^d)$ such that 
${\mathcal D}(O_2)=1$ and, for any $m \in O_2$, 
\begin{align*}
{\mathbf P}^{t_0, \emp^{-1}(m)}
\left( \left\{\forall t \in [0,T], 
\quad     g(\mu^{\infty}_T)  =  u_t(\mu^{\infty}_t)  
    + \int_t^T f_r(\mu^{\infty}_r) \de r 
     + {\mathcal I}^{t_0}_T - {\mathcal I}^{t_0}_t
\right\} 
\right) = 1.
\end{align*}
On $O_1 \cap O_2$ (which has again measure equal to 1 under ${\mathcal D}$), 
the following two expansions hold true 
${\mathbf P}^{t_0, \emp^{-1}(m)}$-a.s., 
for any $t \in [t_0,T]$,
\begin{align}
g(\mu^{\infty}_T) 
&= u_t(\mu^{\infty}_t)  
    + \int_t^T f_r(\mu^{\infty}_r) \de r 
   + {\mathcal I}^{t_0}_T - {\mathcal I}^{t_0}_t
   \label{eq:expansion:ito:with:I}
\\
&=  u_t(\mu^{\infty}_t)  
    + \int_t^T f_r(\mu^{\infty}_r) \de r 
   +\sum_{i=1}^{\infty}\int^T_t \sqrt{2 \varsigma_i}{\rm D} u_r(\mu^{\infty}_r, \xi^i_r) \cdot \de \beta^i_r.
   \nonumber
\end{align}
Finally, we recall that, by Proposition \ref{prop:mut}, 
\begin{equation*}
{\mathbb{E}}^{{\mathbf P}^{t_0,\emp^{-1}_{\sharp} (\mathcal D)}}\left[\sum_{i =1}^\infty
\int_{0}^T  \varsigma^i  \left\vert 
{\rm D}u_r(\mu^\infty_r,\xi_r^i) \right\vert^2\de r\right]
 = 
 \int_{0}^T \int_{\cP(\T^d)}\left(  \int_{\T^d} \vert {\rm D}u_r(\mu,x)  \vert^2
 \de \mu(x) \right) \de {\mathcal D}(\mu)\de r < \infty,
\end{equation*}
which implies, again thanks to \eqref{eq:disintegration:Pcursm0}, that there exists a Borel subset $O_3$ of $\cP(\T^d)$, with 
${\mathcal D}(O_3)=1$, such that, for any $m \in O_3$, 
\begin{equation}
\label{eq:bound:L2:Du:mxx}
{\mathbb{E}}^{{\mathbf P}^{t_0,\emp^{-1}(m)}}
\left[
\sum_{i =1}^\infty
\int_{0}^T  \varsigma^i  \left\vert 
{\rm D}u_r(\mu^\infty_r,\xi_r^i) \right\vert^2\de r
\right] < \infty.
\end{equation}
\vskip 2pt

For the remaining part of the proof, we set $O:=O_1 \cap O_2 \cap O_3$, and we note that ${\mathcal D}(O)=1$. 
\medskip

\textit{Second step:}
Handling the general case. 
\\
The purpose of this step is to transfer
\eqref{eq:expansion:ito:with:I}
from the canonical space equipped with 
${\mathbf P}^{t_0, \emp^{-1}(m)}$
to the canonical space equipped with
${\mathbf Q}^{t_0, \emp^{-1}(m),\boldsymbol \alpha}$,
for a given collection 
of uniformly
bounded and measurable functions $
{\boldsymbol \alpha}=(
\alpha^i: [0,T] \times \prob(\T^d) \times \T^d \to \R^d)_i$, and for any $m \in O$. This step is divided in several sub-steps.\\
\textit{Sub-step 2.a:} Revisiting some definitions and results from Subsection \ref{ssec:well_posedness_partycles}. \\
We first recall  formula \eqref{eq:new:expression:Qcursm0} (with $\boldsymbol b$ therein being now understood 
as $\boldsymbol \alpha$), which holds true for any initial measure 
$\cursm$ on $T^\infty_o \times (\T^d)^\infty_o$ and in particular under 
$\delta_{\emp^{-1}(m)}$, for $m \in O$:
\begin{equation}
\label{eq:formula:Q:ito:new}
{\mathbf Q}^{t_0, \emp^{-1}(m),\boldsymbol \alpha}= {\cercle{$\Psi$}}^{\boldsymbol{\alpha}}_\sharp \left( \frac1{z_T^{t_0, \emp^{-1}(m),\boldsymbol \alpha}}  
{\mathbf P}^{t_0, \emp^{-1}(m)}\right).
\end{equation}
Note that we emphasized the dependence of $z_T$ on both the initial measure  $\emp^{-1}(m)$ and the collection of velocity fields ${\boldsymbol \alpha}$, as well as the dependence 
of $\cercle{$\Psi$}^{\boldsymbol \alpha}$ on $\boldsymbol \alpha$.
We also recall (as a consequence of Girsanov theorem) that 
\begin{equation}
\label{eq:girsanov:expectation=1}
{\mathbb{E}}^{{\mathbf P}^{t_0, \emp^{-1}(m)}}\left[ \frac1{z_T^{t_0, \emp^{-1}(m),\boldsymbol \alpha}}\right]=
1.
\end{equation}
The second point we wish to emphasize is that, from the proof of 
Lemma~\ref{lem:universal:mathcalI0}, we can construct a sequence of Riemann sums $({\mathcal I}^{t_0,(n)})_n$ such that, for any 
$m \in O$, 
\begin{equation*}
\forall \varepsilon >0, \quad 
\lim_{n \rightarrow \infty}
{\mathbf P}^{t_0, \emp^{-1}(m)}
\left( \left\{
\sup_{t \in [0,T]}
\vert {\mathcal I}^{t_0}_t - {\mathcal I}_t^{t_0,(n)}
\vert \geq \varepsilon 
\right\} 
\right) 
= 0. 
\end{equation*}
By combining the last two displays, we deduce that, for the same fixed collection
${\boldsymbol \alpha}$ as before, 
\begin{equation}
\label{eq:proof:ito:new:point:2}
\forall \varepsilon >0, \quad 
\lim_{n \rightarrow \infty}
\left[ \frac1{z_T^{t_0, \emp^{-1}(m),\boldsymbol \alpha}}
{\mathbf P}^{t_0, \emp^{-1}(m)}
\right]
\left( \left\{
\sup_{t \in [0,T]}
\vert {\mathcal I}^{t_0}_t - {\mathcal I}_t^{t_0,(n)}
\vert \geq \varepsilon 
\right\} 
\right) 
= 0. 
\end{equation}
For the purpose of the proof, we introduce, 
for each $i \in \N_+$,   
\begin{equation*}
\begin{split}
\theta_i^{\boldsymbol \alpha} : T^\infty_o\times 
\rmC([0,T];{\mathbb R}^d)^{\infty} \times 
\rmC([0,T];{\mathbb R}^d)
&\rightarrow    \rmC([0,T];{\mathbb R}^d)
\\
\left( (s_j)_j, (x^j)_j,
w\right) &\mapsto  
\left(
w_t + \sqrt{\frac{s_i}{2}} \int_0^t \alpha_r^i\left( 
\sum_{j =1 }^\infty
s_j \delta_{x^j_r}
,x_r^i
\right) \de r
\right)_{t \in [0,T]},
\end{split}
\end{equation*}
where we are using the convention as in Remark \ref{rem:toruspb} to evaluate $\alpha_r^i$ at $x_r^i$, see also \eqref{eq:project} for the meaning of $\sum_{j =1 }^\infty
s_j \delta_{x^j_r}$. We define the mapping ${\cercle{$\Theta$}}^{\boldsymbol \alpha}\colon\Xi\to\Xi$ as
\begin{equation*}
{\cercle{$\Theta$}}^{\boldsymbol \alpha} : \left((s_j)_j,(x_j^0)_j,(w^j)_j,(x^j)_j\right) 
\mapsto \left( (s_j)_j,(x_j^0),\left(\theta_j^{\boldsymbol \alpha}((s_i)_i,(x^i)_i,w^j)\right)_j,(x^j)_j
\right).
\end{equation*}
Recalling the definition of $\psi_i^{\boldsymbol \alpha}$ given in \eqref{eqn:psi_i_new}, we observe that
\begin{equation*}
\theta_i^{\boldsymbol \alpha}\left((s_j)_j, (x^j)_j,
\psi_i^{\boldsymbol \alpha}\left((s_j)_j, (x^j)_j,
w\right)\right)=(w_t)_{t \in [0,T]} 
\end{equation*}
for every $i \in \N_+$ and every $\left((s_j)_j,(x_j^0)_j,(w^j)_j,(x^j)_j\right) \in T^\infty_o\times 
\rmC([0,T];{\mathbb R}^d)^{\infty} \times 
\rmC([0,T];{\mathbb R}^d)$,
from which we deduce that $\cercle{$\Theta$}^{\boldsymbol \alpha} \circ \cercle{$\Psi$}^{\boldsymbol \alpha}$ is the identity on 
$\Xi$. This makes it possible to rewrite 
\eqref{eq:proof:ito:new:point:2} in the form 
\begin{equation*}
\forall \varepsilon >0, \quad 
\lim_{n \rightarrow \infty}
\left[ \frac1{z_T^{t_0, \emp^{-1}(m),\boldsymbol \alpha}}
{\mathbf P}^{t_0, \emp^{-1}(m)}
\right]
\left( \left\{
\sup_{t \in [0,T]}
\left\vert 
{\mathcal I}_t^{t_0} \circ \cercle{$\Theta$}^{\boldsymbol \alpha} \circ \cercle{$\Psi$}^{\boldsymbol \alpha} 
-   {\mathcal I}^{t_0, (n)}_t
\circ \cercle{$\Theta$}^{\boldsymbol \alpha} \circ \cercle{$\Psi$}^{\boldsymbol \alpha}
\right\vert \geq \varepsilon 
\right\} 
\right) 
= 0,
\end{equation*}
which, by \eqref{eq:formula:Q:ito:new}, reads as 
\begin{equation}
\label{eq:convergence:proba:underQ}
\forall \varepsilon >0, \quad 
\lim_{n \rightarrow \infty}
{\mathbf Q}^{t_0, \emp^{-1}(m),{\boldsymbol \alpha}}
\left( \left\{
\sup_{t \in [0,T]}
\left\vert  {\mathcal I}^{t_0}_t \circ \cercle{$\Theta$}^{\boldsymbol \alpha}
 -  {\mathcal I}^{t_0, (n)}_t
\circ \cercle{$\Theta$}^{\boldsymbol \alpha}
\right\vert \geq \varepsilon 
\right\} 
\right) 
= 0.
\end{equation}
\vskip 2pt

\textit{Sub-step 2.b:}
addressing the limit 
of the sequence $(
{\mathcal I}^{t_0, (n)}
\circ \cercle{$\Theta$}^{\boldsymbol \alpha})_{n \geq 1}$ under 
${\mathbf Q}^{t_0, \emp^{-1}(m),{\boldsymbol \alpha}}$.
\vskip 2pt

We first provide an explicit expression of 
${\mathcal I}^{t_0, (n)}
\circ \cercle{$\Theta$}^{\boldsymbol \alpha}$, for a given fixed $n \in \N_+$. 
Recalling 
\eqref{eq:mathcalIn0} in the proof of Lemma \ref{lem:universal:mathcalI0}, we observe that 
each ${\mathcal I}^{t_0, (n)}$ can be written in the form
\begin{equation*}
{\mathcal I}_t^{t_0, (n)}
\circ \cercle{$\Theta$}^{\boldsymbol \alpha} 
=
 \sum_{i =1}^\infty \sum_{k=1}^{N(n)} \left( H_{k}^{(n),i} \circ \cercle{$\Theta$}^{\boldsymbol \alpha} 
 \right)\cdot \left( \beta_{ t_k^{(n)}\wedge t}^i \circ \cercle{$\Theta$}^{\boldsymbol \alpha}
  -
 \beta_{ t_{k-1}^{(n)}\wedge t}^i \circ \cercle{$\Theta$}^{\boldsymbol \alpha}
  \right), \quad t \in [0,T],
\end{equation*}
with $0=t_0^{(n)}<t_1^{(n)}< \cdots < t_{N(n)}^{(n)}=T$ being a subdivision of $[0,T]$ of step-size
\begin{equation}\label{eq:stepsize}
\max_{k=1,\cdots,N(n)} \left\vert t_k^{(n)} - t_{k-1}^{(n)} 
\right\vert \le n^{-3}, \end{equation}
(the rationale for this choice, which implies in particular that $N(n)$ grows at least like $n^3$ with $n$, will become clear later in the proof), and $H^{(n),i}_k$ being as in
\cite[Lemma 4.3.2]{StroockVaradhan}, i.e.,
\begin{equation}
\label{eq:expression:Hnik}
H^{(n),i}_k =
h^{(n),i}_{t_k^{(n)}}
\quad 
\textrm{\rm with}
\quad
h^{(n),i}_s:=
\sqrt{ 2 \varsigma^i } \, n  \int_{0}^{T}
\varrho\left( n(s - r) \right) \eta_n\left({\rm D}u_r\left(\mu_r^\infty,\xi_r^i\right)\right)
\de r, \quad s \in [0,T],
\end{equation}
where $\varrho$
is a a smooth density on ${\mathbb R}$ supported on the interval 
$[0,1]$ and, for any $n \geq 1$, 
$\eta_n$ is a bounded function, equal to the identity on the hypercube $[-n,n]^d$ and satisfying the growth property $\vert \eta_n(x) \vert \leq \vert x \vert$ for any $x \in {\mathbb R}^d$.  
By using the definition of $\cercle{$\Theta$}^{\boldsymbol \alpha}$, we obtain 
\begin{equation*}
{\mathcal I}_t^{t_0, (n)}
\circ \cercle{$\Theta$}^{\boldsymbol \alpha} 
=
 \sum_{i =1}^\infty \sum_{k=1}^{N(n)} \left( H_{k}^{(n),i} \circ \cercle{$\Theta$}^{\boldsymbol \alpha} 
 \right)\cdot \left( \left[\beta_{ t_k^{(n)}\wedge t}^i
 -
 \beta_{ t_{k-1}^{(n)}\wedge t}^i
 \right]
 +  \sqrt{\frac{\varsigma^i}2}
 \int_{t_{k-1}^{(n)} \wedge t}^{t_k^{(n)}\wedge t}
 \alpha_r^i \left( \mu_r^{\infty},\xi_r^i \right) \de r
\right), \quad t \in [0,T].
\end{equation*} 
Moreover,
using the fact that $\mu^{\infty} \circ \cercle{$\Theta$}^{\boldsymbol \alpha}=\mu^{\infty}$ and 
$\xi^{i} \circ \cercle{$\Theta$}^{\boldsymbol \alpha}=\xi^i$, we also notice that 
\begin{equation*}
H^{(n),i}_k \circ   \cercle{$\Theta$}^{\boldsymbol \alpha}
= H^{(n),i}_k,
\end{equation*}
for any $k \in \{1,\cdots,N(n)\}$ and $i \in {\mathbb N}_+$.
Then,
\begin{equation}
\label{eq:ito:new:3:a}
 {\mathcal I}^{t_0, (n)}_t 
\circ \cercle{$\Theta$}^{\boldsymbol \alpha} 
=
 \sum_{i =1}^\infty \sum_{k=1}^{N(n)}   H_{k}^{(n),i}  
 \cdot \left( \left[\beta_{t_k^{(n)}\wedge t}^i
 -
 \beta_{t_{k-1}^{(n)}\wedge t}^i
 \right]
 +  \sqrt{\frac{\varsigma^i}2}
 \int_{t_{k-1}^{(n)}\wedge t}^{t_k^{(n)}\wedge t}
 \alpha_r^i \left( \mu_r^{\infty},\xi_r^i \right) \de r
\right), \quad t \in [0,T].
\end{equation}
In order to complete this sub-step, we must derive further estimates on the rate of convergence of the time-discrete approximation considered above. We compute
\begin{equation}\label{eq:whystep}   
\begin{split} 
&\sum_{i=1}^\infty \int_0^T \left | \sum_{k=1}^{N(n)}
{\mathbf 1}_{[t_{k-1}^{(n)},t_{k}^{(n)})}(s)
H_{k}^{(n),i} -h_s^{(n),i}\right |^2 \de s \\
&=  \sum_{i=1}^\infty \sum_{i=1}^{N(n)} \int_{t_{k-1}^{(n)}}^{t_k^{(n)}} \left | h_{t_k^{(n)}}^{(n),i}-h_s^{(n),i} \right|^2 \de s \\
& \le \sum_{i=1}^\infty 2 \varsigma^i n^2 \sum_{i=1}^{N(n)} \int_{t_{k-1}^{(n)}}^{t_k^{(n)}} \left | \int_0^T \left [ \varrho (n(t_k^{(n)}-r)) - \varrho (n(s-r)) \right  ] \eta_n \left ( {\rm D}u_r\left(\mu_r^\infty,\xi_r^i\right) \right ) \de r\right|^2 \de s \\
&  \le \sum_{i=1}^\infty 2 \|\varrho'\|_\infty^2 \varsigma^i n^4 T^2 \max_{k=1,\cdots,N(n)} \left\vert t_k^{(n)} - t_{k-1}^{(n)}
\right\vert^2 \int_0^T \left |{\rm D}u_r\left(\mu_r^\infty,\xi_r^i\right) \right |^2 \de r \\
& = 2 T^2 \|\varrho'\|_\infty^2 n^{-2} \sum_{i=1}^\infty \int_0^T \varsigma^i \left |{\rm D}u_r\left(\mu_r^\infty,\xi_r^i\right) \right |^2 \de r, 
\end{split}
\end{equation}
where we used the choice of the step size in \eqref{eq:stepsize}. As a consequence of \eqref{eq:bound:L2:Du:mxx}, \eqref{eq:whystep} and standard approximation techniques in $L^2$, we conclude that, for $m \in O$, we have,  ${\mathbf P}^{t_0,\emp^{-1}(m)}$ a.s.,
\begin{equation}
\label{eq:convergence:L2:a.s.:H(n)toDu}
\lim_{n \rightarrow \infty}
\sum_{i =1}^\infty
\int_0^T 
\left\vert \sqrt{2 \varsigma^i} {\rm D} u_r(\mu^{\infty}_r,\xi_r^i)
-
\sum_{k=1}^{N(n)}
{\mathbf 1}_{[t_{k-1}^{(n)},t_{k}^{(n)})}(r)
H_{k}^{(n),i} 
\right\vert^2 \de r = 0.
\end{equation}
Using
\eqref{eq:convergence:L2:a.s.:H(n)toDu}, once again that $\mu^{\infty} \circ \cercle{$\Theta$}^{\boldsymbol \alpha}=\mu^{\infty}$,  
$\xi^{i} \circ \cercle{$\Theta$}^{\boldsymbol \alpha}=\xi^i$, $
H^{(n),i}_k \circ   \cercle{$\Theta$}^{\boldsymbol \alpha}
= H^{(n),i}_k,$
for any $k \in \{1,\cdots,N(n)\}$ and $i \in {\mathbb N}_+$, and \eqref{eq:formula:Q:ito:new}, we notice that,  for any $m \in O$, 
${\mathbf Q}^{t_0, \emp^{-1}(m),{\boldsymbol \alpha}}$-a.s.
\begin{equation}\label{eq:butunderq}
\lim_{n \rightarrow \infty}
\sum_{i =1}^\infty
\int_0^T 
\left\vert \sqrt{2 \varsigma^i} {\rm D} u_r(\mu^{\infty}_r,\xi_r^i)
-
\sum_{k=1}^{N(n)}
{\mathbf 1}_{[t_{k-1}^{(n)},t_{k}^{(n)})}(r)
H_{k}^{(n),i} 
\right\vert^2 \de r = 0.
\end{equation}

The objective, in the rest of this sub-step, is to prove that 
the 
above   convergence
also holds in
$L^1$ under ${\mathbf Q}^{t_0,\emp^{-1}(m), \boldsymbol{\alpha}}$.
Using the contraction property of the convolution in $L^2$-norm, 
it is standard to check that 
\begin{equation*}
 \sum_{i=1}^\infty \int_0^T \vert h_s^{(n),i}\vert^2 \de s  \leq 2 \sum_{i=1}^\infty \varsigma^i 
 \int_0^T
\left\vert
{\rm D}u_r\left(\mu_r^\infty,\xi_r^i\right)\right\vert^2 \de r.
\end{equation*}
From this and \eqref{eq:whystep}, we  deduce that 
\begin{equation}\label{eq:norm:L2:approx:stoch}
\sum_{i=1}^\infty 
  \int_0^T
 \left | \sum_{k=1}^{N(n)}
{\mathbf 1}_{[t_{k-1}^{(n)},t_{k}^{(n)})}(r)
H_{k}^{(n),i} \right |^2 \le  4(1+T^2 \|\varrho'\|^2_\infty)  \sum_{i=1}^\infty \varsigma^i 
 \int_0^T
\left\vert
{\rm D}u_r\left(\mu_r^\infty,\xi_r^i\right)\right\vert^2 \de r.
\end{equation}

\vskip 2pt

We further notice
that, since the collection
$(\alpha^i)_i$ is uniformly bounded, we have an improved version of the estimate \eqref{eqn: bound_girsanov_exp}. Precisely, for any $m \in O$,
\eqref{eq:girsanov:expectation=1}
can be strengthened into 
\begin{equation*}
{\mathbb{E}}^{{\mathbf P}^{t_0, \emp^{-1}(m)}}\left[ \left( \frac1{z_T^{t_0, \emp^{-1}(m),\boldsymbol \alpha}} \right)^p \right] < \infty,
\end{equation*}
for every $p\geq 1$.
Recalling \eqref{eq:bound:L2:Du:mxx} and \eqref{eq:formula:Q:ito:new}, and using the Cauchy--Schwarz inequality, 
together with the facts that 
$\mu^{\infty} \circ \cercle{$\Theta$}^{\boldsymbol \alpha}=\mu^{\infty}$ and 
$\xi^{i} \circ \cercle{$\Theta$}^{\boldsymbol \alpha}=\xi^i$, 
we then deduce that, for every $q \in (1,2)$,
\begin{equation}
\label{eq:bound:L2:Du:m}
\begin{split}
&{\mathbb{E}}^{{\mathbf Q}^{t_0, \emp^{-1}(m),{\boldsymbol \alpha}}}
\left[ \left( 
\sum_{i =1}^\infty
\int_{0}^T  \varsigma^i  \left\vert 
{\rm D}u_r(\mu^\infty_r,\xi_r^i) \right\vert^2\de r
\right)^{q/2}
\right] 
\\
&={\mathbb{E}}^{{\mathbf Q}^{t_0, \emp^{-1}(m),{\boldsymbol \alpha}}}
\left[ \left( 
\sum_{i =1}^\infty
\int_{0}^T  \varsigma^i  \left\vert 
{\rm D}u_r\left(\mu^\infty_r \circ \cercle{$\Theta$}^{\boldsymbol \alpha},\xi_r^i\circ \cercle{$\Theta$}^{\boldsymbol \alpha}
\right) \right\vert^2\de r
\right)^{q/2}
\right]
\\
&= {\mathbb{E}}^{{\mathbf P}^{t_0,\emp^{-1}(m)}}
\left[ \frac1{z_T^{t_0, \emp^{-1}(m),{\boldsymbol \alpha}}}  \left( 
\sum_{i =1}^\infty
\int_{0}^T  \varsigma^i  \left\vert 
{\rm D}u_r\left(\mu^\infty_r  ,\xi_r^i 
\right) \right\vert^2\de r
\right)^{q/2}
\right] \\
& \le {\mathbb{E}}^{{\mathbf P}^{t_0, \emp^{-1}(m)}}\left[ \left( \frac1{z_T^{t_0, \emp^{-1}(m),{\boldsymbol \alpha}}} \right)^{p} \right]^{1/p} {\mathbb{E}}^{{\mathbf P}^{t_0, \emp^{-1}(m)}}\left[ \sum_{i =1}^\infty
\int_{0}^T  \varsigma^i  \left\vert 
{\rm D}u_r(\mu^\infty_r,\xi_r^i) \right\vert^2\de r \right ]^{q/2} < \infty,
\end{split}
\end{equation}
where $p=2/(2-q)$ is the conjugate exponent of $2/q$.
Thanks to \eqref{eq:bound:L2:Du:m} and \eqref{eq:butunderq}, we deduce from a uniform integrability argument that 
for any 
$m \in O$, 
\begin{equation}\label{eq:fordoob}
\lim_{n \rightarrow \infty}
{\mathbb{E}}^{{\mathbf Q}^{t_0, \emp^{-1}(m),{\boldsymbol \alpha}}}
\left[ \left (
\sum_{i =1}^\infty
\int_{0}^T 
\left\vert \sqrt{2 \varsigma^i} {\rm D} u_r(\mu^{\infty}_r,\xi_r^i)
-
\sum_{k=1}^{N(n)}
{\mathbf 1}_{[t_{k-1}^{(n)},t_{k}^{(n)})}(r)
H_{k}^{(n),i} 
\right\vert^2 \de r \right )^{1/2}
\right]
= 0.
\end{equation}
Using the expression in \eqref{eq:ito:new:3:a},  Burkholder-Davis-Gundy inequality,  and \eqref{eq:fordoob}, we obtain for any $m \in O$
\begin{equation*}
\begin{split}
\lim_{n \rightarrow \infty}
&{\mathbb{E}}^{{\mathbf Q}^{t_0, \emp^{-1}(m), {\boldsymbol \alpha}}}
\left[ \sup_{t \in [0,T]}
\left\vert  {\mathcal I}^{t_0,(n)}_t
\circ \cercle{$\Theta$}^{\boldsymbol \alpha} 
\phantom{\sum_{i =1}^\infty
\int_0^t }
\right. \right.
\\
&\hspace{15pt} -   \left. \left.
\sum_{i =1}^\infty
\int_0^t \sqrt{2 \varsigma^i} {\rm D}u_r(\mu^\infty_r,\xi^i_r) \cdot \de \beta^i_r 
- \sum_{i =1}^\infty
\varsigma^i
\int_0^t  {\rm D}u_r(\mu^\infty_r,\xi^i_r) \cdot \alpha_r^i(\mu_r^{\infty},\xi^i_r)
\de r \right\vert\right]=0.
\end{split}
\end{equation*}
\vskip 2pt

Returning to 
\eqref{eq:convergence:proba:underQ}, this gives, for 
any 
$m \in O$, 
${\mathbf Q}^{t_0, \emp^{-1}(m),{\boldsymbol \alpha}}$-a.s., for all 
$t \in [0,T]$, 
\begin{equation*}
{\mathcal I}^{t_0}_t \circ \cercle{$\Theta$}^{\boldsymbol \alpha} 
= \sum_{i =1}^\infty
\int_0^t \sqrt{2 \varsigma^i} {\rm D}u_r(\mu^\infty_r,\xi^i_r) \cdot \de \beta^i_r 
+ \sum_{i =1}^\infty
\varsigma^i
\int_0^t  {\rm D}u_r(\mu^\infty_r,\xi^i_r) \cdot \alpha^i_r(\mu_r^{\infty},\xi^i_r)
\de r.
\end{equation*}
\vskip 2pt

\textit{Sub-step 2.c:} conclusion.\\
It remains to come back to 
\eqref{eq:expansion:ito:with:I}, which we write, under 
${\mathbf P}^{t_0, \emp^{-1}(m)}$ for $m \in O$, 
for any $t \in [t_0,T]$,
\begin{equation*}
\begin{split}
g\left(\mu^{\infty}_T \circ \cercle{$\Theta$}^{\boldsymbol \alpha} \circ \cercle{$\Psi$}^{\boldsymbol \alpha}\right)
&= u_t\left(\mu^{\infty}_t
\circ \cercle{$\Theta$}^{\boldsymbol \alpha} \circ \cercle{$\Psi$}^{\boldsymbol \alpha}
\right)  
\\
&\hspace{15pt}
    + \int_t^T f_r\bigl(\mu^{\infty}_r  
    \circ \cercle{$\Theta$}^{\boldsymbol \alpha} \circ \cercle{$\Psi$}^{\boldsymbol \alpha}
    \bigr) \de r 
   + {\mathcal I}^{t_0}_T\circ \cercle{$\Theta$}^{\boldsymbol \alpha} \circ \cercle{$\Psi$}^{\boldsymbol \alpha} - {\mathcal I}^{t_0}_t\circ \cercle{$\Theta$}^{\boldsymbol \alpha} \circ \cercle{$\Psi$}^{\boldsymbol \alpha}.
   \end{split}
\end{equation*}
This gives, for any 
$m \in O$, 
under 
${\mathbf Q}^{t_0, \emp^{-1}(m),{\boldsymbol \alpha}}$, 
for any $t \in [t_0,T]$,
\begin{equation*}
\begin{split}
g\left(\mu^{\infty}_T \circ \cercle{$\Theta$}^{\boldsymbol \alpha} \right)
&= u_t\left(\mu^{\infty}_t
\circ \cercle{$\Theta$}^{\boldsymbol \alpha}  
\right)  
    + \int_t^T f_r\bigl(\mu^{\infty}_r  
    \circ \cercle{$\Theta$}^{\boldsymbol \alpha}  
    \bigr) \de r 
   + {\mathcal I}^{t_0}_T\circ \cercle{$\Theta$}^{\boldsymbol \alpha} - {\mathcal I}^{t_0}_t\circ \cercle{$\Theta$}^{\boldsymbol \alpha}.
   \end{split}
\end{equation*}
Using the fact that $\mu^{\infty}
\circ \cercle{$\Theta$}^{\boldsymbol \alpha}=\mu^{\infty}$ and then inserting the expression of 
${\mathcal I}^{t_0}_t\circ \cercle{$\Theta$}^{\boldsymbol \alpha}$, we complete the proof.
\end{proof}

\subsection{Finite-dimensional approximation}\label{ssec:finite_dim} 
Let $b:[0,T] \times \prob(\T^d) \times \T^d \to \R^d$ be a bounded and measurable function and $g \in H$ (to ease the presentation, the source term $f$ is assumed be null in this subsection, but
the results stated below could be adapted, without any difficulty, to the case when $f$ is non-zero). The main result of this section is to provide a suitable finite-dimensional reduction of the PDE \eqref{eq:pdeaaa} and the related particle system \eqref{eqn: particles_drift}. Unlike what is typically done in the classical mean-field literature, here, it is not possible to directly substitute the (uniform) empirical measure of a finite  particle system for the measure $\mu^\infty$.
What we do instead is to intervene at the level of the coefficients, introducing approximations of the velocity field $b$ and of the terminal condition $g$ that depend on a probability measure only through a finite number of (sufficiently massive) atoms.
The goal of this section is to show how this approximation propagates at both the PDE and particle levels, and, in particular, 
how it allows us to reduce \eqref{eq:pdeaaa}
to a PDE set in finite dimension, 
and 
\eqref{eqn: particles_drift}
to a system with finitely many particles.

To make it precise, the approximations that are used next rely on the following notion of compatibility between two probability measures.

\begin{definition}
\label{def:epsilon:equivalence:proba:measures}
    For a given $\eps\in (0,1]$, two measures $\mu, \nu \in \prob(\T^d)$ are said to be $\eps$-compatible if they have the same atoms of mass greater than or equal to $\eps$ and assign the same mass to these atoms, i.e., 
\[ \mu_x = \nu_x \text{ for every } x \in \{ x \in \T^d : \mu_x \ge \eps\} = \{x \in \T^d : \nu_x \ge \eps\}.\]
\end{definition}
This induces an equivalence relation between elements of $\prob(\T^d)$ which we denote by $\mu \approx_{\varepsilon} \nu$.
The $\R^m$-valued, $m \in \N_+$,  functions defined 
on the quotient space 
(and also depending on time and space)
can be identified with 
functions belonging to the set 
\[ \AA_\eps^{m} := \set{\fru:  [0,T] \times \prob(\T^d) \times \T^d \to \R^{m} : \fru_{t}(\mu,x) = \fru_{t}(\nu,x) 
\begin{array}{l}
\text{for every } \mu, \nu \in \prob(\T^d) \, \,\text{s.t. $\mu \approx_{\varepsilon} \nu$} 
\\
\text{and for every }
 (t,x) \in [0,T]\times\T^d
\end{array}}
,\]
which is an algebra. Elements of 
$\AA_\eps^{m}$ depend on the measure argument only through  atoms of mass greater than or equal to $\eps$. Note that if $\eps < \delta$ then $\AA_\delta^{m} \subseteq \AA_\eps^{m}$. More generally, we say  that a function $v\colon\prob(\T^d)\to\R^{m}$  belongs to $\AA_{\eps}^{m}$ if its trivial extension to $ [0,T] \times \prob(\T^d)\times\T^d$ lies in $\AA_{\eps}^{m}$.
\vskip 4pt

\smallskip
We then consider three sequences $(\eps_n)_n$, $(g^n)_n$ and $(b^n)_n$ such that:
\begin{enumerate}[i.]
\item 
the sequence 
$(\varepsilon_n)_n$ takes values in $(0,1)$, is decreasing and tends to $0$ as $n \rightarrow \infty$;
\item for each $n \in {\mathbb N}_+$, 
$b^n \in \AA_{\varepsilon_n}^{d}$, 
$\sup_{n \in {\mathbb N}_+} \| b^n \|_\infty < \infty$, 
and $b^n \rightarrow b$ in the space 
$L^2([0,T] \times \prob(\T^d) \times \T^d, \mathscr{L}^{[0,T]} \otimes \overline{\cD}; \R^d)$ as $n \rightarrow \infty$;
\item for each $n \in {\mathbb N}_+$, $g^n \in \AA_{\varepsilon_n}^{1}$
and 
$\| g^n\|_\infty < \infty$;
and, $g^n \rightarrow g$ in the space $H$
 as $n \rightarrow \infty$.
\end{enumerate}

Existence of the three sequences $(\varepsilon_n)_n$, 
$(g^n)_{n}$ and 
$(b^n)_n$ is guaranteed by Proposition \ref{prop:approx_trunc} and property \eqref{eq:density}. However, other constructions are conceivable; we come back to this point next.
\smallskip

For each $n \in {\mathbb N}_+$, 
we denote by $\mathbf{Q}^{t_0,(\mbfs,\mbfx),b^n}$ the solution to the particle system \eqref{eqn: particles_drift} (with $i$-independent drift) with data $b^n,t_0$ and $ (\mbfs,\mbfx)\in T^\infty_o\times(\T^d)^\infty_o$. Since $b^n\in\AA_{\eps_n}^{d}$, we can consider the following finite-dimensional sub-system of \eqref{eqn: particles_drift}, which is given in closed form by
\begin{equation}\label{eqn: finite_particles_drift_trunc}
    \begin{cases}
        \de X^{i}_t & = b_t^n(\emp^{N(\eps_n, \mbfs)}(\mbfs, (X^{i}_t)_{i=1}^{N(\eps_n, \mbfs)}), X^i_t)\de t +  \sqrt{\frac{2}{s_i}}\de W^i_t,\quad t\in(t_0,T),\, i=1,\dots,N(\eps_n, \mbfs),\\
        X^{i}_{t_0} & = \iota(x_i), \quad i=1,\dots,N(\eps_n, \mbfs),
    \end{cases}
\end{equation}
where $N(\eps_n, \mbfs)$ is as in \eqref{eq:Neps:mbfs}.
Since the drift 
$b^n \circ \emp^{N(\varepsilon_,n\mbfs)(\mbfs,\cdot)}$ is bounded, the system \eqref{eqn: finite_particles_drift_trunc} has a unique strong solution, 
see 
\cite{Veretennikov80}. Of course, this solution is also unique in law. In this context, the purpose of the next statement is to clarify how the finite-dimensional structure of the particle system \eqref{eqn: finite_particles_drift_trunc} is reflected in the PDE \eqref{eq:pdeaaa} (with $b^n$ being  substituted for $b$). Notice indeed that the latter is still posed in infinite dimension despite the seemingly simpler form of $b^n$.

In order to provide a finite-dimensional version of \eqref{eq:pdeaaa}, we first introduce the following functional spaces. For any $N \in \N_+$, we denote by 
$\rmC^{0,1,N}$ the space of continuous real-valued functions 
on $[0,T] \times (\T^d)^{N}$ that are differentiable in space and whose space derivative is jointly continuous in time and space on $[0,T) \times (\T^d)^N$, and for any real $p>1$, 
we denote by
$W^{1,2,N}_p$ the space of functions $k : [0,T] \times ({\mathbb T}^d)^N \rightarrow {\mathbb R}$ that belong to 
$L^p([0,T] \times ({\mathbb T}^d)^N, \mathscr{L}^{[0,T]} \otimes \textrm{vol}_d^N)$
and  possess generalized Sobolev derivatives $\partial_t k$, 
$(\partial_{x_i^j} k)_{i=1,\cdots,N;j=1,\cdots,d}$ and 
$(\partial^2_{x_{i}^{j} x_{i'}^{j'}} k)_{i,i'=1,\cdots,N;j,j'=1,\cdots,d}$ in
$L^p([0,T] \times ({\mathbb T}^d)^N, \mathscr{L}^{[0,T]} \otimes \textrm{vol}_d^N)$, with the convention that an element of 
$({\mathbb T}^d)^N$ is represented in the form $(x_i^j)_{i=1,\cdots,N;j=1,\cdots,d}$.

\begin{theorem}
\label{thm:finite:dim:PDE}
Let $n \in \N_+$, $b^n$ and $g^n$ be as above, and $u^n$ be the solution of $\eqref{eq:pdeaaa}$ with drift $b^n$ and final datum $g^n$. Then, for every $t_0 \in [0,T)$ and $\cD$-a.e.~$m \in \prob(\T^d)$,
\[ u^n_{t_0}(m) = h_{t_0}^{\mbfs(m)}(x_1(m), \dots, x_{N(\eps_n, \mbfs(m))}(m)),\]
where, for $\mbfs \in T_o^\infty$, $h^{\mbfs}$ is the unique solution, in the space
$\rmC^{0,1,N(\varepsilon_n,\mbfs)} \cap (\cap_{p > 1} W^{1,2,N(\varepsilon_n,\mbfs)}_p$), 
of the PDE, set on $[0,T] \times (\T^d)^{N(\eps_n, \mbfs)}$:
\begin{equation}\label{eqn:finite_dim_kolmogorov} \partial_t h_t + \sum_{i=1}^{N(\eps_n, \mbfs)} \frac{1}{s_i} \Delta_i h_t + \sum_{i=1}^{N(\eps_n, \mbfs)}  b^n_t (\emp^{N(\eps_n, \mbfs)}(\mbfs, \cdot), \cdot_i) \cdot \nabla_i h_t =0 \quad \text{ for a.e. } t \in (0,T)
\end{equation}
with final condition
\begin{equation*} h_t \rvert_{t=T} = g^n \circ \emp^{N(\eps_n, \mbfs)}(\mbfs, \cdot). \end{equation*}
\end{theorem}

\begin{proof}
From Theorem \ref{thm:repr}, we know that, for every $t_0 \in [0,T)$ and $\cD$-a.e.~$m \in \prob(\T^d)$,
$u^n_{t_0}(m) = \mathbb{E}^{\mathbf{Q}^{t_0,\emp^{-1}(m),b^n}}[g^n(\mu^\infty_T)]$. For the rest of the proof, we fix $m$ such that the last equality holds true, and for this $m$, we set $(\mbfs, \mbfx):=(\mbfs(m), \mbfx(m))$ and, then, $N:= N(\eps_n, \mbfs)$. Since $g^n\in
\AA_{\eps_n}^{1}$, we have that $g^n(m) = g^n\circ\emp^{N} (\mbfs, (x_i)_{i=1}^{N})$. Moreover, the mapping 
\[
    (y_1,\dots,y_{N})\in(\T^d)^{N}\mapsto g^n\circ\emp^{N}(\mbfs, (y_i)_{i=1}^{N})
\]
belongs to $\rmC^\infty((\T^d)^{N})$. Therefore, we have
\begin{align*}
    u_{t_0}(m) & = \mathbb{E}^{\mathbf{Q}^{t_0,(\mbfs,\mbfx),b^n}}\left[(g^n\circ\emp^{N})(\mbfs,(\xi^1_T,\dots,\xi^{N}_T))\right].
\end{align*}
It is plain to see that, under the probability measure 
${\mathbf Q}^{t_0,(\mbfs,\mbfx),b^n},$
the process $(\xi^{i})_{i=1}^N$ is a solution to the particle system \eqref{eqn: finite_particles_drift_trunc} (with $\beta^1,\cdots,\beta^N$ as driving Brownian motions).

Consider now the PDE
\eqref{eqn:finite_dim_kolmogorov}. Existence and uniqueness of a solution 
(within the class specified in the statement of Theorem \ref{thm:finite:dim:PDE}) are standard in the literature (viewing the equation as a PDE posed on $[0,T] \times ({\mathbb R}^d)^{N}$, with periodic coefficients); 
for convenience, 
we refer to \cite[Theorem 2.1]{DELARUE20061712}, as the result therein fits exactly our needs.
Since the operator driving the PDE \eqref{eqn:finite_dim_kolmogorov} is the generator of 
\eqref{eqn: finite_particles_drift_trunc}, we deduce from Krylov's version of It\^o's formula, see \cite[Theorem 2.10.1]{KrylovBookControl}, that the right-hand side in the above display is also equal to
\begin{align*}
    \mathbb{E}^{\mathbf{Q}^{t_0,(\mbfs,\mbfx),b^n}}\left[(g^n\circ\emp^{N})(\mbfs,(\xi^1_T,\dots,\xi^{N}_T))\right]= h^{\mbfs}_{t_0}(x_1,\dots,x_{N}),
    \end{align*}
    which completes the proof.
\end{proof}

\begin{remark}
\label{rem:rate:convergence}
It must be stressed that  the 
finite-dimensional approximation constructed above satisfies the assumption of 
Proposition \ref{prop_stab_new}.
In particular, with the same notation 
as in the statement 
of Theorem 
\ref{thm:finite:dim:PDE}, Proposition
\ref{prop_stab_new}
guarantees that \[  \sup_{t \in [0,T]} |u^n_t-u_t|_H   \to 0 \quad \text{ as } n \to \infty. \]
Of course, the rate of convergence depends on the rates at which 
$(b^n)_{n \geq 1}$ and 
$(g^n)_{n \geq 1}$ converge to $b$ and $g$ respectively. 
When 
$(b^n)_{n \geq 1}$
and $(g^n)$ are obtained by applying
Proposition \ref{prop:approx_trunc} and property \eqref{eq:density}, 
the rate is not explicit (as the approximation is constructed by means of a Stone-Weierstrass
argument). 
Still, we can wonder whether it is possible to obtain  more explicit rates 
when $b$ and $g$ are  continuous. We explain in  Appendix \ref{app:B} that, 
 when $b$ and $g$ are indeed continuous (the space $\prob(\T^d)$ being equipped with the weak topology), there exists (at least, in \textit{reasonable} situations) a function 
$\vartheta : {\mathbb R}_+ \rightarrow {\mathbb R}_+$, with 
$\lim_{r \downarrow 0}
\vartheta(r)=0$, such that 
\begin{equation*}
\begin{split}
\vert g^n - g \vert_H^2 &\leq \vartheta(\varepsilon_n),
\\
\int_0^T \int_{\prob(\T^d)} \left[ \int_{\T^d}
\vert b_r(\mu,x) - 
b_r^n(\mu,x) \vert^2 
\de\mu(x) 
\right] \de {\mathcal D}(\mu) \de r &\leq \vartheta(\varepsilon_n),
\end{split}
\end{equation*}
with the behaviour of the function 
$\vartheta$
being dictated by the shape of a certain modulus of continuity of $b$ and $g$.
Inserting these two bounds in \eqref{eq:prop:3.6:main:bound}, we see that
the term 
$\mathcal{R(}b,b^n,\rmD u,T)$ therein remains non-explicit: this is due to the presence, inside the integral, of $\lvert \rmD u_r(\mu,x) \rvert^2$, which is just known to be integrable. Without  stronger integrability properties, the best bound we can obtain is: 
\begin{equation*}
\begin{split}
\mathcal{R(}b,b^n,\rmD u,T) \leq \inf_{a>0} \left[a^2 
\vartheta(\varepsilon_n) + 
4 \| b \|_\infty^2 
\int_0^T \int_{\prob(\T^d)} \int_{\T^d} 
{\mathbf 1}_{\{
\vert \rmD u_r(\mu,x) \vert>a 
\}
}
\lvert \rmD u_r(\mu,x) \rvert^2 \de \mu(x) \de \cD(\mu) \de r \right].
\end{split}
\end{equation*}
The  above 
inequality is not completely explicit; this explains why we refrain from providing more details on the construction of the function $\vartheta$.
Actually, the fact that we cannot get a better bound should not come as a surprise: since we are working with low regularity properties on the coefficients $b$ and $g$, the resulting rate of convergence 
remains rather poor and depends on the regularity of $u$ through the integrability properties of the gradient.
Obviously, if $\rmD u$  
were bounded, the rate would be directly given by $\vartheta(\eps_n)$, but, without any 
stronger regularization property of the semi-group of the Dirichlet-Ferguson diffusion, we cannot expect this bound to hold, except if we assume $b$ and $g$ to be regular in a 
sufficiently strong sense. Obviously, the latter assumption would not fit the framework of this article.

That said, the rate becomes explicit when $b=0$ and $g$ is continuous. It is
given by $\vartheta(\varepsilon_n)$ for the same function $\vartheta$ as in the last paragraph. As we explain in Appendix 
\ref{app:B}, the rate of growth of the function 
$\vartheta$ is algebraic if the modulus of continuity of $g$ (in total variation distance) is also algebraic. In this case, the rate of convergence  
can be bounded by 
$C\varepsilon_n^{\alpha}$, for a certain
constant $C>0$ and a certain
$\alpha >0$. Observing that $N(\varepsilon_n,\mbfs)$ is always less than $\lceil 1/\varepsilon_n \rceil$, this says that the rate of convergence decays, in this situation, algebraically fast with the number of particles that are used in the approximation. This observation is consistent with the rates that can be obtained by 
using the \textit{semi-group (or master equation)}
approach in standard mean field systems; see, for example, 
\cite{cardaliaguetdelaruelasrylions,carmonadelarue1}.
\end{remark}

\section{The semilinear case}\label{sec: control}
The aim of this section is to apply the tools developed earlier in a nonlinear framework inspired by mean-field control theory. In Subsection~\ref{ssec: hjb}, we establish a general existence and uniqueness result for a class of semilinear PDEs driven by the operator ${\boldsymbol \Delta}$. 
When the nonlinear term can be interpreted as a Hamiltonian, i.e., as the Legendre transform of a Lagrangian, we provide an interpretation of the solution to the nonlinear PDE as the value function of a control problem posed on a particle system of the form~\eqref{eqn: particles_drift}. The main result in this regard is Theorem~\ref{thm: verif}. In Subsection \ref{sec:apprcontr} we introduce an approximation of the infinite-dimensional control problem in the same spirit of Subsection \ref{ssec:finite_dim}.

In the whole section we let $T$ be the same finite time horizon as at the beginning of Section \ref{se:2}. 
As in Section~\ref{sec: massive}, we make repeated use of the material introduced in Section~\ref{se:2} and of the results established thus far.

\subsection{Well-posedness of a semilinear PDE}\label{ssec: hjb}

Given a measurable function $\HH\colon\T^d\times\R^d\to\R$, we define a 
new function $\cH : \prob(\T^d) \times L^2(\prob(\T^d)\times \T^d, \overline{\mathcal{D}}; \R^d) \to \R$ by letting
\begin{equation}\label{eqn: int_hamiltonian}
    \cH(\mu,\gamma):=\int_{\T^d}\HH(x,\gamma(\mu, x))\de\mu (x),
\end{equation}
whenever the integral makes sense. For coefficients $\mathscr{F}:[0,T] \times \prob(\T^d) \to \R$ and $\mathscr{G} : \prob(\T^d) \to \R$, we then aim at solving the following semilinear PDE:
\begin{equation}\label{eqn: HJB}
        \partial_t u_t + \boldsymbol{\Delta}u_t - \cH(\cdot, \rmD u_t) + \mathscr{F}_t = 0\quad \text{for a.e.~} t \in (0,T), \quad u_t  \bigr |_{t=T} = \mathscr{G}.  
    \end{equation}
We work under the following assumption.
\begin{assumption} \label{hp: hj} 
The function $\HH\colon\T^d\times\R^d\to\R$  is jointly measurable and satisfies the following two properties:
    \begin{enumerate}[a)]
        
        \item  it holds
        \[\int_{\prob(\T^d)}\int_{\T^d}\lvert \HH(x,0)\rvert^2 \de \mu(x)\de\mathcal{D}(\mu)<\infty.\]
        \item 
        The mapping $p \in {\mathbb R}^d\mapsto \HH(x,p)$ is Lipschitz continuous, uniformly with respect to $x\in\T^d$, i.e., there exists a constant $L>0$, such that, for 
        any $x \in \T^d$ and
        any $p,q\in\R^d$, 
        \begin{equation*}
                \lvert \HH(x, p) - \HH(x,q)\rvert \leq  L \lvert p  - q\rvert.
            \end{equation*}
    \end{enumerate}
\end{assumption}
\begin{remark}\label{rmk: H_integral}
We notice from  Assumption \ref{hp: hj}a) that, for
$\mathcal{D}$-a.e.~$\mu \in \prob(\T^d)$, 
\begin{equation}
\label{eq:HH(x,0)<infty}
\int_{{\mathbb T}^d}
\lvert \HH(x,0)\rvert^2 \de \mu(x) <   \infty.
\end{equation}
And then, Assumption \ref{hp: hj}b)
gives, for any $\mu$ in the Borel subset of 
$\prob(\T^d)$ on which the above inequality holds true, and any 
$\gamma \in L^2(\prob(\T^d)\times \T^d, \overline{\mathcal{D}}; \R^d)$,
\begin{equation*}
\left( 
\int_{\T^d}
\lvert \HH\left( x,\gamma(\mu,x)\right)\rvert^2 \de \mu(x)
\right)^{1/2}
\leq 
\left( 
\int_{\T^d}
\lvert \HH\left( x,0\right)\rvert^2 \de \mu(x)
\right)^{1/2}
+ 
L 
\left( 
\int_{\T^d}
\lvert \gamma(\mu,x) \rvert^2 \de \mu(x) \right)^{1/2},
\end{equation*}
so that the left-hand side is automatically finite if, in addition to 
\eqref{eq:HH(x,0)<infty}, 
it holds 
$\int_{\T^d}
\lvert \gamma(\mu,x) \rvert^2 \de \mu(x)<\infty$. 
Obviously, the latter is true 
for $\mathcal{D}$-a.e.~$\mu \in \prob(\T^d)$. 
We deduce that $\cH(\mu, \gamma)$ is well defined for $\mathcal{D}$-a.e.~$\mu \in \prob(\T^d)$ and 
induces a mapping
${\mathcal H}(\cdot,\gamma) : \mu \mapsto 
{\mathcal H}(\mu,\gamma)
$ that can be regarded as an element of $H$. Moreover, for any 
$\gamma^1,\gamma^2 \in L^2(\prob(\T^d)\times \T^d, \overline{\mathcal{D}}; \R^d)$, 
for any $\mu \in \prob(\T^d)$ such that 
\eqref{eq:HH(x,0)<infty} holds true and 
$\int_{\T^d}
( \lvert \gamma^1(\mu,x) \rvert^2 + \lvert \gamma^2(\mu,x) \rvert^2) \de \mu(x)$ is finite (and thus for 
${\mathcal D}$-a.e $\mu$), 
we get, from Assumption 
\ref{hp: hj}b), 
\begin{equation}
\label{eq:Lipschitz:difference:mathcalH}
\begin{split}
        \lvert \cH(\mu, \gamma^1) - \cH(\mu, \gamma^2)\rvert &\leq \int_{\T^d} \lvert \HH(x,\gamma^1(\mu, x)) - \HH(x,\gamma^2(\mu, x)) \rvert\de \mu(x)
        \\
        &\leq L \int_{\T^d}\lvert \gamma^1(\mu, x) - \gamma^2(\mu, x)\rvert\de\mu(x).
        \end{split}
    \end{equation}

\end{remark}
\begin{proposition}\label{prop:well_pos_HJ} Let Assumption \ref{hp: hj} be in force, $\mathscr{F} \in L^2([0,T]; H)$, and $\mathscr{G} \in H$. Then there exists a unique $u \in \rmC([0,T];H) \cap \mathrm{AC}_{loc}((0,T); H) \cap L^2([0,T]; H^{1,2})$ such that $u_t \in D(\boldsymbol{\Delta})$ for a.e.~$t \in (0,T)$ and 
    \begin{equation*}
        \partial_t u_t + \boldsymbol{\Delta}u_t - \cH(\cdot, \rmD u_t) + \mathscr{F}_t = 0\quad \text{for a.e.~} t \in (0,T), \quad u_t  \bigr |_{t=T} = \mathscr{G}.  
    \end{equation*}
\end{proposition}
\begin{proof}
   Thanks to the Lipschitz regularity of $\HH$, the proof proceeds along the same lines as in Theorem~\ref{thm:pde}, which notation we adopt here.
 We briefly outline the reasoning, emphasizing the key differences. As a preliminary remark, note that for $v \in X_{\bar{T}}$, the function 
 $t \mapsto {\mathcal H}(\cdot,D v_t)$ belongs to $X_{\bar{T}}$. Indeed, by Assumption \ref{hp: hj},
\begin{align*}
    \int_0^{\bar{T}}\int_{\prob(\T^d)} \lvert \cH(\mu,\rmD v_t)\rvert^2\de\mathcal{D}(\mu)\de t & = \int_0^{\bar{T}}\int_{\prob(\T^d)}\left \lvert \int_{\T^d}\HH(x,\rmD v_t(\mu,x))\de\mu(x)\right\rvert^2\de\mathcal{D}(\mu)\de t\\
    & \leq 2L^2\int_0^{\bar{T}}\int_{\prob(\T^d)}\int_{\T^d}\lvert\rmD v_t(\mu,x)\rvert^2\de\mu(x)\de\mathcal{D}(\mu)\de t\\
    &\quad + 2\int_0^{\bar{T}}\int_{\prob(\T^d)}\int_{\T^d}\lvert \HH(x,0)\rvert^2 \de \mu(x)\de\mathcal{D}(\mu)\de t<\infty.
\end{align*}
 
 The core of the argument is to show that, for a suitably small $\bar{T}>0$, for every $w_0 \in H$, 
 the operator  $\Phi_{\bar{T},w_0}\colon X_{\bar{T}}\to X_{\bar{T}}$ that maps every \( v\in X_{\bar{T}} \) to the unique solution of  
\begin{equation}\label{eqn: hjb_v}
    \partial_t u^v_t - \boldsymbol{\Delta}u^v_t  = -\cH(\cdot, \rmD v_t)+ \mathscr{F}_t\quad \text{for a.e.~} t \in (0,\bar{T}), \quad u^v_t  \bigr |_{t=0} = w_0\in H,
\end{equation}  
is a contraction. As in the proof of Theorem \ref{thm:pde}, we can focus on the forward-in-time equation and then recover the result for the backward-in-time version via a simple time-reversal argument. 
\\ 

    \textit{Claim 1}: 
    For any $\bar{T} \in (0,T)$, $w_0 \in H$ and $v \in X_{\bar{T}}$ there exists a unique $u^v \in X_{\bar{T}}$ solving \eqref{eqn: hjb_v}. This  defines a map $\Phi_{\bar{T},w_0}: X_{\bar{T}} \to X_{\bar{T}}$ associating to every $v \in X_{\bar{T}}$ the function $u^v$ as above.\\
\smallskip
\textit{Proof of claim 1:}  This follows directly from Proposition \ref{prop:pde}, with $f$ given by $t\mapsto f_t:=-\cH(\cdot,\rmD v_t)+ \mathscr{F}_t$, noting that the latter belongs to to $L^2([0,\bar{T}]; H)$ \\

\textit{Claim 2}: There exists $\bar{T} \in (0,T)$, only depending on the Lipschitz constant $L$ of $\HH$, such that, for any $w_0 \in H$, the map $\Phi_{\bar{T}, w_0}$ defined in claim 1 is a contraction on $X_{\bar{T}}$.\\
\textit{Proof of claim 2:} For $w_0 \in H$ and $v,\tilde{v}\in X_{\bar{T}}$, set $u := \Phi_{\bar{T}, w_0}(v)$, $\tilde{u} := \Phi_{\bar{T}, w_0}(\tilde{v})$. The argument used in the proof of Theorem \ref{thm:pde} can be repeated directly. The only new point concerns the treatment of the difference between the two terms driven by ${\mathcal H}$. Here, we handle the difference by means of 
\eqref{eq:Lipschitz:difference:mathcalH}:
    \begin{align*}
        &\int_0^t \int_{\prob(\T^d)}\lvert \cH(\mu,\rmD v_r) - \cH(\mu,\rmD \tilde{v}_r) \rvert^2\de\mathcal{D}(\mu)\de r \\ 
        &\qquad\leq L^2\int_0^t\int_{\prob(\T^d)}\int_{\T^d}\lvert\rmD v_r(\mu,x) - \rmD\tilde{v}_r(\mu,x)\rvert^2\de\mu(x)\de\mathcal{D}(\mu) \de r
        \leq L^2\norm{v-\tilde{v}}^2_{X_{\bar{T}}}.
    \end{align*}
    \\
    
    \textit{Conclusion}: The end of the proof is similar to that of Theorem~\ref{thm:pde}. 
The main point is to note that $\bar{T}$ depends only on the Lipschitz constant $L$ of ${\mathcal H}$ and, in particular, is independent of the initial condition $w_0 \in H$ in \eqref{eqn: hjb_v}. 
The argument of Claim~2 can then be iterated over successive subintervals of $[0,T]$ of length smaller than $\bar{T}$, as in the proof of Theorem~\ref{thm:pde}.
\end{proof}
\begin{remark}\label{rml: H_general}
The form of ${\mathcal H}$ imposed by relation~\eqref{eqn: int_hamiltonian}, with 
$\HH$ satisfying Assumption~\ref{hp: hj}, is motivated by the application to optimal control given below. 
Nevertheless, the result of Proposition~\ref{prop:well_pos_HJ} still holds when $\cH$ is a more general measurable function from 
$\prob_{o}(\T^d) \times L^2(\prob(\T^d)\times \T^d, \overline{\mathcal{D}}; \R^d)$
into $\R$ (recall that 
$\prob_o(\T^d) \times \T^d$
has full measure under 
$\overline{\mathcal{D}}$; 
therefore, restricting the first factor in the space
$\prob(\T^d) \times L^2(\prob(\T^d)\times \T^d, \overline{\mathcal{D}}; \R^d)$
to $\prob_o(\T^d)$ is consistent with our choice of equipping 
$\prob(\T^d)\times \T^d$ with the measure $\overline{\mathcal D}$)
satisfying the following two conditions:
\begin{enumerate}
    \item[$\text{a}^\star$)]it holds
      \[\int_{\prob(\T^d)}\lvert \cH(\mu,0)\rvert^2 \de\mathcal{D}(\mu)<\infty;\]
      \item[$\text{b}^\star$)] there exists a constant $L>0$ such that for any $\mu\in\prob_{o}(\T^d)$, $\gamma^1,\gamma^2\in L^2(\prob(\T^d)\times \T^d, \overline{\mathcal{D}} ; \R^d)$,
            \begin{equation*}
                \lvert \cH(\mu, \gamma^1) - \cH(\mu, \gamma^2)\rvert \leq L\int_{\T^d}\lvert \gamma^1(\mu, x) - \gamma^2(\mu, x)\rvert\de\mu(x).
            \end{equation*}
    \end{enumerate}
\end{remark}

\subsection{Infinite-dimensional control problem}\label{ssec:verification}

We now use the results of the previous subsection in order to characterize the value function of a stochastic control problem set over a controlled version of the massive particle system. 

For a compact and convex subset $A \subset \R^d$ such that $0\in A$, 
we define the 
space of admissible controls (in Markov feedback form), denoted by ${\mathcal A}$, as the 
space of sequences $\boldsymbol{\alpha} = (\alpha^i)_i$, 
where each $\alpha^i\colon[0,T]\times\prob(\T^d)\times\T^d\to A\subset\R^d$ is a measurable function.
Throughout, we work within the framework of Subsection \ref{ssec:transfer}. Given an initial time $t_0 \in [0,T)$ and an initial distribution $\cursm \in \prob(T_o^\infty \times (\T^d)_o^\infty)$, we denote by $\mathbf Q^{t_0, \cursm, \boldsymbol{\alpha}}$ the law of $((\mbfs, \mbfx), (X^i)_i, (W^i)_i)$ on the canonical space $\Xi$, where $({\Omega},{\cF},{\Q},\{{\cF}_t\}_{t\in[t_0,T]}, (({\mbfs},{\mbfx}), ({X}^i)_i, ({W}^i)_i))$ is a (unique in law) weak solution (according to the definition given in  Subsection \ref{ssec:well_posedness_partycles}) to the system
\begin{equation}\label{eqn: controlled_particles}
    \begin{cases}
        \de X^i_t & = \alpha^i_t(\mu^\infty_t,X_t^i)\de t + {\sqrt{\frac{2}{s_i}}}\de W^i_t,\quad t\in[t_0,T],\quad i\in\N_+,\\
        X^i_{t_0} & = \iota(x_i).
    \end{cases}
\end{equation}
Since the controls $\boldsymbol{\alpha} = (\alpha^i)_i$ are assumed to be uniformly bounded (as the set $A$ itself is bounded), 
the system \eqref{eqn: controlled_particles} can be regarded as a system of 
the same type as 
\eqref{eqn: particles_drift}, and the approach developed in 
Subsection \ref{ssec:well_posedness_partycles}
applies. In particular, 
Propositions \ref{prop: weak_existence} and
\ref{prop:weak:uniqueness} guarantee existence and uniqueness in law of weak solutions.
In this approach, the index $i$ should be viewed as the label of a rational agent; accordingly, the massive particle system  becomes, in this subsection, a countable system of agents (or players). 
The control (or feedback function) played by the agent $i$ is allowed to depend on $i$, on  the collective state of the system (here encoded through $\mu^\infty$), and on the private state of the agent (encoded through $X_t^i$). Whilst this form looks rather restrictive, it covers in fact a wide class of feedback functions. Indeed, in  contrast with standard mean-field models, which are symmetric, the current model is not symmetric: for instance, the first player is, by construction, the most massive one (i.e., $s_1$ is greater than all the other masses $(s_i)_{i \geq 2}$); more generally, the 
measure $\mu^\infty_t$ provides a complete description of the masses and the atom locations of the agents, in the sense that the mass and location of the agent number $i$ can be deduced from the observation of $\mu^\infty_t$. This follows from the fact that the masses are ordered and (almost-surely) pairwise distinct and the locations are also (almost-surely) pairwise distinct. Mathematically, we can retrieve the masses via the map 
$\mbfs$ and the atom locations via the map $\mbfx$ as in~\eqref{eq:phiinv}: observing $\mu_t^\infty$, we can observe both 
$\mbfs(\mu^\infty_t)$ and 
$\mbfx(\mu^\infty_t)$. As such, any (measurable) function of the masses and the atom locations can be rewritten as a (measurable) function of $\mu^\infty_t$. In particular, with the given form of feedback functions, we allow each agent to play a control depending on the state of the \textit{whole} system. By the way, notice that, in the notation 
$\alpha^i_{t}(\mu_t^\infty,X_t^i)$, the variable $X_t^i$ is somewhat redundant: as we just explained, 
$X_t^i$ can be observed from 
the observation of $\mu^\infty_t$. That said, we feel more consistent with the usual mean-field control theory to stick to the notation $\alpha^i_{t}(\mu_t^\infty,X_t^i)$.

\smallskip
We associate a cost with $ \mathbf Q^{t_0,m,\boldsymbol{\alpha}}$; to do so, we consider measurable coefficients 
$\LL\colon\T^d\times A\to\R$, $\mathscr{F}:[0,T] \times \prob(\T^d) \to \R$, and $\mathscr{G}\colon\prob(\T^d)\to\R$, with 
$\LL$ being understood as the Lagrangian of the control problem, and 
${\mathscr F}$ and ${\mathscr G}$ as (respectively) running and terminal interaction potentials. For 
an initial time $t_0\in[0,T)$, an initial distribution $\cursm \in \prob(T_o^\infty \times (\T^d)_o^\infty)$, and
an admissible control 
${\boldsymbol \alpha} \in {\mathcal A}$, the cost associated to $(t_0,\cursm,{\boldsymbol \alpha})$ is defined as
\begin{equation*}
J(t_0,\cursm,\boldsymbol{\alpha}):=  \mathbb{E}^{\mathbf{Q}^{t_0, \cursm, \boldsymbol{ \alpha}}}\left[\int_{t_0}^T\left(\mathscr{F}_r(\mu_r^\infty) +  \sum_{i=1}^{\infty} \varsigma_i \LL(\xi^i_r,\alpha^i_r(\mu^\infty_r,\xi^i_r) )\right) \de r+ \mathscr{G}(\mu^\infty_T)\right]. 
\end{equation*}
Here we are using the same convention as in Remark \ref{rem:toruspb} to evaluate $\mathscr L$ and $\alpha^i_r$ at $\xi_r^i$.
Conditions are given below under which the right-hand side is well-defined for any 
${\boldsymbol \alpha} \in {\mathcal A}$. Of course, the purpose is to minimize 
$J(t_0,\cursm,{\boldsymbol \alpha})$ with respect to ${\boldsymbol \alpha} \in \cA$.}
The corresponding value function is defined as
\begin{equation}\label{eqn: value}
    V(t_0,\cursm):= \inf_{\boldsymbol{\alpha}\in\cA} J(t_0,\cursm,\boldsymbol{\alpha}),\quad (t_0,\cursm)\in[0,T)\times \prob( T_o^\infty \times (\T^d)_o^\infty).
\end{equation}
In analogy with the shorthand notation $\mathbf Q^{t_0, (\mbfs, \mbfx)}$
introduced in \eqref{eq:Q:t0:s,x} as a substitute for the notation 
$\mathbf Q^{t_0, \delta_{(\mbfs, \mbfx)}}$, we write $V(t_0, (\mbfs, \mbfx))$ in place of $V(t_0, \delta_{(\mbfs, \mbfx)})$ when the particle system is initialized with the deterministic initial condition $(\mbfs, \mbfx)$.

We study the control problem under the following assumption.
\begin{assumption}\label{hp: control} The functions $\LL\colon\T^d\times A\to\R$, $\mathscr{F}:[0,T] \times \prob(\T^d) \to \R$, and $\mathscr{G}\colon\prob(\T^d)\to\R$ satisfy:
    \begin{enumerate}[a)]
        \item for every $x \in {\mathbb T}^d$, the function    
     $A \ni a \mapsto \LL(x,a)$ is strictly convex and lower semicontinuous;
        \item there exists $C \ge 0$ such that, for every $x \in \T^d$, 
        $|\inf_a \LL(x,a)| \le C$; 
        \item $\mathscr{F} \in L^2([0,T];H)$ and $\mathscr{G} \in H$.
    \end{enumerate}
\end{assumption}

In the following, we often assume that Assumption \ref{hp: control} is in force.

In order to connect the control problem introduced above with 
Subsection~\ref{ssec: hjb}, we associate to the Lagrangian $\LL$ the 
Hamiltonian
\begin{equation}\label{eq:hdef}
\HH(x,p) := \sup_{a\in A} \left\{ -\LL(x,a) - p\cdot a \right\}, 
\qquad (x,p)\in \T^d \times \R^d,
\end{equation}
which will play, in what follows, the same role as 
$\HH$ in 
Subsection~\ref{ssec: hjb}.

By construction, $\HH$ is convex \cite[Theorem 12.2]{Rockafellar70} and Lipschitz in $p$ uniformly with respect to $x$. Moreover,  
 $|\HH(x,0)| = | \inf_{a \in A} \LL(x,a) | \le C$, for any $x \in \T^d.$
In particular, $\HH$ satisfies Assumption \ref{hp: hj} with a Lipschitz constant $L$ depending only on  $A$. Moreover, since $\LL$ is strictly convex, then $\R^d\ni p\mapsto \HH(x,p)$ is everywhere differentiable for any $x\in\T^d$ \cite[Theorem 26.3]{Rockafellar70}, with derivative bounded again by a constant depending solely on $A$. By convexity and lower semicontinuity in $a$ of $\LL$, it holds $\LL(x,a) = \sup_{p\in\R^d}\{-\HH(x,p) - a\cdot p\}$ for any $(x,a)\in\T^d\times A$ \cite[Corollary 12.2.1]{Rockafellar70}. 

\begin{proposition}\label{prop:ito_general_HJB} Let Assumption \ref{hp: control} be in force and $t_0 \in [0,T)$ be an initial time. If $u$ is the unique solution of \eqref{eqn: HJB}, then there exists a Borel subset $O$ of $\cP(\T^d)$, with ${\mathcal D}(O)=1$, such that, 
for any $m \in O$ and any $\boldsymbol{\alpha} \in \mathcal{A}$
\begin{align*}
    \mathscr{G}(\mu^\infty_T) & =  u_{t}(\mu^\infty_{t})  
    - \int_{t}^T \mathscr{F}_r(\mu^\infty_r) \de r   
    + \int_{t}^T\int_{\T^d}\HH(x,\rmD u_r(\mu_r^\infty,x))\de\mu^\infty_r(x)\de r 
    \\
    & \quad + \sum_{i=1}^\infty \int_t^T \varsigma_i \rmD u_r(\mu_r^\infty, \xi_r^i) \cdot \alpha_r^i(\mu_r^\infty, \xi_r^i) \de r \\
    & \quad +\sum_{i=1}^{\infty}\int^T_{t} \sqrt{2 \varsigma_i}\rmD  u_r(\mu^\infty_r, \xi^i_r) \cdot \de \beta^i_r,\quad t\in[t_0,T],\quad{\mathbf Q}^{t_0,\emp^{-1}(m),\boldsymbol{\alpha}}\text{-a.s.}
\end{align*}
\end{proposition}
Note that in the above display we are using the convention explained in Remark \ref{rem:toruspb} to evaluate $\rmD u_r(\mu_r^\infty, \cdot)$ and $\alpha_r^i(\mu_r^\infty, \cdot)$ at $\xi_r^i$.
\begin{proof} The proof is a consequence of Proposition \ref{prop:ito_general:new},
    replacing $t \in [0,T] \mapsto f_t$
    by 
    $t \in [0,T] \mapsto 
    {\mathcal H}(\cdot, \rmD u_t) - {\mathscr F_t}$, which is indeed in $L^2([0,T]; H)$ (see the proof of Proposition \ref{prop:well_pos_HJ}).
    \end{proof}
We now have all the ingredients to state and prove a verification theorem for our mean-field control problem.

\begin{theorem}\label{thm: verif}  Let Assumption \ref{hp: control} be in force and $u$ be the unique solution of \eqref{eqn: HJB}; then  for every $t_0 \in [0,T)$ there exists a Borel subset $\tilde{O}_{t_0}$ of $\prob(\T^d)$, with $\cD(\tilde{O}_{t_0})=1$, such that
\begin{equation}\label{eq:verific_eq}
    u_{t_0}(m) = V(t_0, \emp^{-1}(m))\quad \text{ for every } m \in \tilde{O}_{t_0}.
\end{equation}
Moreover, if  $t_0 \in [0,T)$ and $m \in  \tilde{O}_{t_0}$, then the control 
    $\hat{\boldsymbol \alpha}=(\hat{\alpha}^i)_i$, with the 
    $\hat{\alpha}^i$'s being all equal to 
    the same mapping $\hat{\alpha} : [0,T] \times \prob(\T^d) \times \T^d \ni (t,\mu,x) \mapsto \hat{\alpha}_t(\mu,x) := -\nabla_p \HH (x, \rmD u_t(\mu,x))$, is optimal for $V(t_0, \emp^{-1}(m))$ and it is unique in the following sense:  if $\tilde{\boldsymbol{\alpha}} \in \mathcal{A}$ is also optimal for $V(t_0, \emp^{-1}(m))$, then,
for any $i \in {\mathbb N}_+$, for a.e. $r \in (t_0,T)$, 
\begin{equation*}
{\mathbf Q}^{t_0, \emp^{-1}(m), \tilde{\boldsymbol{\alpha}}}
\left( \left\{\tilde{\alpha}_r^i(\mu_r^\infty,\xi_r^i) = 
\hat{\alpha}_r(\mu_r^\infty,\xi_r^i)  \right\}
\right) = 1.
\end{equation*}    
\end{theorem}

\begin{remark}
\label{rem:open-loop}
The following remarks are in order, regarding the shape of the optimal feedback function.
\begin{enumerate}[i)]
\item It is worth mentioning that the definition of the function 
$\hat{\alpha}$ depends in fact on the choice of a version of the mapping $(t,\mu,x) \mapsto
\rmD u_t(\mu,x)$, which is viewed as an element of $L^2([0,T] \times \prob(\T^d) \times \T^d,{\mathscr L}^{[0,T]} \otimes \overline{\mathcal D})$. 
Proposition 
\ref{prop:consistency:solutions}
says that the law of the particle system would 
remain the same (at least for 
${\mathcal D}$-almost every 
initial 
condition $m \in 
\prob(\T^d)$) if we used another version of this mapping. 

\item We also notice that the resulting form of the optimal feedback function $\hat{\alpha}$ is independent of 
the label of the player: at the optimum, all the agents play the same feedback function (but implemented at a different location, as it depends on the own state of the agent). This observation is consistent with the theory of mean-field control.

Consequently, the control problem \eqref{eqn: value} can be equivalently formulated by considering controls that are identical across particles. Therefore, 
    for any 
    $t_0 \in [0,T]$, and for 
    ${\mathcal D}$-a.e. $m \in {\mathcal P}({\mathbb T}^d)$, 
    \begin{equation*}
        V(t_0, \emp^{-1}(m))= \inf_{\boldsymbol{\alpha}\in\cA} J(t_0,\emp^{-1}(m),\boldsymbol{\alpha}) = \inf_{\boldsymbol{\alpha}\in\cA_u} J(t_0,\emp^{-1}(m),\boldsymbol{\alpha}),
    \end{equation*}where $\cA_u$ denotes  the set of controls $\boldsymbol{\alpha}= (\alpha^i)_i$ with the $\alpha^i$ being all equal to a single measurable function $\alpha\colon[0,T]\times\prob(\T^d)\times\T^d\to A$. Notice that,  for such an $\boldsymbol{\alpha}\in\cA_u$, the cost functional can be written as 
    \begin{equation*}
         J(t_0,\cursm,\alpha) =  \mathbb{E}^{\mathbf{Q}^{t_0, \cursm, \boldsymbol{ \alpha}}}\left[\int_{t_0}^T\left(\mathscr{F}_r(\mu_r^\infty) +  \int_{\T^d} \LL(x,\alpha_r(\mu^\infty_r,x) ) \de \mu_r^\infty(x) \right) \de r+ \mathscr{G}(\mu^\infty_T)\right].  
\end{equation*}

\item We strongly believe that  the proof of Theorem 
\ref{thm: verif}
that is achieved below 
    could be adapted, at least up to some extent, 
to accommodate open-loop controls. 
The result would be as follows. For an initial time $t_0 \in [0,T)$, a filtered probability space 
$(\Omega,{\mathcal F},{\mathbb P},\{{\mathcal F}_t\}_{t \in [0,T]})$,  equipped with a collection $(B^i)_i$
of independent $d$-dimensional $\{\cF_t\}_{t\in[t_0,T]}$-Brownian motions, 
for any collection $(\alpha^i)_i=((\alpha^i_r)_{r \in [t_0,T]})_i$ of $A$-valued $\{{\mathcal F}_t\}_{t \in [t_0,T]}$-progressively measurable processes,
for any
${\mathcal F}_{t_0}$-measurable 
random initial distributions $(\mbfs, \mbfx)$ with values in $T_o^\infty \times (\T^d)_o^\infty$ such that
\begin{equation*}
\mu_{t_0}^{\infty}= \sum_{i=1}^\infty s_i \delta_{x_i}
\sim {\mathcal D},
\end{equation*}
the conditional cost 
\begin{equation}
\mathbb{E}^{\mathbb P} \left[\int_{t_0}^T\left(\mathscr{F}_r\left(\mu_r^{\infty}
\right) +  \sum_{i=1}^{\infty} \varsigma_i \LL(X^i_r,\alpha^i_r )\right) \de r+ \mathscr{G}\left( \mu_T^{\infty}
\right) \, \vert \, {\mathcal F}_{t_0} \right]
\end{equation}
would be higher than $u_{t_0}(\mu_{t_0}^{\infty})$ with probability 1 under ${\mathbb P}$, 
where 
\begin{equation*}
\mu_{r}^{\infty}:= \sum_{i=1}^{\infty} s_i \delta_{X_{r}^i}, \quad r \in [t_0,T],
\end{equation*}
and 
\begin{equation*}
X_t^i = \iota(x_i) + 
\int_{t_0}^t \alpha_r^i \de r +
\sqrt{\frac{2}{s_i}}\left( B_t^i - B_{t_0}^i \right), \quad t \in [t_0,T],
\end{equation*}
with $\iota: \T^d \to \R^d$ the canonical projection and we are using the same convention as in \eqref{eq:project}.
Equivalently, our guess is that
\begin{equation}\label{eq:open:loop}
\underset{(\alpha^i)_i}{\rm essinf} \
\mathbb{E}^{\mathbb P} \left[\int_{t_0}^T\left(\mathscr{F}_r\left(\mu_r^{\infty}
\right) +  \sum_{i=1}^{\infty} \varsigma_i \LL(X^i_r,\alpha^i_r )\right) \de r+ \mathscr{G}\left( \mu_T^{\infty}
\right) \, \vert \, {\mathcal F}_{t_0} \right]
\geq u_{t_0}(\mu_{t_0}).
\end{equation}
Conditioning on ${\mathcal F}_{t_0}$
is here a way to fix the initial value 
$\mu_{t_0}^{\infty}$. As 
$\mu_{t_0}^{\infty}$ is 
assumed to follow the distribution ${\mathcal D}$, this is consistent with the fact that 
the result of 
Theorem 
\ref{thm: verif} is stated for $m$ in a   ${\mathcal D}$-full subset of ${\mathcal P}(\T^d)$. That said, Theorem 
\ref{thm: verif} is stronger. In comparison, it can be summarized in the following way: for 
${\mathcal D}$-a.e. $m \in {\mathcal P}(\T^d)$,
\begin{equation*}
\inf_{{\boldsymbol \alpha}} J(t_0,m,{\boldsymbol \alpha})\geq
 u_{t_0}(m),
 \end{equation*}
with equality when 
${\boldsymbol \alpha}=(\hat{\alpha})_i$; implicitly, the ${\mathcal D}$-full subset 
of ${\mathcal P}(\T^d)$ on which the inequality is true is independent of ${\boldsymbol \alpha}$, which explains why the essential infimum can be replaced by an infimum.

The proof of 
\eqref{eq:open:loop}
would rely on the same principle as the first step in the proof of 
Theorem 
\ref{thm: verif} below, this step itself being a consequence of the version of It\^o's formula stated in Proposition \ref{prop:ito_general_HJB}. In fact, one of the difficulties in the proof of It\^o's formula is to establish that the subset $O$ (on which the expansion holds true) is independent of ${\boldsymbol \alpha}$, but in order to derive 
\eqref{eq:open:loop}, 
a simpler version would be in fact enough: it would suffice to get the
It\^o expansion under ${\mathbb P}$, for a fixed ${\boldsymbol \alpha}$ (recalling that 
$\mu_{t_0}^{\infty} \sim_{\mathbb P} {\mathcal D}$), and then to take conditional expectation given ${\mathcal F}_{t_0}$. For this reason, the proof would be easier. To get the expansion for a fixed ${\boldsymbol \alpha}$, one could come back to the proof of Proposition \ref{prop:ito_general} and then expand  
$(u^n_t(\mu_t^{\infty}))_{t \in [t_0,T]}$ under the tilted probability measure
${\mathbb Q}$, given by 
\begin{equation*}
\frac{\de {\mathbb Q}}{\de {\mathbb P}}
:= \exp \left( 
- \sum_{i=1}^\infty 
\int_{t_0}^T \sqrt{\frac{s_i}2} \alpha_r^i \cdot \de B_r^i - \frac12 
\sum_{i=1}^\infty \int_{t_0}^T 
\frac{s_i}2 \vert \alpha_r^i \vert^2 \right). 
\end{equation*}
Details are left to the reader. 

We come back to the notion of open-loop controls in the framework of the finite-dimensional approximation addressed in Subsection \ref{subsubse:approximate:control}.

\item 
The case where the feedback functions are just required to depend on $t$ and $\mu_t^\infty$
is also very interesting
from a practical viewpoint. 
This 
corresponds to 
the framework used in the mean-field approach
of residual networks with a continuum of layers, see for instance \cite{EHanLi2019}.
While 
entering the details of this model would be out of the scope 
of this article, it is worth mentioning that the related HJB equation has the form 
\begin{equation*}
\begin{split}
&\partial_t u_t(\mu)  + \sup_{\nu}
 \left[ 
 \int_A \left[ 
 \int_{\R^d}
 b(x,a) \cdot  {\rm D} u_t(\mu,x,y) \de \mu(x,y)
 \right] 
 \de\nu(a)
 + {\mathscr K}(\nu)
 \right] = 0,
\quad  t \in (0,T),
\\
&u_t  \bigr |_{t=T}(\mu) = \mathscr{G}(\mu),
\end{split}
\end{equation*}
for $(t,\mu) \in [0,T] \times \prob(\R^d)$, where $b$ 
is a generic activation function 
taking as inputs the current state $x$ of an $n$-dimensional feature and the current state $a$ of a neuron (in a set $A$), ${\mathscr K}$ is a possible penalization function, and 
$\mathscr G$ is a  loss function, measuring the 
distance between the outputs of the network and the labels. 
The variable $\mu$ stands for the (initial) joint  distribution of the features (represented by the variable $x$) and the labels (represented by the variable $y$), so that the dimension $d$ is greater than $n$; the variable $\nu$ represents  
an instantaneous 
distribution of the neurons and is thus viewed as an element of $\prob(A)$. We refer to \cite{daudin2025genericitypolyaklojasiewiczinequalitiesentropic} for a complete description of this model. 

When the state space is the torus and a Laplacian is added, this equation is covered by Remark \ref{rml: H_general} (provided that the coefficients satisfy mild conditions). The fact that our results are stated without any convexity assumption is especially adapted to this situation.
\end{enumerate}

\end{remark}

\begin{proof}[Proof of Theorem 
\ref{thm: verif}.]
The proof is divided into two steps: the first establishes the upper bound 
in~\eqref{eq:verific_eq}, and the second establishes the lower bound.\\
    \smallskip
    \textit{First step: proving the upper bound} $ u_{t_0}(m) \leq V(t_0, \emp^{-1}(m))$.
    
    Back to the main expansion in the statement of Proposition \ref{prop:ito_general_HJB}, we first establish that the local martingale therein is in fact a true martingale (up to a possibly new choice of the set $O= O_{t_0}$ in the statement). To do so, we come back to the proof of Theorem 
    \ref{thm:repr}, and especially to  \eqref{eqn: unif_int}. 
    We note that the bound \eqref{eqn: bound_girsanov_exp} is, in the present context, independent of $\boldsymbol{\alpha}=(\alpha^i)_i$ since each 
    $\alpha^i$ is required to take values in the bounded set $A$. Moreover, the same bound holds not only when $\cursm=\emp_\sharp^{-1}(\mathcal{D})$ but also when $\cursm=\emp^{-1}(m)$, for any $m\in\mathcal{P}_o(\T^d)$. 
    Then,    we can repeat the computations carried out in the first step of the proof of Theorem \ref{thm:repr}, especially those achieved to handle \eqref{eqn: mg_est_1}
and \eqref{eqn: mg_est_2}, and  deduce that there exists a constant $C$ such that, for any ${\boldsymbol \alpha} \in {\mathcal A}$
and any 
$m \in {\mathcal P}_o({\mathbb T}^d)$, 
    \[ \mathbb{E}^{\mathbf Q^{t_0, \emp^{-1}(m), \boldsymbol{\alpha}}} \left [ \sup_{t \in [t_0, T]} \left |\sum_{i=1}^{\infty} \int_{t_0}^T \sqrt{2\varsigma_i} \rmD u_r(\mu_r^\infty, \xi_r^i) \cdot \de \beta_r^i \right |\right ] \le C \mathbb{E}^{\mathbf P^{t_0, \emp^{-1}(m)}} \left [ \sum_{i=1}^{\infty} \int_{t_0}^T \varsigma_i| \rmD u_r(\mu_r^\infty, \xi_r^i) |^2 \de r\right]. \]

We already know from \eqref{eq:thefirstbound} that the right-hand side above is finite when we substitute $\emp^{-1}(m)$ for $\emp_{\sharp}^{-1}(\mathcal{D})$.
In particular, we can find a subset 
$O'_{{t_0}}$, of full measure under ${\mathcal D}$, such that, for any 
${\boldsymbol \alpha} \in {\mathcal A}$
and any 
$m \in O'$, the left-hand side is finite. This shows in particular that,
for any 
${\boldsymbol \alpha} \in {\mathcal A}$ and 
any $m \in O'_{t_0}$, 
the process
\begin{equation*}
\left( \sum_{i=1}^{\infty} \int_{t_0}^t \sqrt{2\varsigma_i} \rmD u_r(\mu_r^\infty, \xi_r^i) \cdot \de \beta_r^i
\right)_{t \in [t_0,T]}
\end{equation*}
is ${\mathbf Q}^{t_0, \emp^{-1}(m), \boldsymbol{\alpha}}$-martingale.
This says that, for any 
${\boldsymbol \alpha} \in {\mathcal A}$
and any $m \in \tilde{O}_{t_0}:=O_{t_0} \cap O'_{t_0}$, 
the stochastic integral in the 
statement of Proposition 
\ref{prop:ito_general_HJB}
has zero expectation.
We obtain, for any 
${\boldsymbol \alpha} \in {\mathcal A}$ and any 
$m \in \tilde{O}_{t_0}$,
\begin{equation}
\label{eq:prop:5.6:upper:bound}
\begin{split}
&{\mathbb{E}}^{{\mathbf Q}^{t_0, \emp^{-1}(m), \boldsymbol{\alpha}}}
\left[ {\mathscr G}(\mu_T^\infty) 
+ \int_{t_0}^T 
\left( 
{\mathscr F}_r(\mu_r^\infty)
+
\sum_{i=1}^\infty 
\varsigma_i 
{\mathscr L}(\xi_r^i,\alpha_r^i(\mu_r^\infty,\xi_r^i)) \right) 
\de r
\right]
\\
&=u_{t_0}(m) 
+ 
{\mathbb{E}}^{{\mathbf Q}^{t_0, \emp^{-1}(m), \boldsymbol{\alpha}}}
\sum_{i=1}^\infty 
\int_{t_0}^T 
\varsigma_i  
\left[
 \HH(\xi_r^i,\rmD u_r(\mu_r^\infty,\xi_r^i)) \right.
 \\
 &\hspace{150pt} \left.
+
{\mathscr L}(\xi_r^i,\alpha_r^i(\mu_r^\infty,\xi_r^i))
+
\rmD u_r(\mu_r^\infty, \xi_r^i) \cdot \alpha_r^i(\mu_r^\infty, \xi_r^i) 
 \right] \de r.
\end{split}
\end{equation}
And then, recalling \eqref{eq:hdef}, we deduce that 
$J(t_0,m,{\boldsymbol \alpha}) \geq u_{t_0}(m)$. 
Since 
${\boldsymbol \alpha}$ is arbitrary, 
this implies 
$V(t_0,\emp^{-1}(m)) \geq u_{t_0}(m).$

We now study the case when
$J(t_0,m,{\boldsymbol \alpha}) \geq u_{t_0}(m)$
is in fact an equality, i.e., 
$J(t_0,m,{\boldsymbol \alpha}) = u_{t_0}(m)$. 
Recalling that 
$\emp^{-1}(m)$ belongs to  $T^\infty_o \times ({\mathbb T}^d)^\infty_o$ (since 
$m \in {\mathcal P}_o({\mathbb T}^d)$), we observe that all the weights $(\varsigma_i)_i$
are (strictly) positive. Back to 
\eqref{eq:prop:5.6:upper:bound}, this implies that, for any $i \in {\mathbb N}_+$, for a.e. 
$r \in [t_0,T]$, 
\begin{equation*}
{\mathbf Q}^{t_0, \emp^{-1}(m), \boldsymbol{\alpha}}
\left( \left\{
\HH(\xi_r^i,\rmD u_r(\mu_r^\infty,\xi_r^i)) 
+
{\mathscr L}(\xi_r^i,\alpha_r^i(\mu_r^\infty,\xi_r^i))
+
\rmD u_r(\mu_r^\infty, \xi_r^i) \cdot \alpha_r^i(\mu_r^\infty, \xi_r^i)
=0 \right\}
\right) = 1.
\end{equation*}
Since, for any $x \in {\mathbb T}^d$, ${\mathscr L}(x,\cdot)$ is strictly convex in $a \in A$, this implies, for any $i \in {\mathbb N}_+$, for a.e. $r \in [t_0,T]$, 
\begin{equation*}
{\mathbf Q}^{t_0, \emp^{-1}(m), \boldsymbol{\alpha}}
\left( \left\{\alpha_r^i(\mu_r^\infty,\xi_r^i) = 
\hat{\alpha}_r(\mu_r^\infty,\xi_r^i)  \right\}
\right) = 1.
\end{equation*}
This proves that 
there is at most one control
whose cost 
is equal to $u_{t_0}(m)$. 
\quad \\
\quad \\
\smallskip
    \textit{Second step: proving the lower bound 
    $u_{t_0}(m) \leq V(t_0,\emp^{-1}(m))$.}
    
 In order to complete the proof, it suffices to show that 
the control 
$\hat{\boldsymbol \alpha}:=(\hat{\alpha})_i$ (all the entries are the same), which obviously belongs to 
${\mathcal A}$, satisfies 
$u_{t_0}(m) = J(t_0,m,\hat{\boldsymbol \alpha})$, for 
$m$ in a full subset of ${\mathcal D}$. 
This is a consequence
of \eqref{eq:prop:5.6:upper:bound}, noticing that the right-hand side therein then reduces to $u_{t_0}(m).$
\end{proof}

\subsection{Finite-dimensional approximation}\label{sec:apprcontr}
\subsubsection{The approximate PDE}
In this subsection we construct a finite dimensional approximation of the PDE \eqref{eqn: HJB}. By the density in \eqref{eq:density} we can find sequences $(\eps_n)_n, (\mathscr{F}^n)_n, (\mathscr{G}^n)_n$ such that $\eps_n \in (0,1)$, $\mathscr{F}^n_t \in \hat{\mathfrak{Z}}^\infty_{\eps_n}$ for every $t \in [0,T]$, $\mathscr{G}^n \in \hat{\mathfrak{Z}}_{\eps_n}^\infty$ for every $n \in \N_+$, and $\mathscr{F}^n \to \mathscr{F}$ in $L^2([0,T]; H)$, $\mathscr{G}^n \to \mathscr{G}$ in $H$, and $\eps_n \downarrow 0$ as $n \to \infty$. Starting from $\HH$ as in \eqref{eq:hdef}, for every $n \in \N$, we define a function $\cH^n : \prob_o(\T^d) \times \R^d \to \R$ as
\[\cH^n(\mu,p) := 
\sum_{j=1}^{N(\eps_n, \mbfs)}s_j \HH(x_j,p), \quad \mu = \emp(\mbfs, \mbfx) \in\prob_o(\T^d), \, p \in \R^d,\]
where $N(\eps_n, \mbfs)$ is as in \eqref{eq:Neps:mbfs}.
With a small abuse of notation, we still denote by $\cH^n$ the map over $\prob_o(\T^d) \times L^2(\prob(\T^d)\times \T^d, \overline{\mathcal{D}}; \R^d)$ defined by 
\begin{equation}\label{eq:hn}
    \cH^n(\mu,\gamma):=\sum_{j=1}^{N(\eps_n, \mbfs)}s_j\HH(x_j,\gamma(\mu,x_j)), \quad \mu = \emp(\mbfs, \mbfx) \in\prob_o(\T^d), \, \gamma \in L^2(\prob(\T^d)\times \T^d, \overline{\mathcal{D}}; \R^d). 
\end{equation}
A first object we are interested in is the Hamilton-Jacobi equation associated with $\cH^n, \mathscr{F}^n$ and $\mathscr{G}^n$. 
\begin{proposition}\label{prop:existhn}  Let Assumption \ref{hp: control} be in force; for a given 
$n \in {\mathbb N}_+$, let $\cH^n$, $\mathscr{F}^n$, and $\mathscr{G}^n$ be as above. Then $\cH^n$ satisfies the conditions $\textrm{\rm a}^\star$\textrm{\rm )} and $\textrm{\rm b}^\star$\textrm{\rm )}
in Remark \ref{rml: H_general} with the same Lipschitz constant $L$ as the one of $\HH$ in Assumption \ref{hp: hj}. In particular, there exists a unique $u^n \in \rmC([0,T];H) \cap \mathrm{AC}_{loc}((0,T); H) \cap L^2([0,T]; H^{1,2})$ such that $u^n_t \in D(\boldsymbol{\Delta})$ for a.e.~$t \in (0,T)$ and
    \begin{equation}\label{eqn: HJB_n}
        \partial_t u^n_t + \boldsymbol{\Delta}u^n_t - \cH^n(\cdot, \rmD u^n_t) + \mathscr{F}^n_t = 0\quad \text{for a.e.~} t \in (0,T), \quad u_t^n  \bigr |_{t=T} = \mathscr{G}^n.  
    \end{equation}
\end{proposition}
\begin{proof} 
As far 
    as
    $\text{a}^\star$)  is concerned, it follows from the following straightforward inequalities
    \begin{align*}
        \int_{\prob(\T^d)}\lvert \cH^n(\mu,0)\rvert^2 \de\mathcal{D}(\mu) &\leq \int_{\prob(\T^d)}\Bigg(\sum_{j=1}^{N(\eps_n, \mbfs(\mu))}s_j(\mu)\lvert\HH(x_j(\mu), 0)\rvert\Bigg)^2 \de \mathcal{D}(\mu)\\
        &\leq\int_{\prob(\T^d)}\Bigg(\int_{\T^d}\lvert \HH(x,0)\rvert \de\mu(x)\Bigg)^2\de\mathcal{D}(\mu)<\infty,
    \end{align*}
    where we have denoted by $(\mbfs(\mu), \mbfx(\mu))$ the inverse of $\emp$ applied to $\mu$, see \eqref{eq:phiinv}, and used that $\HH$ satisfies Assumption \ref{hp: hj} a).
    
    We now check  $\text{b}^\star$).
    For any $\mu = \emp(\mbfs, \mbfx)\in\prob_o(\T^d)$, $\gamma^1,\gamma^2\in L^2(\prob(\T^d)\times \T^d, \overline{\mathcal{D}}; \R^d)$, we have, 
    from the Lipschitz property of $\HH$ in Assumption \ref{hp: hj} b),
    \begin{align}
    \label{eq:regularity:Hn:C}
         \lvert \cH^n(\mu, \gamma^1) - \cH^n(\mu, \gamma^2)\rvert & \leq \sum_{j=1}^{N(\eps_n, \mbfs)}s_j\lvert \HH(x_j,\gamma^1(\mu,x_j)) - \HH(x_j,\gamma^2(\mu,x_j) )\rvert \\& \leq L\int_{\T^d}\lvert \gamma^1(\mu, x) - \gamma^2(\mu, x)\rvert \de\mu(x),
    \end{align}
    and so $\text{b}^\star$) holds. Since $\mathscr{F}^n \in L^2([0,T]; H)$ and $\mathscr{G}^n\in H$,
   the existence of a solution to \eqref{eqn: HJB_n}
    (as claimed in the
    statement) follows directly from Proposition \ref{prop:well_pos_HJ} and Remark \ref{rml: H_general}.
\end{proof}

The following stability result shows that the solutions to the PDE \eqref{eqn: HJB_n} converge to the one of the PDE \eqref{eqn: HJB}.

\begin{proposition}
\label{prop:convergence:nonlinear}Let Assumption \ref{hp: control} be in force, $(\cH^n)_n$, $(\mathscr{F}^n)_n$, $(\mathscr{G}^n)_n$ be as above, and $u^n$ be the solution of \eqref{eqn: HJB_n} for each $n \in {\mathbb N}_+$. Then,
for $u$ the solution of \eqref{eqn: HJB}, it holds
\begin{equation}\label{eq:to0}
\begin{split}
\sup_{t \in [0,T]} |u^n_t-u_t|_H^2 + &\int_0^T\|\rmD u_t^n - \rmD u_t \|^2 \de t\\
 &\leq 
C(L,T)  \left ( 
 |\mathscr{G}^n-\mathscr{G}|^2_H  +  \int_0^T|\mathscr{F}^n_t-\mathscr{F}_t|^2_H \de t
+ \mathcal{R}^n(\mathscr H, \rmD u, T) \right ) \to 0 \text{ as } n \to \infty,
    \end{split}
    \end{equation}
    where 
    \[
   \mathcal{R}^n(\mathscr H, \rmD u, T):= \int_0^T \int_{\prob(\T^d)}
   \sum_{i\geq N(\eps_n, \mbfs(\mu)) + 1}^{\infty} s_i(\mu) |\HH(x_i(\mu), \rmD u_r(x_i(\mu), \mu))|^2 \de \cD(\mu) \de r,
    \]
and $C(L,T):= (1+L^{-2})(1+ \rme^{T(1+2L^2)} + (1+2L^2)T\rme^{T(1+2L^2)})$.
In particular, the left-hand side of \eqref{eq:to0} tends to $0$ as $n$ tends to $\infty$.
\end{proposition}
\begin{proof}
We use the same stability arguments as in the proof of Proposition \ref{prop_stab_new}. Writing the equation solved by $u-u^n$, testing it against $u-u^n$, integrating 
with respect to time on the interval $[t,T]$ for any $t \in [0,T]$, and integrating by parts the term driven by $\boldsymbol{\Delta}$ using \eqref{eq:ibp0}, we get
\begin{align}
    &\frac{1}{2} |u_t^n-u_t|_H^2 + \int_t^T \| \rmD u_r^n - \rmD u_r \|^2 \de r 
    \\
    &=  \frac{1}{2} |\mathscr{G}^n-\mathscr{G}|^2_H + \int_t^T ( \mathcal{H}(\cdot, \rmD u_r) - \mathcal{H}^n(\cdot, \rmD u^n_r) , u_r^n-u_r)_H  \de r  + \int_t^T (\mathscr{F}^n_r-\mathscr{F}_r, u_r^n-u_r)_H \de r \nonumber
    \\
    &  \le \frac{1}{2} |\mathscr{G}^n-\mathscr{G}|^2_H + \frac{1}{2} \int_t^T |u_r^n - u_r|^2_H \de r + L^2 \int_t^T |u_r^n-u_r|^2_H \de r + \frac{1}{2} \int_0^T |\mathscr{F}^n_r-\mathscr{F}_r|^2_H \de r\nonumber
    \\
    & \quad + \frac{1}{4L^2} \int_t^T  \left | \mathcal{H}(\cdot, \rmD u_r) - \mathcal{H}^n(\cdot, \rmD u^n_r) \right |^2_H \de r \nonumber
    \\
    & \le \frac{1}{2} |\mathscr{G}^n-\mathscr{G}|^2_H + (1/2+L^2) \int_t^T |u_r^n - u_r|^2_H \de r + \frac{1}{2} \int_0^T |\mathscr{F}^n_r-\mathscr{F}_r|^2_H \de r 
    \label{eq:proof:Proposition5.11:term1}
    \\
    & \quad + \frac{1}{2L^2} \int_t^T \left | \mathcal{H}(\cdot, \rmD u_r) - \mathcal{H}^n(\cdot, \rmD u_r) \right |^2_H\de r \label{eq:proof:Proposition5.11:term2}
    \\
    & \quad + \frac{1}{2L^2} \int_t^T  \left | \mathcal{H}^n(\cdot, \rmD u_r) - \mathcal{H}^n(\cdot, \rmD u^n_r) \right |^2_H \de r,
    \label{eq:proof:Proposition5.11:term3}
\end{align}
where we have used Young's inequality  and we recall that $L$ is the Lipschitz constant (depending solely on the set $A$) appearing in Assumption \ref{hp: hj} b) for $\HH$ as well as in condition $\text{b}^\star$) in Remark \ref{rml: H_general} for $\cH^n$, see the statement of Proposition \ref{prop:existhn}.

Next, we explain how to handle the last two terms 
\eqref{eq:proof:Proposition5.11:term2} and 
\eqref{eq:proof:Proposition5.11:term3}.
For \eqref{eq:proof:Proposition5.11:term3}, we get, using once again $\text{b}^\star$) in Remark \ref{rml: H_general} for $\cH^n$,
\begin{align}
    \frac{1}{2L^2} \int_t^T \left | \cH^n(\cdot, \rmD u_r) - \cH^n(\cdot, \rmD u^n_r) \right |^2_H \de r &=
    \frac{1}{2L^2} \int_t^T \int_{\prob(\T^d)} \left | \mathcal{H}^n(\mu, \rmD u_r) - \mathcal{H}^n(\mu, \rmD u^n_r) \right |^2 \de \mathcal{D}(\mu) \de r
    \nonumber
    \\
    & \le\frac{1}{2L^2} \int_t^T \int_{\prob(\T^d)} \left [ L\int_{\T^d} |\rmD u_r(x,\mu)-\rmD u_r^n(x,\mu) | \de \mu(x) \right ]^2 \de \mathcal{D}(\mu) \de r
    \nonumber
    \\
    & \le \frac{1}{2} \int_t^T \| \rmD u_r - \rmD u_r^n \|^2 \de r.
\label{eq:proof:Proposition5.11:term4}
\end{align}
For \eqref{eq:proof:Proposition5.11:term2}, we observe that we can write
\begin{equation}\label{eq:523}
    \begin{split}
    \frac{1}{2L^2} \int_t^T  |\mathcal{H}(\mu, \rmD u_r) &- \mathcal{H}^n(\mu, \rmD u_r)  |^2_H \de r  \\
    & = \frac{1}{2L^2} \int_t^T \int_{\prob(\T^d)} \left | \sum_{i=N(\eps_n, \mbfs(\mu))+1}^{\infty} s_i(\mu) \HH(x_i(\mu), \rmD u_r(x_i(\mu), \mu)) \right |^2 \de \mathcal{D}(\mu) \de r \\
    &  \le \frac{1}{2L^2}  \int_t^T \int_{\prob(\T^d)}\sum_{i= N(\eps_n, \mbfs(\mu)) + 1}^{\infty} s_i(\mu) |\HH(x_i(\mu), \rmD u_r(x_i(\mu), \mu))|^2 \de \cD(\mu) \de r
    \\
    & \le \frac{1}{2L^2} \mathcal{R}^n(\mathscr H, \rmD u, T).
    \end{split}
   \end{equation}
   We write $\mathcal{R}^n(\mathscr H, \rmD u, T)$ in the form 
   \begin{equation*}
\begin{split}
\mathcal{R}^n(\mathscr H, \rmD u, T)&= 
     \int_0^T \int_{\prob(\T^d)} r_n(r,\mu) \de\mathcal{D}(\mu) \de r
\\
{\rm with} \quad 
r_n(r,\mu) &:= \sum_{i= N(\eps_n, \mbfs(\mu)) + 1}^{\infty} s_i(\mu) |\HH(x_i(\mu), \rmD u_r(x_i(\mu), \mu))|^2.
    \end{split}
\end{equation*}
Next, we prove by means of the dominated convergence theorem that 
$\mathcal{R}^n(\mathscr H, \rmD u, T) \to 0$ as 
$n \rightarrow \infty$. Domination follows from the bound
\begin{align*}
    |r_n(r,\mu)| 
    \le \sum_{i=1}^{\infty} s_i(\mu) |\HH(x_i(\mu), \rmD u_r(x_i(\mu), \mu))|^2 
    =  \int_{\prob(\T^d)} | \mathscr{H}(x, \rmD u_r(x, \mu))|^2 \de \mu(x),
\end{align*}
the right-hand side, seen as a function of $(r,\mu)$, belonging, as a consequence of Remark \ref{rmk: H_integral} and the fact that $\mathscr{H}$ satisfies Assumption \ref{hp: hj}, to
the space $L^1([0,T] \times \prob(\T^d), \mathscr{L}^{[0,T]}\otimes \cD)$. The above bound proves that the sequence $(\vert r_n(r,\mu)\vert)_{n}$ is dominated by a common function in $L^1([0,T] \times \prob(\T^d), \mathscr{L}^{[0,T]}\otimes \cD)$. At the same time, it shows that 
$r_n(r,\mu) \to 0$ pointwise as $n \to \infty$, since it is bounded by the remainder 
of a convergent series. Therefore,
\begin{equation}
\label{eq:proof:Proposition5.11:term5}
\lim_{n \rightarrow \infty} \mathcal{R}^n(\mathscr H, \rmD u, T)
=0.
\end{equation}
Combining  \eqref{eq:proof:Proposition5.11:term1},
\eqref{eq:proof:Proposition5.11:term2},
\eqref{eq:proof:Proposition5.11:term3},
\eqref{eq:proof:Proposition5.11:term4}, and \eqref{eq:523}, we deduce that, for every $t \in [0,T]$, 
\begin{align*}
\frac{1}{2} &|u_t^n-u_t|_H^2 + \frac{1}{2} \int_t^T \| \rmD u_r^n - \rmD u_r \|^2 \de r 
\\
&\le \frac{1}{2} |\mathscr{G}^n-\mathscr{G}|^2_H + (1/2+L^2) \int_t^T |u_r^n - u_r|^2_H \de r + \frac{1}{2} \int_0^T |\mathscr{F}^n_r-\mathscr{F}_r|^2_H \de r + \frac{1}{2L^2} \mathcal{R}^n(\mathscr H, \rmD u, T). 
\end{align*}
The inequality in the statement follows then by the above inequality and a standard application of Gronwall's lemma. The convergence in the statement follows from \eqref{eq:proof:Proposition5.11:term5}.    
\end{proof}

\begin{remark}
   Questions related to the rate of convergence can be approached as in the linear case, see Remark \ref{rem:rate:convergence}. 
Since the integrand in the definition of ${\mathcal R}^n$ (in the statement of Proposition
\ref{prop:convergence:nonlinear}) grows like 
$\vert {\rmD} u_r\vert^2$, the difficulties are the same, especially those stemming from the poor integrability properties of the gradient of $u$.
\end{remark}

It is not difficult to see that an appropriate version of Proposition \ref{prop:ito_general_HJB} 
also holds in the framework of this subsection. We state the result in the following proposition, omitting the proof since it can be obtained by repeating exactly the argument used for proving Proposition \ref{prop:ito_general_HJB}.

\begin{proposition}\label{prop: ito_for_HJB_truncated}
Let Assumption \ref{hp: control} be in force. For $n \in {\mathbb N}_+$, let $\cH^n$, $\mathscr{F}^n$, $\mathscr{G}^n$ be as above, and $u^n$ be the (corresponding) unique solution of \eqref{eqn: HJB_n}.
Let
$t_0 \in [0,T)$ be an initial time. There exists a Borel subset $O$ of $\cP(\T^d)$, with ${\mathcal D}(O)=1$, such that, 
for any $m \in O$ and any $\boldsymbol{\alpha} \in \mathcal{A}$
\begin{align*}
    \mathscr{G}^n(\mu^\infty_T) & =  u_{t}^n(\mu^\infty_{t})  
    - \int_{t}^T \mathscr{F}^n_r(\mu^\infty_r) \de r   
    + \int_{t}^T \sum_{i=1}^{N(\eps_n, \mbfs(m))}\varsigma_i\HH( \xi_r^i,\rmD u_r^n(\mu_r^\infty,\xi_r^i))\de r 
    \\
    & \quad + \sum_{i=1}^\infty \int_t^T \varsigma_i \rmD u_r^n(\mu_r^\infty, \xi_r^i) \cdot \alpha_r^i(\mu_r^\infty, \xi_r^i) \de r \\
    & \quad +\sum_{i=1}^{\infty}\int^T_{t} \sqrt{2 \varsigma_i}\rmD  u_r^n(\mu^\infty_r, \xi^i_r) \cdot \de \beta^i_r,\quad t\in[t_0,T],\quad{\mathbf Q}^{t_0,\emp^{-1}(m),\boldsymbol{\alpha}}\text{-a.s.}
\end{align*}
\end{proposition}

\subsubsection{The approximate control problem}  
\label{subsubse:approximate:control}
Here we discuss, similarly 
to Subsection \ref{ssec:finite_dim}, a finite-dimensional approximation of the mean-field control problem introduced in 
Subsection \ref{ssec:verification}. 
The purpose of this approximation is twofold:
(i) on the one hand, it aims to provide, within the context of this paper, a counterpart to previously established results on the approximation of a mean-field control problem by a finite-dimensional (in that case, exchangeable) control problem;
(ii) on the other hand, it also aims to give an interpretation of the partial differential equation \eqref{eqn: HJB_n} and to explain how the solution to the latter can be viewed as the value function of an auxiliary control problem.

Our first step is thus to introduce a suitable class of controls to connect the approximate PDE \eqref{eqn: HJB_n} to the system of particles \eqref{eqn: controlled_particles}
(taking advantage of the fact that 
the 
feedback functions $\alpha^i$
may differ from one index 
$i$ to another).
For every $n \in \N_+$, we define
\[ \mathcal{A}_n := \left \{(\alpha^i)_i \in \mathcal{A} : \alpha_t^i(\mu, x) = 0 \text{ for every } (t,\mu, x) \in [0,T] \times \prob(\T^d) \times \T^d \text{ and } i \in \N_+ \text{ s.t.~} N(\eps_n, \mbfs(\mu)) <i\right \}.\]
For 
an initial time $t_0\in[0,T)$, an initial distribution $\cursm \in \prob(T_o^\infty \times (\T^d)_o^\infty)$, and
an admissible control 
${\boldsymbol \alpha} \in {\mathcal A}_n$, the $n$-cost associated to $(t_0,\cursm,{\boldsymbol \alpha})$ is defined as
\begin{equation}\label{eqn: approx_cost}
J^n(t_0,\cursm,\boldsymbol{\alpha}):=   \mathbb{E}^{\mathbf{Q}^{t_0, \cursm, \boldsymbol{ \alpha}}}\left[\int_{t_0}^T\left(\mathscr{F}^n_r(\mu_r^\infty) +  \sum_{i=1}^{N(\varepsilon_n,(\varsigma_i)_i)
} \varsigma_i \LL(\xi^i_r,\alpha^i_r(\mu_r^\infty, \xi_r^i) )\right) \de r+ \mathscr{G}^n(\mu^\infty_T)\right]. 
\end{equation}
The corresponding value function is defined as
\begin{equation}\label{eqn: valuen}
    V^n(t_0,\cursm):= \inf_{\boldsymbol{\alpha}\in\cA_n} J^n(t_0,\cursm,\boldsymbol{\alpha}),\quad (t_0,\cursm)\in[0,T)\times \prob( T_o^\infty \times (\T^d)_o^\infty).
\end{equation}
The following result is the analogue of Theorem \ref{thm: verif} for the approximate PDE \eqref{eqn: HJB_n}.

\begin{theorem}\label{thm: verif_n} Let Assumption \ref{hp: control} be in force, $n \in \N_+$, and $u^n$ be the solution of \eqref{eqn: HJB_n}. Then for every $t_0 \in [0,T)$ there exists a Borel subset $\tilde{O}_{t_0}$ of $\prob(\T^d)$, with $\cD(\tilde{O}_{t_0})=1$, such that
\begin{equation}\label{eq:verific_eqn}
    u^n_{t_0}(m) = V^n(t_0, \emp^{-1}(m))\quad\text{ for all } m \in \tilde{O}_{t_0}.
\end{equation}
Moreover, if $t_0 \in [0,T)$ and $m \in \tilde{O}_{t_0}$, then the control 
    $\hat{\boldsymbol \alpha}^n=(\hat{\alpha}^{i,n})_i$, with the 
    $\hat{\alpha}^{i,n}$'s being given by 
\[ \hat{\alpha}^{i,n}_t(\mu,x) := -\nabla_p \HH (x, \rmD u_t^n(\mu,x)) \mathbf{1}_{i \le N(\eps_n, \mbfs(\mu))}, \quad (t,\mu,x) \in [0,T] \times \prob(\T^d) \times \T^d,\]
is optimal for $V^n(t_0, \emp^{-1}(m))$ and it is unique in the following sense: if $\tilde{\boldsymbol{\alpha}} \in \mathcal{A}_n$ is also optimal for $V^n(t_0, \emp^{-1}(m))$, then,
for any $i \in \N_+$, for a.e. $r \in (t_0,T)$, 
\begin{equation*}
{\mathbf Q}^{t_0, \emp^{-1}(m), \tilde{\boldsymbol{\alpha}}}
\left( \left\{\tilde{\alpha}_r^i(\mu_r^\infty,\xi_r^i) = 
\hat{\alpha}_r^{i,n}(\mu_r^\infty,\xi_r^i)  \right\}
\right) = 1.
\end{equation*}
\end{theorem}
\begin{proof} The proof follows the same argument as the one of Theorem \ref{thm: verif}, using Proposition \ref{prop: ito_for_HJB_truncated} instead of  Proposition \ref{prop:ito_general_HJB}, also noting that 
\[ \sum_{i=1}^\infty \int_{t_0}^T \varsigma_i \rmD u_r^n(\mu_r^\infty, \xi_r^i) \cdot \alpha_r^i(\mu_r^\infty, \xi_r^i) \de r = \sum_{i=1}^{N(\eps_n,\mbfs(m))}\int_{t_0}^T \varsigma_i \rmD u_r^n(\mu_r^\infty, \xi_r^i) \cdot \alpha_r^i(\mu_r^\infty, \xi_r^i) \de r   \quad \mathbf{Q}^{t_0, \emp^{-1}(m), \boldsymbol{\alpha}}\text{-a.s.}\]
for every $\boldsymbol{\alpha} \in \mathcal{A}_n$, $t_0 \in [0,T)$, and $m \in \tilde{O}_{t_0}$.
    \end{proof}

\begin{remark}
\label{rem:N:dimension}
The following observations are in order:
\begin{enumerate}[1)]
\item Above, the threshold $N(\varepsilon_n,{\mbfs })$ is defined as a function of ${\mbfs}$, see \eqref{eq:Neps:mbfs}. Recalling that 
$N(\varepsilon,\mbfs) \leq \lceil 1/\varepsilon \rceil$, we could also use 
$\lceil 1/\varepsilon \rceil$ as a universal cap for the index $i$. The statement of Theorem 
\ref{thm: verif_n} would remain unchanged. 
\item Although the feedback functions $(\alpha^i)_i$ in ${\mathcal A}_n$ are identically zero for $i > N(\varepsilon_n, \mathbf{s}(\mu))$, the corresponding particle system \eqref{eqn: controlled_particles} remains infinite-dimensional, and each $\alpha^i$ is still defined on an infinite-dimensional space.
This leads us to consider, in the statement below, a truly finite-dimensional version of the control problem.
\end{enumerate}
\end{remark}

In the statement below, we use the same finite-dimensional functional spaces as in the statement of  Theorem \ref{thm:finite:dim:PDE}.

\begin{theorem} 
\label{thm:5.14:PDE:finite:dim}
Let Assumption \ref{hp: control} be in force, $n \in \N_+$, and $u^n$ be the solution of $\eqref{eqn: HJB_n}$. Then for every $t_0 \in [0,T)$ and $\cD$-a.e.~$m \in \prob(\T^d)$
\[ u^n_{t_0}(m) = k_{t_0}^{\mbfs(m)}(x_1(m), \dots, x_{N(\eps_n, \mbfs(m))}(m)),\]
where, for $\mbfs \in T_o^\infty$, $k^{\mbfs}$ is the unique solution, in the space
$\rmC^{0,1,N(\varepsilon_n,\mbfs)} \cap (\cap_{p > 1} W^{1,2,N(\varepsilon_n,\mbfs)}_p$), 
of the PDE, set on $[0,T] \times (\T^d)^{N(\eps_n, \mbfs)}$,
\begin{equation}\label{eq:HJBFD}
  \partial_t k_t + \sum_{i=1}^{N(\eps_n, \mbfs)} \frac{1}{s_i} \Delta_i k_t + \sum_{i=1}^{N(\eps_n, \mbfs)}  s_i \mathscr H(\cdot, s_i^{-1} \nabla_i k_t) + \mathscr{F}^n_t \circ \emp^{N(\eps_n, \mbfs)}(\mbfs, \cdot) =0 \quad \text{ for a.e. } t \in (0,T)   
\end{equation}
with final condition
\[ k_t \rvert_{t=T} = \mathscr{G}^n \circ \emp^{N(\eps_n, \mbfs)}(\mbfs, \cdot). \] 
\end{theorem}

Before presenting the proof of Theorem \ref{thm:5.14:PDE:finite:dim}, we explain its meaning from the perspective of stochastic control.
For a fixed 
$\mbfs \in T^\infty_o$, 
equation \eqref{eq:HJBFD} is indeed a finite-dimensional Hamilton-Jacobi-Bellman equation. It is associated 
with a stochastic control problem in finite dimension, which can be described as follows.
With the shorthand notation 
$N:=N(\varepsilon_n,\mbfs)$, 
and for an arbitrary filtered probability space $(\Omega,{\mathcal F},{\mathbb P},\{{\mathcal F}_t\}_{t \in [0,T]})$ endowed with a collection $\{B^i\}_{i=1}^N$ of independent $d$-dimensional $\{\cF_t\}_{t\in[t_0,T]}$-Brownian motions,  we consider the collection of $A^N$-valued $\{{\mathcal F}_t\}_{t \in [0,T]}$-progressively-measurable open-loop controls
\[ \Gamma^N_o := \left \{ \ddelta^N:=(\delta^1, \dots, \delta^N) : \delta^{i} : [0,T] \times \Omega \to A \text{ is progressively measurable } \forall \, i \in \{1, \dots, N\} \right \}.\]
For the
 same collection 
of weights 
$\mbfs \in T^\infty_o$,
and with
an initial condition 
$(t_0,\mbfx^N=(x_i)_{i=1}^{N}) \in
[0,T) \times (\T^d)^N$, and  a control ${\boldsymbol \delta}^N \in \Gamma^N_o$, we associate the cost 
\begin{equation}
\label{eq:JoN}
 {J}_o^{N, \mbfs}(t_0, \mbfx^N, \ddelta^N) := \mathbb{E}^{\mathbb Q} \left [ \int_{t_0}^T\left(\mathscr{F}^n_r \circ \emp^N(\mbfs, \mbfX^N_r) +  \sum_{i=1}^N s_i \LL(X_r^i,\delta^i_r) )\right) \de r+ \mathscr{G}^n\circ \emp^N(\mbfs, \mbfX^N_T)\right], 
\end{equation}
where $ \mbfX^N=\{(X^i_t)_{t \in [t_0,T]} \}_{i=1}^N$ is the stochastic process on $(\Omega,{\mathcal F},{\mathbb P},\{{\mathcal F}_t\}_{t \in [0,T]})$ defined by
\begin{equation}\label{eqn:gammasisto}
    \begin{cases}
        \de X^i_t & = \delta^i_t\de t + {\sqrt{\frac{2}{s_i}}}\de B^i_t,\quad t\in[t_0,T],\quad i\in\{1, \dots, N\},\\
        X^i_{t_0} & = \iota(x_i), \quad i\in\{1, \dots, N\}.
    \end{cases}
\end{equation} 
As a corollary of the proof of 
Theorem \ref{thm:5.14:PDE:finite:dim}, we will identify the value function with the solution of \eqref{eq:HJBFD}.
\begin{corollary}
\label{corol:open:loop}
Within the framework of 
Theorem \ref{thm:5.14:PDE:finite:dim}, 
the value function of the control problem 
\eqref{eq:JoN}--\eqref{eqn:gammasisto}
is the function $k^{\mbfs}$, i.e., 
\begin{equation}\label{eq:kolmohjbo}
 k_{t_0}^\mbfs(\mbfx^N) = \inf_{\ddelta^N \in \Gamma^N_o} {J}_{o}^{N, \mbfs}(t_0, \mbfx^N, \ddelta^N), \quad t_0 \in [0,T), \, \mbfx^N \in (\T^d)^N.  
\end{equation}
\end{corollary}
Consistently
with the control problem formulated in \eqref{eqn: approx_cost}, we also introduce a version of \eqref{eq:JoN}--\eqref{eqn:gammasisto} set over 
(Marokvian) closed-loop controls. Here, the set of closed-loop controls  is given by 
\[ \Gamma^N_c := \left \{ \ggamma^N:=(\gamma^1, \dots, \gamma^N) : \gamma^{i} : [0,T] \times (\T^d)^N \to A \text{ is measurable } \forall \, i \in \{1, \dots, N\} \right \},\]
and, for given
$t_0 \in [0,T)$
and
${\mathbf x}^N \in (\T^d)^N$,  
the cost function is
\[ {J}_c^{N, \mbfs}(t_0, \mbfx^N, \ggamma^N) := \mathbb{E}^{\mathbb Q} \left [ \int_{t_0}^T\left(\mathscr{F}^n_r \circ \emp^N(\mbfs, \mbfX^N_r) +  \sum_{i=1}^N s_i \LL(X_r^i,\gamma^i_r(\mbfX^N_r) )\right) \de r+ \mathscr{G}^n\circ \emp^N(\mbfs, \mbfX^N_T)\right], \]
where, on the space  $(\Omega,{\mathcal F},{\mathbb P},\{{\mathcal F}_t\}_{t \in [0,T]}, \{(B^i_t)_{t \in [t_0,T]} \}_{i=1}^N) $, 
$\mbfX^N$ is 
the solution to
\begin{equation}\label{eqn:gammasistc}
    \begin{cases}
        \de X^i_t & = \gamma^i_t(\mbfX^N_t)\de t + {\sqrt{\frac{2}{s_i}}}\de B^i_t,\quad t\in[t_0,T],\quad i\in\{1, \dots, N\},\\
        X^i_{t_0} & = \iota(x_i) , \quad i\in\{1, \dots, N\}.
    \end{cases}
\end{equation}
Note that, in the above equations, we are using the same convention as in Remark \ref{rem:toruspb} to evaluate $\gamma_r^i$ at $\mbfX^N_r \in (\T^d)^N$.
Note also that, even though the functions 
$\gamma^1,\cdots,\gamma^N$ are just assumed to be measurable (and bounded since they are valued in $A$), the system \eqref{eqn:gammasistc} is uniquely solvable in the strong sense, see \cite{Veretennikov80}. As such, it is also uniquely solvable in the weak sense, but there is no need here to deal with weak solutions; this makes a difference with the infinite dimensional setting 
introduced in Subsection \ref{ssec:verification}.

It is a standard result in the theory of stochastic control that the control problems defined respectively over open and closed controls share the same value function. We reestablish this fact below as a consequence of the verification argument used in the proof of Theorem \ref{thm:5.14:PDE:finite:dim}. In particular, we have 
the following analogue of Corollary 
\ref{corol:open:loop}.
\begin{corollary}
\label{corol:closed:loop}
Within the framework of 
Theorem \ref{thm:5.14:PDE:finite:dim}, 
the value function of the control problem 
\eqref{eq:JoN}--\eqref{eqn:gammasisto}
is the function $k^{\mbfs}$, i.e., 
\begin{equation}
\label{eq:kolmohjbc}
 k_{t_0}^\mbfs(\mbfx^N) = \inf_{\ggamma^N \in \Gamma^N_c} {J}_{c}^{N, \mbfs}(t_0, \mbfx^N, \ggamma^N), \quad t_0 \in [0,T), \, \mbfx^N \in (\T^d)^N.
\end{equation}
\end{corollary}
We stress once again that the dimension $N$ in \eqref{eq:kolmohjbc} depends on $\mbfs$. That said, as mentioned in the 
first item of Remark 
\ref{rem:N:dimension}, one could define another truncation in \eqref{eqn: approx_cost}, with $N$ therein being equal to $\lceil 1/\varepsilon_n \rceil$. With this definition, all the PDEs \eqref{eq:HJBFD} (that still depend on $\mbfs$) would be set over the same space. 

We now turn to the proof of the last three statements.
 
\begin{proof}[Proofs of Theorem \ref{thm:5.14:PDE:finite:dim} and Corollaries \ref{corol:open:loop} and \ref{corol:closed:loop}] 
Existence and uniqueness of a solution to the PDE
\eqref{eq:HJBFD}
(within the class specified in the statement of Theorem \ref{thm:5.14:PDE:finite:dim})
follow from 
\cite[Theorem 2.1]{DELARUE20061712}.
By the 
Krylov version of It\^o's formula, see 
\cite[Theorem 2.10.1]{KrylovBookControl}, 
one can expand, for 
a collection of weights 
$\mbfs \in T^\infty_{o}$,
an initial condition $(t_0,{\mathbf x}^N)$
(with $N:=N(\varepsilon_n,\mbfs)$)
and an open-loop control ${\boldsymbol \delta}^N$ (in $\Gamma^N_o$), the process 
\begin{equation*}
\left(k_t^{\mbfs}(\mbfX_t^N) + \int_{t_0}^t 
\left(\mathscr{F}^n_r \circ \emp^N(\mbfs, \mbfX^N_r) +  \sum_{i=1}^N s_i \LL(X_r^i,\delta^i_r) \right) \de r
\right)_{t \in [t_0,T]}.
\end{equation*}
Reproducing the computations carried out in the proofs of Theorems \ref{thm: verif} and \ref{thm: verif_n} (the setting is even simpler here since the underlying space is of finite dimension), we can prove that 
\begin{equation}
\label{eq:proof:coroll:open}
 {J}_{o}^{N, \mbfs}(t_0, \mbfx^N, \ddelta^N)
 \geq 
  k_{t_0}^\mbfs(\mbfx^N), 
  \end{equation}
  with equality if and only if, almost everywhere under 
  ${\mathscr L}^{[t_0,T]} \otimes {\mathbb P}$, 
\begin{equation*}
\delta_t^i =- \nabla_p
{\mathscr H}\left(X^i_t,s_i^{-1} \nabla_i k_t(\mbfX^N_t)
\right),
  \end{equation*}
  i.e. 
  $\mbfX^N$ is driven by the 
  closed-loop
control $\gamma^N \in \Gamma^N_c$ given by 
\begin{equation*}
\gamma_t^i(\mbfx)=-
\nabla_p
{\mathscr H}\left(x_i,s_i^{-1} \nabla_i k_t(\mbfx)
\right),
\end{equation*}
which proves, as a by-product, the two Corollaries \ref{corol:open:loop}
and
\ref{corol:closed:loop}.

Therefore, by Theorem \ref{thm: verif_n}, the
conclusion of Theorem \ref{thm:5.14:PDE:finite:dim} follows if we show that, for every $(\mbfs, \mbfx)\in T_o^\infty \times (\T^d)^\infty_o$ and every $t_0 \in [0,T)$, setting $N:= N(\eps_n, \mbfs)$ and $\mbfx^N := \pi_N(\mbfx)$, 
(recalling \eqref{eqn: valuen}
for the definition of $V^n$):
\begin{itemize}
    \item[(i)] the infimum in \eqref{eq:kolmohjbc} is larger than $V^n(t_0, (\mbfs, \mbfx))$;
    \item[(ii)] the infimum in \eqref{eq:kolmohjbo} is smaller than $V^n(t_0, (\mbfs, \mbfx))$.
\end{itemize}
To prove (i), we proceed as follows: let $\ggamma^N \in \Gamma^N_c$ be given and let us define
\[ \alpha_t^i (\mu, x):= \gamma_t^i\left(\pi_{N(\eps_n,\mbfs(\mu))} (\mbfx(\mu))\right) \mathbf{1}_{i \le N(\eps_n, \mbfs(\mu))} \quad (t, \mu, x) \in [0,T] \times \prob_o(\T^d) \times \T^d, \,\, i \in \{1, \dots, N\},\]
and $\alpha^i=0$ if $i >N$. We also set $\alpha^i(\mu,x)=0$ if $\mu \in \prob(\T^d) \setminus \prob_o(\T^d)$. Notice that $\alpha^i$ is Borel measurable, since $\prob_o(\T^d)$ is a Borel subset of $\prob(\T^d)$, see the discussion below \eqref{eq:Intro:Phi}.
For this ${\boldsymbol \alpha}$, we easily check that the
$N$ first coordinates
of the 
(weak) solution to \eqref{eqn: controlled_particles}
(constructed on 
some 
set-up 
$({\Omega},{\cF},{\Q},\{{\cF}_t\}_{t\in[t_0,T]}, (({\mbfs},{\mbfx}),  ({W}^i)_i))$) coincides with the solution 
to \eqref{eqn:gammasistc}, when constructed on the same set-up
(and when the initial conditions are the same). 
Since the law of the 
solution to 
\eqref{eqn:gammasistc} is independent of the probabilistic set-up, we deduce that 
\[ {J}_{c}^{N, \mbfs}(t_0, \mbfx^N, \ggamma^N) =J^n(t_0, (\mbfs, \mbfx), \boldsymbol{\alpha}) \ge V^n(t_0, (\mbfs, \mbfx)).\]
Taking the infimum over $\ggamma^N \in \Gamma^N_c$ in the above inequality, we conclude the proof of (i).\\
To prove (ii), we proceed in the opposite way: let $\boldsymbol{\alpha} \in \mathcal{A}_n $ be given and, on the filtered probability space $(\Xi, \cG, \mathbf{Q}^{t_0, (\mbfs, \mbfx), \boldsymbol{\alpha}}, \{ \cG_t\}_{t \in [0,T]})$ endowed with the collection $(\beta^i)_{i=1}^N$ of independent $d$-dimensional $\{\cG_t\}_{t\in[t_0,T]}$-Brownian motions, let us define the progressively measurable functions
\[ \delta^i_t := \alpha_t^i(\mu_t^\infty, \xi_t^i) \quad t \in [0,T], \, i \in \{1, \dots, N\}.\]
Then $\ddelta^N:=(\delta^1, \dots, \delta^N) \in \Gamma_o^N$ and the process $\mbfX^N:=(\xi^1, \dots, \xi^N)$ satisfy \eqref{eqn:gammasisto} (on the canonical set-up). Reproducing the proof of 
\eqref{eq:proof:coroll:open}, we deduce that
\[ 
k_{t_0}^{\mbfs}(\mbfx^N) \le {J}_{o}^{N, \mbfs}(t_0, \mbfx^N, \ddelta^N)=J^n(t_0, (\mbfs, \mbfx), \boldsymbol{\alpha}).\]
Taking the infimum over $\boldsymbol{\alpha} \in \mathcal{A}_n$ in the above inequality, we conclude the proof of (ii).
\end{proof}
\appendix
\section{L-derivative and its relation with the vector-valued gradient \texorpdfstring{$\rmD$}{D} }\label{app:lder}
We compare the notion of derivative discussed in Subsection \ref{sec:diffprob} with another popular definition, which is particularly relevant in the context of mean-field games and mean-field control. 

\begin{definition} We say that a continuous function $u:\prob(\T^d) \to \R$ is linearly differentiable at $\mu \in \prob(\T^d)$ if the limit
\begin{equation*}
    \rmD_F u(\mu,x):=\lim_{t\downarrow0}\frac{u((1-t)\mu + t\delta_x) - u(\mu)}{t}
\end{equation*}
exists and it is finite for every $x \in \T^d$. The real number above is called \emph{flat or linear functional derivative} of $u$ at $(\mu,x)$. The space of continuous functions which are everywhere linearly differentiable and whose flat derivative is everywhere continuous in $\prob(\T^d) \times \T^d$ is denoted by $\rmC^1_{\rmF}(\prob(\T^d))$.
\end{definition}
The flat derivative of $u \in \rmC^1_{\rmF}(\prob(\T^d))$ can be also characterized as the unique continuous function $G: \prob(\T^d) \times \T^d \to \R$ satisfying
\begin{equation}\label{eqn: flat_relation} u(\mu)-u(\nu) = \int_0^1\int_{\T^d} G(t\mu + (1-t)\nu, x) \de(\mu - \nu)(x)\de t, \quad \int_{\T^d}G(\mu,x)\de\mu(x) = 0
\end{equation}
for every $\mu, \nu \in \prob(\T^d)$ (see ~\cite[Chapter 2.2]{cardaliaguetdelaruelasrylions}).
\begin{definition} We say that $u: \prob(\T^d) \to \R$ is L-differentiable at $\mu \in \prob(\T^d)$ if $u$ is linearly differentiable at $\mu$ and the gradient
\[
\rmD_L u(\mu,x):=\nabla \rmD_F(\mu,x)
\]
exists for every $x \in \T^d$, where $\nabla$ denotes the gradient with respect to the $x$-variable. The vector above is called \emph{L-derivative of} $u$ at $(\mu,x)$. The space of continuous functions which are everywhere L-differentiable and whose L-derivative is everywhere continuous (and so bounded) is denoted by $\rmC^1(\prob(\T^d))$.
\end{definition}
    \begin{remark}
    The L-derivative introduced here is strongly related, at least in some cases, to the one introduced by P.-L. Lions in \cite{lions_course} (see also \cite[Chapter 5.2]{carmonadelarue1}) through the so-called lifting approach. For our purposes, it is enough to note that if a function $u$ belongs to $\rmC^1(\prob(\T^d))$, then it is differentiable in the sense of Lions, and the two derivatives coincide (see \cite[Proposition 5.84]{carmonadelarue1}). The converse may also be true under further assumptions on the Fréchet derivative of the lifting of $u$ (see \cite[Proposition 5.51]{carmonadelarue1}).
    \end{remark}
    \begin{remark}\label{rem:lips}
        It is easy to see that if $u\in\rmC^1(\prob(\T^d))$, then it is $W_2$-Lipschitz with Lipschitz constant $\norm{\rmD_Lu}_\infty$. Indeed, since $\rmD_L u$ is bounded, we have that for any $\mu\in\prob(\T^d)$ the map $x\mapsto\rmD_F u(\mu,x)$ is Lipschitz with Lipschitz constant bounded by $\norm{\rmD_L u}_\infty$. Then, by \eqref{eqn: flat_relation} and the dual representation of the $1$-Wasserstein distance $W_1$, see e.g.~\cite[Theorem 1.14]{Villani03}, and the inequality $W_1 \le W_2$, it holds
        \begin{equation*}
            \lvert u(\mu) - u (\nu)\rvert\leq \norm{\rmD_L u}_\infty\sup_{\phi\colon\Lip \phi \leq1}\int_{\T^d}\phi\de(\mu-\nu) =\norm{\rmD_L u}_\infty W_1(\mu,\nu) \leq\norm{\rmD_L u}_\infty W_2(\mu,\nu),
        \end{equation*}
        where we have denoted by $\Lip \phi$ the global Lipschitz constant of $\phi$.
    \end{remark}
    It is clear that if $u \in \mathfrak{Z}^\infty$ then $u \in \rmC^1(\prob(\T^d))$ and $\boldnabla u = \rmD_L u$.
\begin{proposition}\label{prop:approx}
    Let $u\in\rmC^1(\prob(\T^d))$. Then for every Borel probability measure $\mm$ on $\prob(\T^d)$ there exists a sequence of cylinder functions $(u_n)_n\subset \mathfrak{Z}^\infty$ such that
    \begin{equation}\label{eq:conv}
		u_n \to u \text{ in } L^2(\prob(\T^d), \mm), \quad \boldnabla u_n \to \rmD_L u \text{ in } L^2(\prob(\T^d) \times \T^d , \bmm; \R^d),
    \end{equation}
    where $\bmm$ is the measure over $\prob(\T^d)\times\T^d$ defined by \eqref{eq:bmm} with $\mm$
 in place of $\cD$.
 \end{proposition}
\begin{proof}
    We divide the proof into three claims.\\
	\textit{Claim 1}. We define, for every $n \in \N_+$, the functions $h_n: (\T^d)^n \to \R$ and $v_n : \prob(\T^d) \to \R$ as
 \begin{align*}
 h_n(x_1, \dots, x_n)&:= u \left ( \frac{1}{n} \sum_{i=1}^n \delta_{x_i} \right ), \quad &&(x_1, \dots, x_n) \in (\T^d)^n,\\
 v_n(\mu)&:= \int_{(\T^d)^n} h_n \de \mu^{\otimes n}, \quad &&\mu \in \prob(\T^d),
\end{align*}
where $\mu^{\otimes n}$ denotes the measure on $(\T^d)^n$ obtained as the product of $n$ copies of $\mu$. Then $(v_n)_n \subset \rmC^1(\prob(\T^d))$, $(h_n)_n \subset \rmC^1((\T^d)^n)$, $\|v_n\|_\infty \le \|u\|_\infty$, $\|\rmD_L v_n\|_\infty \le \|\rmD_L u\|_\infty$ and the following formulas hold
\begin{align*}\label{eqn: der_Phi_n}
    \nabla_j h_n(x_1, \dots, x_n) & = \frac{1}{n}\rmD_L u\left(\frac{1}{n} \sum_{i=1}^n \delta_{x_i},x_j\right), \quad &&(x_1, \dots, x_n) \in (\T^d)^n,\\
    \rmD_L v_n(\mu, x) & = \int_{\T^{d(n-1)}} \rmD_L u \left ( \frac{1}{n} \sum_{i=1}^{n-1} \delta_{x_i} + \frac{1}{n}\delta_x, x \right ) \de \mu^{\otimes (n-1)}(x_1, \dots, x_{n-1}), \quad &&(\mu, x) \in \prob(\T^d) \times \T^d,
\end{align*}
where $\nabla_j$ denotes the gradient w.r.t.~the $j$-th variable. Moreover,
\begin{equation}\label{eq:conv1}
\lim_{n \to \infty} v_n(\mu) = u(\mu), \quad \lim_{n \to \infty} \rmD_L v_n(\mu, x) = \rmD_L u(\mu, x) \quad \text{ for every } (\mu, x) \in \prob(\T^d)\times \T^d.  
\end{equation}
\emph{Proof of claim 1}. The fact that $(h_n)_n \subset \rmC^1((\T^d)^n)$ and the formula for $ \nabla_j h_n$ can be found for instance in \cite[Proposition 5.35]{carmonadelarue1}. Then, it is not difficult to check (see, e.g.,~\cite[proof of Theorem 4.4]{cossomartini}) that $v_n$ is L-differentiable and that 
\begin{align*} \rmD_L v_n(\mu, x) &=  n \int_{\R^{d(n-1)}} \nabla_n h_n (\cdot, x) \de \mu^{\otimes (n-1)} \\
&= \int_{\R^{d(n-1)}} \rmD_L u \left ( \frac{1}{n} \sum_{i=1}^{n-1} \delta_{x_i} + \frac{1}{n}\delta_x, x \right ) \de \mu^{\otimes (n-1)}(x_1, \dots, x_{n-1}), \quad (\mu, x) \in \prob(\T^d) \times \T^d.
\end{align*}
This, together with the definition of $v_n$, gives the desired bounds and the fact that $(v_n)_n \subset \rmC^1(\prob(\T^d))$. The convergence in \eqref{eq:conv1} follows by the law of large numbers and the continuity of $u$ and $\rmD_L u$. This concludes the proof of the first claim.\\

\textit{Claim 2}. Let $n \in \N_+$ be fixed and let $v_n$ be as in claim 1. Then, there exists a sequence of polynomials $(p^n_k)_k$ in $(\T^d)^n$ satisfying 
\begin{equation}\label{eq:simm}
p_k^n(x_{\sigma(1)}, \dots, x_{\sigma(n)}) = p_k^n(x_1, \dots, x_n) \text{ for every } (x_1, \dots, x_n) \in (\T^d)^n
\end{equation}
and for every permutation $\sigma$ of $\{1, \dots, n\}$, such that, setting
\[
w_k^n(\mu):= \int_{\R^{dn}} p_k^n \de \mu^{\otimes n}, \quad \mu \in \prob(\T^d),
\]
we have that $w_k^n$ is $L$-differentiable at every $\mu \in \prob(\T^d)$, there exists a constant $C=C(\|u\|_\infty, \|\rmD_L u\|_\infty)>0$ such that $\|w_k^n\|_\infty \le C$, $\|\rmD_L w_k^n\|_\infty \le C$ for every $k \in \N$, and
\begin{equation}\label{eq:conv2}
\lim_{k \to \infty} w_k^n(\mu) = v_n(\mu), \quad \lim_{k \to \infty} \rmD_L w_k^n(\mu, x) = \rmD_L v_n (\mu, x) \quad \text{ for every } (\mu, x) \in \prob(\T^d)\times \T^d.  
\end{equation}
\emph{Proof of claim 2}.
We take any sequence of polynomials $(p_k^n)_k$ in $(\T^d)^n$ satisfying \eqref{eq:simm} such that
\begin{equation}\label{eq:convpol}
\norm{h_n - p_{k}^n}_{\rmC^1((\T^d)^n)} \to 0 \quad \text{ as } k \to \infty.
\end{equation}
Up to a unrelabeled subsequence, we can assume that $\|p_k^n\|_{C^1((\T^d)^n)} \le 2 \|h_n\|_\infty$ for every $k \in \N_+$. We have (again by e.g.,~\cite[proof of Theorem 4.4]{cossomartini}) that $w_n^k$ is L-differentiable at every $\mu \in \prob(\T^d)$ and, since $p_k^n$  satisfies \eqref{eq:simm}, it also holds
\[
    \rmD_L w_k^n (\mu,x) = n \int_{\T^{d(n-1)}} \nabla_n p_k^n (\cdot, x) \de \mu^{\otimes (n-1)}, \quad (\mu, x) \in \prob(\T^d) \times \T^d,
\]
where $\nabla_n$ denotes the gradient w.r.t.~the $n$-th variable. We clearly have $\|w_k^n\|_\infty \le 2\|u\|_\infty$ and 
\[ |\rmD_L w_k^n(\mu, x)| \le n \|\nabla_n p_k\|_{C(K^n)} \le 2n \|\nabla_n h_n\|_\infty \le 2\|\rmD_L v_n\|_\infty\le 2\|\rmD_L u\|_\infty , \]
thanks to the estimates obtained in claim 1.
This proves the bounds. Using the dominated convergence theorem and \eqref{eq:convpol} we deduce both the convergences. This concludes the proof of the second claim.
\\
\emph{Claim 3}. Conclusion.\\
\emph{Proof of claim 3}. Let $\mm$ be any Borel probability measure on $\prob(\T^d)$. By claim 1, we can find a sequence $(v_n)_n \subset \rmC^1(\prob(\R^d))$ of the form 
\[ v_n(\mu) = \int_{(\T^d)^n} h_n \de \mu^{\otimes n}, \quad \mu \in \prob(\T^d), \]
for functions $h_n \in \rmC^1((\T^d)^n)$ satisfying \eqref{eq:simm}, such that
\[ v_n \to u \text{ in } L^2(\prob(\T^d), \mm), \quad \rmD_L v_n \to \rmD_L u \text{ in } L^2(\prob(\T^d) \times \T^d, \bmm) \text{ as } n \to \infty.\]
By claim 2, we can find a sequence of polynomials $(p_{n,k})_k$ in $(\T^d)^n$ such that, defining
\[ u_{n,k}(\mu):= \int_{(\T^d)^n} p_{n,k} \de \mu^{\otimes n}, \quad \mu \in \prob(\T^d),\]
then $ u_{n,k}\in \mathfrak{Z}^\infty$, it is L-differentiable at every $\mu \in \prob(\T^d)$, $\| u_{n,k}\|_\infty \le 2\|u\|_\infty$, $\|\rmD_L u_{n,k}\|_\infty \le 2\|\rmD_L u\|_\infty$ and
\[ u_{n,k}(\mu) \to v_n(\mu), \quad \rmD_L u_{n,k} (\mu, x) \to \rmD_L v_n (\mu, x) \quad \text{ for every } (\mu, x) \in \prob(\T^d) \times \T^d\]
as $k \to \infty$. 
Moreover, by dominated convergence, $u_{n,k}\to v_n$ in $L^2(\prob(\T^d),\mm)$ and $\rmD_L u_{n,k}\to\rmD_L v_n$ in $L^2(\prob(\T^d)\times\T^d,\bmm)$ as $k\to\infty$. Then, a diagonal argument in $L^2(\prob(\T^d), \mm) \times L^2(\prob(\T^d) \times \R^d, \bmm)$ concludes the proof of the third claim.
\end{proof}

\begin{proposition} \label{prop:equiv_der}
Let $\cD$ be the Dirichlet--Ferguson measure on $\prob(\T^d)$. Then 
\[ \rmC^1(\prob(\T^d)) \subset H^{1,2}(\prob(\T^d), W_2, \cD)\]
and,
\[ \rmD_L u = \rmD u  \quad \text{ for every } u \in \rmC^1(\prob(\T^d)),\]
where $\rmD$ is as in Proposition \ref{prop:vectorgrad}.
\end{proposition}
\begin{proof} Since, by Remark \ref{rem:lips}, we have that $ \rmC^1(\prob(\T^d)) \subset \Lip_b(\prob(\T^d), W_2)$ and, by definition,  $\Lip_b(\prob(\T^d), W_2) \subset H^{1,2}(\prob(\T^d), W_2, \cD)$, we get the first inclusion in the statement. Let $u \in\rmC^1(\prob(\T^d))$ and recall that closability holds (see Subsection \ref{sec:geo}); by Proposition \ref{prop:approx}, we can find a sequence of cylinder functions $(u_n)_n\subset \mathfrak{Z}^\infty$ satisfying \eqref{eq:conv} for $u$. Let also  $(u_n')_n \subset \mathfrak{Z}^\infty$ be as in \eqref{eq:apprpr} and let us define $z_n:= u_n-u_n'$; we have that $z_n \to 0$ in $L^2(\prob(\T^d), \cD)$ and $\boldnabla z_n = \boldnabla u_n - \boldnabla u_n' \to \rmD_L u- \rmD_\mm u$ in $L^2(\prob(\T^d) \times \T^d, \overline{\cD}; \R^d)$. By definition of closability (cf.~\eqref{eq:closab} and \cite[Definition 5.6]{FSS22}), we deduce that it must be $\rmD_L u = \rmD u$. 
\end{proof}
\begin{remark}
    It is easy to see that the inclusion in Proposition \ref{prop:equiv_der} holds not only for $\cD$, but for every Borel probability measure $\mm$ on $\prob(\T^d)$. In particular  $\rmC^1(\prob(\T^d)) \subset H^{1,2}(\prob(\T^d), W_2, \mm)$, where a precise definition of $H^{1,2}(\prob(\T^d), W_2, \mm)$ can be found in \cite{FSS22}. On the other hand, the equivalence between $\rmD_Lu$ and $\rmD u$ is strongly related to the choice of $\cD$ and the closability of the quadratic form associated.
\end{remark}

\section{Explicit approximation 
of continuous functions on \texorpdfstring{${\mathcal P}(\T^d)$}{P}
}
\label{app:B}

The purpose of this section is to provide a quantitative and explicit method for approximating continuous functions on $\prob(\T^d)$ (or more generally, functions on $ [0,T] \times \prob(\T^d)\times \T^d$), in contrast 
with Proposition
\ref{prop:approx_trunc}
which relies on a non-explicit approach. 

Throughout, we make use of the equivalence 
relation introduced, for a given $\varepsilon \in (0,1]$, in Definition
\ref{def:epsilon:equivalence:proba:measures}, and of the related 
class $\AA_\epsilon^1$
of functions defined on the subsequent quotient space. The same discussion and the same results can be also done when $\AA_\epsilon^1$ is replaced by $\AA_\epsilon^m$ for any $m \in \N_+$, but, for the sake of simplicity, we work here with the simpler choice $m=1$. For
a continuous function $\fru\colon[0,T] \times\prob(\T^d)\times\T^d\to\R$, we look for a family $(\fru^\varepsilon)_{\varepsilon \in (0,1]}$
that  converges 
to $\fru$ in a \textit{suitable} sense as $\eps \downarrow 0$.
For a given $\eps \in (0,1]$, we define the mapping $T_\eps\colon\prob(\T^d)\to \prob^{pa}(\T^d)$, recalling that the latter is the subset of $\prob(\T^d)$ of purely atomic measures, by
\begin{equation*}
    T_\eps(\mu):=
    \begin{cases}
        \frac{1}{\mu(\{x\in\T^d\colon\mu_x>\eps\})} \mu\rvert_{\{x\in\T^d\colon\mu_x>\eps\}},\quad&\text{ if } \mu(\{x\in\T^d\colon\mu_x>\eps\})>0,\\
        \delta_0,\quad&\text{ if } \mu(\{x\in\T^d\colon\mu_x>\eps\})=0.
    \end{cases}
\end{equation*}
In words, 
$T_{\varepsilon}(\mu)$ is the conditional probability of $\mu$ given the subset of its atoms with mass greater than $\eps$.
Notice in particular that, if $\mu = \sum_{j=1}^\infty s_{j}\delta_{x_j}$, with $(x_j)_{j}\subset\T^d$
and $0 < s_{j+1} \le s_{j}$ for every $j \in\N_+$, then
$$T_\eps(\mu) = \frac{\sum_{j=1}^{N^{\mu,\eps}}s_{j}\delta_{x_j}}{\sum_{j=1}^{N^{\mu,\eps}}s_{j}},$$
under the additional notation $N^{\mu,\eps}:=\max\{j\in\N_+\colon s_{j}\geq \eps\}$.
Also, it is easy to see that if $\mu$ is purely atomic, then  $T_\eps(\mu)\to\mu$ weakly as $\eps\downarrow0$.
The rate of convergence can be computed explicitly, using the total variation distance, here denoted by 
$d_{\rm TV}$:
\begin{equation}
\label{eq:
dTV:Tepsilon:mu-mu}
\begin{split}
d_{\rm TV}
\left( T_\varepsilon(\mu),\mu \right) 
&= 
2 \sup_{B \in {\mathcal B}(\T^d)}
\left\vert 
[T_\varepsilon(\mu)](B) - \mu(B)
\right\vert
\\
&\leq 2 \sum_{j=1}^{N^{\mu,\varepsilon}}
\Bigl\vert s_j - \frac{s_j}{
\sum_{i=1}^{N^{\mu,\varepsilon}}
s_i
} \Bigr\vert+ 
2 \sum_{j \geq N^{\mu,\varepsilon}+1}
s_j
\leq 4 \sum_{j \geq N^{\mu,\varepsilon}+1}
s_j.
\end{split}
\end{equation}
Thus, given $\fru\colon [0,T] \times \prob(\T^d)\times\T^d\to\R$, we can define its approximation $\fru^\eps\colon [0,T] \times \prob(\T^d)\times\T^d\to\R$ as
\[\fru_{t}^\eps(\mu,x):=\fru_{t}(T_\eps(\mu),x),\quad (t, \mu,x)\in  [0,T] \times  \prob(\T^d)\times\T^d.\]
Notice that,
if $\mu$ and $\nu$ are $\eps$-compatible as in in Definition
\ref{def:epsilon:equivalence:proba:measures} for a given $\eps \in (0,1]$, then by definition of $T_\eps$, it holds $\fru_t^\eps(\mu,x) = \fru_t^\eps(\nu,x)$ for any $(t,x)\in [0,T] \times \T^d$. This implies that $\fru^\eps \in\AA_\eps^1$.
Moreover, by construction, we also have $\norm{\fru^\eps}_\infty\leq \norm{\fru}_\infty$ for any $\eps \in (0,1]$. The convergence of $\fru^\eps$ to $\fru$ is immediate since $\fru$ is continuous, and, by dominated convergence, the convergence also  holds
in $L^2([0,T] \times \prob(\T^d) \times \T^d,\mathcal{Q})$, for any Borel probability measure $\mathcal{Q}$ on $[0,T] \times \prob(\T^d) \times \T^d$ whose marginal on $\prob(\T^d)$ is concentrated on $\prob^{pa}(\T^d)$. 
The rate of convergence can be computed under ${\mathscr L}^{[0,T]} \otimes \overline{\mathcal D}$, using the modulus of continuity (in the $\mu$-variable) of $\fru$:
\begin{equation*}
\omega(\delta) : =
\sup_{t \in [0,T], x \in {\mathbb T}^d}
\sup_{\mu,\nu \in \prob^{pa}{(\T^d)} : 
d_{\textrm{\rm TV}
(\mu,\nu) \leq \delta}} \vert \fru(t,\mu,x) - 
\fru(t,\nu,x) \vert, 
\quad \delta \in [0,2]. 
\end{equation*}
Since $[0,T] \times \prob(\T^d) \times \T^d$
is compact and
$\fru$ is continuous
(on the product space, equipped with the product topology) and because the square of the $W_2$ distance (that metricizes the weak topology on $\prob(\T^d)$)
is dominated by the total variation distance, it holds that 
\begin{equation*}
\lim_{\delta \rightarrow 0} \omega(\delta) =0.
\end{equation*}
And then,  
\begin{equation*}
\int_0^T \int_{\prob(\T^d)}
\left[ \int_{\T^d}
\left\vert 
\fru^{\varepsilon}_t(\mu,x) - \fru_t(\mu,x) \right\vert^2 \de \mu(x) 
\right] \de {\mathcal D}(\mu)  \de t
\leq 
T \int_{\prob(\T^d)}
\omega^2 
\left( d_{\textrm{\rm TV}}(T_\varepsilon(\mu),\mu) 
\right) 
\de {\mathcal D}(\mu).
\end{equation*} 
Consider now any concave function 
$\vartheta : [0,2] \rightarrow {\mathbb R}$ that dominates 
$T \omega^2$ 
(e.g., if $\omega$ is linear, then $\vartheta$ is also linear; if $\omega$ is $1/2$-H\"older, then 
$\vartheta$ becomes linear). Deduce that 
\begin{equation}
\label{eq:modulus:continuity:in:app:B}
\int_0^T \int_{\prob(\T^d)}
\left[ \int_{\T^d}
\left\vert 
\fru^{\varepsilon}_t(\mu,x) - \fru_t(\mu,x) \right\vert^2 \de \mu(x) 
\right] \de {\mathcal D}(\mu)  \de t
\leq 
\vartheta 
\left( 
\int_{\prob(\T^d)}
d_{\textrm{\rm TV}}(T_\varepsilon(\mu),\mu)  
\de {\mathcal D}(\mu)\right).
\end{equation} 
The remaining question is thus to estimate 
$\int_{\prob(\T^d)}
d_{\textrm{\rm TV}}(T_\varepsilon(\mu),\mu)  
\de {\mathcal D}(\mu)$. To do so, consider a Poisson--Dirichlet distributed sequence $(S_i)_{i \geq 1}$, constructed on some probability space 
$(\Omega,{\mathcal F}, 
{\mathbb P})$. Back to 
\eqref{eq:
dTV:Tepsilon:mu-mu}, the right-hand side therein can be rewritten (up to a factor 4) as  
\begin{equation*}
{\mathbb E} \left[ 
\sum_{i \geq 1} S_i {\mathbf 1}_{\{ 
S_i > \varepsilon\}}
\right]. 
\end{equation*}
We use notion of biased-pick sampling, see for instance \cite[Subsection 2.4]{Feng}. We call
$I$ an ${\mathbb N}_+$-valued random variable such that ${\mathbb P}(\{I=i\} \vert (S_j)_j)=S_i$, for 
all $i \in {\mathbb N}_+$. The above expectation can be reformulated as
\begin{equation*}
{\mathbb E} \left[ 
\sum_{i \geq 1} S_i {\mathbf 1}_{\{ 
S_i < \varepsilon\}}
\right] 
= {\mathbb E}\left[ {\mathbf 1}_{\{S_I < \varepsilon\}}\right].
\end{equation*}
By \cite[Theorem 2.7]{Feng}, we know that the law of $S_I$ is exactly the law of $Y_1$ in 
Remark \ref{rem:Poisson-Dirichlet:rep}. 
Therefore, 
\begin{equation*}
{\mathbb E} \left[ 
\sum_{i \geq 1} S_i {\mathbf 1}_{\{ 
S_i < \varepsilon\}}
\right] 
= {\mathbb E}\left[ {\mathbf 1}_{\{S_I < \varepsilon\}}\right] = 
{\mathbb P}(\{Y_1 < \varepsilon\}) = \eps.
\end{equation*}
Back to 
\eqref{eq:modulus:continuity:in:app:B}, we obtain 
\begin{equation}
\label{eq:modulus:continuity:in:app:B:2}
\int_0^T \int_{\prob(\T^d)}
\left[ \int_{\T^d}
\left\vert 
\fru^{\varepsilon}_t(\mu,x) - \fru_t(\mu,x) \right\vert^2 \de \mu(x) 
\right] \de {\mathcal D}(\mu)  \de t
\leq 
\vartheta(4\eps),
\end{equation} 
where $\vartheta$ is extended to the interval $(2,\infty)$
by letting $\omega(r) = \omega(2)$ for $r >2$. 

\mysubsubsection{Acknowledgments.} Part of this research was carried out during a visit of G. Sodini to the Laboratoire J.~A.~Dieudonné in Nice, whose hospitality is gratefully acknowledged. This research was funded in part by the Austrian Science Fund (FWF) project \href{https://doi.org/10.55776/F65}{10.55776/F65}. F. Delarue and M. Martini acknowledge the financial support of the European Research Council (ERC) under the European Union’s Horizon Europe research and innovation program (ELISA project, Grant agreement No. 101054746). Views and opinions expressed are however those of the authors only and do not necessarily reflect those of the European Union or the European Research Council Executive Agency. Neither the European Union nor the granting authority can be held responsible for them. M. Martini was
a member of INdAM-GNAMPA when this work was conceived.

\end{document}